%% file: d.tex
\documentclass[12pt,oneside]{book}

\input{pre}

\begin{document}

\include{official_front_matter}

\pagenumbering{arabic}
\include{intro}
\include{prelims}

\include{enr_orth_factn}
\include{compl_cl_dens_orth}

\include{sm_closed_adj}
\include{fin_comm_alg_th}
\include{nat_acc_dist}
\include{dcompl_vint}

\bibliographystyle{amsplain}
\bibliography{d}

\end{document}

%% file: pre.tex
\usepackage{amsmath, amsthm, amssymb}
\usepackage[all]{xy}
\usepackage{fullpage}
\usepackage{titlesec}
\usepackage{mathrsfs}
\usepackage{amsbsy}
\usepackage{turnstile}
\usepackage{bbm}
\usepackage{yhmath}
\usepackage{tensor}

\usepackage{setspace}

\usepackage[pdftex,bookmarks=true]{hyperref}

\numberwithin{equation}{subsection}

\theoremstyle{plain}

\newtheorem{PropSub}[subsection]{Proposition}
\newtheorem{LemSub}[subsection]{Lemma}
\newtheorem{CorSub}[subsection]{Corollary}
\newtheorem{ThmSub}[subsection]{Theorem}

\newtheorem{PropSubSub}[subsubsection]{Proposition}

\newtheorem{ThmSubSub}[subsubsection]{Theorem}

\theoremstyle{definition}

\newtheorem{DefSub}[subsection]{Definition}
\newtheorem{ExaSub}[subsection]{Example}
\newtheorem{RemSub}[subsection]{Remark}
\newtheorem{ParSub}[subsection]{}

\newtheorem{DefSubSub}[subsubsection]{Definition}
\newtheorem{RemSubSub}[subsubsection]{Remark}
\newtheorem{ExaSubSub}[subsubsection]{Example}
\newtheorem{ParSubSub}[subsubsection]{}

\newcommand*{\emptybox}{\leavevmode\hbox{}}

\newcommand{\dint}[1]{\ensuremath{\aint{\text{\rlap{$*$}\hspace{0.6ex}}}{#1}}}

\newcommand{\pint}[1]{\ensuremath{\aint{\text{\rlap{\tiny P}\hspace{0.8ex}}}{#1}}}

\newcommand{\aint}[2]{\ensuremath{\tensor*[_{#1}]{\int}{_{#2}}}}

\newdir{ >}{{}*!/-10pt/@{>}}

\DeclareMathAlphabet{\mathpzc}{OT1}{pzc}{m}{it}
\DeclareMathAlphabet{\mathcalligra}{T1}{calligra}{m}{n}

\newcommand{\bref}[1]{\textnormal{\ref{#1}}}
\newcommand{\pbref}[1]{\textnormal{(\ref{#1})}}

\newcommand{\bD}{\ensuremath{D}}

\newcommand{\bH}{\ensuremath{H}}

\newcommand{\bJ}{\ensuremath{J}}
\newcommand{\bK}{\ensuremath{K}}

\newcommand{\bT}{\ensuremath{T}}

\newcommand{\A}{\ensuremath{\mathscr{A}}}
\newcommand{\B}{\ensuremath{\mathscr{B}}}
\newcommand{\C}{\ensuremath{\mathscr{C}}}

\newcommand{\E}{\ensuremath{\mathscr{E}}}
\newcommand{\F}{\ensuremath{\mathscr{F}}}

\newcommand{\sH}{\ensuremath{\mathscr{H}}}

\newcommand{\J}{\ensuremath{\mathscr{J}}}
\newcommand{\K}{\ensuremath{\mathscr{K}}}
\newcommand{\sL}{\ensuremath{\mathscr{L}}}
\newcommand{\M}{\ensuremath{\mathscr{M}}}
\newcommand{\N}{\ensuremath{\mathscr{N}}}
\newcommand{\sO}{\ensuremath{\mathscr{O}}}
\newcommand{\sS}{\ensuremath{\mathscr{S}}}
\newcommand{\T}{\ensuremath{\mathscr{T}}}
\newcommand{\U}{\ensuremath{\mathscr{U}}}
\newcommand{\V}{\ensuremath{\mathscr{V}}}
\newcommand{\W}{\ensuremath{\mathscr{W}}}
\newcommand{\X}{\ensuremath{\mathscr{X}}}
\newcommand{\Y}{\ensuremath{\mathscr{Y}}}

\newcommand{\tL}{\ensuremath{\widetilde{\mathscr{L}}}}

\newcommand{\Pett}[2]{\ensuremath{#1_{\textnormal{P}(#2)}}}
\newcommand{\pL}[1]{\ensuremath{\Pett{\sL}{#1}}}

\newcommand{\eA}{\ensuremath{\mathscr{A}}}
\newcommand{\eB}{\ensuremath{\mathscr{B}}}

\newcommand{\eL}{\ensuremath{\mathscr{L}}}

\newcommand{\teL}{\ensuremath{\widetilde{\mathscr{L}}}}

\newcommand{\aeX}{\ensuremath{\underline{\mathscr{X}}}}
\newcommand{\aeL}{\ensuremath{\underline{\mathscr{L}}}}
\newcommand{\aeU}{\ensuremath{\underline{\mathscr{U}}}}
\newcommand{\aeV}{\ensuremath{\underline{\mathscr{V}}}}
\newcommand{\aeW}{\ensuremath{\underline{\mathscr{W}}}}

\newcommand{\aetL}{\ensuremath{\widetilde{\underline{\mathscr{L}}}}}

\newcommand{\uX}{\ensuremath{\underline{\mathscr{X}}}}
\newcommand{\uL}{\ensuremath{\underline{\mathscr{L}}}}
\newcommand{\uU}{\ensuremath{\underline{\mathscr{U}}}}
\newcommand{\uV}{\ensuremath{\underline{\mathscr{V}}}}
\newcommand{\uW}{\ensuremath{\underline{\mathscr{W}}}}

\newcommand{\tdelta}{\ensuremath{\widetilde{\delta}}}
\newcommand{\tkappa}{\ensuremath{\widetilde{\kappa}}}
\newcommand{\tDelta}{\ensuremath{\widetilde{\Delta}}}
\newcommand{\txi}{\ensuremath{\widetilde{\xi}}}

\newcommand{\CAT}{\ensuremath{\operatorname{\textnormal{\textbf{CAT}}}}}
\newcommand{\eCAT}[1]{\ensuremath{#1\textnormal{-\textbf{CAT}}}}
\newcommand{\VCAT}{\ensuremath{\V\textnormal{-\textbf{CAT}}}}
\newcommand{\WCAT}{\ensuremath{\W\textnormal{-\textbf{CAT}}}}
\newcommand{\XCAT}{\ensuremath{\X\textnormal{-\textbf{CAT}}}}
\newcommand{\LCAT}{\ensuremath{\sL\textnormal{-\textbf{CAT}}}}

\newcommand{\MMCCAATT}{\ensuremath{\operatorname{\textnormal{\textbf{MCAT}}}}}
\newcommand{\SSMMCCAATT}{\ensuremath{\operatorname{\textnormal{\textbf{SMCAT}}}}}
\newcommand{\CCllSSMMCCAATT}{\ensuremath{\operatorname{\textnormal{\textbf{CSMCAT}}}}}
\newcommand{\CCllCCAATT}{\ensuremath{\operatorname{\textnormal{\textbf{CCAT}}}}}
\newcommand{\TWOCCAATT}{\ensuremath{\operatorname{\textnormal{\textbf{2CAT}}}}}

\newcommand{\CC}{\ensuremath{\mathbb{C}}}
\newcommand{\DD}{\ensuremath{\mathbb{D}}}
\newcommand{\tDD}{\ensuremath{\widetilde{\mathbb{D}}}}
\newcommand{\dDD}{\ensuremath{\wideparen{\mathbb{D}}}}
\newcommand{\HH}{\ensuremath{\mathbb{H}}}
\newcommand{\tHH}{\ensuremath{\widetilde{\mathbb{H}}}}
\newcommand{\dHH}{\ensuremath{\wideparen{\mathbb{H}}}}
\newcommand{\LL}{\ensuremath{\mathbb{L}}}
\newcommand{\MM}{\ensuremath{\mathbb{M}}}
\newcommand{\NN}{\ensuremath{\mathbb{N}}}

\newcommand{\RR}{\ensuremath{\mathbb{R}}}
\newcommand{\SSS}{\ensuremath{\mathbb{S}}}
\newcommand{\TT}{\ensuremath{\mathbb{T}}}
\newcommand{\tTT}{\ensuremath{\widetilde{\mathbb{T}}}}

\newcommand{\ONEONE}{\ensuremath{\mathbbm{1}}}

\newcommand{\tpartial}{\ensuremath{\widetilde{\fd}}}
\newcommand{\tgamma}{\ensuremath{\widetilde{\fk}}}

\newcommand{\fd}{\ensuremath{\mathfrak{d}}}
\newcommand{\fk}{\ensuremath{\mathfrak{k}}}
\newcommand{\tfd}{\ensuremath{\widetilde{\fd}}}

\newcommand{\dfd}{\ensuremath{\wideparen{\fd}}}
\newcommand{\dfk}{\ensuremath{\wideparen{\fk}}}

\newcommand{\tD}{\ensuremath{\widetilde{D}}}
\newcommand{\tH}{\ensuremath{\widetilde{H}}}
\newcommand{\tT}{\ensuremath{\widetilde{T}}}

\newcommand{\dD}{\ensuremath{\wideparen{D}}}
\newcommand{\dH}{\ensuremath{\wideparen{H}}}
\newcommand{\dJ}{\ensuremath{\wideparen{J}}}
\newcommand{\dK}{\ensuremath{\wideparen{K}}}

\newcommand{\ob}{\ensuremath{\operatorname{\textnormal{\textsf{Ob}}}}}
\newcommand{\mor}{\ensuremath{\operatorname{\textnormal{\textsf{Mor}}}}}
\newcommand{\Ob}{\ob}
\newcommand{\Mor}{\mor}
\newcommand{\Iso}{\ensuremath{\operatorname{\textnormal{\textsf{Iso}}}}}
\newcommand{\Epi}{\ensuremath{\operatorname{\textnormal{\textsf{Epi}}}}}
\newcommand{\Mono}{\ensuremath{\operatorname{\textnormal{\textsf{Mono}}}}}
\newcommand{\StrMono}{\ensuremath{\operatorname{\textnormal{\textsf{StrMono}}}}}
\newcommand{\Sub}{\ensuremath{\operatorname{\textnormal{\textsf{Sub}}}}}
\newcommand{\SigmaClSub}{\ensuremath{\Sigma\textnormal{-\textsf{ClSub}}}}

\newcommand{\Refl}{\ensuremath{\operatorname{\textnormal{\textsf{Refl}}}}}
\newcommand{\IdmMnd}{\ensuremath{\operatorname{\textnormal{\textsf{IdmMnd}}}}}
\newcommand{\Lan}{\ensuremath{\operatorname{\textnormal{\textsf{Lan}}}}}

\newcommand{\Dense}[1]{\ensuremath{#1\textnormal{-\textsf{Dense}}}}
\newcommand{\ClEmb}[1]{\ensuremath{#1\textnormal{-\textsf{ClEmb}}}}

\newcommand{\SigmaDense}{\Dense{\Sigma}}
\newcommand{\SigmaClEmb}{\ClEmb{\Sigma}}

\newcommand{\Ev}{\ensuremath{\textnormal{\textsf{Ev}}}}

\newcommand{\y}{\ensuremath{\textnormal{\textsf{y}}}}

\newcommand{\Set}{\ensuremath{\operatorname{\textnormal{\textbf{Set}}}}}
\newcommand{\Two}{\ensuremath{\operatorname{\textnormal{\textbf{2}}}}}
\newcommand{\SET}{\ensuremath{\operatorname{\textnormal{\textbf{SET}}}}}

\newcommand{\RMod}{\ensuremath{R\textnormal{-\textbf{Mod}}}}
\newcommand{\Top}{\ensuremath{\operatorname{\textnormal{\textbf{Top}}}}}
\newcommand{\Conv}{\ensuremath{\operatorname{\textnormal{\textbf{Conv}}}}}
\newcommand{\Smooth}{\ensuremath{\operatorname{\textnormal{\textbf{Smooth}}}}}
\newcommand{\Adj}{\ensuremath{\operatorname{\textnormal{\textbf{Adj}}}}}
\newcommand{\Mnd}{\ensuremath{\operatorname{\textnormal{\textbf{Mnd}}}}}
\newcommand{\Alg}[1]{\ensuremath{#1\operatorname{\textnormal{-\textbf{Alg}}}}}

\newcommand{\subs}{\ensuremath{\subseteq}}
\newcommand{\sups}{\ensuremath{\supseteq}}

\newcommand{\lt}{\leqslant}

\newcommand{\op}{\ensuremath{\textnormal{op}}}

\newcommand{\btimes}{\ensuremath{\boxtimes}}

\newcommand{\pushoutcorner}{\ar@{}[dr]|(.3)\ulcorner}
\newcommand{\pullbackcorner}{\ar@{}[dr]|(.3)\lrcorner}

\newcommand{\widegrave}[1]{\grave{\wideparen{#1}}}
\newcommand{\wideacute}[1]{\acute{\wideparen{#1}}}

%% file: official_front_matter.tex
\pagenumbering{roman}
\pagestyle{empty} 
\onehalfspacing
\begin{center}
\noindent
\begin{tabular}[textwidth]{@{}c@{}}
\\
\\
\\
\\
\\
RIESZ-SCHWARTZ EXTENSIVE QUANTITIES AND\\
VECTOR-VALUED INTEGRATION IN CLOSED CATEGORIES\\
\\
\\
\\
\\
\\
RORY B. B. LUCYSHYN-WRIGHT\\
\\
\\
\\
\\
A DISSERTATION SUBMITTED TO\\
THE FACULTY OF GRADUATE STUDIES\\
IN PARTIAL FULFILMENT OF THE REQUIREMENTS\\
FOR THE DEGREE OF\\
DOCTOR OF PHILOSOPHY\\
\\
\\
GRADUATE PROGRAM IN MATHEMATICS AND STATISTICS\\
YORK UNIVERSITY,\\
TORONTO, ONTARIO\\
\\
JUNE 2013\\
\\
\\
\copyright\ Rory B. B. Lucyshyn-Wright, 2013
\end{tabular}
\end{center}
\singlespacing
\newpage
\pagestyle{plain}
\section*{Abstract}
We develop aspects of functional analysis in an abstract axiomatic setting, through monoidal and enriched category theory.  We work in a given closed category, whose objects we call \textit{spaces}, and we study \hbox{$R$-module} objects therein (or algebras of a commutative monad), which we call \textit{linear spaces}.  Building on ideas of Lawvere and Kock, we study functionals on the space of \hbox{scalar-valued} maps, including \hbox{compactly-supported} Radon measures and Schwartz distributions.  We develop an abstract theory of \hbox{vector-valued} integration with respect to these scalar functionals and their relatives.  We study three axiomatic approaches to vector integration, including an abstract \hbox{Pettis-type} integral, showing that all are encompassed by an axiomatization via monads and that all coincide in suitable contexts.  We study the relation of this vector integration to relative notions of completeness in linear spaces.  One such notion of completeness, defined via enriched orthogonality, determines a symmetric monoidal closed reflective subcategory consisting of exactly those separated linear spaces that support the vector integral.  We prove \hbox{Fubini-type} theorems for the vector integral.  Further, we develop aspects of several supporting topics in category theory, including enriched orthogonality and factorization systems, enriched associated idempotent monads and adjoint factorization, symmetric monoidal adjunctions and commutative monads, and enriched commutative algebraic theories.
\newpage
\onehalfspacing
\emptybox
\begin{center}
\noindent
\begin{tabular}[textwidth]{@{}c@{}}
\\
\\
\\
\\
\\
\\
\\
\\
\\
\\
\textit{To Anna}
\end{tabular}
\end{center}
\singlespacing
\newpage
\section*{Acknowledgments}
I would like to thank Alexander Nenashev for his investment of time in reading this dissertation as member of the supervisory committee, and also for his contributions to my education in several subjects in mathematics.  I am grateful to Peter Gibson both for his enthusiasm and for his expenditure of time in reading this dissertation as member of both the supervisory committee and the examining committee.  I would like to thank Nantel Bergeron and Jeff Edmonds for their expenditure of time in reading this dissertation as members of the examining committee.  I am grateful to F. William Lawvere not only for his contribution of time in serving as external examiner and reading this dissertation, but also for having in his own work laid much of the conceptual foundation for this work.  I am grateful to Walter Tholen, my supervisor, for his support and encouragement, for his enthusiasm, for his contribution of time in reading this dissertation, and for providing me with the freedom to study and explore according to my fancy.  I also thank Prof. Tholen for suggesting at an early point in this research that I read the Cassidy-H\'ebert-Kelly paper, which proved to provide a crucial component of the methodology of this dissertation when taken in conjunction with related work of Day.  Most importantly, I am grateful to Anna Frey for her steadfast support and encouragement.
\newpage
\tableofcontents

%% file: intro.tex
\chapter{Introduction}
\setcounter{subsection}{0}

The present text develops certain aspects of functional analysis in an abstract setting through \textit{monoidal} and \textit{enriched} category theory.  In such an approach, similar concepts and results in different contexts can be treated as instances of the same abstract phenomenon and studied by the same methods.  For example, compactly-supported Radon measures and Schwartz distributions are both instances of the notion of \textit{natural distribution} studied herein.  Further, a structural approach to the subject yields definitions and results of a generality and canonicity that would seem unavailable in the familiar approach through Banach spaces and topological vector spaces.  

In Chapters \bref{ch:nat_acc_dist} and \bref{ch:dcompl_vint} we study natural distributions and their kin, as well as vector-valued integration with respect to them, examining in particular the relation between this vector integration and notions of completeness in linear spaces.  In the preceding chapters, we first develop several topics in category theory that provide the foundation for our abstract functional analysis, which is inspired in part by ideas of Lawvere and Kock \cite{Law:CatsSpQu,Law:VoFu,Law:AxEd,Kock:ProResSynthFuncAn,Kock:Dist}.  Certain of our results in Chapters \bref{ch:nat_acc_dist} and \bref{ch:dcompl_vint} shed light on long-standing problems posed by these authors \cite{Kock:ProResSynthFuncAn}.  Each chapter includes an introduction that summarizes its content.

In the terminology of Lawvere \cite{Law:CatsSpQu,Law:AxEd}, natural distributions are certain \textit{extensive quantities}, namely those formed according to the Riesz-Schwartz dualization paradigm from \textit{representable} intensive quantities.

\section{Linear spaces}

In the category $\RMod$ of $R$-modules for a commutative ring $R$, and in particular for vector spaces, one can form a canonical $R$-module of linear maps that is in adjoint relation to the usual tensor product (see \ref{eqn:tens_hom}), and as a result $\RMod$ is \textit{symmetric monoidal closed}.  In the category of Banach spaces and bounded linear maps, one has a canonical Banach space of bounded linear maps that is again in adjoint relation to the projective tensor (which represents bounded bilinear maps).  But upon moving to the category of topological vector spaces, this canonical symmetric monoidal closed structure is lost.  This owes to a deficiency in the category $\Top$ of topological spaces itself, as one cannot in general form in $\Top$ a canonical space of continuous maps in adjoint relation to the \textit{cartesian} product, so that unlike the category $\Set$ of sets, $\Top$ fails to be \textit{cartesian closed}.  When one instead considers $R$-module objects in any cartesian closed category $\X$ with enough limits and colimits, the desired tensor and `hom' familiar from the $\Set$-based case are regained and $\RMod(\X)$ is symmetric monoidal closed \pbref{exa:rmod_smclosed_monadic}.  Several authors have developed aspects of functional analysis in a variety of specific cartesian closed categories, including \textit{compactly generated spaces}, \textit{bornological spaces}, \textit{smooth spaces}, and several flavours of \textit{convergence spaces} (e.g., in \cite{Fro:CpctGenSpDu,HN,FroBu,BeBu,Fro:SmthStr}).

In the present text, we instead work with an arbitrary given cartesian closed category $\X$, whose objects and morphisms we call simply \textit{spaces} and \textit{maps}, and study on an axiomatic basis the category $\sL = \RMod(\X)$ of $R$-module objects for a commutative ring $R$ in $\X$.  We call the objects and morphisms of $\sL$ \textit{linear spaces} and \textit{linear maps}.  More generally, we allow $\X$ to be any suitable \textit{symmetric monoidal} closed category and $\sL$ to be the category of algebras of any suitable \textit{symmetric monoidal monad} or \textit{commutative monad} $\LL$ on $\X$; $\sL$ is then symmetric monoidal closed, and we let $R$ be the unit object in $\sL$.  A similar general setting was employed in recent work of Kock \cite{Kock:Dist}, who had much earlier introduced the fundamental notion of commutative monad \cite{Kock:BilCartClMnds, Kock:ClsdCatsGenCommMnds,Kock:Comm,Kock:SmComm}.

In order to illustrate and apply our work, we make repeated reference to two specific examples of cartesian closed categories $\X$, namely the category $\Conv$ of Fischer-Kowalsky \textit{Limesr\"aume} or \textit{convergence spaces} (see \cite{BeBu}) and the category $\Smooth$ of Fr\"olicher's \textit{smooth spaces} \cite{Fro:SmthStr,Fro:CccAnSmthMaps,FroKr}.  The former contains $\Top$ (as a full subcategory), and the latter contains the category of paracompact smooth manifolds.

\section{Completeness}

In developing functional analysis in the above abstract setting, we no longer have at our disposal the usual notions \textit{completeness} of normed or topological vector spaces.  Instead, we define suitable notions of completeness and completion of linear spaces with reference to only the given abstract categorical data.  Whereas the endofunctor on $\sL$ sending each linear space $E$ to its \textit{double-dual} $E^{**}$ is part of an $\sL$-enriched monad on $\sL$, we form the associated $\sL$-enriched \textit{idempotent monad}, which determines a symmetric monoidal closed reflective subcategory of $\sL$ whose objects we call \textit{functionally complete} linear spaces.  A study of \textit{enriched factorization systems} and \textit{orthogonality} \pbref{ch:enr_orth_fact} forms the basis for our use of the induced idempotent monad \pbref{sec:assoc_idm_mnd}, and the same techniques also allow us to define further notions of completeness adapted to specific purposes \pbref{ch:compl_enr_orth}.  Each comes with accompanying notions of \textit{completion}, \textit{closure}, and \textit{density}.

\section{Natural and accessible distributions}

For each space $X \in \X$ there is an associated linear space $[X,R]$ whose elements are the maps $f:X \rightarrow R$.  We call linear maps $\mu:[X,R] \rightarrow R$ \textit{natural distributions} on $X$.  For each map $f:X \rightarrow R$ we define $\int f\;d\mu := \mu(f)$.  For example, when $\X = \Conv$ and $X$ is a locally compact Hausdorff space, natural distributions on $X$ are $R$-valued Radon measures of compact support, where $R = \RR$ or $\CC$.  When $\X = \Smooth$ and $X$ is a separable paracompact smooth manifold, natural distributions on $X$ are Schwartz distributions of compact support on $X$.  We can form the linear space $DX = \sL([X,R],R) = [X,R]^*$ of natural distributions on any given space $X$, and the resulting endofunctor $D$ on $\X$ is part of an $\X$-enriched monad $\DD$ on $\X$.  With categories of smooth spaces in mind, Kock has called $\DD$ the \textit{Schwartz double-dualization monad} \cite{Kock:ProResSynthFuncAn,Kock:Dist}.

Whereas to each space $X$ there is associated a \textit{free} linear space $FX$ on $X$, we observe that
$$DX = [X,R]^* \cong (FX)^{**}$$
so that the space of natural distributions is the double-dual of the free span of $X$.  We define the linear space of \textit{accessible distributions} $\tD X$ as instead the \textit{functional completion} of $FX$, and it follows that $\tD X$ embeds as a subspace of $DX$.  We obtain a submonad $\tDD$ of $\DD$.

We prove that if each linear space $[X,R]$ is \textit{reflexive} (i.e., $[X,R] \cong [X,R]^{**}$ via the canonical map) then $\tDD \cong \DD$, so that accessible distributions are the same as natural distributions.  By a result of Butzmann \cite{Bu}, this is applicable in particular in the example with $\X = \Conv$.

We show that $\tDD$ is always \textit{commutative}, and it follows that integration with respect to accessible distributions satisfies a Fubini-type theorem \pbref{thm:fub_for_acc_distns}.  In the case of convergence spaces, where $\tDD \cong \DD$, this yields a Fubini theorem for natural distributions that generalizes a classical Fubini theorem for compactly-supported Radon measures.

\section{Vector-valued integration}

We develop an abstract theory of vector-valued integration with respect to natural distributions and their relatives, reconciling three distinct axiomatic approaches to vector integration and studying their relation to the above notions of completeness of linear spaces.  We first formulate a notion of \textit{abstract distribution monad} $(\MM,\Delta,\xi)$ that encompasses the natural and accessible distribution monads $\DD$ and $\tDD$.  

Relative to a given abstract distribution monad $\MM$, we define a notion of \textit{Pettis-type integral} that accords with Bourbaki's adaptation of the 1938 vector integral of Pettis \cite{Bou,Pett}.  We define a property of linear spaces $E$ tantamount to the existence of \textit{all} Pettis-type integrals $\pint{} f\;d\mu$ in $E$, saying that $E$ is \textit{$\MM$-Pettis} if it has this property.

We consider also an approach inspired by work of Schwartz \cite{Schw} and Waelbroeck \cite{Wae} that was formulated for $\DD$ in the setting of ringed toposes by Lawvere and Kock \cite{Kock:ProResSynthFuncAn}.  Given a linear space $E$, one may ask whether for each space $X$ the Dirac map $\delta_X:X \rightarrow MX$ induces an isomorphism $\X(X,E) \cong \sL(MX,E)$ in $\X$.  If so then we say that $E$ is \textit{$\MM$-distributionally complete}, and we define the integral $\int f\;d\mu$ of each map $f:X \rightarrow E$ with respect to $\mu \in MX$ as the result of applying the unique linear extension $f^\sharp:MX \rightarrow E$ of $f$ to $\mu$.  

We prove that the separated $\tDD$-distributionally complete linear spaces are exactly the \hbox{$\tDD$-Pettis} spaces, and further that they form a symmetric monoidal closed reflective subcategory of $\sL$.

We show that both $\MM$-Pettis and also $\MM$-distributionally complete linear spaces are \hbox{\textit{$\MM$-algebras}}, and further that the associated vector integrals can be expressed in terms of the $\MM$-algebra structure.  For an \textit{arbitrary} $\MM$-algebra $(Z,a)$, we define
\begin{equation}\label{eqn:int_in_malg_intro}\int f\;d\mu := (a \cdot Mf)(\mu)\end{equation}
for each map $f:X \rightarrow Z$ and each $\mu \in MX$, thus generalizing both the Pettis- and Schwartz-Waelbroeck-type integrals above.  Via a monad morphism $\Delta:\LL \rightarrow \MM$, each $\MM$-algebra carries the structure of a linear space, so this too is a \textit{vector-valued} integral.  The integral notation \eqref{eqn:int_in_malg_intro} for algebras of a monad was defined independently by Kock \cite{Kock:Dist} and the author \cite{Lu:AlgThVectInt}.  In the latter paper, we showed in the measure-theoretic context that the resulting notion of vector integration includes the classical Pettis integral for bounded functions.

We prove that the \textit{linearity} of the vector integral in $\MM$-algebras is tantamount to the equality of two canonical transformations $\otimes^\Delta,\widetilde{\otimes}^\Delta$, generalizing the notion of commutative monad.  Calling $\MM$ a \textit{linear} abstract distribution monad if it has this property, we show that both the natural and accessible distribution monads $\DD$ and $\tDD$ are linear.

We show that if $\MM$ is commutative then the associated vector integral satisfies a Fubini-type theorem.  Since $\tDD$ is always commutative, this yields a vector Fubini theorem for accessible distributions.  In particular, when the cotensors $[X,R]$ are reflexive, as in the example of convergence spaces, we obtain a vector Fubini theorem for natural distributions.

We show that for accessible distributions, the general vector integral in $\tDD$-algebras only \textit{slightly} generalizes the Pettis- and Schwartz-Waelbroeck-type integrals (which for separated linear spaces coincide in this case).  In particular, whereas the category of $\tDD$-Pettis spaces embeds as a full reflective subcategory of the category of $\tDD$-algebras, an arbitrary $\tDD$-algebra arises from a $\tDD$-Pettis space as soon as its underlying linear space is \textit{separated}.  One thus obtains a coincidence of all three notions of vector integral in separated linear spaces, and a resulting symmetric monoidal closed reflective subcategory of $\sL$ consisting of exactly those separated linear spaces that support the integral.  For the example of convergence spaces, where $\tDD \cong \DD$, these results apply to arbitrary natural distributions, and compactly-supported Radon measures in particular. 

%% file: prelims.tex
\chapter{Preliminaries} \label{ch:prelims}
\setcounter{subsection}{0}

In the present chapter, we survey certain well-known basic notions and results in category theory that shall be required for the sequel.

\section{Stability and cancellation for cartesian arrows and squares} \label{sec:stab_canc_cart_arr_sq}

In the present section, we record some well-known closure properties of the class of \textit{cartesian arrows} of a functor, and so in particular, of pullback squares.  The omitted proofs are easy and elementary exercises.  We also note in passing that the cartesian arrows are part of a generalized prefactorization system of the sort considered in \cite{Th:FactConesFunc}, which we might call a \textit{prefactorization system for cones along a functor}, and the given closure properties may alternatively be proved on this basis.

\begin{DefSub}
Let $P:\A \rightarrow \Y$ be a functor.
\begin{enumerate}
\item Given morphisms $f':A' \rightarrow B$ and $f:A \rightarrow B$ in $\A$ and a morphism $v:PA' \rightarrow PA$ in $\Y$ with $Pf \cdot v = Pf'$, we call the triple $(f',f,v)$ a \textit{$P$-lifting problem}, and we say that a morphism $u:A' \rightarrow A$ is a is a \textit{P-lift} for $(f',f,v)$ if $Pu = v$ and $f \cdot u = f'$.
\item We say that a morphism $f:A \rightarrow B$ in $\A$ is \textit{$P$-cartesian} if there exists a unique $P$-lift for every $P$-lifting problem of the form $(f',f,v)$.
\end{enumerate}
\end{DefSub}

\begin{ExaSub} \label{exa:pb_cart}
Given a category $\B$, functors from the two-element chain $\Two$ to $\B$ are simply arrows in $\B$, whereas morphisms $(f,f'):a \rightarrow b$ in the functor category $[\Two,\B]$ are commutative squares
\begin{equation}\label{eq:comm_sq}
\xymatrix{
A' \ar[r]^{f'} \ar[d]_a & B' \ar[d]_b \\
A \ar[r]^{f}           & B 
}
\end{equation}
in $\B$.  There is an obvious functor $P:[\Two,\B] \rightarrow \B$ sending an arrow in $\B$ to its codomain, and the $P$-cartesian morphisms in $[\Two,\B]$ are exactly the pullback squares of $\B$.
\end{ExaSub}

\begin{PropSub} \label{thm:cart_iso_comp_canc}
Let $P:\A \rightarrow \Y$ be a functor, and let $f:A \rightarrow B$, $g:B \rightarrow C$ in $\A$.  
\begin{enumerate}
\item Isomorphisms in $\A$ are $P$-cartesian.
\item If $f$ and $g$ are $P$-cartesian, then $g \cdot f$ is $P$-cartesian.
\item If both $g \cdot f$ and $g$ are $P$-cartesian, then $f$ is $P$-cartesian.
\end{enumerate}
\end{PropSub}

\begin{CorSub}[Pullback Cancellation and Composition] \label{prop:pb_canc}
Suppose given a commutative diagram
$$
\xymatrix{
A_1 \ar[r]^{f_1} \ar[d]_a &  B_1 \ar[r]^{g_1} \ar[d]_b &  C_1 \ar[d]_c\\
A_2 \ar[r]^{f_2}          &  B_2 \ar[r]^{g_2}          &  C_2
}
$$
in a category $\C$.
\begin{enumerate}
\item If both the left and right squares are pullbacks, then the outer rectangle is a pullback.
\item If both the outer rectangle and the rightmost square are both pullbacks, then the leftmost square is a pullback.
\end{enumerate}
Further, if a commutative square has a pair of opposites sides each of which is an isomorphism, then the square is a pullback.
\end{CorSub}

\begin{CorSub} \label{prop:pb_cube}
Suppose given a commutative cube
$$
\xymatrix{
                                                         & \cdot \ar[rr] \ar@{.>}[dd] \ar[dl] &                              & \cdot \ar[dd] \ar[dl] \\ 
\cdot \ar[rr] \ar[dd] &                                                         & \cdot \ar[dd] &                                  \\
                                                         & \cdot \ar@{.>}[rr] \ar@{.>}[dl]          &                                  & \cdot \ar[dl]                  \\
\cdot \ar[rr]                                      &                                                 & \cdot                           &                                  
}
$$
in a category, and suppose that the back, front, and right faces are pullbacks.  Then the left face is a pullback.
\end{CorSub}
\begin{proof}
The commutative rectangle obtained by composing the back face and the right face is a pullback.  But this composite rectangle is equally the composite of the left face and the front face, so the latter composite is a pullback.  But the front face is also a pullback, so by pullback cancellation \pbref{prop:pb_canc}, the left face is a pullback.
\end{proof}

\begin{PropSub} \label{prop:cart_fp}
Let $P:\A \rightarrow \Y$ be a functor.  Suppose that $f:A \rightarrow B$ is a fiber product in $\A$ of a family $(f_i:A_i \rightarrow B)_{i \in I}$ of $P$-cartesian morphisms, indexed by a class $I$, and suppose that this fiber product is preserved by $P$.  Then $f$ is $P$-cartesian.
\end{PropSub}

\begin{CorSub} \label{cor:fp_pb}
Suppose we are given a commutative square in a category $\B$ as in \eqref{eq:comm_sq} and, for each element $i$ of some class $I$, a commutative triangular prism
$$
\xymatrix{
A' \ar[rr]^{f'} \ar[dr]_{\pi'_i} \ar[dd]_a &                                   & B' \ar[dd]^b \\
                                           & A'_i \ar[ur]_{f'_i} \ar[dd]_(.3){a_i} &             \\
A \ar@{.>}[rr]^(0.7)f \ar[dr]_{\pi_i}           &                                   & B           \\
                                           & A_i \ar[ur]_{f_i}                 &
}
$$
whose right face is a pullback.  Suppose that the top and bottom faces ($i \in I$) present $f'$ and $f$, respectively, as fiber products.  Then the given square is a pullback.
\end{CorSub}
\begin{proof}
The morphisms $(\pi_i,\pi'_i):a \rightarrow a_i$ in $[\Two,\B]$ $(i \in I)$ present $(f,f'):a \rightarrow b$ as a fiber product of the morphisms $(f_i,f'_i):a_i \rightarrow b$, and this fiber product is preserved by the codomain functor $P:[\Two,\B] \rightarrow \B$.  Hence, since each morphism $(f_i,f'_i):a_i \rightarrow b$ is $P$-cartesian, we deduce by \bref{prop:cart_fp} that the fiber product $(f,f'):a \rightarrow b$ is $P$-cartesian.
\end{proof}

\begin{PropSub}\label{prop:pb_canc_mono}
Suppose given a commutative diagram
$$
\xymatrix{
A_1 \ar[r]^{f_1} \ar[d]_a &  B_1 \ar[r]^{g_1} \ar[d]_b &  C_1 \ar[d]_c\\
A_2 \ar[r]^{f_2}          &  B_2 \ar[r]^{g_2}          &  C_2
}
$$
in a category $\C$, and suppose that the outer rectangle is a pullback and $b$ is mono.  Then the leftmost square is a pullback.
\end{PropSub}
\begin{proof}
Suppose $x:D \rightarrow A_2$, $y:D \rightarrow B_1$ have $f_2 \cdot x = b \cdot y$.  Then $(g_2 \cdot  f_2) \cdot x = g_2 \cdot b \cdot y = c \cdot (g_1 \cdot y)$, so since the outer rectangle is a pullback, there is a unique $w:D \rightarrow A_1$ such that $a \cdot w = x$ and $(g_1 \cdot f_1) \cdot w = g_1 \cdot y$.  But for any morphism $w:D \rightarrow A_1$, if $a \cdot w = x$ then it automatically follows that $f_1 \cdot w = y$, since $b \cdot f_1 \cdot w = f_2 \cdot a \cdot w = f_2 \cdot x = b \cdot y$.  The result follows. 
\end{proof}

\begin{PropSub}
Let $P:\A \rightarrow \Y$ be a functor, and let $f:A \rightarrow B$ in $\A$ be such that $Pf$ is iso.  Then $f$ is $P$-cartesian iff $f$ is iso.
\end{PropSub}

\begin{CorSub} \label{prop:pb_iso}
Suppose given a commutative square
$$
\xymatrix{
A \ar[r]^{f'} \ar[d]_a & B' \ar[d]_b \\
A \ar[r]^{f}           & B 
}
$$
in which $f$ is iso.  Then the square is a pullback if and only if $f'$ is iso.
\end{CorSub}

\section{2-categories} \label{sec:two_cats}

We shall require the following basic notions and results in the 2-categorical context.

\begin{DefSub}
Let $\K$ be a 2-category.
\begin{enumerate}
\item An \textit{adjunction} $f \nsststile{\varepsilon}{\eta} g:B \rightarrow A$ in $\K$ consists of 1-cells $g:B \rightarrow A$, $f:A \rightarrow B$ together with 2-cells $\eta:1_A \rightarrow gf$, $\varepsilon:fg \rightarrow 1_B$ such that $g\varepsilon \cdot \eta g = 1_g$ and $\varepsilon f \cdot f \eta = 1_f$.
\item Given an object $A$ in $\K$, a \textit{monad} $\TT = (t,\eta,\mu)$ on $A$ in $\K$ consists of a 1-cell $t:A \rightarrow A$ together with 2-cells $\eta:1_A \rightarrow t$, $\mu:tt \rightarrow t$ such that $\mu \cdot t\eta = 1_t = \mu \cdot \eta t$ and $\mu \cdot \mu t = \mu \cdot t \mu$.
\end{enumerate}
\end{DefSub}

\begin{ExaSub}\label{exa:enr_adj_mnd}
There is a 2-category $\CAT$ with objects all categories, 1-cells all functors, and 2-cells all natural transformations.  Adjunctions and monads in $\CAT$ are simply adjunctions and monads in the usual sense.

Given a monoidal category $\V$ \pbref{sec:enr_cats}, there is a 2-category $\VCAT$ with objects all $\V$-enriched categories, 1-cells all $\V$-functors, and 2-cells all $\V$-natural transformations.  Adjunctions and monads in $\VCAT$ are called \textit{$\V$-adjunctions} and \textit{$\V$-monads}, respectively.

There are 2-categories as follows; see \cite{EiKe}.
\begin{center}
     \begin{tabular}{ | p{10ex} | p{20ex} | p{20ex} | p{26ex} |}
     \hline
     Name  & Objects & 1-cells & 2-cells \\ \hline
     \MMCCAATT & monoidal categories & monoidal functors & monoidal transformations\\ \hline
     \SSMMCCAATT & symmetric monoidal categories & symmetric monoidal functors & monoidal transformations \\ \hline
     \CCllCCAATT & closed categories & closed functors & closed transformations \\
     \hline
     \end{tabular}
\end{center}
Thus we obtain the notions of \textit{monoidal} (resp. \textit{symmetric monoidal, closed}) \textit{adjunction} (resp. \textit{monad}).
\end{ExaSub}

\begin{ParSub}
As in the familiar case of $\K = \CAT$, every adjunction $f \dashv g:B \rightarrow A$ in a 2-category $\K$ determines an associated monad on $A$ in $\K$.
\end{ParSub}

\begin{PropSub} \label{thm:mon_func_detd_by_adj}
Let $f \nsststile{\varepsilon}{\eta} g : B \rightarrow A$ be an adjunction in a 2-category $\K$.  Then there is an associated monoidal functor $[f,g] := \K(f,g) : \K(B,B) \rightarrow \K(A,A)$ with the following property:  For any adjunction $f' \nsststile{\varepsilon'}{\eta'} g':C \rightarrow B$ with induced monad $\TT'$ on $B$, the monad $[f,g](\TT')$ on $A$ is equal to the monad induced by the composite adjunction 
$$\xymatrix{A \ar@/_0.5pc/[rr]_f^(0.4){\eta}^(0.6){\varepsilon}^{\top} & & B \ar@/_0.5pc/[ll]_g \ar@/_0.5pc/[rr]_{f'}^(0.4){\eta'}^(0.6){\varepsilon'}^{\top} & & {C\;.} \ar@/_0.5pc/[ll]_{g'}}$$
\end{PropSub}
\begin{proof}
The monoidal structure on the functor $[f,g]$ consists of the morphisms $gh\varepsilon hf:ghfgkf \rightarrow ghkf$ in $\K(A,A)$ (for all objects $h$, $k$ in $\K(B,B)$) and the morphism $\eta:1_A \rightarrow gf$ in $\K(A,A)$.  The verification is straightforward.
\end{proof}

\begin{ParSub} \label{par:adj_and_mnd}
Given objects $A$, $B$ in a 2-category $\K$, there is a category $\Adj_\K(A,B)$ whose objects are adjunctions $f \nsststile{\varepsilon}{\eta} g : B \rightarrow A$ in $\K$ and whose morphisms $(\phi,\psi):(f \nsststile{\varepsilon}{\eta} g) \rightarrow (f' \nsststile{\varepsilon'}{\eta'} g')$ consist of 2-cells $\phi:f \rightarrow f'$ and $\psi:g \rightarrow g'$ such that $(\psi \circ \phi) \cdot \eta = \eta'$ and $\varepsilon' \cdot (\phi \circ \psi) = \varepsilon$.

There is a category $\Mnd_\K(A)$ whose objects are monads on $A$ and whose morphisms $\theta:(t,\eta,\mu) \rightarrow (t',\eta',\mu')$ consist of a 2-cell $\theta:t \rightarrow t'$ such that $\theta \cdot \eta = \eta'$ and $\mu' \cdot (\theta \circ \theta) = \theta \cdot \mu$.  The identity monad $\ONEONE_A$ is an initial object in $\Mnd_\K(A)$, since for each monad $\TT = (t,\eta,\mu)$ on $A$, the 2-cell $\eta$ is the unique monad morphism $\eta:\ONEONE_A \rightarrow \TT$.

There is a functor $\Adj_{\K}(A,B) \rightarrow \Mnd_\K(A)$ sending an adjunction to its induced monad and a morphism $(\phi,\psi):(f \nsststile{\varepsilon}{\eta} g) \rightarrow (f' \nsststile{\varepsilon'}{\eta'} g')$ to the morphism $\psi \circ \phi:\TT \rightarrow \TT'$ between the induced monads.  Hence, in particular, isomorphic adjunctions induce isomorphic monads.
\end{ParSub}

\begin{ParSub}\label{par:general_mnd_mor}
Further to \bref{par:adj_and_mnd}, we shall also make use of a more general notion of morphism of monads $\TT \rightarrow \SSS$ in a 2-category $\K$, where $\TT = (t,\eta,\mu)$ is a monad on $A$ and $\SSS = (s,\phi,\nu)$ is a monad on (a possibly different object) $B$.  First let us define the following weaker notion, which we shall also require:  Given endomorphisms $t$ of $A$ and $s$ of $B$ in $\K$, a \textit{morphism of endo-1-cells} $(\lambda,g):t \rightarrow s$ in $\K$ consists of a 1-cell $g:B \rightarrow A$ and a 2-cell $\lambda:tg \rightarrow gs$.  We then define a \textit{morphism of monads} $(\lambda,g):\TT \rightarrow \SSS$ to be a morphism of the underlying endo-1-cells $(\lambda,g):t \rightarrow s$ such that the following diagrams commute:
\begin{eqnarray*}
\xymatrix{
tg \ar[rr]^\lambda &                                    & gs \\
                   & g \ar[ul]^{\eta g} \ar[ur]_{g\phi} &
} & \text{(Unit Law)}
\end{eqnarray*}
\begin{eqnarray*}
\xymatrix{
ttg \ar[r]^{t\lambda} \ar[d]_{\mu g} & tgs \ar[r]^{\lambda s} & gss \ar[d]^{g\nu} \\
tg \ar[rr]^\lambda                   &                        & gs
} & \text{(Associativity Law)}
\end{eqnarray*}
The narrower notion of morphism given in \bref{par:adj_and_mnd} is recovered in the case that $A = B$ by taking $g = 1_A$ to be the identity 1-cell.  With these more general morphisms, the monads in $\K$ form a category.  Morphisms of this sort were considered in \cite{St:FormalThMnds} but there were called \textit{monad functors} and, accordingly, were construed as morphisms in the opposite direction --- i.e., in the direction of the component 1-cell rather than that of the 2-cell.
\end{ParSub}

\begin{PropSub} \label{thm:uniq_adj}
Let $f \nsststile{\varepsilon}{\eta} g$ and $f' \nsststile{\varepsilon'}{\eta'} g$ be adjunctions, having the same right adjoint $g:B \rightarrow A$, in a 2-category $\K$.  Then these adjunctions are isomorphic and hence induce isomorphic monads on $A$.
\end{PropSub}
\begin{proof}
The 2-cell $\phi := \varepsilon f' \cdot f \eta':f \rightarrow f'$ has inverse $\varepsilon' f \cdot f' \eta$ (\cite{Gray}, I,6.3), and one checks that $(\phi,1_g)$ serves as the needed isomorphism of adjunctions.
\end{proof}

\begin{PropSub} \label{thm:mnd_morph_from_adj_factn}
Let $\xymatrix{A \ar@/_0.5pc/[rr]_f^(0.4){\eta}^(0.6){\varepsilon}^{\top} & & B \ar@/_0.5pc/[ll]_g \ar@/_0.5pc/[rr]_{f'}^(0.4){\eta'}^(0.6){\varepsilon'}^{\top} & & C \ar@/_0.5pc/[ll]_{g'}}$ and $\xymatrix{A \ar@/_0.5pc/[rr]_{f''}^(0.4){\eta''}^(0.6){\varepsilon''}^{\top} & & C \ar@/_0.5pc/[ll]_{g''}}$ be adjunctions in a 2-category $\K$, with respective induced monads $\TT$, $\TT'$, $\TT''$, and suppose that $gg' = g''$.  Then there is an associated monad morphism $\TT \rightarrow \TT''$.
\end{PropSub}
\begin{proof}
Let $f_c \nsststile{\varepsilon_c}{\eta_c} g''$ be the composite adjunction, and let $\TT_c$ be its induced monad.  By \bref{thm:mon_func_detd_by_adj}, we have that $\TT_c = [f,g](\TT')$, whereas $\TT = [f,g](\ONEONE_A)$.  By applying $[f,g]$ to the monad morphism $\eta':\ONEONE_A \rightarrow \TT'$, we obtain a monad morphism
$$g \eta' f = [f,g](\eta') : \TT = [f,g](\ONEONE_A) \rightarrow [f,g](\TT') = \TT_c\;.$$
Also, by \bref{thm:uniq_adj}, there is an isomorphism of monads $\xi:\TT_c \rightarrow \TT''$, and we obtain a composite morphism of monads
\begin{equation}\label{eqn:comp_morph_mnds}\TT \xrightarrow{g \eta' f} \TT_c \xrightarrow{\xi} \TT''\;.\end{equation}
\end{proof}

\begin{ParSub}\label{par:idm_mnd}
A monad $(t,\eta,\mu)$ on an object $A$ in a 2-category $\K$ is said to be \textit{idempotent} if any of the following equivalent conditions holds (see, e.g., \cite{Sto}):  (i) $t\eta = \eta t$, (ii) $t\mu = \mu t$, (iii) $\mu$ is iso.
\end{ParSub}

\section{Monoidal, closed, and enriched categories} \label{sec:enr_cats}

We shall employ the theory of monoidal categories, closed categories, and categories enriched in a monoidal category $\V$, as documented in the classic \cite{EiKe} and the comprehensive references \cite{Ke:Ba}, \cite{Dub}.  In the absence of any indication to the contrary, $\V$ is assumed to be a \textit{closed} symmetric monoidal category.  Since we shall at times avail ourselves of enrichment with respect to each of multiple given categories, we shall include an explicit indication of $\V$ when employing notions such as $\V$-category, $\V$-functor, and so on, omitting the prefix $\V$ only when concerned with the corresponding notions for non-enriched or \textit{ordinary} categories.  When such ordinary notions and terminology are applied to a given $\V$-category $\A$, they should be interpreted relative to the underlying ordinary category of $\A$.

\begin{ParSubSub}\label{par:cat_classes}
The ordinary categories considered in this text are \textit{not} assumed \textit{locally small} --- that is, they are not necessarily $\Set$-enriched categories.  Rather, we assume that for any finite family of categories under consideration, there is a category $\SET$ of \textit{classes} in which lie the hom-classes of each of the given categories, so that each is $\SET$-enriched.  The category $\SET$ is assumed to have finite limits, and the terminal object $1$ is assumed to be a dense generator, but $\SET$ is not assumed cartesian closed.
\end{ParSubSub}

\begin{ParSubSub}\label{par:clsmcat_notn}
Given a closed symmetric monoidal category $\V$, we denote by $\uV$ the canonically associated $\V$-category whose underlying ordinary category is isomorphic to $\V$ (and shall be identified with $\V$ whenever convenient); in particular, the internal homs in $\V$ will therefore be denoted by $\uV(V_1,V_2)$.  The canonical `evaluation' morphisms $V_1 \otimes \aeV(V_1,V_2) \rightarrow V_2$ are denoted by $\Ev_{V_1 V_2}$, or simply $\Ev$.  We sometimes omit subscripts and names of morphisms when they are clear from the context.
\end{ParSubSub}

\begin{ParSubSub}\label{par:well_pointed_ccc}
While this text is written in terms of abstract symmetric monoidal closed categories, we shall have occasion to consider examples of locally small cartesian closed categories $\X$ that are \textit{well-pointed}, in the sense that the associated functor $\X(1,-):\X \rightarrow \Set$ is faithful.  In this case, each object $X$ is equipped with an \textit{underlying set} $\X(1,X)$ and the morphisms $f$ in $\X$ may be identified with the functions $\X(1,f)$ that they induce between the underlying sets; further, the underlying set of an internal hom $\uX(X,Y)$ may be identified with $\X(X,Y)$, and the evaluation morphisms are then identified with the restrictions of those in $\Set$.  In practice, a cartesian closed category $\X$ may be equipped with a faithful functor $U:\X \rightarrow \Set$ that does not coincide with $\X(1,-)$, but as long as the terminal object $1$ of $\X$ is preserved by $U$ and is \textit{$U$-discrete} (meaning that $U_{1 X}:\X(1,X) \rightarrow \Set(U1,UX)$ is bijective for all $X \in \X$), then it follows that $U \cong \X(1,-)$.  In the terminology of \cite{AHS}, $(\X,U)$ is said to be a \textit{construct with function spaces}.
\end{ParSubSub}

\begin{ExaSubSub}\label{exa:conv_sm_sp}
The category $\Conv$ of \textit{convergence spaces} \cite{BeBu} and the category $\Smooth$ of Fr\"olicher's \textit{smooth spaces} \cite{Fro:SmthStr,Fro:CccAnSmthMaps,FroKr} are locally small well-pointed cartesian closed categories.

A \textit{convergence space} is a set $X$ together with an assignment to each point $x \in X$ a set $c(x)$ of (proper) filters on $X$ such that the following conditions hold for all filters $\mathfrak{x},\mathfrak{y}$ on $X$ and all points $x \in X$, where we say that $\mathfrak{x}$ \textit{converges to} $x$ if $\mathfrak{x} \in c(x)$:
\begin{enumerate}
\item The filter $\{A \subs X | x \in A\}$ converges to $x$.
\item If $\mathfrak{x}$ converges to $x$ and $\mathfrak{x} \subs \mathfrak{y}$, then $\mathfrak{y}$ converges to $x$.
\item If each of $\mathfrak{x}$ and $\mathfrak{y}$ converges to $x$, then their intersection $\mathfrak{x} \cap \mathfrak{y}$ converges to $x$.
\end{enumerate}
Given convergence spaces $X$, $Y$, a \textit{continuous map} is a function $f:X \rightarrow Y$ such that if $\mathfrak{x}$ converges to $x$ in $X$ then the \textit{image filter} $f(\mathfrak{x}) := \{B \subs Y \mid f^{-1}(B) \in \mathfrak{x}\}$ converges to $f(x)$ in $Y$.  Convergence spaces and continuous maps constitute a locally small well-pointed cartesian closed category $\Conv$ into which the category $\Top$ of topological spaces embeds as a full subcategory.

A \textit{smooth space} (in the sense of Fr\"olicher) is a set $X$ equipped with a set $C_X$ of mappings from $\RR$ to $X$ and a set $F_X$ of mappings from $X$ to $\RR$, such that the following conditions hold:
\begin{enumerate}
\item An arbitrary mapping $c:\RR \rightarrow X$ lies in $C_X$ if and only if the composite $\RR \xrightarrow{c} X \xrightarrow{f} \RR$ is smooth for each $f \in F_X$.
\item An arbitrary mapping $f:X \rightarrow \RR$ lies in $F_X$ if and only if the composite $\RR \xrightarrow{c} X \xrightarrow{f} \RR$ is smooth for each $c \in C_X$.
\end{enumerate}
We call the elements of $C_X$ \textit{smooth curves} in $X$ and the elements of $F_X$ \textit{smooth functions} on $X$.  Given smooth spaces $X$ and $Y$, a mapping $g:X \rightarrow Y$ is said to be a \textit{smooth map} if for each $c \in C_X$, the composite $\RR \xrightarrow{c} X \xrightarrow{g} Y$ lies in $C_Y$.  It is equivalent to require that $f \cdot g \in F_X$ for each $f \in F_Y$.  Smooth spaces and smooth maps constitute a locally small well-pointed cartesian closed category into which the category of paracompact smooth manifolds and smooth maps embeds as a full subcategory \cite{Fro:CccAnSmthMaps}.
\end{ExaSubSub}

\begin{ExaSubSub}\label{exa:rmodx_smc_and_xenr}
Given a commutative ring object $R$ in a countably-complete and \linebreak[4] -cocomplete cartesian closed category $\X$, the category $\sL := \RMod(\X)$ of $R$-module objects in $\X$ is symmetric monoidal closed \pbref{exa:rmod_smclosed_monadic}.  Also, $\sL$ is (the underlying ordinary category of) an $\X$-enriched category \pbref{exa:rmod_smclosed_monadic}.

For example, for $\X = \Conv$ and $R = \RR$ or $\CC$, $\sL = \RMod(\X)$ is the category of \textit{convergence vector spaces}; see \cite{BeBu}.  For $\X = \Smooth$ and $R = \RR$, $\sL = \RMod(\X)$ is the category of \textit{smooth vector spaces} \cite{Fro:SmthStr}.
\end{ExaSubSub}

\subsection{Enriched monomorphisms and conical limits} \label{sec:enr_monos_lims}

The following notions are defined in \cite{Dub}.

\begin{DefSubSub} \label{def:enr_mono_limit}
Let $\B$ be a $\V$-category.
\begin{enumerate}
\item A morphism $m:B_1 \rightarrow B_2$ in $\B$ (i.e., in the underlying ordinary category of $\B$) is a \textit{$\V$-mono} if $\B(A,m):\B(A,B_1) \rightarrow \B(A,B_2)$ is a monomorphism in $\V$ for every object $A$ of $\B$.  A morphism $e$ in $\B$ is a \textit{$\V$-epi} if $e$ is a $\V$-mono in $\B^\op$.
\item $\Mono_\V\B$ and $\Epi_\V\B$ are the classes of all $\V$-monos and $\V$-epis, respectively, in $\B$.
\item A \textit{$\V$-limit} of an ordinary functor $D:\J \rightarrow \B$ consists of a cone for $D$ that is sent by each functor $\B(A,-):\B \rightarrow \V$ to a limit cone for $\B(A,D-)$.  Equivalently, a $\V$-limit is a limit of $D$ that is preserved by each functor $\B(A,-)$.  As special cases of $\V$-limits we define \textit{$\V$-products, $\V$-pullbacks, $\V$-fiber-products}, etc.  \textit{$\V$-colimits} are defined as $\V$-limits in $\B^\op$.
\end{enumerate}
\end{DefSubSub}
\begin{RemSubSub}
$\V$-limits coincide with the \textit{conical limits} of \cite{Ke:Ba}.  Note that every $\V$-mono (resp. $\V$-epi, $\V$-limit, $\V$-colimit) in $\B$ is a mono (resp. epi, limit, colimit) in (the underlying ordinary category of) $\B$.
\end{RemSubSub}

\begin{PropSubSub} \label{thm:lim_mono_epi_in_base_are_enr}
Any ordinary mono (resp. epi, limit, colimit) in $\V$ is a $\V$-mono (resp. $\V$-epi, $\V$-limit, $\V$-colimit) in $\uV$.  Hence $\Mono_{\V}\uV = \Mono\V$ and $\Epi_{\V}\uV = \Epi\V$.
\end{PropSubSub}
\begin{proof}
Regarding limits and monos, each ordinary functor $\uV(V,-):\V \rightarrow \V$ is right adjoint and hence preserves limits and monos.  Regarding epis and colimits, each functor $\uV^\op(V,-) = \uV(-,V) : \V^\op \rightarrow \V$ is right adjoint (to $\uV(-,V)$) and hence sends monos (resp. limits) in $\V^\op$ (i.e. epis, resp. colimits, in $\V$) to monos (resp. limits).
\end{proof}

\begin{PropSubSub} \label{prop:enr_mono_kp}
A morphism $m:B \rightarrow C$ in a $\V$-category $\B$ is a $\V$-mono if and only if $m$ has a $\V$-kernel-pair $\pi_1,\pi_2:P \rightarrow B$ with $\pi_1 = \pi_2$.
\end{PropSubSub}
\begin{proof}
If $m$ is a $\V$-mono then $1_B,1_B:B \rightarrow B$ is a $\V$-kernel-pair in $\B$, as one readily checks.  Conversely, if $m$ has a $\V$-kernel-pair with $\pi_1 = \pi_2$, then since each functor $\B(A,-):\B \rightarrow \V$ ($A \in \B$) preserves the given kernel pair, the needed conclusion follows from the analogous result for ordinary categories, which is immediate.
\end{proof}

\begin{PropSubSub} \label{prop:intersection}
Let $(m_i:B_i \rightarrow C)_{i \in I}$ be a family of $\V$-monos in $\V$-category $\B$, where $I$ is a class, and let $m:B \rightarrow C$ be a $\V$-fiber-product of this family, with associated projections $\pi_i:B \rightarrow B_i$ ($i \in I$).  Then $m$ is a $\V$-mono.
\end{PropSubSub}
\begin{proof}
For each $A \in \B$, we must show that $\B(A,m):\B(A,B) \rightarrow \B(A,C)$ is mono, but $\B(A,m)$ is a fiber product in $\V$ of the monomorphisms $\B(A,m_i)$.  Hence it suffices to show that the analogous proposition holds for ordinary categories, and in this case the verification is straightforward and elementary.
\end{proof}

\begin{DefSubSub} \label{def:intersection}
In the situation of \ref{prop:intersection}, we say that the $\V$-mono $m$ is a \textit{$\V$-intersection} of the $m_i$.  
\end{DefSubSub}
\begin{RemSubSub}
Definition \ref{def:intersection} is an enriched analogue of the notion of intersection defined for ordinary categories in \cite{Ke:MonoEpiPb}.
\end{RemSubSub}

\begin{DefSubSub} \label{def:closed_under_tensors}
Given a class of morphisms $\E$ in a $\V$-category $\B$, we say that $\E$ is \textit{closed under tensors} in $\B$ if for any morphism $e:A_1 \rightarrow A_2$ in $\E$ and any object $V \in \V$ for which tensors $V \otimes A_1$ and $V \otimes A_2$ exist in $\B$, the induced morphism $V \otimes f:V \otimes A_1 \rightarrow V \otimes A_2$ lies in $\E$.  Dually, one defines the property of being \textit{closed under cotensors} in $\B$.
\end{DefSubSub}

\begin{PropSubSub} \label{thm:closure_props_of_monos}
For any $\V$-category $\B$, the following hold:
\begin{enumerate}
\item If $g \cdot f \in \Mono_\V\B$, then $f \in \Mono_\V\B$.
\item $\Mono_\V\B$ is closed under composition, cotensors, arbitrary $\V$-fiber-products, and $\V$-pullbacks along arbitrary morphisms in $\B$.
\end{enumerate}
\end{PropSubSub}
\begin{proof}
Both 1 and the needed closure under composition follow immediately from the analogous statements for ordinary categories, upon applying each functor $\B(A,-):\B \rightarrow \V$ ($A \in \B$).  We have already established closure under $\V$-fiber-products in \bref{prop:intersection}, and closure under $\V$-pullbacks is proved by an analogous method, since the corresponding statement for ordinary categories holds.  Lastly, given an object $V$ of $\V$ and a $\V$-mono $m:B_1 \rightarrow B_2$ in $\B$ for which cotensors $[V,B_1], [V,B_2]$ exist in $\B$, the induced morphism $[V,m]:[V,B_1] \rightarrow [V,B_2]$ is $\V$-mono, as follows.  Indeed, for each $A \in \B$, $\B(A,[V,m]) \cong \uV(V,\B(A,m))$ in the arrow category $[\Two,\V]$, and since $\B(A,m)$ is mono and the right-adjoint functor $\V(V,-):\V \rightarrow \V$ preserves monos, $\uV(V,\B(A,m))$ is mono.
\end{proof}

\subsection{Faithful enriched functors}\label{sec:faithful_enr_funcs}

\begin{DefSubSub}
A $\V$-functor $G:\B \rightarrow \A$ is said to be \textit{$\V$-faithful} if each structure morphism
$$G_{B B'}:\B(B,B') \rightarrow \A(GB,GB')\;\;\;\;(B,B' \in \B)$$
is a monomorphism in $\V$.  $G$ is \textit{$\V$-fully-faithful} if each $G_{B B'}$ is an isomorphism in $\V$.
\end{DefSubSub}

\begin{PropSubSub}\label{thm:factn_through_faithful_vfunctor}
Let $Q:\C \rightarrow \A$, $G:\B \rightarrow \A$ be $\V$-functors, and suppose that $G$ is $\V$-faithful.  Suppose we are given an assignment to each object $C$ of $\C$ an object $PC$ of $\B$ such that $GPC = QC$, and suppose that for all $C,C' \in \C$ the structure morphism $Q_{C C'}$ factors through the monomorphism $G_{PC PC'}$ in $\V$.  Then there is a unique $\V$-functor $P:\C \rightarrow \B$ such that $\Ob P$ is the given assignment and $GP = Q$.
\end{PropSubSub}
\begin{proof}
Letting $G_{PC PC'} \cdot P_{C C'} = Q_{C C'}$ be the factorization in question for each pair $C,C' \in \C$, it is straightforward to check that the diagrammatic axioms for the $\V$-functoriality of $P$ follow from those of $Q,G$ and the fact that the $G_{PC PC'}$ are mono.
\end{proof}

\begin{PropSubSub}\label{thm:crit_for_equal_faithful_vfunctors}
Let $P^{(1)}:\B_1 \rightarrow \A$ and $P^{(2)}:\B_2 \rightarrow \A$ be $\V$-faithful $\V$-functors with $\ob\B_1 = \ob\B_2$ and $\ob P^{(1)} = \ob P^{(2)}$, and suppose that for all $B,B' \in \ob\B_1 = \ob\B_2$, $\B_1(B,B') = \B_2(B,B')$ and $P^{(1)}_{B B'} = P^{(2)}_{B B'}$.  Then $\B_1 = \B_2$ and $P^{(1)} = P^{(2)}$.
\end{PropSubSub}
\begin{proof}
Under the given assumptions, the diagrammatic unit and composition laws for the $\V$-functors $P^{(i)}$, together with the fact that the structure morphisms of the $P^{(i)}$ are mono, entail immediately that the unit and composition morphisms for $\B_1$ and $\B_2$ are identical, as one readily checks.
\end{proof}

\subsection{Enriched adjunctions and monads} \label{sec:enr_adj}

We shall require the following basic results on $\V$-adjunctions and $\V$-monads \pbref{exa:enr_adj_mnd}.

\begin{PropSubSub} \label{prop:adjn_via_radj_and_unit}
Let $G:\C \rightarrow \B$ be a $\V$-functor, and suppose that for each $B \in \B$ we are given an object $FB$ in $C$ and a morphism $\eta_B:B \rightarrow GFB$ in $\B$ such that for each $C \in \C$, the composite
$$\phi_{B C} := \left(\C(FB,C) \xrightarrow{G_{FB C}} \B(GFB,GC) \xrightarrow{\B(\eta_B,GC)} \B(B,GC)\right)$$
is an isomorphism in $\V$.  Then we obtain a $\V$-functor $F:\B \rightarrow \C$ with structure morphisms the composites
\begin{equation}\label{eqn:ladj_formula}F_{B_1 B_2} := \left(\B(B_1,B_2) \xrightarrow{\B(B_1,\eta_{B_2})} \B(B_1,GFB_2) \xrightarrow{\phi_{B_1 FB_2}^{-1}} \C(FB_1,FB_2) \right)\;,\;\;B_1,B_2 \in \B\;,\end{equation}
and the given morphisms determine a $\V$-natural transformation $\eta$ and a $\V$-adjunction \mbox{$F \nsststile{}{\eta} G$}.
\end{PropSubSub}
\begin{proof}
Fixing an object $B \in \B$, the weak Yoneda lemma (\cite{Ke:Ba}, 1.9) yields a bijection between morphisms $B \rightarrow GFB$ and $\V$-natural transformations $\C(FB,-) \rightarrow \B(B,G-)$, under which $\eta_B$ corresponds to $\phi_{B -} := (\phi_{B C})_{C \in \C}$.  In particular, $\phi_{B -}:\C(FB,-) \rightarrow \B(B,G-)$ is thus a $\V$-natural isomorphism, showing that $\B(B,G-):\C \rightarrow \uV$ is a representable $\V$-functor.  Since this holds for all $B \in \B$, the result follows from \cite{Ke:Ba}, \S 1.11.
\end{proof}

\begin{PropSubSub} \label{thm:enradj_detd_by_radj_and_univ_arr}
\emptybox
\begin{enumerate}
\item Given a $\V$-adjunction $F \nsststile{\varepsilon}{\eta} G:\C \rightarrow \B$, the left $\V$-adjoint $F$ is necessarily given as in \bref{prop:adjn_via_radj_and_unit}.
\item A $\V$-adjunction $F \nsststile{\varepsilon}{\eta} G:\C \rightarrow \B$ is uniquely determined by the following data: (1) the right $\V$-adjoint $G$, (2) the family of objects $FB$ $(B \in \B)$, and (3) the family of morphisms $\eta_B:B \rightarrow GFB$ $(B \in \B)$.
\end{enumerate}
\end{PropSubSub}
\begin{proof}
\cite{Ke:Ba}, \S 1.11.
\end{proof}

\begin{PropSubSub}\label{prop:enr_of_ord_adj}
Given a $\V$-functor $G:\C \rightarrow \B$ and an ordinary adjunction  \linebreak[4]\mbox{$F \nsststile{}{\eta} G:\C \rightarrow \B$} (in which the right adjoint is the underlying ordinary functor of the given $\V$-functor $G$), suppose that each associated morphism $\phi_{B C}$ \pbref{prop:adjn_via_radj_and_unit} is an isomorphism in $\V$.  Then the given adjunction is the underlying ordinary adjunction of a $\V$-adjunction gotten via \bref{prop:adjn_via_radj_and_unit}.
\end{PropSubSub}
\begin{proof}
Our task is simply to verify that the two ordinary adjunctions in question are identical, but both share the same right adjoint functor and the same family of universal arrows $(FB,\eta_B:B \rightarrow GFB)_{B \in \B}$ for the right adjoint and hence are identical.
\end{proof}

\begin{ParSubSub}\label{par:em_adj}
Let $\TT$ be a $\V$-monad on a $\V$-category $\B$.  By \cite{Dub} II.1, if $\V$ has equalizers, then there is an associated $\V$-adjunction
$$F^\TT \dashv G^\TT : \B^\TT \rightarrow \B$$
which we call \textit{the Eilenberg-Moore $\V$-adjunction} and whose underlying ordinary adjunction is the usual Eilenberg-Moore adjunction for the underlying ordinary monad of $\TT$.  For each pair of $\TT$-algebras $(A,a),(B,b)$, the associated structure morphism of the $\V$-functor $G^\TT$ is defined as the equalizer
$$\B^\TT((A,a),(B,b)) \rightarrow \B(A,B)$$
in $\V$ of the pair of morphisms
\begin{equation}\label{eq:pair_defining_em_hom}\B(A,B) \xrightarrow{T_{A B}} \B(TA,TB) \xrightarrow{\B(TA,b)} \B(TA,B)\end{equation}
$$\B(A,B) \xrightarrow{\B(a,B)} \B(TA,B)\;.$$
Obviously even if $\V$ does not have all equalizers, we can just as well define the Eilenberg-Moore $\V$-adjunction as long as we have for each pair of $\TT$-algebras an associated equalizer of the given morphisms, and in this case we say that \textit{the Eilenberg-Moore $\V$-category for $\TT$ exists}.
\end{ParSubSub}

\begin{ExaSubSub}
Given a commutative ring object $R$ in a countably-complete and \linebreak[4] -cocomplete cartesian closed category $\X$, the category $\sL := \RMod(\X)$ of $R$-module objects in $\X$ is isomorphic to the underlying ordinary category of $\uX^\TT$ for an $\X$-monad $\TT$ on $\uX$ \pbref{thm:rmod_commutative_monadic}.
\end{ExaSubSub}

\begin{DefSubSub}\label{def:monadic}
Let $F \nsststile{\varepsilon}{\eta} G:\C \rightarrow \B$ be a $\V$-adjunction with induced $\X$-monad $\TT$, and suppose that the Eilenberg-Moore $\V$-category for $\TT$ exists.  
\begin{enumerate}
\item We say that the given $\V$-adjunction is \textit{$\V$-monadic} if the comparison $\V$-functor $\C \rightarrow \B^\TT$ (\cite{Dub} II.1) is an equivalence of $\V$-categories.
\item We say that the given $\V$-adjunction is \textit{strictly $\V$-monadic} if the comparison $\V$-functor is an isomorphism.
\end{enumerate}
\end{DefSubSub}

\begin{RemSubSub}
By \cite{Dub} II.1.6, the comparison $\V$-functor \pbref{def:monadic} commutes with both the `forgetful' $\V$-functors $G,G^\TT$ and the `free' $\V$-functors $F,F^\TT$.
\end{RemSubSub}

\begin{ThmSubSub}[Crude Monadicity Theorem]\label{thm:crude_mndcity}
Let $F \nsststile{\varepsilon}{\eta} G:\C \rightarrow \B$ be a $\V$-adjunction.
\begin{enumerate}
\item If $G$ detects, preserves, and reflects $\V$-coequalizers of reflexive pairs, then the given $\V$-adjunction is $\V$-monadic.
\item If $G$ creates $\V$-coequalizers of reflexive pairs, then the given $\V$-adjunction is strictly $\V$-monadic.
\end{enumerate}
\end{ThmSubSub}
\begin{proof}
A proof is obtained by adapting the proof of Theorem II.2.1 of \cite{Dub}.  Indeed, whereas the latter theorem considers $\V$-coequalizers of \textit{$G$-contractible} pairs rather than reflexive pairs, the same proof can be used, since the pairs of $\V$-functors whose pointwise $\V$-coequalizers are taken therein are not only $G$-contractible (as noted there) but also reflexive.  Indeed, the first pair considered therein, namely $\varepsilon F G^\TT, FG^\TT\varepsilon^\TT$ has a common section $F\eta G^\TT$, and the second pair considered, namely $\varepsilon FG, FG\varepsilon$ has common section $F\eta G$.
\end{proof}

\begin{PropSubSub}\label{prop:enr_func_induced_by_mnd_mor}
Let $(\lambda,Q):\TT \rightarrow \SSS$ be a morphism of $\V$-monads \pbref{par:general_mnd_mor}, where $\TT$ and $\SSS$ are $\V$-monads on $\A$ and $\B$, respectively, and $Q:\B \rightarrow \A$.  Then, assuming that the Eilenberg-Moore $\V$-categories for $\TT$, $\SSS$ exist \pbref{par:em_adj}, there is an associated $\V$-functor $Q^\lambda$ yielding a commutative diagram
\begin{equation}\label{eqn:comm_diag_func_ind_by_mnd_mor}
\xymatrix{
\B^\SSS \ar[d] \ar[r]^{Q^\lambda} & \A^\TT \ar[d] \\
\B \ar[r]^Q & \A
}
\end{equation}
in which the vertical arrows are the forgetful $\V$-functors.  On objects,
\begin{equation}\label{eqn:action_of_mndmor_on_algs}Q^\lambda(B,b:SB \rightarrow B) = (QB,TQB \xrightarrow{\lambda_B} QSB \xrightarrow{Qb} QB)\;.\end{equation}
Further, this process defines a contravariant functor from the category of $\V$-monads to $\VCAT$.
\end{PropSubSub}

\begin{ParSubSub}\label{par:alg_endofunc}
Given an endofunctor $T$ on a category $\A$, a \textit{$T$-algebra} $(A,a)$ consists of an object $A$ of $\A$ and a morphism $a:TA \rightarrow A$.  We then obtain also the notion of \textit{$T$-homomorphism} between $T$-algebras, defined in the same way as one defines the notion of $\TT$-homomorphism for a monad $\TT$.  One thus defines the \textit{category of $T$-algebras} $\A^T$.  Given endofunctors $T,S$ on categories $\A,\B$, respectively, one has a notion of \textit{morphism of endofunctors} $(\lambda,Q):T \rightarrow S$, per \bref{par:general_mnd_mor}.  As in \bref{prop:enr_func_induced_by_mnd_mor}, such a morphism induces a functor $Q^\lambda:\B^S \rightarrow \A^T$, given on objects by the formula \eqref{eqn:action_of_mndmor_on_algs} and inducing an a commutative diagram analogous to \eqref{eqn:comm_diag_func_ind_by_mnd_mor}.  Again this process defines a contravariant functor from an evident \hbox{(meta-)category} of endofunctors to the \hbox{(meta-)category} of categories.
\end{ParSubSub}

\begin{PropSubSub}\label{thm:mnd_mor_as_homom}
Assume given $\V$-monads $\TT = (T,\eta^\TT,\mu^\TT)$ on $\A$ and $\SSS = (S,\eta^\SSS,\mu^\SSS)$ on $\B$, and a morphism of $\V$-endofunctors $(\lambda,Q):T \rightarrow S$ \pbref{par:general_mnd_mor}.  Then $\lambda:\TT \rightarrow \SSS$ is a morphism of $\V$-monads if and only if each component $\lambda_B:TQB \rightarrow QSB$ $(B \in \B)$ is a $T$-homomorphism 
$$\lambda_B:(TQB,\mu^\TT_{QB}) \rightarrow Q^\lambda(SB,\mu^\SSS_B)$$
whose restriction along $\eta^\TT_{QB}:QB \rightarrow TQB$ is $Q\eta^\SSS_B:QB \rightarrow QSB$.
\end{PropSubSub}
\begin{proof}
The requirement that $\lambda_B$ be a $T$-homomorphism amounts to exactly the associativity law for the $\V$-monad morphism $\lambda$, and the remaining condition is the unit law.
\end{proof}

\begin{PropSubSub}\label{thm:equality_of_mnd_morphs_via_free_algs}
Morphisms of $\V$-monads $(\lambda^{(1)},Q),(\lambda^{(2)},Q'):\TT \rightarrow \SSS$ are equal if and only if the underlying $\V$-functors $Q,Q':\B \rightarrow \A$ are equal and the induced $\V$-functors $Q^{\lambda^{(1)}},Q^{\lambda^{(2)}}:\B^\SSS \rightarrow \A^\TT$ agree on free $\SSS$-algebras.
\end{PropSubSub}
\begin{proof}
Let $B \in \B$.  One implication is trivial.  For the other, suppose that $Q = Q'$ and that the $Q^{\lambda^{(i)}}$ agree on free $\SSS$-algebras.  For each $i = 1,2$, the component $\lambda^{(i)}_B:TQB \rightarrow QSB$ is characterized in \bref{thm:mnd_mor_as_homom} as the unique $\TT$-homomorphism
$$\lambda^{(i)}_B:(TQB,\mu^\TT_{QB}) \rightarrow Q^{\lambda^{(i)}}(SB,\mu^\SSS_B)$$
whose restriction along $\eta^\TT_{QB}:QB \rightarrow TQB$ is $Q\eta^\SSS_B:QB \rightarrow QSB$.  But by assumption $Q^{\lambda^{(1)}}(SB,\mu^\SSS_B) = Q^{\lambda^{(2)}}(SB,\mu^\SSS_B)$, so $\lambda^{(1)}_B = \lambda^{(2)}_B$.
\end{proof}

\begin{DefSubSub}\label{def:refl_subcats_and_idm_mnds}
Let $\B$ be a $\V$-category.
\begin{enumerate}
\item A \textit{$\V$-reflective-subcategory} of $\B$ is a full replete sub-$\V$-category $\B'$ of $\B$ for which the inclusion $\V$-functor $J:\B' \hookrightarrow \B$ has a $\V$-left-adjoint $K$.
\item A $\V$-adjunction $K \nsststile{}{\rho} J : \B' \hookrightarrow \B$ for which $J$ is the inclusion of a $\V$-reflective-subcategory is called a \textit{$\V$-reflection} (on $\B$).
\item Given an idempotent $\V$-monad $\SSS = (S,\rho,\lambda)$ \pbref{par:idm_mnd} on $\B$, we let $\B^{(\SSS)}$ denote the full sub-$\V$-category of $\B$ consisting of those objects $B$ for which $\rho_B$ is iso.
\end{enumerate}
\end{DefSubSub}

\begin{PropSubSub} \label{thm:refl_idemp_mnd}
There is a bijection $\Refl_\V(\B) \cong \IdmMnd_\V(\B)$ between the class $\Refl_\V(\B)$ of all $\V$-reflections on $\B$ and the class $\IdmMnd_\V(\B)$ of all idempotent $\V$-monads on $\B$, which associates to each $\V$-reflection on $\B$ its induced $\V$-monad.  The $\V$-reflective-subcategory associated to a given idempotent $\V$-monad $\SSS$ via this bijection is $\B^{(\SSS)}$.
\end{PropSubSub}

\begin{PropSubSub}\label{thm:em_adj_idm_mnd}
Let $\SSS$ be an idempotent $\V$-monad on a $\V$-category $\B$, with associated $\V$-reflection $K \nsststile{}{\rho} J : \B^{(\SSS)} \hookrightarrow \B$.
\begin{enumerate}
\item The Eilenberg-Moore $\V$-category for $\SSS$ exists \pbref{par:em_adj}.
\item The comparison $\V$-functor $\B^{(\SSS)} \rightarrow \B^\SSS$ is an isomorphism, and its structure morphisms are the identity morphisms; corresponding hom-objects of these $\V$-categories are identical.
\end{enumerate}
\end{PropSubSub}
\begin{proof}
For each pair of $\SSS$-algebras $(A,\rho_A^{-1}), (B,\rho_B^{-1})$ the associated morphisms in \eqref{eq:pair_defining_em_hom} are equal, so their equalizer may be taken as simply the identity morphism on $\B(A,B)$.
\end{proof}

\begin{RemSubSub}\label{rem:em_adj_idm_mnd}
In view of \bref{thm:em_adj_idm_mnd}, we shall often identify a $\V$-reflection with the Eilenberg-Moore $\V$-adjunction of its corresponding idempotent $\V$-monad.
\end{RemSubSub}

%% file: enr_orth_factn.tex
\chapter{Enriched orthogonality and factorization} \label{ch:enr_orth_fact}
\setcounter{subsection}{0}

The notion of \textit{factorization system} and the requisite concept of \textit{orthogonality} provide an abstraction for the phenomenon whereby morphisms in many categories may be canonically decomposed into a suitable `surjective' morphism followed by an `embedding' of the appropriate sort.  Yet the notion of factorization system is much more general, and in Chapter \bref{ch:compl_enr_orth} we shall use it in defining abstract notions of \textit{dense morphism}, \textit{closed embedding}, and \textit{closure} \pbref{sec:cmpl_cl_dens}.  The notion of orthogonality of morphisms in a category (see \cite{FrKe}, \S 2) has a natural counterpart for enriched categories \pbref{def:enr_orth_fact} that was employed by Day \cite{Day:AdjFactn}.  Since all categorical notions employed in Day's paper are ``assumed to be relative to a suitable symmetric monoidal closed category'', Day's use of the notion of factorization system must be interpreted in the enriched context as well.  Yet to date there has been no substantial account of the basic notions and theory of the enriched analogues of the notions of factorization system and prefactorization system of \cite{FrKe}, notwithstanding the brief treatment of a particular special case in \cite{Ke:Ba} 6.1, the substantial work on related notions at a greater level of generality in \cite{Ang:thesis,Ang:SemiInitFin}, and the brief account of \textit{functorial weak factorizations} for tensored enriched categories in \cite{Rie}.  The objective of the present chapter is to fill this gap to the extent needed for the sequel.

\section{Enriched factorization systems} \label{sec:enr_factn_sys}

\begin{DefSub} \label{def:enr_orth_fact}
Let $\B$ be a $\V$-category.
\begin{enumerate}
\item For morphisms $e:A_1 \rightarrow A_2$, $m:B_1 \rightarrow B_2$ in $\B$ we say that $e$ is \textit{$\V$-orthogonal} to $m$, written $e \downarrow_\V m$, if the commutative square
\begin{equation}\label{eqn:orth_pb}
\xymatrix {
\B(A_2,B_1) \ar[rr]^{\B(A_2,m)} \ar[d]_{\B(e,B_1)}  & & \B(A_2,B_2) \ar[d]^{\B(e,B_2)} \\
\B(A_1,B_1)  \ar[rr]^{\B(A_1,m)}                     & & \B(A_1,B_2)
}
\end{equation}
is a pullback in $\V$.

\item Given classes $\E$, $\M$ of morphisms in $\B$, we define
$$\E^{\downarrow_\V} := \{m \in \mor\B \;|\; \forall e \in \E \;:\; e \downarrow_\V m\}\;,$$
$$\M^{\uparrow_\V} := \{e \in \mor\B \;|\; \forall m \in \M \;:\; e \downarrow_\V m\}\;.$$

\item A \textit{$\V$-prefactorization-system} on $\B$ is a pair $(\E,\M)$ of classes of morphisms in $\B$ such that $\E^{\downarrow_\V} = \M$ and $\M^{\uparrow_\V} = \E$.

\item For a pair $(\E,\M)$ of classes of morphisms in $\B$, we say that \textit{(\E,\M)-factorizations exist} if every morphism in $\B$ factors as a morphism in $\E$ followed by a morphism in $\M$.  More precisely, we require an assignment to each morphism $f$ of $\B$ an associated pair $(e,m) \in \E \times \M$ with $f = m \cdot e$.

\item A \textit{$\V$-factorization-system} is a $\V$-prefactorization-system such that (\E,\M)-factorizations exist.
\end{enumerate}
\end{DefSub}

\begin{RemSub}
The given definition of orthogonality relative to $\V$ appears in \cite{Day:AdjFactn}, where the notion of enriched factorization system is also implicitly used to a limited extent.
\end{RemSub}

\begin{RemSub}
Analogous notions for ordinary categories were given in \cite{FrKe}, and for these we omit the indication of $\V$ from the notation; for $\V$-enriched categories with $\V = \Set$ (i.e. locally small ordinary categories) the $\V$-enriched notions coincide with the ordinary.  Given morphisms $e$ and $m$ in an ordinary category $\B$, $e \downarrow m$ if and only if for any commutative square as in the periphery of the following diagram
$$
\xymatrix{
\cdot \ar[r] \ar[d]_e & \cdot \ar[d]^m \\
\cdot \ar[r] \ar@{-->}[ur]|w    & \cdot
}
$$
there exists a unique morphism $w$ making the diagram commute.

For general $\V$, we note that $\V$-orthogonality implies ordinary orthogonality.
\end{RemSub}

\begin{ExaSub}
Let us return to the example of \bref{exa:rmodx_smc_and_xenr}, in which $\sL = \RMod(\X)$ is the symmetric monoidal closed category of $R$-module objects for a commutative ring object $R$ in a suitable cartesian closed category $\X$; suppose also that $\X$ is \textit{finitely well-complete} \pbref{def:enr_fwc}.  Particular examples include the categories $\sL$ of convergence vector spaces and smooth vector spaces.  We have the following $\sL$-factorization-systems on $\uL$; see \pbref{exa:rmod_factn_systems}.
\begin{enumerate}
\item $(\Epi\sL,\StrMono\sL)$, where $\Epi\sL$ consists of all epis in $\sL$ (equivalently, $\sL$-epis in $\uL$) and $\StrMono\sL$ consists of all strong monos in $\sL$ (equivalently, $\sL$-strong-monos in $\uL$, which we call \textit{embeddings}). 
\item $(\Dense{\Sigma_H},\ClEmb{\Sigma_H})$, consisting of the classes of \textit{functionally dense} morphisms and \textit{functionally closed embeddings}, respectively.
\end{enumerate}
\end{ExaSub}

\begin{RemSub}
Given a class of morphisms $\sH$ in a $\V$-category $\B$, both $(\sH^{\downarrow_\V\uparrow_\V},\sH^{\downarrow_\V})$ and $(\sH^{\uparrow_\V},\sH^{\uparrow_\V\downarrow_\V})$ are $\V$-prefactorization-systems on $\B$.
\end{RemSub}

\begin{RemSub} \label{rem:factn_dualization}
The above notions \pbref{def:enr_orth_fact} obviously depend on the choice of $\V$-category $\B$. Given morphisms $e:A_1 \rightarrow A_2$, $m:B_1 \rightarrow B_2$ in $\B$, one immediately finds that $e \downarrow_\V m$ in $\B$ iff $m \downarrow_\V e$ in $\B^\op$.  Given classes of morphisms $\E$, $\M$ in $\B$, we shall write $\E^\op$, $\M^\op$ for these classes when they are considered as classes of morphisms in $\B^\op$.  One then has that ${\E^\op}^{\uparrow_\V} = \E^{\downarrow_\V}$ and ${\M^\op}^{\downarrow_\V} = \M^{\uparrow_\V}$.  In particular, each $\V$-(pre)factorization-system $(\E,\M)$ determines a $\V$-(pre)factorization-system $(\M^\op,\E^\op)$ on $\B^\op$.
\end{RemSub}

\begin{PropSub} \label{thm:orth_yielding_retr_and_iso}
Let $g,f$ be morphisms in a $\V$-category $\B$ such that the composite $g \cdot f$ is defined.
\begin{enumerate}
\item If $g \cdot f \downarrow_\V g$ then $g$ is a retraction in $\B$.
\item If $f \downarrow_\V g \cdot f$ then $f$ is a section in $\B$.
\item If $f \downarrow_\V f$ then $f$ is an isomorphism.
\end{enumerate}
\end{PropSub}
\begin{proof}
If $g \cdot f \downarrow_\V g$ then there exists a unique $s$ such that the following diagram commutes
$$
\xymatrix{
\cdot \ar[r]^f \ar[d]_{g \cdot f} & \cdot \ar[d]^g \\
\cdot \ar@{=}[r] \ar@{-->}[ur]|s    & \cdot
}
$$
and hence $g$ is a retraction of $s$.  Statement 2 follows dually.  Lastly, if $f \downarrow_\V f$ then $f \cdot 1 \downarrow_\V f$ and hence $f$ is a retraction, but also $f \downarrow_\V 1 \cdot f$ and hence $f$ is a section, so $f$ is an isomorphism.
\end{proof}

The following is an enriched analogue of \cite{FrKe}, 2.2.1.

\begin{PropSub} \label{thm:crit_factn_sys}
Let $\E$, $\M$ be classes of morphisms in a $\V$-category $\B$.  Suppose that (i) each of $\E$ and $\M$ is closed under composition with isomorphisms, (ii) $\E \subs \M^{\uparrow_\V}$, and (iii) $(\E,\M)$-factorizations exist.  Then $(\E,\M)$ is a $\V$-factorization-system on $\B$.
\end{PropSub}
\begin{proof}
We have
\begin{equation}\label{eqn:orth_incls}\E \subs \M^{\uparrow_\V} \subs \M^\uparrow\;,\;\;\;\;\M \subs \E^{\downarrow_\V}\subs \E^\downarrow\;.\end{equation}  
Hence by \cite{AHS} 14.6, $(\E,\M)$ is an ordinary factorization system on $\B$, so $\E = \M^\uparrow$ and $\M = \E^\downarrow$ and hence (by \eqref{eqn:orth_incls}) $\E = \M^{\uparrow_\V}$ and $\M = \E^{\downarrow_\V}$.
\end{proof}

\begin{CorSub} \label{thm:vfactn_sys_is_factn_sys}
Every $\V$-factorization-system $(\E,\M)$ on $\B$ is an ordinary factorization system on $\B$.
\end{CorSub}
\begin{proof}
Since $\E \subs \M^{\uparrow_\V} \subs \M^\uparrow$, we may invoke \bref{thm:crit_factn_sys} (or rather \cite{AHS} 14.6, which applies to categories that are not locally small) with respect to the underlying ordinary category $\A$.
\end{proof}

\section{Stability and cancellation for enriched prefactorization systems} \label{sec:stab_canc_enr_prefactn}

In the present section we establish several stability properties of the left and right classes of an enriched prefactorization system.  Most are proved on the basis of analogous properties given in \bref{sec:stab_canc_cart_arr_sq} for cartesian arrows and hence pullback squares.

\begin{PropSub} \label{thm:pref_iso_comp_canc}
Let $(\E,\M) = (\sH^{\downarrow_\V\uparrow_\V},\sH^{\downarrow_\V})$ be a $\V$-prefactorization-system on a $\V$-category $\B$, and let $f:A \rightarrow B$, $g:B \rightarrow C$ in $\B$.
\begin{enumerate}
\item Each of $\E$ and $\M$ contains all the isomorphisms of $\B$.
\item Both $\E$ and $\M$ are closed under composition.
\item If $g \cdot f \in \M$ and $g \in \M$, then $f \in \M$.
\item Suppose $\sH \subs \Epi_\V\B$.  Then if $g \cdot f \in \M$, it follows that $f \in \M$.
\end{enumerate}
\end{PropSub}
\begin{proof}
(1).  Given an isomorphism $m:B_1 \rightarrow B_2$ in $\B$, we have that for each morphism $e:A_1 \rightarrow A_2$ in $\E$, the square \eqref{eqn:orth_pb} has a pair of opposite sides that are isos and hence is a pullback square \pbref{prop:pb_canc}, whence $e \downarrow_\V m$.  The statement regarding $\E$ follows by dualizing.

(2),(3),(4).  For any morphism $h:H_1 \rightarrow H_2$ in $\sH$ we have a commutative diagram
$$
\xymatrix{
\B(H_2,A) \ar[r]^{\B(H_2,f)} \ar[d]_{\B(h,A)} & \B(H_2,B) \ar[r]^{\B(H_2,g)} \ar[d]|{\B(h,B)} & {\C(H_2,C)} \ar[d]^{\C(h,C)} \\
\B(H_1,A) \ar[r]^{\B(H_1,f)}                  & \B(H_1,B) \ar[r]^{\B(H_1,g)}                  & {\C(H_1,C)\;.}
}
$$
Regarding (2), observe that if $f,g \in \M$ then the left and right squares are pullbacks and hence the outer rectangle is a pullback \pbref{prop:pb_canc}, whence $h \downarrow_\V g \cdot f$.  Regarding (3) and (4), suppose instead that $g \cdot f \in \M$.  Then the outer rectangle is a pullback.  If further $g \in \M$, then the rightmost square is a pullback, so by pullback cancellation \pbref{prop:pb_canc}, the leftmost square is a pullback --- i.e. $h \downarrow_\V f$.  If instead we assume that $\sH \subs \Epi_\V\B$, then $h$ is $\V$-epi and hence $\B(h,B)$ is mono, so that by \bref{prop:pb_canc_mono}, the leftmost square is again a pullback.
\end{proof}

\begin{CorSub}\label{thm:int_of_upper_and_lower_classes_is_iso}
For a $\V$-prefactorization-system $(\E,\M)$ on a $\V$-category $\B$, $\E \cap \M = \Iso\B$.
\end{CorSub}
\begin{proof}
By \bref{thm:orth_yielding_retr_and_iso} 3, $\E \cap \M \subs \Iso\B$, and the opposite inclusion holds by \bref{thm:pref_iso_comp_canc} 1.
\end{proof}

\begin{PropSub} \label{prop:upperclass_closed_under_tensors}
Let $(\E,\M)$ be a $\V$-prefactorization-system on $\V$-category $\B$.  Then $\E$ is closed under tensors in $\B$ \pbref{def:closed_under_tensors}.  Dually, $\M$ is closed under cotensors in $\B$.
\end{PropSub}
\begin{proof}
Letting $e:A_1 \rightarrow A_2$ lie in $\E$, we must show that $V \otimes e:V \otimes A_1 \rightarrow V \otimes A_2$ lies in $\E = \M^{\uparrow_\V}$ for each $V \in \V$ for which the given tensors exist in $\B$.  Letting $m:B_1 \rightarrow B_2$ lie in $\M$, the square
$$
\xymatrix {
\B(V \otimes A_2,B_1) \ar[rr]^{\B(V \otimes A_2,m)} \ar[d]_{\B(V \otimes e,B_1)}  && \B(V \otimes A_2,B_2) \ar[d]^{\B(V \otimes e,B_2)} \\
\B(V \otimes A_1,B_1)  \ar[rr]^{\B(V \otimes A_1,m)}                     && {\B(V \otimes A_1,B_2)\;,}
}
$$
is isomorphic to the square
$$
\xymatrix {
\uV(V,\B(A_2,B_1)) \ar[rr]^{\uV(V,\B(A_2,m))} \ar[d]_{\uV(V,\B(e,B_1))}  && \uV(V,\B(A_2,B_2)) \ar[d]^{\uV(V,\B(e,B_2))} \\
\uV(V,\B(A_1,B_1)) \ar[rr]^{\uV(V,\B(A_1,m))}                     && {\uV(V,\B(A_1,B_2))\;.}
}
$$
But the latter square is gotten by applying the limit-preserving ordinary functor \linebreak[4] \mbox{$\uV(V,-):\V \rightarrow \V$} to the pullback square \eqref{eqn:orth_pb} and hence is a pullback.
\end{proof}

\begin{PropSub} \label{prop:orth_stab_pb}
Let $e:A_1 \rightarrow A_2$ in a $\V$-category $\B$, let
$$
\xymatrix {
B'_1 \ar[r]^{f_1} \ar[d]_{m'} \pullbackcorner & B_1 \ar[d]^m \\
B'_2 \ar[r]_{f_2}                             & B_2
}
$$
be a $\V$-pullback in $\B$, and suppose that $e \downarrow_\V m$.  Then $e \downarrow_\V m'$.
\end{PropSub}
\begin{proof}
We have a commutative cube
$$
\xymatrix{
                                                         & {\B(A_2,B'_1)} \ar[rr]^{\B(A_2,f_1)} \ar@{.>}[dd]|(0.7){\B(A_2,m')} \ar[dl]|{\B(e,B'_1)} &                              &  {\B(A_2,B_1)} \ar[dd]|{\B(A_2,m)} \ar[dl]|{\B(e,B_1)}\\ 
 {\B(A_1,B'_1)} \ar[rr]^(0.7){\B(A_1,f_1)} \ar[dd]_{\B(A_1,m')} &                                                         & {\B(A_1,B_1)} \ar[dd]|(0.7){\B(A_1,m)} &                                  \\
                                                         & {\B(A_2,B'_2)} \ar@{.>}[rr]^(0.7){\B(A_2,f_2)} \ar@{.>}[dl]|{\B(e,B'_2)}          &                                  &  {\B(A_2,B_2)} \ar[dl]|{\B(e,B_2)}                  \\
{\B(A_1,B'_2)} \ar[rr]_{\B(A_1,f_2)}                      &                                                         & {\B(A_1,B_2)}                                  
}
$$
whose back and front faces are pullbacks, since they are obtained by applying the functors $\B(A_2,-),\B(A_1,-):\B \rightarrow \V$ to the given $\V$-pullback square.  The right face is also a pullback, since $e \downarrow_\V m$, so by \bref{prop:pb_cube}, the left face is a pullback --- i.e. $e \downarrow_\V m'$.
\end{proof}

\begin{CorSub} \label{thm:lowerclass_stab_pb}
Let $(\E,\M)$ be a $\V$-prefactorization-system on a $\V$-category $\B$.  Then $\M$ is closed under $\V$-pullbacks along arbitrary morphisms in $\B$.
\end{CorSub}
\begin{proof}
Since $\M = \E^{\downarrow_\V}$, this follows from \bref{prop:orth_stab_pb}.
\end{proof}

\begin{PropSub} \label{prop:orth_stab_fp}
Let $e:A_1 \rightarrow A_2$ in a $\V$-category $\B$, let $f:B \rightarrow C$ be a $\V$-fiber-product in $\B$ of the morphisms $f_i:B_i \rightarrow C$ (for $i$ in some class $I$), and suppose that $e \downarrow_\V f_i$ for all $i \in I$.  Then $e \downarrow_\V f$.
\end{PropSub}
\begin{proof}
Let the morphisms $\pi_i:B \rightarrow B_i$ ($i \in I$) present $f$ as a $\V$-fiber-product of the $f_i$.  For each $i \in I$ we have a commutative triangular prism
$$
\xymatrix{
{\B(A_2,B)} \ar[rr]^{\B(A_2,f)} \ar[dr]_{\B(A_2,\pi_i)} \ar[dd]|{\B(e,B)} &                                   & {\B(A_2,C)} \ar[dd]|{\B(e,C)} \\
                                           & \B(A_2,B_i) \ar[ur]_{\B(A_2,f_i)} \ar[dd]|(.3){\B(e,B_i)} &             \\
{\B(A_1,B)} \ar@{.>}[rr]^(0.7){\B(A_1,f)} \ar[dr]_{\B(A_1,\pi_i)}           &                                   & \B(A_1,C)           \\
                                           & \B(A_1,B_i) \ar[ur]_{\B(A_1,f_i)}                 &
}
$$
whose right face is a pullback since $e \downarrow_\V f_i$.  Further, since the given fiber product is preserved by the functors $\B(A_2,-),\B(A_1,-):\B \rightarrow \V$, the top and bottom faces ($i \in I$) present $\B(A_2,f)$ and $\B(A_1,f)$ respectively as fiber products.  Hence by \bref{cor:fp_pb}, the back face is a pullback.
\end{proof}

\begin{CorSub} \label{prop:lowerclass_stab_fp}
Let $(\E,\M)$ be a $\V$-prefactorization-system on a $\V$-category $\B$.  Then $\M$ is closed under arbitrary $\V$-fiber-products in $\B$.
\end{CorSub}
\begin{proof}
Since $\M = \E^{\downarrow_\V}$, this follows from \bref{prop:orth_stab_fp}.
\end{proof}

\section{A certain coincidence of enriched and ordinary orthogonality} \label{sec:coinc_enr_ord_orth}

A statement of the following proposition appears in an entry on the collaborative web site \textit{nLab} \cite{Nlab:enr_factn_sys}:
\begin{PropSub} \label{thm:enr_orth_from_ord}
Let $\E$ be a class of morphisms in a $\V$-category $\B$.  Suppose that $\B$ is tensored, and suppose that $\E$ is closed under tensors in $\B$ \pbref{def:closed_under_tensors}.  Then $\E^{\downarrow_\V} = \E^{\downarrow}$.
\end{PropSub}
\begin{proof}
Let $\SET$ be a category of classes \pbref{par:cat_classes} in which lie the hom-classes of $\V$.  Enriched orthogonality implies ordinary, so it suffices to show that $\E^\downarrow \subs \E^{\downarrow_\V}$.  Letting $m:B_1 \rightarrow B_2$ lie in $\E^\downarrow$ and $e:A_1 \rightarrow A_2$ lie in $\E$, we must show that $e \downarrow_\V m$.  It suffices to show that each functor $\V(V,-):\V \rightarrow \SET$ ($V \in \V$) sends the square \eqref{eqn:orth_pb} to a pullback square in $\SET$.  Since $\eA$ is tensored, we have in particular that
$$\V(V,\eA(A,B)) \cong \B(V \otimes A,B)$$
naturally in $A,B \in \B$.  The diagram of sets obtained by applying $\V(V,-)$ to the square \eqref{eqn:orth_pb} is therefore isomorphic to the following diagram
$$
\xymatrix@R=2ex {
\B(V \otimes A_2,B_1) \ar[rr]^{\B(V \otimes A_2,m)} \ar[d]_{\B(V \otimes e,B_1)}  && \B(V \otimes A_2,B_2) \ar[d]^{\B(V \otimes e,B_2)} \\
\B(V \otimes A_1,B_1)  \ar[rr]^{\B(V \otimes A_1,m)}                     && {\A(V \otimes A_1,B_2)\;,}
}
$$
which is a pullback in $\SET$ since $V \otimes e \in \E$ and hence $V \otimes e \downarrow m$.
\end{proof}

\begin{PropSub} \label{prop:enr_vs_ord_orth_for_pref_sys}
Let $(\E,\M)$ be a $\V$-prefactorization-system on a $\V$-category $\B$.
\begin{enumerate}
\item If $\B$ is tensored, then $\M = \E^{\downarrow_\V} = \E^\downarrow$.
\item If $\B$ is cotensored, then $\E = \M^{\uparrow_\V} = \M^\uparrow$.
\end{enumerate}
\end{PropSub}
\begin{proof}
Since $\E$ is closed under tensors in $\B$ (by \bref{prop:upperclass_closed_under_tensors}), the first statement follows from \bref{thm:enr_orth_from_ord}, and the second follows dually.
\end{proof}

\begin{CorSub}
Every $\V$-prefactorization-system on a tensored and cotensored $\V$-category $\B$ is an ordinary prefactorization system on $\B$.
\end{CorSub}

\section{Enriched strong monomorphisms} \label{sec:enr_str_monos}

The notion of strong monomorphism was introduced in \cite{Ke:MonoEpiPb}, and the following enriched generalization of this notion was given in \cite{Day:AdjFactn}.

\begin{DefSub} \label{def:enr_str_mon}
Let $\B$ be a $\V$-category.
\begin{enumerate}
\item A \textit{$\V$-strong-mono} in $\B$ is a $\V$-mono $m:B_1 \rightarrow B_2$ such that $e \downarrow_\V m$ for every $\V$-epi $e$ in $\B$.
\item We denote the class of all $\V$-strong-monos in $\B$ by 
$$\StrMono_\V\B := (\Epi_\V\B)^{\downarrow_\V} \cap \Mono_\V\B\;.$$
\end{enumerate}
\end{DefSub}

\begin{DefSub} \label{def:emb}
In this text we shall often use the following streamlined terminology and notation, provided the choice of $\V$ is clear from the context.
\begin{enumerate}
\item  We call $\V$-strong-monos in a $\V$-category $\B$ \textit{embeddings (in $\B$)}.
\item We write $m:B_1 \rightarrowtail B_2$ to indicate that $m$ is an embedding in a given $\V$-category.
\item Given an object $B \in \B$, we denote by $\Sub_\V(B)$, or simply $\Sub(B)$, the category whose objects are embeddings with codomain $B$ in $\B$ and whose morphisms $f:m \rightarrow n$ consist of a morphism $f$ in $\B$ such that $n \cdot f = m$.  Note that $\Sub(B)$ is a preordered class.
\item Given embeddings $m:M \rightarrowtail B$ and $n:N \rightarrowtail B$, we say that \textit{$n$ is an embedding onto $m$}, or equivalently, \textit{$n$ presents $N$ as $m$}, if $m \cong n$ in $\Sub(B)$.
\end{enumerate}
\end{DefSub}

\begin{PropSub} \label{thm:stability_and_canc_for_strmonos}
For any $\V$-category $\B$, the following hold:
\begin{enumerate}
\item If $g \cdot f \in \StrMono_\V\B$, then $f \in \StrMono_\V\B$.
\item $\StrMono_\V\B$ is closed under composition, cotensors, arbitrary $\V$-fiber-products, and $\V$-pullbacks along arbitrary morphisms in $\B$.
\end{enumerate}
\end{PropSub}
\begin{proof}
By \bref{sec:stab_canc_enr_prefactn} and \bref{thm:closure_props_of_monos}, each of the classes $(\Epi_\V\B)^{\downarrow_\V}$ and $\Mono_\V\B$ possesses the needed closure properties, so the intersection $(\Epi_\V\B)^{\downarrow_\V} \cap \Mono_\V\B = \StrMono_\V\B$ does as well.
\end{proof}

\begin{LemSub} \label{prop:sect_str_mono}
Every section in a $\V$-category $\B$ is a $\V$-strong-mono in $\B$.
\end{LemSub}
\begin{proof}
Suppose the composite $A \xrightarrow{s} B \xrightarrow{r} A$ in $\B$ is the identity.  Then $s$ is a $\V$-mono since $s$ is an absolute mono and hence is sent by each functor $\B(A',-):\B \rightarrow \V$ ($A' \in \B)$ to a mono in $\V$.  Further, for each $\V$-epi $e:A' \rightarrow B'$ in $\V$, we have a commutative diagram
$$
\xymatrix{
\B(B',A) \ar[r]^{\B(B',s)} \ar[d]|{\B(e,A)} & \B(B',B) \ar[r]^{\B(B',r)} \ar[d]|{\B(e,B)} & \B(B',A) \ar[d]|{\B(e,A)}\\
\B(A',A) \ar[r]^{\B(A',s)}     & \B(A',B) \ar[r]^{\C(A',r)}       & \B(A',A)
}
$$
whose outer rectangle is a pullback since its top and bottom sides are iso \pbref{prop:pb_iso}, and the middle arrow $\B(e,B)$ is mono since $e$ is $\V$-epi, so by \pbref{prop:pb_canc_mono}, the left square is a pullback.
\end{proof}

The following is an enriched generalization of part of \cite{FrKe}, 2.1.4.
\begin{PropSub} \label{prop:pref_mono_epi}
Let $(\E,\M)$ be a $\V$-prefactorization-system on a $\V$-category $\B$.
\begin{enumerate}
\item If $\B$ has $\V$-kernel-pairs and every retraction in $\B$ lies in $\E$, then $\M \subs \Mono_\V\B$.
\item Dually, if $\B$ has $\V$-cokernel-pairs and every section in $\B$ lies in $\M$, then $\E \subs \Epi_\V\B$.
\end{enumerate}
\end{PropSub}
\begin{proof}
1.  Let $m:B \rightarrow C$ lie in $\M$.  We have a $\V$-pullback as in the following diagram
$$
\xymatrix{
B \ar@/^/@{=}[drr] \ar@/_/@{=}[ddr] \ar@{-->}[dr]|s &   &  \\
                                      & P \ar[d]|{\pi_1} \ar[r]|{\pi_2} \pullbackcorner & B \ar[d]^m \\
                                      & B \ar[r]_m \ar@{-->}[ur]_w                       & C
}
$$
and an induced $s$ such that $\pi_1 \cdot s = 1_B = \pi_2 \cdot s$.  In particular, $\pi_1$ is a retraction and hence lies in $\E$, so there is a unique $w$ making the diagram commute.  Hence $1_B = \pi_2 \cdot s = w \cdot \pi_1 \cdot s = w \cdot 1_B = w$, so $\pi_2 = w \cdot \pi_1 = 1_B \cdot \pi_1 = \pi_1$ and the result follows by \bref{prop:enr_mono_kp}.
\end{proof}

\begin{PropSub} \label{thm:strmono_lowerclass_epi_upperclass}
Let $\B$ be a $\V$-category with $\V$-kernel-pairs.  Then we have the following:
\begin{enumerate}
\item $\StrMono_\V\B = (\Epi_\V\B)^{\downarrow_\V}$.
\item If $\B$ also has $\V$-cokernel-pairs, then $\Epi_\V\B = (\StrMono_\V\B)^{\uparrow_\V}$.
\end{enumerate}
\end{PropSub}
\begin{proof}
Let $(\E,\M) := ((\Epi_\V\B)^{\downarrow_\V\uparrow_\V},(\Epi_\V\B)^{\downarrow_\V})$.  Every retraction in $\B$ is an absolute epi and hence is a $\V$-epi and so lies in $\E$, so by \bref{prop:pref_mono_epi} 1, $\M \subs \Mono_\V\B$ and hence $\StrMono_\V\B = \M \cap \Mono_\V\B = \M$.  If $\B$ also has $\V$-cokernel-pairs, then since every section lies in $\M$ \pbref{prop:sect_str_mono} we deduce by \bref{prop:pref_mono_epi} 2 that $\E \subs \Epi_\V\B$, so $\Epi_\V\B = \E = \M^{\uparrow_\V} = (\StrMono_\V\B)^{\uparrow_\V}$.
\end{proof}

\begin{RemSub} \label{rem:ord_epi_strm_prefactn}
For ordinary categories, it was recognized in \cite{CHK} that 2.1.4 of \cite{FrKe} entails that the epimorphisms and strong monomorphisms of a finitely-complete or -cocomplete ordinary category together constitute a prefactorization system.  \bref{thm:strmono_lowerclass_epi_upperclass} is a partial analogue of this result in the enriched context, but note the additional assumption of $\V$-cokernel-pairs in 2.
\end{RemSub}

\begin{PropSub} \label{thm:ord_str_mono_in_base_is_enriched}
$\StrMono_{\V}\uV = \StrMono\V$, and $\Epi_\V\uV = \Epi\V$.
\end{PropSub}
\begin{proof}
The second equation was established in \bref{thm:lim_mono_epi_in_base_are_enr}.  Since epis in $\V$ are preserved by each left-adjoint functor $V \otimes (-):\V \rightarrow \V$, the class $\E := \Epi\V$ satisfies the hypotheses of \bref{thm:enr_orth_from_ord}, and we compute that
$$(\Epi_\V\uV)^{\downarrow_\V} = (\Epi\V)^{\downarrow_\V} = (\Epi\V)^{\downarrow}$$
and hence by \bref{thm:lim_mono_epi_in_base_are_enr}
$$\StrMono_{\V}\uV = (\Epi_\V\uV)^{\downarrow_\V} \cap \Mono_\V\uV = (\Epi\V)^{\downarrow} \cap \Mono\V = \StrMono\V\;.$$
\end{proof}

The following notion is a $\V$-enriched analogue of the notion of \textit{finitely well-complete (ordinary) category} of \cite{CHK}:
\begin{DefSub} \label{def:enr_fwc}
A $\V$-category $\B$ is \textit{$\V$-finitely-well-complete ($\V$-f.w.c.)} if
\begin{enumerate}
\item $\B$ has all finite $\V$-limits, and
\item $\B$ has $\V$-intersections of arbitrary (class-indexed) families of $\V$-strong-monos.
\end{enumerate}
\end{DefSub}
\begin{RemSub} \label{rem:fwc}
By \bref{thm:stability_and_canc_for_strmonos}, the $\V$-intersections of $\V$-strong-monos required in \bref{def:enr_fwc} are necessarily $\V$-strong-monos.
\end{RemSub}

\begin{PropSub} \label{thm:base_fwc}
$\uV$ is $\V$-finitely-well-complete if and only if $\V$ is finitely well-complete.
\end{PropSub}
\begin{proof}
Finite $\V$-limits, $\V$-intersections, and $\V$-strong-monos in $\uV$ are the same as the corresponding ordinary notions in $\V$ (by \bref{thm:lim_mono_epi_in_base_are_enr}, \bref{thm:ord_str_mono_in_base_is_enriched}).
\end{proof}

\begin{ExaSub}\label{exa:fwc}
Any small-complete category $\B$ whose objects each have but a set of strong subobjects is f.w.c.  Hence, in particular, every category $\B$ topological over $\Set$ is f.w.c., since the strong monomorphisms in $\B$ are exactly the initial injections (e.g. by \cite{Wy:QuTo} 11.9), so that strong subobjects correspond bijectively to subset inclusions.

In particular, then, since each of the cartesian closed categories $\Conv$ and $\Smooth$ \pbref{exa:conv_sm_sp} is topological over $\Set$, each is therefore f.w.c.
\end{ExaSub}

\begin{ExaSub}\label{exa:rmod_fwc}
Given a commutative ring object $R$ in a countably-complete and \linebreak[4] -cocomplete cartesian closed category $\X$, the symmetric monoidal closed category $\sL := \RMod(\X)$ of $R$-module objects in $\X$ \pbref{exa:rmodx_smc_and_xenr} is finitely well-complete as soon as $\X$ is so (\bref{exa:smc_adj_with_l_fwc} 2).  In particular, the categories of convergence vector spaces and smooth vector spaces \pbref{exa:rmodx_smc_and_xenr} are finitely well-complete, in view of \bref{exa:fwc}.
\end{ExaSub}

\section{Enriched factorization systems determined by classes of monos} \label{sec:enr_factn_det_cls_monos}

In view of the stability properties of enriched prefactorization systems \pbref{sec:stab_canc_enr_prefactn}, the following is essentially an enriched variant of Lemma 3.1 of \cite{CHK}.  The assumption of a cotensored $\V$-category allows us to employ \bref{prop:enr_vs_ord_orth_for_pref_sys} within the proof in order to obtain a certain reduction of enriched orthogonality to ordinary and then apply a method of proof analogous to that sketched in \cite{CHK}.

\begin{PropSub} \label{prop:enr_factn_sys_det_lowercls_monos}
Let $(\F,\N)$ be a $\V$-prefactorization-system on a cotensored $\V$-category $\B$ with $\N \subs \Mono_\V\B$, and suppose that $\B$ has arbitrary $\V$-intersections of $\N$-morphisms and $\V$-pullbacks of $\N$-morphisms along arbitrary morphisms.  Then $(\F,\N)$ is a \linebreak[4] \hbox{$\V$-factorization-system} on $\B$.
\end{PropSub}
\begin{proof}
Let $g:C \rightarrow B$ in $\B$.  Let $m:M \rightarrow B$ be the $\V$-intersection of all $n:N_n \rightarrow B$ in $\N$ through which $g$ factors --- i.e. for which there exists a (necessarily unique) $g_n:C \rightarrow N_n$ with $n \cdot g_n = g$.  Then $m \in \N$ by \bref{prop:lowerclass_stab_fp}.  The family of all $g_n:C \rightarrow N_n$ induces a unique $f:C \rightarrow M$ such that $m \cdot f = g$ and $\pi_n \cdot f = g_n$ for all $n$, where the morphisms $\pi_n:M \rightarrow N_n$ present $m$ as a $\V$-intersection of the morphisms $n$.

Hence we have a factorization $m \cdot f$ of $g$ with $m \in \N$, and it remains to show that $f \in \F$.  But since $\B$ is cotensored, we have by \bref{prop:enr_vs_ord_orth_for_pref_sys} that $\F = \N^{\uparrow_\V} = \N^{\uparrow}$, so it suffices to show that $f \in \N^\uparrow$.  Let $i:B_1 \rightarrow B_2$ lie in $\N$.  For any $h,k$ such that $i \cdot k = h \cdot f$ we obtain a commutative diagram
$$
\xymatrix{
  &                       & C \ar@/_2ex/[dddll]_g \ar@/_/[ddl]^{f} \ar@{-->}[d]|w \ar@/^/[ddr]^k & \\
  &                       & P \ar[dl]^{j} \ar[dr]_{h'}                                           & \\
  & M \ar[dl]^m \ar[dr]_h &                                                                      & B_1 \ar[dl]^i \\
B &                       & B_2                                                                  &  
}
$$
in which the diamond with top vertex $P$ is a $\V$-pullback and $w$ is the unique morphism into this pullback making the diagram commute.  Since $i \in \N$, the $\V$-pullback $j$ of $i$ lies in $\N$ by \bref{thm:lowerclass_stab_pb}.  Hence $m \cdot j \in \N$ since $\N$ is closed under composition \pbref{thm:pref_iso_comp_canc}, so $m \cdot j$ is a morphism in $\N$ through which $g$ factors, so from the definition of $m$ as a $\V$-intersection it follows that $j$ is iso (since $m \cdot j \cdot \pi_{m \cdot j} = m$ and hence the mono $j$ is a retraction of $\pi_{m \cdot j}$ and so is iso).  Letting $u := h' \cdot j^{-1}$, we claim that $u$ is the unique morphism with $u \cdot f = k$ and $i \cdot u = h$.  Indeed, $u \cdot f = h' \cdot j^{-1} \cdot f = h' \cdot w = k$, and $i \cdot u = i \cdot h' \cdot j^{-1} = h \cdot j \cdot j^{-1} = h$, and the uniqueness of $u$ is immediate since $i$ is mono.
\end{proof}

\begin{CorSub} \label{thm:enr_factn_sys_det_lowerclass_strmonos}
Let $\B$ be a $\V$-finitely-well-complete cotensored $\V$-category, and let $\Sigma$ be a class of morphisms in $\B$.  Let $\N := \Sigma^{\downarrow_\V} \cap \StrMono_\V\B$.  Then $(\N^{\uparrow_\V},\N)$ is a \mbox{$\V$-factorization-system} on $\B$.
\end{CorSub}
\begin{proof}
Let $\M := \StrMono_\V\B$.  By \bref{thm:strmono_lowerclass_epi_upperclass}, $\M = (\Epi_\V\B)^{\downarrow_\V}$, so
$$\N = \Sigma^{\downarrow_\V} \cap \M = \Sigma^{\downarrow_\V} \cap (\Epi_\V\B)^{\downarrow_\V} = (\Sigma \cup \Epi_\V\B)^{\downarrow_\V}\;.$$
Hence, letting $\sH := \Sigma \cup \Epi_\V\B$, $(\N^{\uparrow_\V},\N) = (\sH^{\downarrow_\V\uparrow_\V},\sH^{\downarrow_\V})$ is a $\V$-prefactorization-system on $\B$.  Invoking \bref{prop:enr_factn_sys_det_lowercls_monos} w.r.t $(\N^{\uparrow_\V},\N)$, the needed conclusion is obtained.
\end{proof}

\begin{CorSub} \label{thm:fwc_plus_cokerpairs_implies_epistrmono_factn_sys}
Let $\B$ be a $\V$-finitely-well-complete cotensored $\V$-category with $\V$-cokernel-pairs.  Then $(\Epi_\V\B,\StrMono_\V\B)$ is a $\V$-factorization-system on $\B$.
\end{CorSub}
\begin{proof}
Taking $\Sigma = \Iso \B$ in \bref{thm:enr_factn_sys_det_lowerclass_strmonos}, we find that $((\StrMono_\V\B)^{\uparrow_\V},\StrMono_\V\B)$ is a \linebreak[4] \hbox{$\V$-factorization-system}, but by \bref{thm:strmono_lowerclass_epi_upperclass} we have that $(\StrMono_\V\B)^{\uparrow_\V} = \Epi_\V\B$.
\end{proof}

\begin{CorSub} \label{thm:str_img_factns_for_fwc_base}
Suppose that $\V$ is finitely well-complete.  Then $(\Epi_\V\uV,\StrMono_\V\uV) = (\Epi\V,\StrMono\V)$ is a $\V$-factorization-system on $\uV$.
\end{CorSub}
\begin{proof}
$\uV$ is $\V$-finitely-well-complete by \bref{thm:base_fwc} and is also cotensored, so as in the proof of \bref{thm:fwc_plus_cokerpairs_implies_epistrmono_factn_sys} we find that $((\StrMono_\V\uV)^{\uparrow_\V},\StrMono_\V\uV)$ is a $\V$-factorization-system on $\uV$.  By \bref{thm:ord_str_mono_in_base_is_enriched}, we have that $(\Epi_\V\uV,\StrMono_\V\uV) = (\Epi\V,\StrMono\V)$, and by \bref{rem:ord_epi_strm_prefactn}, \linebreak[4] $(\Epi\V,\StrMono\V)$ is an (ordinary) prefactorization system on $\V$.  Hence $(\StrMono_\V\uV)^{\uparrow_\V} = (\StrMono\V)^{\uparrow_\V} \subs (\StrMono\V)^\uparrow = \Epi\V = \Epi_\V\uV$, so in fact $(\StrMono_\V\uV)^{\uparrow_\V} = \Epi_\V\uV$.
\end{proof}

\begin{DefSub} \label{def:image}
Let $\B$ be a $\V$-category.
\begin{enumerate}
\item We say that $\B$ \textit{has $\V$-strong image factorizations} if $(\Epi_\V\B,\StrMono_\V\B)$ is a $\V$-factorization-system on $\B$.
\item In the present text, if $\B$ has $\V$-strong image factorizations, then we shall call the second factor of the $(\Epi_\V\B,\StrMono_\V\B)$-factorization of a given morphism $f$ the \textit{image} of $f$ (provided the choice of $\V$ is clear from the context).
\end{enumerate}
\end{DefSub}

\begin{ExaSub}\label{exa:rmod_factn_systems}
Let $R$ be a commutative ring object in a countably-complete and \linebreak[4] -cocomplete, finitely well-complete cartesian closed category $\X$, and let $\sL$ be the symmetric monoidal closed category $\RMod(\X)$ of $R$-module objects in $\X$ \pbref{exa:rmodx_smc_and_xenr}.  For suitable choices $\X, R$, we obtain as $\sL$ the categories of convergence vector spaces and smooth vector spaces \pbref{exa:rmodx_smc_and_xenr}.  As noted in \pbref{exa:rmod_fwc}, $\sL$ is finitely well-complete; hence the cotensored $\sL$-category $\uL$ is $\sL$-finitely-well-complete \pbref{thm:base_fwc}.
\begin{enumerate}
\item By \bref{thm:str_img_factns_for_fwc_base}, $\uL$ has $\sL$-strong image factorizations, so $(\Epi\sL,\StrMono\sL)$ is an \linebreak[4] $\sL$-factorization-system on $\uL$.
\item Letting $\Sigma = \Sigma_H$ be the class of all morphisms $h:E_1 \rightarrow E_2$ in $\sL$ such that $\uL(h,R):\uL(E_2,R) \rightarrow \uL(E_1,R)$ is an isomorphism in $\sL$, we obtain by \bref{thm:enr_factn_sys_det_lowerclass_strmonos} an $\sL$-factorization-system on $\uL$ written as $(\Dense{\Sigma_H},\ClEmb{\Sigma_H})$, which we employ in Chapters \bref{ch:nat_acc_dist}, \bref{ch:dcompl_vint}.  We call the elements of $\Dense{\Sigma_H}$ \textit{functionally dense} morphisms and those of $\ClEmb{\Sigma_H}$ \textit{functionally closed embeddings}; see \bref{def:func_cl_dens}.
\end{enumerate}
\end{ExaSub}

\section{Orthogonality of morphisms and objects in the presence of an adjunction} \label{sec:orth_pres_adj}

The following notion of orthogonality in the enriched context was employed in \cite{Day:AdjFactn}.

\begin{DefSub} \label{def:orth_subcat_sigma}
\emptybox
\begin{enumerate}
\item For a morphism $f:A_1 \rightarrow A_2$ in a $\V$-category $\B$ and an object $B$ in $\B$, we say that $f$ is \textit{$\V$-orthogonal} to $B$, written $f \bot_\V B$, if \mbox{$\B(f,B):\B(A_2,B) \rightarrow \B(A_1,B)$} is an isomorphism in $\V$.
\item Given a class of morphisms $\Sigma$ in a $\V$-category $\B$, we let $\Sigma^{\bot_\V} := \{B \in \ob\B \:|\; \forall f \in \Sigma \::\: f \bot_\V B\}$.  We let $\B_\Sigma$ be the full sub-$\V$-category of $\B$  whose objects are those in $\Sigma^{\bot_\V}$.
\item Given a class $\sO$ of objects in a $\V$-category $\B$, we let $\sO^{\top_\V} := \{f \in \mor\B \:|\: \forall B \in \sO \::\: f \bot_\V B\}$.
\item Given a functor $P:\B \rightarrow \C$, we denote by $\Sigma_P$ the class of all morphisms in $\B$ inverted by $P$ (i.e. sent to an isomorphism in $\C$).
\end{enumerate}
\end{DefSub}

\begin{RemSub}
If $\B$ has a $\V$-terminal object $1$, then it is easy to show that $f \bot_\V B$ iff $f \downarrow_\V !_B$, where $!_B:B \rightarrow 1$.  One can more generally consider orthogonality of a morphism to a source or cone, thus generalizing both `morphism-morphism' and `morphism-object' orthogonality even in the absence of a terminal object \cite{He:TopFunc,Th:FactConesFunc}.
\end{RemSub}

\begin{PropSub}\label{thm:orth_to_all_is_iso}
For a $\V$-category $\B$, $(\ob\B)^{\top_\V} = \Iso \B = (\mor\B)^{\uparrow_\V}$.
\end{PropSub}
\begin{proof}
$\Iso\B$ is clearly included in both the rightmost and leftmost classes.  Also, the inclusion $(\mor\B)^{\uparrow_\V} \subs \Iso\B$ follows from \bref{thm:orth_yielding_retr_and_iso}.  Lastly, if $h:B_1 \rightarrow B_2$ lies in $(\ob\B)^{\top_\V}$, then the $\V$-natural transformation $\B(h,-):\B(B_2,-) \rightarrow \B(B_1,-)$ is an isomorphism; but by the weak Yoneda lemma (\cite{Ke:Ba} 1.9), the (ordinary) functor \hbox{$Y:\B^\op \rightarrow \VCAT(\B,\uV)$} given by $YB = \B(B,-)$ is fully-faithful, so $h$ is iso.
\end{proof}

\begin{PropSub} \label{thm:orth_of_morphs_in_orth_subcat}
Let $\B$ be a $\V$-category and $\Sigma$ a class of morphisms in $\B$.  Then for any morphisms $e:A_1 \rightarrow A_2$ in $\Sigma$ and $m:B_1 \rightarrow B_2$ in $\B_\Sigma$, we have that $e \downarrow_\V m$ in $\B$.
\end{PropSub}
\begin{proof}
Since $e \bot_\V B_1$ and $e \bot_\V B_2$, the left and right sides of the commutative square \eqref{eqn:orth_pb} are isomorphisms, so the square is a pullback.
\end{proof}

In the non-enriched context, the first of the following equivalences appears in the proof of Lemma 4.2.1 of \cite{FrKe}, and variants of both equivalences are given in \cite{Pum}.

\begin{PropSub}\label{prop:orth_and_adj}
Let $P \dashv Q:\C \rightarrow \B$ be a $\V$-adjunction, and let $f:B_1 \rightarrow B_2$ be a morphism in $\B$.
\begin{enumerate}
\item For any morphism $g:C_1 \rightarrow C_2$ in $\C$,
$$Pf \downarrow_\V g \;\;\Longleftrightarrow\;\; f \downarrow_\V Qg\;.$$
\item For any object $C$ of $\C$,
$$Pf \bot_\V C \;\;\Longleftrightarrow\;\; f \bot_\V QC\;.$$
\end{enumerate}
\end{PropSub}
\begin{proof}
1.  Via the given $\V$-adjunction, the commutative diagram
$$
\xymatrix {
\C(PB_2,C_1) \ar[rr]^{\C(PB_2,g)} \ar[d]_{\C(Pf,C_1)} & & \C(PB_2,C_2) \ar[d]^{\C(Pf,C_2)} \\
\C(PB_1,C_1) \ar[rr]^{\C(PB_1,g)}                     & & \C(PB_1,C_2)
}
$$
is isomorphic to the commutative diagram
$$
\xymatrix {
\B(B_2,QC_1) \ar[rr]^{\B(B_2,Qg)} \ar[d]_{\B(f,QC_1)} & & \B(B_2,QC_2) \ar[d]^{\B(f,QC_2)} \\
\B(B_1,QC_1) \ar[rr]^{\B(B_1,Qg)}                     & & {\B(B_1,QC_2)\;.}
}
$$
2.  $\C(Pf,C) \cong \B(f,QC)$ in the arrow category $[\Two,\V]$.
\end{proof}

In the non-enriched setting, the first of the following equations is noted in \cite{CHK} 3.3 and specializes \cite{FrKe} 4.2.1.

\begin{CorSub} \label{thm:charns_of_sigma_p}
Let $P \dashv Q:\C \rightarrow \B$ be a $\V$-adjunction.  Then 
$$(Q(\mor\C))^{\uparrow_\V} = \Sigma_P = (Q(\ob\C))^{\top_\V}\;.$$
Hence, in particular, $(\Sigma_P,(\Sigma_P)^{\downarrow_\V})$ is a $\V$-prefactorization-system on $\B$.
\end{CorSub}
\begin{proof}
By \bref{prop:orth_and_adj} and \bref{thm:orth_to_all_is_iso}, we may compute as follows:
\begin{description}
\item[] $(Q(\mor\C))^{\uparrow_\V} = P^{-1}((\mor\C)^{\uparrow_\V}) = P^{-1}(\Iso\C) = \Sigma_P$;  
\item[] $(Q(\ob\C))^{\top_\V} = P^{-1}((\ob\C)^{\top_\V}) = P^{-1}(\Iso\C) = \Sigma_P$.
\end{description}
\end{proof}

Clearly any sub-$\V$-category of $\B$ of the form $\B_\Sigma$ is replete.  The following proposition shows that every $\V$-reflective-subcategory of $\B$ is of the form $\B_\Sigma$ for each of two canonical choices of $\Sigma$:

\begin{PropSub} \label{thm:refl_subcats_are_orth_subcats}
Let $K \nsststile{}{\rho} J : \C \hookrightarrow \B$ be a $\V$-reflection.
\begin{enumerate}
\item $\C = \B_{\Sigma_K} = \B_{\Sigma}$, where $\Sigma := \{\rho_B\;|\;B \in \B\}$.
\item $(\mor\C)^{\uparrow_\V} = \Sigma_K = (\ob\C)^{\top_\V}$, where the right- and leftmost expressions are evaluated with respect to $\B$.
\item $(\Sigma_K,(\Sigma_K)^{\downarrow_\V})$ is a $\V$-prefactorization-system on $\B$.
\end{enumerate}
\end{PropSub}
\begin{proof}
2 and 3 follow from \bref{thm:charns_of_sigma_p}.  Regarding 1, first observe that $\Sigma \subs \Sigma_K$, so that $\B_{\Sigma_K} \subs \B_\Sigma$.  Also, by 2, $\Ob\B_{\Sigma_K} = (\Sigma_K)^{\bot_\V} = (\Ob\C)^{\top_\V\bot_\V} \sups \Ob\C$.  Hence it now suffices to show $\B_\Sigma \subs \C$.  Suppose $B \in \B_\Sigma$.  Since we also have that $KB \in \C \subs \B_\Sigma$, the morphism $\rho_B:B \rightarrow KB$ lies in $\B_\Sigma$, so since $\rho_B \in \Sigma$ we deduce by \bref{thm:orth_of_morphs_in_orth_subcat} that $\rho_B \downarrow_\V \rho_B$.  Hence by \bref{thm:orth_yielding_retr_and_iso}, $\rho_B$ is iso, so $B \in \C$.
\end{proof}

%% file: compl_cl_dens_orth.tex
\chapter{Notions of completeness via enriched orthogonality}\label{ch:compl_enr_orth}
\setcounter{subsection}{0}

In functional analysis, an important role is played by \textit{completeness} properties of normed or topological vector spaces, including the usual Cauchy-completeness and variations thereupon.  But if an abstract functional analysis is to be developed in any suitable cartesian closed category $\X$, through the study of the symmetric monoidal closed category $\sL$ of $R$-module objects in $\X$, or \textit{linear spaces}, then the usual manner of defining completeness in terms of Cauchy nets or filters is not available.  However, the endofunctor $\sL \rightarrow \sL$ sending each linear space $E$ to its \textit{double-dual} $E^{**}$ is part of an $\sL$-enriched monad $\HH$ \pbref{def:dualization}, and our approach in \bref{sec:compl} is to form the \textit{idempotent} $\sL$-monad $\tHH$ induced by $\HH$ \pbref{sec:assoc_idm_mnd}, calling the objects of the associated reflective subcategory \textit{functionally complete}.

In the non-enriched context, the idempotent monad $\tTT$ associated to a given monad $\TT$ on a suitable category $\B$ was constructed by Fakir \cite{Fak}, and in the enriched context one can form it by means of a result of Day \cite{Day:AdjFactn}.  $\tTT$ is terminal among idempotent $\V$-monads $\SSS$ for which a (necessarily unique) morphism $\SSS \rightarrow \TT$ exists \pbref{thm:ind_idmmnd_and_refl}.  The associated reflective subcategory consists of those objects of $\B$ to which the class $\Sigma_T$ of morphisms inverted by $T$ is orthogonal, and we call these objects \textit{$\TT$-complete} \pbref{def:tcomplete}.  Notably, every $\TT$-complete object $B$ embeds into $TB$, as the unit component $B \rightarrow TB$ is necessarily a $\V$-strong-monomorphism; we say that $B$ is \textit{$\TT$-separated}.  Day's results entail also that for \textit{any} class of morphisms $\Sigma \subs \Sigma_T$, the $\TT$-separated objects to which $\Sigma$ is orthogonal form a reflective subcategory of $\B$, whose objects we call the \textit{$\Sigma$-complete} $\TT$-separated objects \pbref{def:compl_sep_closed_dense}.

By varying the choice of $\Sigma$ we thus have a means of defining notions of completeness that are weaker than $\TT$-completeness and tailored to a specific application.  This approach is applied in Chapter \bref{ch:dcompl_vint} in defining a notion of completeness in linear spaces that is weaker than functional completeness and tantamount to the existence of certain Pettis-type vector-valued integrals \pbref{thm:dist_compl_sep_vs_pett}.

In the present chapter, we develop the theory of $\Sigma$-complete $\TT$-separated objects and associated idempotent monads by means of enriched factorization systems \pbref{ch:enr_orth_fact}.  Each class $\Sigma \subs \Sigma_T$ determines an associated enriched factorization system $(\Dense{\Sigma},\ClEmb{\Sigma})$ whose component classes determine notions of \textit{density} and \textit{closure} relative to $\Sigma$ and compatible with $\Sigma$-completeness \pbref{def:compl_sep_closed_dense}.

Any enriched adjunction $P \dashv Q$ inducing $\TT$ factors through the reflective subcategory of $\Sigma$-complete $\TT$-separated objects.  Further, the $\TT$-complete objects form the \textit{smallest} reflective subcategory through which the given adjunction factors \pbref{thm:ind_idmmnd_and_refl}, so that in this case one obtains a factorization of the left adjoint $P$ as a reflection followed by a \textit{conservative} left adjoint \pbref{thm:adj_factn}.  For a given adjunction $P \dashv Q$, an approach to proving in the enriched context an \textit{adjoint factorization} result of this type, analogous to a result of Applegate-Tierney in the non-enriched context, was the explicit objective of Day's paper \cite{Day:AdjFactn}, yet a full proof of the result is not provided therein.  Rather, such a result can be proved through a somewhat involved two-step process on the basis of key lemmas that Day does provide, as was done by the author in a recent paper \cite{Lu:Fu}.  

In the present chapter, we avoid this two-step process, preferring to study directly the interplay of the notions we have termed $\TT$-separatedness and $\Sigma$-completeness in the given category $\B$, developing along the way the related notions of $\Sigma$-density and $\Sigma$-closure in $\B$ whose functional-analytic relevance will be evident in \bref{sec:compl}, \bref{sec:acc_distn}, and \bref{sec:distnl_compl_pett_for_acc_distns}.  We work relative to a given monad $\TT$ rather than an adjunction, since most of the theory is independent of the choice of inducing adjunction.  We make full use of the theory of enriched factorization systems \pbref{ch:enr_orth_fact}, bringing into the enriched context certain techniques used by Cassidy-H\'ebert-Kelly \cite{CHK}.

\section{Relative notions of completion, closure, and \hbox{density}} \label{sec:cmpl_cl_dens}

As a preface to the present section, we note the following basic fact.

\begin{PropSub}\label{thm:all_ladjs_ind_given_monad_inv_same_morphs}
Given an adjunction $P \nsststile{}{\eta} Q:\C \rightarrow \B$ with induced endofunctor $T$ on $\B$, $\Sigma_P = \Sigma_T$.  Hence all left adjoints inducing a given monad $\TT$ invert the same morphisms.
\end{PropSub}
\begin{proof}
One inclusion is immediate.  For the other, suppose $f:B \rightarrow B'$ in $\B$ is inverted by $T$.  Then $Tf$ has an inverse $(Tf)^{-1}:TB' \rightarrow TB$, and one easily shows that the transpose $PB' \rightarrow PB$ of the composite $B' \xrightarrow{\eta_{B'}} QPB' \xrightarrow{(Tf)^{-1}} QPB$ under the given adjunction serves as inverse for $Pf$.
\end{proof}

\begin{ParSub} \label{par:data_enr_orth_subcat_subord_adj}
In the present section, we work with given data as follows:
\begin{enumerate}
\item Let $\TT = (T,\eta,\mu)$ be a $\V$-monad on a $\V$-finitely-well-complete and cotensored $\V$-category $\B$.
\item Let $\Sigma \subs \Sigma_T$ be a class of morphisms inverted by $T$.
\end{enumerate}
Choose any $\V$-adjunction $P \nsststile{\varepsilon}{\eta} Q : \C \rightarrow \B$ inducing $\TT$.  By \bref{thm:all_ladjs_ind_given_monad_inv_same_morphs} we know that $\Sigma \subs \Sigma_T = \Sigma_P$.
\end{ParSub}

\begin{RemSub}
For the present section, one could equivalently (by \bref{thm:all_ladjs_ind_given_monad_inv_same_morphs}) work with a given $\V$-adjunction on $\B$ and a class of morphisms inverted by the left adjoint.  But most of the definitions and theorems in the present section are, by their form, clearly independent of the choice of $\V$-adjunction.  Further, our notation shall not permit any implicit dependence on this choice.  As a definite choice, one may employ the Kleisli $\V$-adjunction or the Eilenberg-Moore $\V$-adjunction for $\TT$.
\end{RemSub}

As per \bref{def:enr_str_mon}, we shall call the $\V$-strong-monos of $\B$ \textit{embeddings} in $\B$.  We let $\M := \StrMono_\V\B$.

\begin{DefSub} \label{def:compl_sep_closed_dense}
Given data as in \bref{par:data_enr_orth_subcat_subord_adj}, we make the following definitions:
\begin{enumerate}
\item An object $B$ of $\B$ is \textit{$\Sigma$-complete} if $B \in \B_\Sigma$.
\item An object $B$ of $\B$ is \textit{$\TT$-separated} if $\eta_B:B \rightarrow TB$ is an embedding.
\item We denote the full sub-$\V$-category of $\B$ consisting of the $\Sigma$-complete $\TT$-separated objects by $\B_{(\TT,\Sigma)}$.
\item We say that an embedding $m:B_1 \rightarrowtail B_2$ in $\B$ is \textit{$\Sigma$-closed} if $m \in \Sigma^{\downarrow_\V}$.  We denote by $\SigmaClEmb := \Sigma^{\downarrow_\V} \cap \M$ the class of all $\Sigma$-closed embeddings in $\B$.
\item We say that a morphism in $\B$ is \textit{$\Sigma$-dense} if it lies in $\SigmaDense := \SigmaClEmb^{\uparrow_\V}$.
\end{enumerate}
\end{DefSub}

\begin{PropSub} \label{thm:dense_closed_factn_sys}
$(\SigmaDense,\SigmaClEmb)$ is a $\V$-factorization-system on $\B$.
\end{PropSub}
\begin{proof}
Apply \bref{thm:enr_factn_sys_det_lowerclass_strmonos}.
\end{proof}

\begin{ExaSub}\label{exa:func_distnl_compl}
Let $R$ be a commutative ring object in a finitely well-complete, countably-complete and -cocomplete cartesian closed category $\X$, and let $\sL := \RMod(\X)$ be the symmetric monoidal closed category of $R$-module objects in $\X$ \pbref{exa:rmodx_smc_and_xenr}.  Particular examples include the categories $\sL$ of convergence vector spaces and smooth vector spaces \pbref{exa:rmodx_smc_and_xenr}.  There is an $\sL$-monad $\HH$ \pbref{def:dualization} whose component $\sL$-functor $H:\uL \rightarrow \uL$ sends each $R$-module object $E$ to its \textit{double-dual} $E^{**} = \uL(\uL(E,R),R)$.  
\begin{enumerate}
\item Taking $\Sigma := \Sigma_H$, the data $\uL$, $\HH$, $\Sigma$ satisfy the assumptions of \bref{par:data_enr_orth_subcat_subord_adj}, by \bref{exa:rmod_fwc}.  The resulting notions obtained via \bref{def:compl_sep_closed_dense} are employed in Chapters \bref{ch:nat_acc_dist} and \bref{ch:dcompl_vint} under the names of \textit{functionally complete} and \textit{functionally separated} object, \textit{functionally closed embedding}, and \textit{functionally dense} morphism; see \bref{sec:compl}.
\item By instead taking $\Sigma$ to be a certain subclass $\tDelta$ of $\Sigma_H$, in \bref{def:distnl_compl_cl_dens} we obtain the notions of \textit{$\tDD$-distributionally complete} object, \textit{$\tDD$-distributionally closed} embedding, and \textit{$\tDD$-distributionally dense} morphism.
\end{enumerate}
\end{ExaSub}

\begin{PropSub}\label{thm:props_sigma_cl_dens}\emptybox
\begin{enumerate}
\item Every $\V$-epi $e$ in $\B$ is $\Sigma$-dense.
\item Every morphism in $\Sigma$ is $\Sigma$-dense.
\item If a composite $B_1 \xrightarrow{f} B_2 \xrightarrow{g} B_3$ is $\Sigma$-dense, then its second factor $g$ is $\Sigma$-dense.
\item If a composite $B_1 \xrightarrow{f} B_2 \xrightarrow{g} B_3$ and its second factor $g$ are both $\Sigma$-closed embeddings, then the first factor $f$ is a $\Sigma$-closed embedding.
\item Every $\Sigma$-dense $\Sigma$-closed embedding is an isomorphism.
\end{enumerate}
\end{PropSub}
\begin{proof}
1.  $e$ is $\V$-orthogonal to every embedding and hence to every $\Sigma$-closed embedding.
2.  $\ClEmb{\Sigma} \subs \Sigma^{\downarrow_\V}$, so $\Sigma \subs \Sigma^{\downarrow_\V\uparrow_\V} \subs \ClEmb{\Sigma}^{\uparrow_\V} = \Dense{\Sigma}$.
3.  Since every $\Sigma$-closed embedding is $\V$-mono, this follows from (the dual of) statement 4 of \bref{thm:pref_iso_comp_canc}.
4.  This follows from statement 3 of \bref{thm:pref_iso_comp_canc}.
5.  This follows from \bref{thm:int_of_upper_and_lower_classes_is_iso}.
\end{proof}

\begin{DefSub} \label{def:sigma_closure}
Given an embedding $m:M \rightarrowtail B$ in $\B$, the \textit{$\Sigma$-closure} of $m$ is the $\Sigma$-closed embedding $\overline{m}$ gotten by taking the $(\SigmaDense,\SigmaClEmb)$-factorization
$$
\xymatrix{
M \ar@{ >->}[rr]^{m} \ar@{ >->}[dr] &                                 & B\\
                                    & \overline{M} \ar@{ >->}[ur]_{\overline{m}} &
}
$$
of $m$.  (The first factor of this factorization is also an embedding by \bref{thm:stability_and_canc_for_strmonos}.)
\end{DefSub}

\begin{RemSub}
Given embeddings $m:M \rightarrowtail B$ and $c:C \rightarrowtail B$ in a $\V$-category, one finds that following are equivalent:
\begin{enumerate}
\item $c$ is a $\Sigma$-closed embedding through which $m$ factors via a $\Sigma$-dense morphism.
\item $c \cong \overline{m}$ in $\Sub(B)$.
\item In the terminology of \bref{def:emb}, $c$ presents $C$ as the $\Sigma$-closure of $m$.
\end{enumerate}
\end{RemSub}

\begin{PropSub}
Let $B$ be an object of $\B$.  Let $\SigmaClSub(B)$ be the full subcategory of the preorder $\Sub(B)$ \pbref{def:emb} with objects all $\Sigma$-closed embeddings.  Then the inclusion functor $\SigmaClSub(B) \hookrightarrow \Sub(B)$ has a left adjoint given by taking the $\Sigma$-closure.  Hence the assignment $m \mapsto \overline{m}$ defines a monad (i.e. a closure operator) on $\Sub(B)$.
\end{PropSub}
\begin{proof}
We must show that for any embedding $m:M \rightarrowtail B$ and any $\Sigma$-closed embedding $c:C \rightarrowtail B$,
$$\overline{m} \lt c \;\;\Leftrightarrow\;\; m \lt c\;.$$
The implication $\Rightarrow$ holds since $m \lt \overline{m}$.  For the converse, if $m \lt c$ then $m = c \cdot b$ for some morphism $b$, so that since $m = \overline{m} \cdot d$ for a $\Sigma$-dense morphism $d$ we obtain a commutative square as in
$$
\xymatrix{
M \ar@{ >->}[d]_d \ar@{ >->}[r]^b            & C \ar@{ >->}[d]^c  \\
\overline{M} \ar@{ >->}[r]_{\overline{m}} \ar@{-->}[ur]^w & B
}
$$
with $d \downarrow c$, so there is a unique diagonal $w$ making the diagram commute, whence $\overline{m} \lt c$.
\end{proof}

\begin{PropSub} \label{thm:closure_of_image}
Suppose that $\B$ has $\V$-strong image factorizations \pbref{def:image}.  Let $f:B_1 \rightarrow B_2$ be a morphism in $\B$ with $(\SigmaDense,\SigmaClEmb)$-factorization $B_1 \xrightarrow{d} C \overset{c}{\rightarrowtail} B_2$.  Then, in the terminology of \bref{def:emb}, $c$ presents $C$ as the $\Sigma$-closure of the image of $f$.
\end{PropSub}
\begin{proof}
Letting $B_1 \xrightarrow{e} M \overset{m}{\rightarrowtail} B_2$ be the $(\Epi_\V\B,\StrMono_\V\B)$-factorization of $f$, the $\Sigma$-closure $\overline{m}$ of the image $m$ of $f$ is the second factor of the $(\SigmaDense,\SigmaClEmb)$-factorization $M \xrightarrow{j} \overline{M} \overset{\overline{m}}{\rightarrowtail} B_2$ of $m$.  But by \bref{thm:props_sigma_cl_dens} 1, $e$ is $\Sigma$-dense, so $j \cdot e$ is $\Sigma$-dense and hence $\overline{m}$ is the second factor of a $(\SigmaDense,\SigmaClEmb)$-factorization $f = \overline{m} \cdot (j \cdot e)$ of $f$.  The result follows.
\end{proof}

\begin{PropSub} \label{thm:compl_vs_clemb}
Let $m:B' \rightarrowtail B$ be an embedding in $\B$, and suppose $B$ is $\Sigma$-complete.
Then
$$\text{$B'$ is $\Sigma$-complete}\;\;\Leftrightarrow\;\;\text{$m$ is $\Sigma$-closed}\;.$$
\end{PropSub}
\begin{proof}
For each $h:B_1 \rightarrow B_2$ in $\Sigma$, we have a commutative square
$$
\xymatrix{
\B(B_2,B') \ar[r]^{\B(h,B')} \ar[d]_{\B(B_2,m)} & \B(B_1,B') \ar[d]^{\B(B_1,m)} \\
\B(B_2,B) \ar[r]_{\B(h,B)}                      & \B(B_1,B) 
}
$$
in which $\B(h,B)$ is iso, so by \bref{prop:pb_iso}, the square is a pullback if and only if $\B(h,B')$ is iso. 
\end{proof}

\begin{PropSub} \label{thm:tc_complete}
Let $C$ be an object of $\C$.  Then $QC$ is $\Sigma$-complete.
\end{PropSub}
\begin{proof}
For each $h:B_1 \rightarrow B_2$ in $\Sigma$, we have a commutative square
$$
\xymatrix{
\B(B_2,QC) \ar[r]^{\B(h,QC)} \ar[d]_\wr & \B(B_1,QC) \ar[d]^{\wr} \\
\C(PB_2,C) \ar[r]^{\sim}_{\C(Ph,C)}      & \C(PB_1,C) 
}
$$
whose left and right sides are isomorphisms, and since $h \in \Sigma \subs \Sigma_P$, the bottom side is iso, so the top side is iso.
\end{proof}

\begin{PropSub} \label{thm:sepobj_compl_iff_eta_closed}
Let $B \in \B$ be $\TT$-separated.  Then
$$\text{$B$ is $\Sigma$-complete}\;\;\Leftrightarrow\;\;\text{$\eta_B:B \rightarrowtail TB$ is $\Sigma$-closed}\;.$$
\end{PropSub}
\begin{proof}
Since $TB = QPB$ is $\Sigma$-complete by \bref{thm:tc_complete} and $\eta_B$ is an embedding, this follows from \bref{thm:compl_vs_clemb}.
\end{proof}

\begin{PropSub} \label{thm:sep_iff_exists_emb}
An object $B \in \B$ is $\TT$-separated iff there exists an embedding \mbox{$m:B \rightarrowtail QC$} for some $C \in \C$.
\end{PropSub}
\begin{proof}
One implication is immediate.  For the other, observe that if $m:B \rightarrowtail QC$ is an embedding then we have a commutative triangle
$$
\xymatrix{
B \ar[r]^{\eta_B} \ar@{ >->}[rd]_m & QPB \ar@{-->}[d]^{Qm^\sharp} \\
                                  & QC                
}
$$
for a unique morphism $m^\sharp$ in $\C$, so $\eta_B \in \M$ by \bref{thm:stability_and_canc_for_strmonos}.
\end{proof}

\begin{CorSub}\label{thm:carriers_of_free_talgs_sep_compl}
For each $B \in \B$, $TB$ is $\Sigma$-complete and $\TT$-separated. 
\end{CorSub}
\begin{proof}
Apply \bref{thm:tc_complete} and \bref{thm:sep_iff_exists_emb}.
\end{proof}

\begin{DefSub} \label{def:completion}
For each $B \in \B$, let 
$$
\xymatrix{
B \ar[rr]^{\eta_B} \ar[dr]_{\rho_B} &                        & TB \\
                                    & KB \ar@{ >->}[ur]_{i_B} & 
}
$$
be the $(\SigmaDense,\SigmaClEmb)$-factorization of $\eta_B$.
\end{DefSub}

\begin{PropSub} \label{thm:kb_sep_compl}
For each $B \in \B$, $KB$ is $\Sigma$-complete and $\TT$-separated.
\end{PropSub}
\begin{proof}
Since $i_B:KB \rightarrowtail TB$ is a $\Sigma$-closed embedding and $TB$ is $\Sigma$-complete, $KB$ is $\Sigma$-complete by \bref{thm:compl_vs_clemb}, and $KB$ is $\TT$-separated by \bref{thm:sep_iff_exists_emb}. 
\end{proof}

The following lemma and its proof stem from an idea employed in the proof of 3.3 of \cite{CHK} in the non-enriched context.  However, in proving the consequent enriched reflection and adjoint-factorization theorems \bref{thm:sigma_s_complete_are_sep_and_refl}, \bref{thm:adj_factn}, we do not employ the overall approach used in \cite{CHK} in proving the analogous non-enriched results (3.3 and 3.5 there), since we rather must also prove the more general results \bref{thm:refl_orth} and \bref{thm:adj_factn_for_class_of_morphs}, for which no analogue exists in \cite{CHK}.  Rather, the first of these \pbref{thm:refl_orth} is an augmented variant of 2.3 of \cite{Day:AdjFactn}.

\begin{LemSub} \label{thm:dense_and_sent_to_section_implies_iso}
Let $f:B_1 \rightarrow B_2$ be a $\Sigma$-dense morphism for which $Pf:PB_1 \rightarrow PB_2$ is a section.  Then $Pf$ is iso.
\end{LemSub}
\begin{proof}
Form the following commutative diagram
$$
\xymatrix{
B_1 \ar@/^/[drr]^{\eta_{B_1}} \ar@{-->}[dr]|w \ar@/_/[ddr]_f & & \\
                                                         & B \ar[d]_v \ar[r]^u \pullbackcorner & QPB_1 \ar[d]^{QPf} \\
                                                         & B_2 \ar[r]^{\eta_{B_2}}               & {QPB_2}
}
$$
where the square is a $\V$-pullback.  $QPf$ is a section and hence is an embedding (by \bref{prop:sect_str_mono}).  Also, by \bref{thm:charns_of_sigma_p}, we have that $QPf \in \Sigma_P^{\downarrow_\V} \subs \Sigma^{\downarrow_\V}$ (using the fact that $\Sigma \subs \Sigma_P$), so $QPf$ is a $\Sigma$-closed embedding.  Hence since $(\SigmaDense,\SigmaClEmb)$ is a $\V$-factorization-system (by \bref{thm:dense_closed_factn_sys}), we deduce by \bref{thm:lowerclass_stab_pb} that the $\V$-pullback $v$ of $QPf$ is also a $\Sigma$-closed embedding.  But $f = v \cdot w$ is $\Sigma$-dense and hence $v \cdot w \downarrow_\V v$, and we deduce by \bref{thm:orth_yielding_retr_and_iso} that the monomorphism $v$ is a retraction and hence an isomorphism.

Hence we may assume without loss of generality that $v = 1_{B_2}$ (and $B = B_2$).  Then
$$\left(B_2 \xrightarrow{u} QPB_1 \xrightarrow{QPf} QPB_2\right) = \eta_{B_2}\;,$$
and, taking the transposes of each side of this equation under the adjunction $P \nsststile{\varepsilon}{\eta} Q$, we obtain an equation
$$\left(PB_2 \xrightarrow{u^\sharp} PB_1 \xrightarrow{Pf} PB_2\right) = 1_{PB_2}$$
where $u^\sharp$ is the transpose of $u$, so $Pf$ is a retraction.  But by assumption $Pf$ is also a section, so $Pf$ is iso.
\end{proof}

\begin{PropSub} \label{thm:rho_inverted_by_s}
For each $B \in \B$, $\rho_B:B \rightarrow KB$ is inverted by $P$.
\end{PropSub}
\begin{proof}
Taking the transposes of each side of the equation
$$\left(B \xrightarrow{\rho_B} KB \xrightarrow{i_B} QPB\right) = \eta_B$$
under the adjunction $P \nsststile{\varepsilon}{\eta} Q$, we obtain
$$\left(PB \xrightarrow{P\rho_B} PKB \xrightarrow{Pi_B} PQPB \xrightarrow{\varepsilon_{PB}} PB\right) = 1_{PB}\;,$$
so $P\rho_B$ is a section, so since $\rho_B$ is $\Sigma$-dense, \bref{thm:dense_and_sent_to_section_implies_iso} applies, and we deduce that $P\rho_B$ is iso.
\end{proof}

\begin{PropSub} \label{thm:rho_orth_to_each_sep_compl_ob}
Let $B,B' \in \B$ and suppose $B'$ is $\TT$-separated and $\Sigma$-complete.  Then $\rho_B \bot_\V B'$.
\end{PropSub}
\begin{proof}
We have a commutative diagram as follows.
$$
\xymatrix{
\B(KB,B') \ar[rr]^{\B(\rho_B,B')} \ar[d]_{\B(KB,\eta_{B'})} & & \B(B,B') \ar[d]^{\B(B,\eta_{B'})} \\
\B(KB,QPB') \ar[rr]^{\B(\rho_B,QPB')} \ar[d]^\wr           & & \B(B,QPB') \ar[d]^\wr \\
\C(PKB,PB') \ar[rr]^{\C(P\rho_B,PB')}                       & & \C(PB,PB')
}
$$
Since $B'$ is $\TT$-separated and $\Sigma$-complete, we have by \bref{thm:sepobj_compl_iff_eta_closed} that $\eta_{B'}$ is a $\Sigma$-closed embedding, so since $\rho_B$ is $\Sigma$-dense, $\rho_B \downarrow_\V \eta_{B'}$, so the upper square is a pullback.  Also, $P\rho_B$ is iso by \bref{thm:rho_inverted_by_s}, so the left, bottom, and right sides of the lower square are iso.  Therefore $\B(\rho_B,QPB')$ is iso and hence its pullback $\B(\rho_B,B')$ is iso (by \bref{prop:pb_iso}).
\end{proof}

The following is a variant of Day's 2.3 \cite{Day:AdjFactn}.  Note however that Day's result applies to a given adjunction whose unit is assumed to lie componentwise within the right-class of a fixed proper $\V$-factorization-system, so that via Day's approach one may obtain an analogue of the following through a two-step process by first invoking the statements \bref{thm:sep_refl} and \bref{thm:adj_factn_lemma}.  We avoid this two-step approach, choosing instead to study directly the interplay in $\B$ of the two notions we have termed \textit{$\TT$-separatedness} and \textit{$\Sigma$-completeness}.

\begin{ThmSub} \label{thm:refl_orth}
Let $\TT$ be a $\V$-monad on a $\V$-finitely-well-complete cotensored $\V$-category $\B$, and let $\Sigma \subs \Sigma_T$.  Then the morphisms $\rho_B:B \rightarrow KB$ ($B \in \B$) of \bref{def:completion} exhibit the $\V$-category $\B_{(\TT,\Sigma)}$ of $\Sigma$-complete $\TT$-separated objects \pbref{def:compl_sep_closed_dense} as a $\V$-reflective-subcategory of $\B$.  
\end{ThmSub}
\begin{proof}
For each $B \in \B$ we have by \bref{thm:kb_sep_compl} that $KB$ lies in $\B_{(\TT,\Sigma)}$, and for each $B' \in \B_{(\TT,\Sigma)}$,
$$\B_{(\TT,\Sigma)}(KB,B') = \B(KB,B') \xrightarrow{\B(\rho_B,B')} \B(B,B')$$
is iso by \bref{thm:rho_orth_to_each_sep_compl_ob}.  The result follows by \bref{prop:adjn_via_radj_and_unit}.
\end{proof}

\begin{CorSub} \label{thm:sep_refl}
The full sub-$\V$-category of $\B$ consisting of the $\TT$-separated objects is $\V$-reflective in $\B$.
\end{CorSub}
\begin{proof}
Taking $\Sigma = \emptyset$ in \bref{thm:refl_orth}, the objects of $\B_{(\TT,\Sigma)}$ are exactly the $\TT$-separated objects of $\B$.
\end{proof}

\begin{CorSub} \label{thm:sigma_s_complete_are_sep_and_refl}
Every $\Sigma_T$-complete object of $\B$ is $\TT$-separated, so $\B_{\Sigma_T} = \B_{(\TT,\Sigma_T)}$.  Hence the $\V$-category $\B_{\Sigma_T}$ of $\Sigma_T$-complete objects is a $\V$-reflective-subcategory of $\B$.
\end{CorSub}
\begin{proof}
By \bref{thm:sep_refl} we know that the full sub-$\V$-category $\B_{(\TT,\emptyset)}$ of $\B$ consisting of the $\TT$-separated objects is $\V$-reflective in $\B$, and we denote the components of the unit of the associated $\V$-reflection by $\sigma_B:B \rightarrow LB$ ($B \in \B$).  Hence, $\sigma_B$ is gotten as the morphism $\rho_B$ of \bref{def:completion} in the case that $\Sigma = \emptyset$.  By \bref{thm:rho_inverted_by_s} we know that each such component $\sigma_B$ is inverted by $P$ --- i.e. $\sigma_B \in \Sigma_P = \Sigma_T$.  Hence, given any $\Sigma_T$-complete object $B'$ of $\B$, we have that $\sigma_B \bot_\V B'$ for every $B \in \B$, so by \bref{thm:refl_subcats_are_orth_subcats}, $B' \in \B_{(\TT,\emptyset)}$.  
\end{proof}

In view of \bref{thm:sigma_s_complete_are_sep_and_refl}, we shall extend the terminology of \bref{def:compl_sep_closed_dense} as follows:

\begin{DefSub}\label{def:tcomplete}
We say that an object $B$ of $\B$ is \textit{$\TT$-complete} if $B$ is $\Sigma_T$-complete.  We denote the resulting $\V$-reflective-subcategory of $\B$ by $\B_{(\TT)} := \B_{\Sigma_T} = \B_{(\TT,\Sigma_T)}$, noting that by \bref{thm:sigma_s_complete_are_sep_and_refl}, every $\TT$-complete object is $\TT$-separated.
\end{DefSub}

\begin{ExaSub}\label{exa:func_complete_rmods}
Returning to the example of \bref{exa:func_distnl_compl}, where $\B = \RMod(\X)$ and $\TT = \HH$ is the double-dualization monad, the $\HH$-complete objects are the \textit{functionally complete} $R$-modules in $\X$.
\end{ExaSub}

\section{Generalities on factorization of an adjunction through a reflection}

\begin{LemSub} \label{thm:adj_factn_lemma}
Suppose given a $\V$-adjunction $P \nsststile{\varepsilon}{\eta} Q : \C \rightarrow \B$ and a $\V$-reflection $K \nsststile{}{\rho} J : \B' \hookrightarrow \B$ such that the image of $Q$ lies in $\B'$.  Then there is a $\V$-adjunction $P' \nsststile{\varepsilon'}{\eta'} Q':\C \rightarrow \B'$ with $JQ' = Q$, $P' = PJ$, $J\eta' = \eta J$, and $\varepsilon' = \varepsilon$.
\end{LemSub}
\begin{proof}
$Q'$ is just the corestriction of $Q$, the components of $\eta'$ are just those of $\eta$; the $\V$-naturality of $\eta'$ is immediate, and the triangular equations are readily verified.
\end{proof}

The following terminology is motivated by \bref{thm:adj_factn_lemma}:

\begin{DefSub} \label{def:adj_factors_through_refl}
Let $P \nsststile{\varepsilon}{\eta} Q : \C \rightarrow \B$ be a $\V$-adjunction, and let $\B'$ be a $\V$-reflective-subcategory of $\B$.  We say that $P \nsststile{\varepsilon}{\eta} Q$ \textit{factors through $\B'$} if the image of $Q$ lies in $\B'$.  Given a $\V$-reflection $K \nsststile{}{\rho} J:\B' \hookrightarrow \B$, we say that \textit{$P \nsststile{\varepsilon}{\eta} Q$ factors through $K \nsststile{}{\rho} J$} if $P \nsststile{\varepsilon}{\eta} Q$ factors through $\B'$, and in this case we call the associated adjunction $P' \nsststile{\varepsilon'}{\eta'} Q':\C \rightarrow \B'$ of \bref{thm:adj_factn_lemma} the \textit{induced $\V$-adjunction}.
\end{DefSub}

\begin{LemSub} \label{thm:uniqueness_of_mnd_morph}
Suppose $\SSS = (S,\rho,\lambda)$ and $\TT = (T,\eta,\mu)$ are $\V$-monads on a $\V$-category $\B$, and suppose that $\SSS$ is idempotent.  If a morphism of $\V$-monads $\alpha:\SSS \rightarrow \TT$ exists, then it is unique, and its component at each $B \in \B$ is characterized as the unique morphism $\alpha_B$ such that
\begin{equation}\label{eqn:mm_unit_diag}
\xymatrix {
B \ar[rr]^{\eta_B} \ar[dr]_{\rho_B} &                       & TB \\
                                    & SB \ar[ur]_{\alpha_B} &
}
\end{equation}
commutes.
\end{LemSub}
\begin{proof}
Suppose given a morphism of $\V$-monads $\alpha:\SSS \rightarrow \TT$.  We shall work only with the underlying morphism of ordinary monads, for which we will use the same notation.  Let $K \nsststile{}{\rho} J : \B' \hookrightarrow \B$ be the reflection determined by the idempotent monad $\SSS$, so that $\B'$ may be identified with $\B^\SSS$.  The monad morphism $\alpha$ induces a functor \mbox{$\B^\alpha:\B^\TT \rightarrow \B^\SSS = \B'$} \pbref{prop:enr_func_induced_by_mnd_mor} which commutes with the forgetful functors.  Hence, for each $B \in \B$, the carrier $TB$ of the free $\TT$-algebra on $B$ lies in $\B'$, and \eqref{eqn:mm_unit_diag} commutes since $\alpha$ is a monad morphism, so $\alpha_B$ is the unique extension of $\eta_B$ along the reflection unit $\rho_B$.
\end{proof}

\begin{PropSub}\label{thm:charns_adj_factn_via_mnd}
Let $\TT$ be a $\V$-monad, and let $P \nsststile{\varepsilon}{\eta} Q : \C \rightarrow \B$ be any $\V$-adjunction inducing $\TT$.  Let $K \nsststile{}{\rho} J : \B' \hookrightarrow \B$ be a $\V$-reflection with induced idempotent $\V$-monad $\SSS$.  Then the following are equivalent:
\begin{enumerate}
\item There exists a \textnormal{(}necessarily unique, \bref{thm:uniqueness_of_mnd_morph}\textnormal{)} morphism of $\V$-monads $\alpha:\SSS \rightarrow \TT$.
\item $T$ factors through $J:\B' \hookrightarrow \B$.
\item $\B'$ contains the carriers of all free $\TT$-algebras.
\item $P \nsststile{\varepsilon}{\eta} Q$ factors through $K \nsststile{}{\rho} J$.
\end{enumerate}
\end{PropSub}
\begin{proof}
The equivalence 2 $\Leftrightarrow$ 3 obvious.  Observe that 3 is equivalent to the statement that the Kleisli $\V$-adjunction for $\TT$ factors through $K \nsststile{}{\rho} J$.  Hence it suffices to prove that 1 $\Leftrightarrow$ 4, for then the equivalence 1 $\Leftrightarrow$ 3 follows as a special case.  To this end, first observe that if 4 holds, then the existence of a morphism of $\V$-monads $\alpha:\SSS \rightarrow \TT$ is guaranteed by \bref{thm:mnd_morph_from_adj_factn}.  Conversely, suppose given a morphism of $\V$-monads $\alpha:\SSS \rightarrow \TT$.  We now work only with the underlying morphism of ordinary monads, for which we will use the same notation.  The adjunction $P \nsststile{\varepsilon}{\eta} Q$ determines a comparison functor $\C \rightarrow \B^\TT$, and we have also a functor $\B^\alpha:\B^\TT \rightarrow \B^\SSS$ \pbref{prop:enr_func_induced_by_mnd_mor}.  Both these functors commute with the forgetful functors to $\B$, and so too does their composite $\C \rightarrow \B^\TT \rightarrow \B^\SSS \cong \B'$.  Hence, applying this composite functor to any given $C \in \C$, we find that the carrier $QC$ of the associated $\SSS$-algebra lies in $\B'$.
\end{proof}

\begin{RemSub}
Proposition \bref{thm:charns_adj_factn_via_mnd} shows in particular that the question of whether a  $\V$-adjunction $P \nsststile{\varepsilon}{\eta} Q$ factors through a given $\V$-reflection depends only on the $\V$-monad $\TT$ induced by $P \nsststile{\varepsilon}{\eta} Q$.
\end{RemSub}

\begin{ParSub} \label{def:cats_idmmnd_refl_defn}
Let $\B$ be a $\V$-category.
\begin{enumerate}
\item The class $\IdmMnd_\V(\B)$ \pbref{thm:refl_idemp_mnd} of idempotent $\V$-monads on $\B$ carries the structure of a full subcategory of $\Mnd_{\VCAT}(\B)$ \pbref{par:adj_and_mnd}.
\item The class $\Refl_\V(\B)$ \pbref{thm:refl_idemp_mnd} of $\V$-reflections on $\B$ acquires the structure of a preordered class when ordered by inclusion of the associated $\V$-reflective-subcategories.
\end{enumerate}
\end{ParSub}

\begin{CorSub} \label{thm:refl_iso_idemp_mnd}
For any $\V$-category $\B$, $\IdmMnd_\V(\B)$ is a preordered class isomorphic to $(\Refl_\V(\B))^\op$ via the bijection given in \bref{thm:refl_idemp_mnd}.
\end{CorSub}
\begin{proof}
By \bref{thm:uniqueness_of_mnd_morph}, $\IdmMnd_\V(\B)$ is a preorder.  Hence it suffices to show that the bijection $\Refl_\V(\B) \cong \IdmMnd_\V(\B)$ \pbref{thm:refl_idemp_mnd} and its inverse are contravariantly functorial (i.e. order-reversing).  But this follows from \bref{thm:charns_adj_factn_via_mnd}, since the given preorder relation on $\Refl_\V(\B)$ may equally be described as
$$(K' \nsststile{}{\rho'} J') \lt (K \nsststile{}{\rho} J)\;\;\Longleftrightarrow\;\;\text{$K'\nsststile{}{\rho'} J'$ factors through $K \nsststile{}{\rho} J$}\;.$$
\end{proof}

\begin{CorSub} \label{thm:iso_idmmnds_det_same_refl_subcat}
Isomorphic idempotent $\V$-monads on $\B$ determine the same associated $\V$-reflective-subcategory of $\B$.
\end{CorSub}

\begin{DefSub}
Let $\TT$ be a $\V$-monad on a $\V$-category $\B$.
\begin{enumerate}
\item Let $\Refl_\V(\TT)$ denote the full subcategory of $\Refl_\V(\B)$ \pbref{def:cats_idmmnd_refl_defn} consisting of those $\V$-reflections whose associated $\V$-reflective-subcategory contains the carriers of all free $\TT$-algebras.
\item Let $\IdmMnd_\V(\B) / \TT$ denote the full subcategory of $\IdmMnd_\V(\B)$ \pbref{def:cats_idmmnd_refl_defn} consisting of all idempotent $\V$-monads $\SSS$ on $\B$ for which a \textnormal{(}necessarily unique, \bref{thm:uniqueness_of_mnd_morph}\textnormal{)} morphism of $\V$-monads $\alpha:\SSS \rightarrow \TT$ exists.
\end{enumerate}
\end{DefSub}

\begin{CorSub} \label{thm:refls_for_t_iso_to_idm_mnds_over_q}
Let $\TT$ be a $\V$-monad on a $\V$-category $\B$.  Then the isomorphism $(\Refl_\V(\B))^\op \cong \IdmMnd_\V(\B)$ of \bref{thm:refl_iso_idemp_mnd} restricts to an isomorphism of preordered classes $(\Refl_\V(\TT))^\op \cong \IdmMnd_\V(\B) / \TT$.
\end{CorSub}
\begin{proof}
This follows from \bref{thm:charns_adj_factn_via_mnd}.
\end{proof}

\begin{CorSub}\label{thm:mnd_mor_for_idm_mnds_subord_t}
Let $\TT$ be a $\V$-monad on a $\V$-category $\B$, let $\SSS_1$, $\SSS_2$ be idempotent $\V$-monads on $\B$, and suppose that we have inclusions of classes
$$\{TB \;|\; B \in \B \} \subs \Ob\B^{(\SSS_2)} \subs \Ob\B^{(\SSS_1)}\;.$$
Then there are unique morphisms of $\V$-monads as in the diagram
$$
\xymatrix {
                                         & \SSS_2 \ar[dr]^{\beta} &    \\
\SSS_1 \ar[ur]^{\alpha} \ar[rr]_{\gamma} &                        & {\TT\;,}
}
$$
which necessarily commutes.
\end{CorSub}

\section{Enriched adjoint functor factorization via orthogonality} \label{sec:enr_adj_factn}

\begin{PropSub} \label{thm:adj_factn_for_class_of_morphs}
Let $P \nsststile{\varepsilon}{\eta} Q : \C \rightarrow \B$ be a $\V$-adjunction, with induced $\V$-monad $\TT$, let $\Sigma \subs \Sigma_T$, and suppose that the $\V$-category $\B$ is $\V$-finitely-well-complete and cotensored.  Then the given $\V$-adjunction factors through the $\V$-reflection $K \nsststile{}{\rho} J : \B_{(\TT,\Sigma)} \hookrightarrow \B$, where $\B_{(\TT,\Sigma)}$ is the full sub-$\V$-category of $\B$ consisting of the $\Sigma$-complete $\TT$-separated objects \pbref{thm:refl_orth}.  Further, each component of the unit $\eta'$ of the induced $\V$-adjunction $P' \nsststile{\varepsilon'}{\eta'} Q':\C \rightarrow \B_{(\TT,\Sigma)}$ is a $\V$-strong-mono in $\B$.
\end{PropSub}
\begin{proof}
By \bref{thm:carriers_of_free_talgs_sep_compl}, $\B_{(\TT,\Sigma)}$ contains the carriers of all free $\TT$-algebras, so by \bref{thm:charns_adj_factn_via_mnd} the given $\V$-adjunction factors as needed.  Further, for each $B' \in \B_{(\TT,\Sigma)}$, since $J\eta' = \eta J$ the component $\eta'_{B'}$ is simply $\eta_{B'}$, which is a $\V$-strong-mono in $\B$ since $B'$ is $\TT$-separated.
\end{proof}

\begin{ThmSub} \label{thm:adj_factn}
Let $P \nsststile{\varepsilon}{\eta} Q : \C \rightarrow \B$ be a $\V$-adjunction with induced $\V$-monad $\TT$, and suppose that the $\V$-category $\B$ is $\V$-finitely-well-complete and cotensored.  Then $P \nsststile{\varepsilon}{\eta} Q$ factors through the $\V$-reflection $K \nsststile{}{\rho} J : \B_{(\TT)} \hookrightarrow \B$, where $\B_{(\TT)}$ is the full sub-$\V$-category of $\B$ consisting of the $\TT$-complete objects \pbref{def:tcomplete}.  Further, letting $P' \nsststile{\varepsilon'}{\eta'} Q':\C \rightarrow \B_{(\TT)}$ be the induced $\V$-adjunction, we have that $P'$ is conservative (i.e. $P'$ reflects isomorphisms), and each component of $\eta'$ is a $\V$-strong-mono in $\B$.
\end{ThmSub}
\begin{proof}
By \bref{thm:sigma_s_complete_are_sep_and_refl}, an invocation of \bref{thm:adj_factn_for_class_of_morphs} yields the needed factorization, and it suffices to show that $P'$ is conservative.  Indeed, if a morphism $f:B_1' \rightarrow B_2'$ in $\B_{(\TT)} = \B_{\Sigma_P}$ is such that $P'f$ is an isomorphism in $\C$, then since $P'f$ is simply $Pf$ we have that $f \in \Sigma_P$, so by \bref{thm:orth_of_morphs_in_orth_subcat} we deduce that $f \downarrow_{\V} f$ in $\eB$, whence $f$ is an isomorphism in $\B$ \pbref{thm:orth_yielding_retr_and_iso} and hence in $\B_{(\TT)}$.
\end{proof}

\begin{RemSub} \label{rem:factn_inversion}
In \bref{thm:adj_factn}, $K$ inverts the same morphisms as $P$ (i.e., $\Sigma_{K} = \Sigma_{P}$), since $P  \cong P'K$ and $P'$ is conservative.
\end{RemSub}

\begin{CorSub} \label{thm:base_adj_fact}
If $\V$ is finitely well-complete, then any given $\V$-adjunction \mbox{$P \nsststile{\varepsilon}{\eta} Q : \C \rightarrow \uV$} factors as in \bref{thm:adj_factn} (when we set $\B := \uV$); further, given also a class of morphisms $\Sigma$ inverted by the induced monad, \bref{thm:adj_factn_for_class_of_morphs} also applies.
\end{CorSub}
\begin{proof}
Since $\uV$ is cotensored, an invocation of \bref{thm:base_fwc} shows that the hypotheses of \bref{thm:adj_factn} and \bref{thm:adj_factn_for_class_of_morphs} are satisfied.
\end{proof}

\begin{RemSub}
Day \cite{Day:AdjFactn} shows that the above factorization of a left-adjoint $\V$-functor into a reflection followed by a conservative left adjoint determines a factorization system, in an `up-to-isomorphism' sense, on a suitable category.
\end{RemSub}

\section{Separated completion monads for a class of morphisms}\label{sec:sep_cpl_mnds_class_morphs}

In the present section, we fix a $\V$-finitely-well-complete and cotensored $\V$-category $\B$.

\begin{DefSub}\label{def:tsep_sigma_compl_mnd}
As in \bref{par:data_enr_orth_subcat_subord_adj}, let $\TT = (T,\eta,\mu)$ be a $\V$-monad on $\B$, and let $\Sigma \subs \Sigma_T$ be a class of morphisms inverted by $T$.  By \bref{thm:refl_orth}, the $\Sigma$-complete $\TT$-separated objects of $\B$ constitute a $\V$-reflective-subcategory $\B_{(\TT,\Sigma)}$ of $\B$, and we call the induced idempotent $\V$-monad $\TT_\Sigma$ on $\B$ \textit{the $\TT$-separated $\Sigma$-completion $\V$-monad}.
\end{DefSub}

\begin{ExaSub}\label{exa:func_distnl_completion_mnds}
Let us return to the example of \bref{exa:func_distnl_compl}, \bref{exa:func_complete_rmods}, where $\TT = \HH$ is the double-dualization $\sL$-monad on $\sL = \RMod(\X)$.  With $\Sigma := \Sigma_H$, we call the resulting idempotent $\sL$-monad $\HH_\Sigma$ the \textit{functional completion monad}; see \bref{def:func_compl}.  Instead taking $\Sigma$ to be the subclass $\tDelta \subs \Sigma_H$, we call the resulting idempotent $\sL$-monad $\HH_{\tDelta}$ the \textit{separated $\tDD$-distributional completion monad}; see \bref{def:distnl_compl_cl_dens}.
\end{ExaSub}

\begin{PropSub}\label{thm:uniq_morph_mnds_tsep_sigmacompl_mnd_to_t}
Given a $\V$-monad $\TT = (T,\eta,\mu)$ on $\B$ and a class $\Sigma \subs \Sigma_T$, there is a unique morphism of $\V$-monads
$$i:\TT_\Sigma \rightarrow \TT\;.$$
The components of $i$ are the $\Sigma$-closed embeddings $i_B:KB \rightarrowtail TB$ of \bref{def:completion}.
\end{PropSub}
\begin{proof}
By \bref{thm:carriers_of_free_talgs_sep_compl}, the $\V$-reflective-subcategory $\B_{(\TT,\Sigma)}$ of $\B$ determined by $\TT_\Sigma$ contains the carriers of all free $\TT$-algebras, so by \bref{thm:charns_adj_factn_via_mnd} there is a unique morphism of $\V$-monads $\TT_\Sigma \rightarrow \TT$.  But by \bref{thm:uniqueness_of_mnd_morph} and \bref{def:completion}, the components of this $\V$-monad morphism must necessarily be the $i_B$.
\end{proof}

\begin{PropSub}\label{thm:mor_tsep_sigma_compl_mnds}
Given a $\V$-monad $\TT = (T,\eta,\mu)$ on $\B$ and classes of morphisms $\Sigma_1 \subs \Sigma_2 \subs \Sigma_T$ inverted by $T$, there are unique morphisms of $\V$-monads as in the diagram
$$
\xymatrix {
                                          & \TT_{\Sigma_2} \ar[dr]^{\beta} &    \\
\TT_{\Sigma_1} \ar[ur]^{\alpha} \ar[rr]_{\gamma} &                              & {\TT\;,}
}
$$
which necessarily commutes.
\end{PropSub}
\begin{proof}
This follows from \bref{thm:mnd_mor_for_idm_mnds_subord_t}, since $\Ob\B_{\Sigma_2} \subs \Ob\B_{\Sigma_1}$ and hence, invoking \bref{thm:carriers_of_free_talgs_sep_compl},
$$\{TB \;|\; B \in \B \} \subs \Ob\B_{(\TT,\Sigma_2)} \subs \Ob\B_{(\TT,\Sigma_1)}\;.$$
\end{proof}

\section{The idempotent monad associated to an enriched monad}\label{sec:assoc_idm_mnd}

\begin{DefSub} \label{def:induced_idm_mnd_and_refl}
Let $\TT$ be a $\V$-monad on a $\V$-finitely-well-complete cotensored $\V$-category $\B$.  We call the idempotent $\V$-monad $\tTT := \TT_{\Sigma_T}$ \pbref{def:tsep_sigma_compl_mnd} \textit{the $\TT$-completion monad}, or, alternatively, \textit{the idempotent $\V$-monad induced by $\TT$}.
\end{DefSub}

\begin{ParSub}
The $\V$-reflective-subcategory of $\B$ determined by $\tTT$ is the $\V$-category $\B_{(\TT)}$ of $\TT$-complete objects \pbref{def:tcomplete}, since these objects are by definition the $\Sigma_T$-complete objects, each of which is automatically $\TT$-separated, by \bref{thm:sigma_s_complete_are_sep_and_refl}.  Let 
$$K \nsststile{}{\rho} J:\B_{(\TT)} \hookrightarrow \B$$
be the associated $\V$-reflection.
\end{ParSub}

The following universal characterizations of $\tTT$ and $\B_{(\TT)}$ justify calling $\tTT$ the idempotent $\V$-monad induced by $\TT$:

\begin{ThmSub} \label{thm:ind_idmmnd_and_refl}
Let $\TT$ be a $\V$-monad on a $\V$-finitely-well-complete cotensored $\V$-category $\B$.
\begin{enumerate}
\item $\tTT$ is terminal among idempotent $\V$-monads $\SSS$ for which a \textnormal{(}necessarily unique, \bref{thm:uniqueness_of_mnd_morph}\textnormal{)} morphism of $\V$-monads $\SSS \rightarrow \TT$ exists.  I.e., $\tTT$ is a terminal object of $\IdmMnd_\V(\B) / \TT$.
\item Given any $\V$-adjunction inducing $\TT$, $\B_{(\TT)}$ is the smallest $\V$-reflective-subcategory of $\B$ through which the given $\V$-adjunction factors.
\item $\B_{(\TT)}$ is the smallest $\V$-reflective-subcategory of $\B$ containing the carriers of all free $\TT$-algebras.
\end{enumerate}
\end{ThmSub}
\begin{proof}
First, let us prove 2.  Any $\V$ adjunction $P \nsststile{\varepsilon}{\eta} Q$ inducing $\TT$ factors through $K \nsststile{}{\rho} J$ by \bref{thm:adj_factn}.  Consider any $\V$-reflection $K' \nsststile{}{\rho'} J':\B' \hookrightarrow \B$ through which $P \nsststile{\varepsilon}{\eta} Q$ factors, and let $P' \nsststile{\varepsilon}{\eta'} Q':\C \rightarrow \B'$ be the induced adjunction.  Then $P \cong P'K'$, whence $\Sigma_{K'} \subs \Sigma_P$.  But for each $B \in \B$, $\rho'_B$ is inverted by $K'$, so $\rho'_B$ lies in $\Sigma_P$.  Therefore, letting $\Sigma' := \{\rho'_B\;|\;B \in \B\}$, we have that $\Sigma' \subs \Sigma_P$, so $\B_{(\TT)} = \B_{\Sigma_P} \subs \B_{\Sigma'} = \B'$ (by \bref{thm:refl_subcats_are_orth_subcats}).

By 2, we have in particular that $\B_{(\TT)}$ is the smallest $\V$-reflective-subcategory of $\B$ through which the Kleisli $\V$-adjunction factors.  Hence 3 holds.  Moreover, we deduce by \bref{thm:refls_for_t_iso_to_idm_mnds_over_q} that 1 holds also.
\end{proof}

\begin{PropSub} \label{thm:addnl_facts_on_ind_idmmnd_refl}
Let $\TT$ be a $\V$-monad on a $\V$-finitely-well-complete cotensored $\V$-category $\B$.
\begin{enumerate}
\item $\tT$ inverts exactly the same morphisms as $T$ --- i.e. $\Sigma_{\tT} = \Sigma_T$.
\item The components of the unique morphism of $\V$-monads $\tTT \rightarrow \TT$ are the $\Sigma$-closed embeddings $i_B:KB \rightarrowtail TB$ of \bref{def:completion}, where $\Sigma = \Sigma_T$.
\end{enumerate}
\end{PropSub}
\begin{proof}
Choose any $\V$-adjunction $P \dashv Q$ inducing $\TT$.  By \bref{rem:factn_inversion} and \bref{thm:all_ladjs_ind_given_monad_inv_same_morphs}, $\Sigma_{\tT} = \Sigma_K = \Sigma_P = \Sigma_T$.  2 follows from \bref{thm:uniq_morph_mnds_tsep_sigmacompl_mnd_to_t}.
\end{proof}

\begin{RemSub} \label{rem:ind_idempt_mnd_on_v}
If $\V$ is finitely well-complete, then for any $\V$-monad $\TT$ on $\uV$ itself, the induced idempotent $\V$-monad $\tTT$ \pbref{thm:ind_idmmnd_and_refl} can be formed.  Indeed, as observed in \bref{thm:base_adj_fact}, $\uV$ is cotensored and $\V$-finitely-well-complete.
\end{RemSub}

\begin{ExaSub}
Returning to the example of \bref{exa:func_distnl_compl}, \bref{exa:func_distnl_completion_mnds}, where $\TT = \HH$ is the double-dualization $\sL$-monad on $\sL = \RMod(\X)$, the idempotent $\sL$-monad induced by $\HH$ is the \textit{functional completion monad} $\tHH$, whose associated $\sL$-reflective-subcategory of $\uL$ consists of the \textit{functionally complete} objects; see \bref{def:func_compl}.
\end{ExaSub}

%% file: sm_closed_adj.tex
\chapter{Symmetric monoidal closed adjunctions and commutative monads}\label{ch:smc_adj_comm_mnd}
\setcounter{subsection}{0}

Given a commutative ring $R$, the category of $R$-modules (call it $\sL$ for the moment) is \textit{symmetric monoidal closed}, as we can form for each pair of $R$-modules $E_1,E_2$ the \textit{tensor product} $E_1 \otimes E_2$ and the \textit{hom} $\uL(E_1,E_2)$ such that
\begin{equation}\label{eqn:tens_hom}\uL(E_1 \otimes E_2,E_3) \cong \uL(E_1,E_2;E_3) \cong \uL(E_1,\uL(E_2,E_3))\end{equation}
naturally in $E_1,E_2,E_3 \in \sL$, where $\uL(E_1,E_2;E_3)$ is the $R$-module of bilinear maps.  Taking $\sL$ instead to be the category of Banach spaces, we can again define a tensor product and hom of Banach spaces in such a way that a natural isomorphism of Banach spaces \eqref{eqn:tens_hom} obtains (see, e.g. \cite{Mich} 1.4).  But the situation becomes much more complicated when one moves to the setting of topological vector spaces \cite{Gro:Memoirs}.  

Whereas topological $R$-vector spaces ($R = \RR$ or $\CC$) are $R$-module objects in the category $\Top$ of topological spaces, we can recapture the natural isomorphisms \eqref{eqn:tens_hom} by working instead with $R$-module objects in any category $\X$ which has enough limits and colimits and shares with the category of sets $\Set$ the property of being \textit{cartesian closed} (such as the categories of convergence spaces and smooth spaces, \bref{exa:conv_sm_sp}).  Further, for each object or `space' $X \in \X$ we can form the \textit{free} $R$-module object $FX$ on $X$, so that with $\sL := \RMod(\X)$ we obtain an adjunction
\begin{equation}\label{eqn:smc_adj_intr}
\xymatrix {
\X \ar@/_0.5pc/[rr]_F^(0.4){}^(0.6){}^{\top} & & \sL \ar@/_0.5pc/[ll]_G
}
\end{equation}
which is compatible with the symmetric monoidal closed structure on $\sL$, making it a \textit{symmetric monoidal adjunction} \pbref{exa:enr_adj_mnd}.

The category of $R$-modules in $\X$ is an example of the category of algebras $\X^\TT$ of a \textit{symmetric monoidal monad} $\TT$ on a symmetric monoidal closed category $\X$ \pbref{thm:rmod_commutative_monadic}.  Further, given \textit{any} symmetric monoidal monad $\TT$ on $\uX$, if there exist certain equalizers in $\X$ and certain coequalizers in $\sL = \X^\TT$, then $\sL$ is again symmetric monoidal closed and the Eilenberg-Moore adjunction \eqref{eqn:smc_adj_intr} is symmetric monoidal \pbref{thm:smclosed_em_adj}.  Indeed, Kock \cite{Kock:ClsdCatsGenCommMnds} showed that $\X^\TT$ is \textit{closed} (but not necessarily \textit{monoidal} closed) and we assemble a proof on the basis of this and more recent work in the monoidal setting \cite{Jac:SemWkngContr},\cite{Seal:TensMndActn}.

The category of $\TT$-algebras $\sL$ is both $\sL$-enriched and $\X$-enriched, and moreover the adjunction \eqref{eqn:smc_adj_intr} is $\X$-enriched.  Indeed, the latter $\X$-enrichment arises in two equivalent ways.  Firstly, as shown by Kock \cite{Kock:SmComm}, symmetric monoidal monads $\TT$ are (up to a bijection) the same as \textit{commutative} $\X$-enriched monads on $\uX$, so the Eilenberg-Moore adjunction \eqref{eqn:smc_adj_intr} is $\X$-enriched.  Secondly, we give a proof of a `folklore' result \pbref{thm:assoc_enr_adj} to the effect that every symmetric monoidal adjunction of closed categories carries an associated $\X$-enrichment.  We show that the two resulting $\X$-enrichments are the same.  Moreover, we clarify the relation of the later enrichment to Kock's bijection between symmetric monoidal monads and commutative enriched monads, we show that this enrichment commutes with composition of adjunctions \pbref{thm:compn_assoc_enr_adj}, and we elucidate its status in relation to the 2-functorial \textit{change-of-base} for enriched categories \pbref{par:ch_base}, studied by Cruttwell \cite{Cr}.

Chapter \bref{ch:compl_enr_orth} provides a means of identifying certain $\sL$-enriched reflective subcategories $\sL' \hookrightarrow \sL$ consisting of suitably `complete' objects.  We show that the theory of commutative monads entails that the tensor-hom isomorphisms for $\sL$ \eqref{eqn:tens_hom} induce analogous isomorphisms for $\sL'$, when one employs the same `hom' but instead the \textit{completion} $E_1 \otimes' E_2$ of the tensor product $E_1 \otimes E_2$.  More generally, we show that every \textit{idempotent} $\V$-monad $\SSS$ on a symmetric monoidal closed category $\V$ is commutative \pbref{thm:idm_mnd_comm}; by the above it then follows that the corresponding reflective subcategory $\W$ of $\V$ (which is the category of algebras of $\SSS$) is symmetric monoidal closed and the associated reflection is symmetric monoidal \pbref{thm:enr_refl_smadj}.  We thus obtain a novel proof of Day's result \cite{Day:Refl} that a reflective subcategory $\W$ of $\V$ is symmetric monoidal closed as soon as the hom $\V(V,W)$ lies in $\W$ whenever $W \in \W$ \pbref{thm:day_refl}.  Indeed, the latter property entails that $\W$ is a $\V$-enriched reflective subcategory of $\V$, so that Day's result is subsumed by the theory of commutative monads.

\section{Enriched functors arising from monoidal functors} \label{sec:enr_mon_func}

\begin{ParSub} \label{par:ch_base}
Cruttwell \cite{Cr} defines a 2-functor $(-)_*:\MMCCAATT \rightarrow \TWOCCAATT$, from the 2-category of monoidal categories to the 2-category of 2-categories, sending each monoidal functor $M:\V \rightarrow \W$ to the \textit{change-of-base} 2-functor $M_*:\VCAT \rightarrow \WCAT$.  For a $\V$-category $\eA$, the $\W$-category $M_*\eA$ has objects those of $\eA$ and homs given by
$$(M_*\eA)(A_1,A_2) = M\eA(A_1,A_2)\;\;\;\;(A_1,A_2 \in \eA)\;.$$
Given a monoidal transformation $\alpha:M \rightarrow N$, where $M,N:\V \rightarrow \W$ are monoidal functors, we obtain an associated 2-natural transformation $\alpha_*:M_* \rightarrow N_*$ whose components are identity-on-object $\W$-functors
$$\alpha_*\A:M_*\A \rightarrow N_*\A\;,\;\;\;\;\A \in \VCAT$$
whose structure morphisms are simply the components
$$(\alpha_*\A)_{A_1 A_2} = \alpha_{\A(A_1,A_2)}:M\A(A_1,A_2) \rightarrow N\A(A_1,A_2)\;,\;\;\;\;(A_1,A_2 \in \A)$$
of $\alpha$.

Given a symmetric monoidal functor $M:\V \rightarrow \W$ between closed symmetric monoidal categories, we obtain a canonical $\W$-functor $\grave{M}:M_*\aeV \rightarrow \aeW$, given on objects just as $M$, and with each
\begin{equation}\label{eq:str_morph_mgrave}\grave{M}_{V_1 V_2} : (M_*\aeV)(V_1,V_2) = M\aeV(V_1,V_2) \rightarrow \aeW(MV_1,MV_2)\;\;\;\;(V_1,V_2 \in \V)\end{equation}
gotten as the transpose of the composite
$$MV_1 \otimes M\aeV(V_1,V_2) \rightarrow M(V_1 \otimes \aeV(V_1,V_2)) \xrightarrow{M(\Ev)} MV_2\;.$$
\end{ParSub}

\begin{ParSub} \label{par:cl_func}
Let $\SSMMCCAATT$ be the 2-category of symmetric monoidal categories, and let $\CCllSSMMCCAATT$ be the full sub-2-category of $\SSMMCCAATT$ with objects all closed symmetric monoidal categories.  Letting $\CCllCCAATT$ be the 2-category of closed categories \cite{EiKe}, there is a 2-functor $c:\CCllSSMMCCAATT \rightarrow \CCllCCAATT$, sending a symmetric monoidal functor $M:\V \rightarrow \W$ (with $\V$ and $\W$ closed symmetric monoidal categories) to the closed functor $cM:\V \rightarrow \W$ with the same underlying ordinary functor $M$ and the same unit morphism $I_{\W} \rightarrow MI_{\V}$, but with each structure morphism
$$M\aeV(V_1,V_2) \rightarrow \aeW(MV_1,MV_2)\;\;\;\;(V_1,V_2 \in \V)$$
equal to the morphism $\grave{M}_{V_1 V_2}$ \eqref{eq:str_morph_mgrave} associated to the $\W$-functor $\grave{M}:M_*\aeV \rightarrow \aeW$ of \bref{par:ch_base}.
\end{ParSub}

\begin{PropSub} \label{prop:comp_assoc_enr_fun}
Let $M:\U \rightarrow \V$, $N:\V \rightarrow \W$ be symmetric monoidal functors between closed symmetric monoidal categories.  Then the $\W$-functor
$$\widegrave{NM}:(NM)_*\aeU \rightarrow \aeW$$
is equal to the composite
$$N_*M_*\aeU \xrightarrow{N_*\grave{M}} N_*\aeV \xrightarrow{\grave{N}} \aeW\;.$$
\end{PropSub}
\begin{proof}
Both $\W$-functors are given on objects just as $NM$.  Letting $C$ be the given composite $\W$-functor, the structure morphisms $C_{U_1 U_2}$ of $C$ (where $U_1,U_2$ are objects of $N_*M_*\aeU$, equivalently, of $\U$) are the composites
$$NM\aeU(U_1,U_2) \xrightarrow{N(\grave{M}_{U_1 U_2})} N\aeV(MU_1,MU_2) \xrightarrow{\grave{N}_{MU_1 MU_2}} \aeW(NMU_1,NMU_2)\;,$$
but by \bref{par:cl_func}, we have that $\grave{M}_{U_1 U_2}$ and $\grave{N}_{MU_1 MU_2}$ are equally the structure morphisms of the closed functors $cM$ and $cN$, respectively, and so $C_{U_1 U_2}$ is the structure morphism of the composite $(cN)(cM)$ of these closed functors.  Since $(cN)(cM) = c(NM)$, $C_{U_1 U_2}$ is therefore equally the structure morphism of the closed functor $c(NM)$ associated to $NM$, which by \bref{par:cl_func} is equal to
$$\widegrave{NM}_{U_1 U_2}:((NM)_*\aeU)(U_1,U_2) \rightarrow \aeW(NMU_1,NMU_2)\;,$$
the structure morphism of the $\aeW$ functor $\widegrave{NM}$.
\end{proof}

\section{Enrichment of a symmetric monoidal closed adjunction} \label{sec:enr_smc_adj}

\begin{ParSub} \label{par:given_smcadj}
Let
\begin{equation}\label{eqn:smcadj}
\xymatrix {
\X \ar@/_0.5pc/[rr]_F^(0.4){\eta}^(0.6){\varepsilon}^{\top} & & \sL \ar@/_0.5pc/[ll]_G
}
\end{equation}
be a symmetric monoidal adjunction, where $\X = (\X,\boxtimes,I)$ and $\sL = (\sL,\otimes,R)$ are \textit{closed} symmetric monoidal categories.  By applying $(-)_*$ \pbref{par:ch_base} to this monoidal adjunction, we obtain an adjunction
$$
\xymatrix {
\XCAT \ar@/_0.5pc/[rr]_{F_*}^(0.4){\eta_*}^(0.6){\varepsilon_*}^{\top} & & \LCAT \ar@/_0.5pc/[ll]_{G_*}
}
$$
in $\TWOCCAATT$.  We have an $\X$-functor $\grave{G}:G_*\aeL \rightarrow \aeX$ and an $\sL$-functor $\grave{F}:F_*\aeX \rightarrow \aeL$ \pbref{par:ch_base}, and we obtain an $\X$-functor $\acute{F}:\aeX \rightarrow G_*\aeL$ as the the transpose of $\grave{F}$ under the preceding adjunction.
\end{ParSub}

\begin{ExaSub}\label{exa:em_cat_as_exa_smcadj}
If $\sL$ is the (ordinary) category of Eilenberg-Moore algebras of a symmetric monoidal monad $\TT$ on a symmetric monoidal closed category $\X$ (equivalently, a commutative $\X$-monad on $\uX$, \bref{thm:comm_sm_mnd}), then, assuming that $\X$ has equalizers and that $\sL$ has certain coequalizers \pbref{def:tens_prod_algs}, the Eilenberg-Moore adjunction acquires the structure of a symmetric monoidal adjunction with $\sL$ closed, by \bref{thm:smclosed_em_adj}.
\end{ExaSub}

\begin{ExaSub}\label{exa:rmod_smcadj}
Specializing \bref{exa:em_cat_as_exa_smcadj}, the category of $R$-module objects for a commutative ring $R$ in a countably-complete and~\hbox{-cocomplete} cartesian closed category $\X$ participates in a symmetric monoidal adjunction $F \dashv G:\RMod(\X) \rightarrow \X$ with $\RMod(\X)$ closed; see \bref{exa:rmod_smclosed_monadic}.
\end{ExaSub}

\begin{RemSub}\label{rem:left_adj_strong}
By \cite{Ke:Doctr} 1.4, the left adjoint $F$ is necessarily \textit{strong} symmetric monoidal, meaning that the structure morphisms $FX \otimes FY \rightarrow F(X \boxtimes Y)$ are iso.
\end{RemSub}

\begin{ThmSub}\label{thm:assoc_enr_adj}
Let $F \nsststile{\varepsilon}{\eta} G : \sL \rightarrow \X$ be a symmetric monoidal adjunction between closed symmetric monoidal categories (as in \bref{par:given_smcadj}).
\begin{enumerate}
\item There is an $\X$-adjunction
\begin{equation}\label{eqn:assoc_enr_adj}
\xymatrix {
\aeX \ar@/_0.5pc/[rr]_{\acute{F}}^(0.4){\eta}^(0.6){\varepsilon}^{\top} & & G_*\aeL \ar@/_0.5pc/[ll]_{\grave{G}}
}
\end{equation}
whose underlying ordinary adjunction may be identified with $F \nsststile{\varepsilon}{\eta} G$.
\item The $\X$-category $G_*\aeL$ is cotensored.  For all $X \in \X$ and $E \in \sL$, a cotensor $[X,E]$ in $G_*\aeL$ may be gotten as the internal hom $\aeL(FX,E)$ in $\sL$.
\item Further to 2, one may obtain the needed `hom-cotensor' $\X$-adjunction
$$
\xymatrix {
\aeX \ar@/_0.5pc/[rr]_{[-,E]}^(0.4){}^(0.6){}^{\top} & & {(G_*\aeL)^\op} \ar@/_0.5pc/[ll]_{(G_*\aeL)(-,E)}
}
$$
as exactly the following composite $\X$-adjunction
\begin{equation}\label{eq:composite_hom_cotensor}
\xymatrix {
\aeX \ar@/_0.5pc/[rr]_{\acute{F}}^(0.4){\eta}^(0.6){\varepsilon}^{\top} & & G_*\aeL \ar@/_0.5pc/[ll]_{\grave{G}} \ar@/_0.5pc/[rr]_{G_*(\aeL(-,E))}^(0.4){}^(0.6){}^{\top} & & {G_*(\aeL^\op)\;,} \ar@/_0.5pc/[ll]_{G_*(\aeL(-,E))}
}
\end{equation}
in which the rightmost adjunction is gotten by applying $G_*:\LCAT \rightarrow \XCAT$ to the $\sL$-adjunction
$$
\xymatrix {
\aeL \ar@/_0.5pc/[rr]_{\aeL(-,E)}^(0.4){}^(0.6){}^{\top} & & {\aeL^\op\;.} \ar@/_0.5pc/[ll]_{\aeL(-,E)}
}
$$
\end{enumerate}
\end{ThmSub}

\begin{ParSub}
In order to begin proving \bref{thm:assoc_enr_adj}, recall from \bref{par:clsmcat_notn} that the underlying ordinary category of the $\X$-category $\uX$ is isomorphic to $\X$; similarly for the $\sL$-category $\uL$.  Let $\SET$ be a category of classes \pbref{par:cat_classes} in which the hom-classes of all these categories reside, so that all are $\SET$-categories.  We then have the following.
\end{ParSub}

\begin{PropSub}
There is a diagram of monoidal functors
$$
\xymatrix{
\sL \ar[rr]^G \ar[dr]_U &      & \X \ar[dl]^V \\
                        & \SET &
}
$$
that commutes up to an isomorphism of monoidal functors $\theta:U \xrightarrow{\sim} VG$, where $U = \sL(R,-)$ and $V = \X(I,-)$ are equipped with their canonical monoidal structures (\cite{EiKe}, 8.1).  Explicitly, the isomorphism $\theta$ is the composite
$$U = \sL(R,-) \xrightarrow{G} \X(GR,G-) \xrightarrow{\X(e^G,G-)} \X(I,G-) = VG\;.$$
\end{PropSub}
\begin{proof}
By \cite{EiKe} (8.20), $\theta$ is a monoidal transformation, and it is shown in \cite{Ke:Doctr} 2.1 that $\theta$ is iso as soon as the monoidal functor $G$ serves as right adjoint in a monoidal adjunction, as is presently the case.
\end{proof}

\begin{CorSub}\label{thm:underlying_ord_cats_comm_triangle}
The diagram of 2-functors
$$
\xymatrix{
\LCAT \ar[rr]^{G_*} \ar[dr]_{U_*} &             & \XCAT \ar[dl]^{V_*} \\
                                 & \eCAT{\SET} &
}
$$
commutes up to an isomorphism of 2-functors, namely $\theta_*:U_* \xrightarrow{\sim} (VG)_* = V_*G_*$.
\end{CorSub}

\begin{RemSub}
Note that $V_*$ (resp. $U_*$) is the canonical 2-functor sending an $\X$-category (resp. $\sL$-category) to its underlying ordinary category.
\end{RemSub}

\begin{CorSub}\label{thm:isos_of_underlying_ord_cats}
The underlying ordinary category of $G_*\uL$ is isomorphic to $\sL$, via the composite isomorphism
$$V_*G_*\uL \cong U_*\uL \cong \sL\;.$$
\end{CorSub}

\begin{ParSub}\label{par:desc_iso}
The isomorphisms in \bref{thm:isos_of_underlying_ord_cats} are identity-on-objects
and are given on hom-classes via natural bijections
$$\X(I,G\uL(E_1,E_2)) \cong \sL(R,\uL(E_1,E_2)) \cong \sL(E_1,E_2)\;,\;\;E_1,E_2 \in \sL\;.$$
Under the first, a morphism
$$k:R \rightarrow \uL(E_1,E_2)$$
corresponds to the composite
$$I \xrightarrow{e^G} GR \xrightarrow{Gk} G\uL(E_1,E_2)\;.$$
We also have an identity-on-objects isomorphism $V_*\uX \cong \X$, given on hom-classes by the natural bijection $\X(I,\uX(X,Y)) \cong \X(X,Y)$, $X,Y \in \X$.
\end{ParSub}

\begin{PropSub}\label{thm:underlying_ord_functor_iso_to_g}
The diagram of functors
$$
\xymatrix{
{V_*G_*\uL} \ar[d]_{V_*(\grave{G})} \ar[r]^\sim & {U_*\uL} \ar[r]^\sim & \sL \ar[d]^G \\
{V_*\uX} \ar[rr]^\sim                           &                      & \X
}
$$
commutes (where the indicated isomorphisms are as given in \bref{thm:isos_of_underlying_ord_cats}, \bref{par:desc_iso}).
\end{PropSub}
\begin{proof}
The given isomorphisms are identity-on-objects, and $V_*(\grave{G})$ acts just as $G$ on objects, so the diagram commutes on objects.  Let $E_1,E_2 \in \sL$, let 
$$g \in (V_*G_*\uL)(E_1,E_2) = \X(I,G\uL(E_1,E_2))$$
and let
$$k \in (U_*\uL)(E_1,E_2) = \sL(R,\uL(E_1,E_2))\;,\;\;\;\;\;\;h \in \sL(E_1,E_2)$$
be the arrows associated to $g$ via the isomorphisms at the top of the diagram, so that the upper-right composite sends $g$ to $Gh$.  On the other hand, $V_*(\grave{G})$ sends $g$ to the composite
$$I \xrightarrow{g} G\uL(E_1,E_2) \xrightarrow{\grave{G}_{E_1 E_2}} \uX(GE_1,GE_2)\;,$$
and, letting $x:GE_1 \rightarrow GE_2$ be the associated morphism in $\X$, our task is to show that $x = Gh$.

Using the definition of $\grave{G}_{E_1 E_2}$ \pbref{par:ch_base}, we find that $x$ is the composite of the upper-right path around the periphery of the following diagram
$$
\xymatrix{
GE_1 \ar[r]^{r^{-1}} \ar@/_2ex/[ddr]|{Gr^{-1}} \ar@/_4.5ex/@{=}[dddr] & {GE_1 \boxtimes I} \ar[d]|{1 \boxtimes e^G} \ar[dr]^{1 \boxtimes g} & \\
                                                     & {GE_1 \boxtimes GR} \ar[d]^{m^G} \ar[r]^(.42){G1 \boxtimes Gk} & {GE_1 \boxtimes G\uL(E_1,E_2)} \ar[d]^{m^G} \\
                                                     & G(E_1 \otimes R) \ar[d]^{Gr} \ar[r]^(.42){G(1 \otimes k)} & G(E_1 \otimes \uL(E_1,E_2)) \ar[d]^{G\Ev} \\
                                                     & GE_1 \ar[r]^{Gh}                                     & GE_2
}
$$
in which $e^G$, $m^G$ are the monoidal structures carried by $G$ and $r$ denotes the structure isomorphism in $\X$ or $\sL$ as appropriate.  The upper cell on the left side of the diagram commutes since $G$ is a monoidal functor, the top cell on the right of the diagram commutes by \bref{par:desc_iso}, the upper square commutes by the naturality of $m^G$, and the bottom square commutes since $h$ corresponds to $k$ under the canonical bijection $\sL(R,\uL(E_1,E_2)) \cong \sL(E_1,E_2)$.
\end{proof}

\begin{ParSub} \label{par:ident_underl_ord}
According to the convention of \bref{par:clsmcat_notn}, we shall now identify the underlying ordinary categories of $\uX$ and $\uL$ with $\X$ and $\sL$, resp., via the canonical isomorphisms \pbref{par:desc_iso}.  Moreover, having shown that the `underlying ordinary category' functors $U_*$ and $V_*$ commute with the change-of-base $G_*$ \pbref{thm:underlying_ord_cats_comm_triangle}, we shall identify the isomorphic 2-functors $U_* \cong V_*G_*$ of \bref{thm:underlying_ord_cats_comm_triangle} via the given isomorphism $\theta_*$.  Hence, for example, we thus identify the underlying ordinary category of $G_*\uL$ with that of $\uL$, i.e. with $\sL$ itself.  (Such identifications can be avoided by working instead with monoidal categories equipped with specified \textit{normalizations} as in the classic \cite{EiKe}, but it is common practice to eschew this extra baggage).

In view of \bref{thm:underlying_ord_functor_iso_to_g}, we shall moreover identify $G:\sL \rightarrow \X$ with the underlying ordinary functor $V_*(\grave{G})$ of $\grave{G}:G_*\uL \rightarrow \uX$.
\end{ParSub}

\begin{LemSub}\label{thm:x_nat_of_eta_and_epsilon}
The given natural transformations $\eta,\varepsilon$ serve as $\X$-natural transformations
$$\eta:1 \rightarrow \grave{G}\acute{F}\;,\;\;\;\;\varepsilon:\acute{F}\grave{G} \rightarrow 1\;.$$
\end{LemSub}
\begin{proof}
We first treat $\eta$.  Using \bref{prop:comp_assoc_enr_fun} and the definition of $\acute{F}$ \pbref{par:given_smcadj}, we find that $\grave{G}\acute{F}$ is equal to the lower composite in the following commutative diagram.
\begin{equation}\label{eqn:expr_for_ggrave_of_facute}
\xymatrix{
 & G_*\uL \ar[dr]^{\grave{G}} & \\
\uX \ar[ur]^{\acute{F}} \ar[r]_(.4){\eta_*\uX} & G_*F_*\uX \ar[u]|{G_*\grave{F}} \ar[r]_(.6){\widegrave{GF}} & \uX
}
\end{equation}
Hence the needed $\X$-naturality of $\eta$ amounts to the requirement that the composite
$$\uX(X,Y) \xrightarrow{\eta} GF\uX(X,Y) \xrightarrow{\widegrave{GF}_{X Y}} \uX(GFX,GFY) \xrightarrow{\uX(\eta_X,GFY)} \uX(X,GFY)$$
be equal to $\uX(X,\eta_Y)$ for all objects $X,Y \in \uX$.  But, using the definition of $\widegrave{GF}_{X Y}$ \pbref{par:ch_base}, this composite is the transpose of the upper path along the periphery of the following diagram
$$
\xymatrix{
 & GFX \boxtimes GF\uX(X,Y) \ar[dr]^{m^{GF}} & \\
X \boxtimes \uX(X,Y) \ar[d]_\Ev \ar[ur]^{\eta \boxtimes \eta} \ar[rr]^\eta & & GF(X \boxtimes \uX(X,Y)) \ar[d]^{GF\Ev} \\
Y \ar[rr]^{\eta} & & GFY
}
$$
which commutes since $\eta$ is a monoidal natural transformation.  This proves the needed equality, since the lower-left composite in this diagram is the transpose of $\uX(X,\eta_Y)$.

In order to establish the $\X$-naturality of $\varepsilon$, let us first show that $\varepsilon$ serves as an $\sL$-natural transformation
\begin{equation}\label{eqn:epsilon_differently_construed}
\xymatrix{
F_*G_*\uL \ar@/^2ex/[rr]^{\widegrave{FG}}_{}="0" \ar@/_2ex/[rr]_{\varepsilon_*\uL}^{}="1" \ar@{=>}"0";"1"^\varepsilon & & {\uL\;.}
}
\end{equation}
This $\sL$-naturality is the requirement that the following square commutes for each pair of objects $E_1,E_2 \in \uL$.
$$
\xymatrix{
FG\uL(E_1,E_2) \ar[d]|{(\varepsilon_*\uL)_{E_1 E_2} \:=\: \varepsilon_{\uL(E_1,E_2)}} \ar[rr]^{\widegrave{FG}_{E_1 E_2}} & & \uL(FGE_1,FGE_2) \ar[d]^{\uL(FGE_1,\varepsilon_{E_2})} \\
\uL(E_1,E_2) \ar[rr]_{\uL(\varepsilon_{E_1},E_2)} & & \uL(FGE_1,E_2)
}
$$
But the two composites in this square are the transposes of the two paths around the periphery of the following diagram
$$
\xymatrix{
FGE_1 \otimes FG\uL(E_1,E_2) \ar[dr]_{\varepsilon \otimes \varepsilon} \ar[r]^{m^{FG}} & FG(E_1 \otimes \uL(E_1,E_2)) \ar[d]^\varepsilon \ar[r]^(.7){FG\Ev} & FGE_2 \ar[d]^\varepsilon \\
 & E_1 \otimes \uL(E_1,E_2) \ar[r]_{\Ev} & E_2
}
$$
which commutes since $\varepsilon$ is a monoidal natural transformation.

Applying $G_*$ to the resulting $\sL$-natural transformation \eqref{eqn:epsilon_differently_construed}, we obtain an $\X$-natural transformation \hbox{$G_*\widegrave{FG} \rightarrow G_*\varepsilon_*\uL$} which (as a family of morphisms) is simply $\varepsilon$ itself.  By the definition of $\acute{F}$ \pbref{par:given_smcadj}, $\acute{F}\grave{G}$ is the composite of the nontrivial path around the periphery of the following diagram.
$$
\xymatrix{
& \uX \ar[dr]^{\eta_*\uX} & \\
G_*\uL \ar[ur]^{\grave{G}} \ar[dr]|{\eta_*G_*\uL} \ar@/_6ex/[ddr]_1 & & G_*F_*\uX \ar@/^6ex/[ddl]^{G_*(\grave{F})} \\
& G_*F_*G_*\uL \ar[ur]|{G_*F_*(\grave{G})} \ar@/_3ex/[d]_{G_*\varepsilon_*\uL}^{}="0" \ar@/^3ex/[d]^{G_*\widegrave{FG}}_{}="1" \ar@{=>}"1";"0"_{\varepsilon} & \\
& G_*\uL &
}
$$
The diamond-shaped cell commutes by the naturality of $\eta_*$, the cell at the left commutes by one of the triangular identities for $F_* \nsststile{\varepsilon_*}{\eta_*} G_*$ \pbref{par:given_smcadj}, and the cell at the right commutes by \bref{prop:comp_assoc_enr_fun}.  Hence each empty cell is an identity 2-cell, and by pasting we obtain a 2-cell $\acute{F}\grave{G} \rightarrow 1$.  Inspecting the diagram, the resulting 2-cell is obtained from the indicated 2-cell $\varepsilon$ by whiskering with the identity-on-objects $\X$-functor $\eta_*G_*\uL$, so this $\X$-natural transformation $\acute{F}\grave{G} \rightarrow 1$ is simply $\varepsilon$.
\end{proof}

\begin{CorSub}\label{thm:und_ord_functor_of_acute_f_is_f}
The underlying ordinary functor of $\acute{F}:\uX \rightarrow G_*\uL$ is $F:\X \rightarrow \sL$.
\end{CorSub}
\begin{proof}
It is immediate from the definition of $\acute{F}$ \pbref{par:given_smcadj} that $\acute{F}$ is given objects just as $F$.  By \bref{thm:x_nat_of_eta_and_epsilon}, $\eta:1 \rightarrow \grave{G}\acute{F}$ is an $\X$-natural transformation and so in particular is a natural transformation of the underlying ordinary functors.  Hence for each morphism $f:X \rightarrow Y$ in $\X$ we have a commutative square as follows.
$$
\xymatrix{
X \ar[r]^{\eta_X} \ar[d]_f & GFX \ar[d]^{\grave{G}\acute{F}f}\\
Y \ar[r]_{\eta_Y}          & GFY
}
$$
But $\grave{G}\acute{F}f = G\acute{F}f$ (since the underlying ordinary functor of $\grave{G}$ is $G$, \bref{par:ident_underl_ord}), so by the universal property of $\eta_X$, $\acute{F}f = Ff$.
\end{proof}

\begin{proof}[Proof of \bref{thm:assoc_enr_adj}]\emptybox

1.  In view of \bref{thm:x_nat_of_eta_and_epsilon}, it now suffices to show that the triangular identities for the needed $\X$-adjunction $\acute{F} \nsststile{\varepsilon}{\eta} \grave{G}$ are satisfied.  But the underlying ordinary functors of $\acute{F},\grave{G}$ are $F,G$, respectively (\bref{par:ident_underl_ord},\bref{thm:und_ord_functor_of_acute_f_is_f}), so the needed triangular identities follow immediately from those of the given adjunction $F \nsststile{\varepsilon}{\eta} G$.

3.  It suffices to show that the composite right $\X$-adjoint $C$ in \eqref{eq:composite_hom_cotensor} is equal to $(G_*\uL)(-,E)$.  On objects, this is immediate.  Given objects $E_1,E_2$ in $(G_*\uL)^\op$, the associated structure morphism $C_{E_1 E_2}$ is the composite
$$G\uL(E_2,E_1) \xrightarrow{G(\uL(-,E)_{E_1 E_2})} G\uL(\uL(E_1,E),\uL(E_2,E)) \xrightarrow{\grave{G}} \uX(G\uL(E_1,E),G\uL(E_2,E))\;.$$
Using the definition of $\grave{G}$ \pbref{par:ch_base}, $C_{E_1 E_2}$ is therefore the transpose of the rightmost path along the periphery of the following commutative diagram
$$
\xymatrix{
G\uL(E_1,E) \boxtimes G\uL(E_2,E_1) \ar[d]^{m^G} \ar[rrr]^(.4){G1 \boxtimes G(\uL(-,E)_{E_1 E_2})} & & & G\uL(E_1,E) \boxtimes G\uL(\uL(E_1,E),\uL(E_2,E)) \ar[d]^{m^G} \\
G(\uL(E_1,E) \otimes \uL(E_2,E_1)) \ar[d]^{Gc} \ar[rrr]^(.4){G(1 \otimes \uL(-,E)_{E_1 E_2})} & & & G(\uL(E_1,E) \otimes \uL(\uL(E_1,E),\uL(E_2,E))) \ar[dlll]^{G\Ev} \\
G\uL(E_2,E)
}
$$
in which $c$ denotes the composition morphism in $\sL$.  On the other hand, the leftmost path is the composition morphism carried by $G_*\uL$ and hence is the transpose of the structure morphism $(G_*\uL)(-,E)_{E_1,E_2}$.  Hence $C = (G_*\uL)(-,E)$.

2.  The composite left $\X$-adjoint in \eqref{eq:composite_hom_cotensor} is given on objects by $X \mapsto \uL(FX,E)$.
\end{proof}

\section{Compatibility with composition of adjunctions}

The following result shows that the construction of \bref{thm:assoc_enr_adj} commutes with composition of adjunctions.

\begin{PropSub} \label{thm:compn_assoc_enr_adj}
Let $\xymatrix {\X \ar@/_0.5pc/[rr]_F^(0.4){}^(0.6){}^{\top} & & \sL \ar@/_0.5pc/[ll]_G}$ and $\xymatrix {\sL \ar@/_0.5pc/[rr]_{P}^(0.4){}^(0.6){}^{\top} & & \U \ar@/_0.5pc/[ll]_{Q}}$ be symmetric monoidal adjunctions between symmetric monoidal closed categories.  Then the $\X$-adjunction
$$
\xymatrix {
\uX \ar@/_0.5pc/[rr]_{\wideacute{PF}}^(0.4){}^(0.6){}^{\top} & & (GQ)_*\uU \ar@/_0.5pc/[ll]_{\widegrave{GQ}}
}
$$
associated to the composite symmetric monoidal adjunction $PF \dashv GQ$ is identical to the composite $\X$-adjunction
$$
\xymatrix {
\uX \ar@/_0.5pc/[rr]_{\acute{F}}^(0.4){}^(0.6){}^{\top} & & G_*\uL \ar@/_0.5pc/[ll]_{\grave{G}} \ar@/_0.5pc/[rr]_{G_*\acute{P}}^(0.4){}^(0.6){}^{\top} & & {G_*Q_*\uU\;,} \ar@/_0.5pc/[ll]_{G_*\grave{Q}}
}
$$
\end{PropSub}
\begin{proof}
By \bref{prop:comp_assoc_enr_fun}, we deduce that the right $\X$-adjoints of these $\X$-adjunctions are equal.  According to \bref{par:ident_underl_ord}, the underlying ordinary adjunction of $G_*\acute{P} \nsststile{}{} G_*\grave{Q}$ is the given adjunction $P \nsststile{}{} Q$.  Hence the two $\X$-adjunctions being compared have the same underlying ordinary adjunction, and the result follows by \bref{thm:enradj_detd_by_radj_and_univ_arr}.
\end{proof}

\section{Commutative monads and symmetric monoidal closed adjunctions} \label{sec:comm_mnd_smc_adj}

Let $\X = (\X,\btimes,I)$ be a closed symmetric monoidal category.

\begin{DefSub}\label{def:tens_str}(Kock \cite{Kock:Comm})
Let $T:\uX \rightarrow \uX$ be an $\X$-functor.  For objects $X,Y \in \X$, we define morphisms
$$t'_{XY} = t'^T_{XY}:T X \btimes Y \rightarrow T(X \btimes Y)\;,\;\;\;\;t''_{XY} = t''^T_{XY}:X \btimes T Y \rightarrow T(X \btimes Y)$$
as the transposes of the following composite morphisms
$$Y \rightarrow \aeX(X,X \btimes Y)
\xrightarrow{\bT} \aeX(T X,T(X \btimes Y))\;,\;\;\;\;X \rightarrow \aeX(Y,X \btimes Y) \xrightarrow{\bT} \aeX(T Y,T(X \btimes Y))\;.$$
\end{DefSub}

\begin{RemSub}\label{rem:enr_nat_of_tens_str}
It is straightforward to verify that $t'_{X Y}, t''_{X Y}$ are $\X$-natural in $X,Y \in \uX$.
\end{RemSub}

\begin{PropSub}\label{prop:xnat_transf_commutes_w_tens_str}
Let $\lambda:T \rightarrow S$ be an $\X$-natural transformation, where $S,T$ are $\X$-endofunctors on $\X$.  Then the following diagrams commute:
$$
\xymatrix{
TX \boxtimes Y \ar[d]_{\lambda_X \boxtimes 1} \ar[r]^{t'^T_{X Y}} & T(X \boxtimes Y) \ar[d]^{\lambda_{X \boxtimes Y}} & & X \boxtimes TY \ar[d]_{1 \boxtimes \lambda_Y} \ar[r]^{t''^T_{X Y}} & T(X \boxtimes Y) \ar[d]^{\lambda_{X \boxtimes Y}} \\
SX \boxtimes Y \ar[r]_{t'^S_{X Y}} & S(X \boxtimes Y) & & X \boxtimes SY \ar[r]_{t''^S_{XY}} & S(X \boxtimes Y)
}
$$
\end{PropSub}
\begin{proof}
One finds that the transposes $Y \rightarrow \uX(TX,S(X \boxtimes Y))$ of the two composites in the leftmost diagram are equal, by the definition of $t'$ and the $\X$-naturality of $\lambda$.  The commutativity of the rightmost diagram is deduced analogously.
\end{proof}

\begin{DefSub}\label{def:comm_mnd}(Kock \cite{Kock:Comm})
Let $\TT = (T,\eta,\mu)$ be an $\X$-monad on $\aeX$.
\begin{enumerate}
\item For objects $X,Y \in \X$, we define morphisms $\otimes_{XY} = \otimes^\TT_{XY}$, $\widetilde{\otimes}_{XY} = \widetilde{\otimes}^\TT_{XY}$ as the following composites:
$$\otimes_{XY} := \left(T X \btimes T Y \xrightarrow{t''_{TX Y}} T(TX \btimes Y) \xrightarrow{T t'_{XY}} TT(X \btimes Y) \xrightarrow{\mu} T (X \btimes Y)\right)\;,$$
$$\widetilde{\otimes}_{XY} := \left(T X \btimes T Y \xrightarrow{t'_{X TY}} T (X \btimes T Y) \xrightarrow{T t''_{XY}} T T (X \btimes Y) \xrightarrow{\mu} T (X \btimes Y)\right)\;.$$
\item $\TT$ is \textit{commutative} if $\otimes_{XY} = \widetilde{\otimes}_{XY}$ for all objects $X,Y$ in $\X$.
\end{enumerate}
\end{DefSub}

\begin{ExaSub}\label{exa:rmod_monad_comm}
Letting $R$ be a commutative ring object in a countably-complete and~\hbox{-cocomplete} cartesian closed category $\X$, we show in \bref{thm:rmod_commutative_monadic} that the category $\RMod(\X)$ of $R$-module objects in $\X$ is isomorphic to the category $\X^\TT$ of Eilenberg-Moore algebras of (the underlying ordinary monad of) a commutative $\X$-monad $\TT$ on $\X$.
\end{ExaSub}

\begin{RemSub}\label{rem:enr_nat_of_tensor_maps}
By \bref{rem:enr_nat_of_tens_str}, we immediately deduce the following:
\begin{itemize}
\item[] $t'_{X Y}$, $t''_{X Y}$, $\otimes_{X Y}$, $\widetilde{\otimes}_{X Y}$ are $\X$-natural in $X,Y \in \uX$.
\end{itemize}
\end{RemSub}

We shall require the following in proving in \bref{sec:comm_idm_mnds} that all idempotent $\X$-monads on $\uX$ are commutative:

\begin{PropSub}\label{thm:mon_structs_assoc_enr_mnd}
Let $\TT = (T,\eta,\mu)$ be an $\X$-monad on $\uX$.
\begin{enumerate}
\item The morphisms $\otimes_{X Y}$, $\widetilde{\otimes}_{X Y}$ of \bref{def:comm_mnd} serve as the structure morphisms of monoidal functors  
$$(T,\otimes,\eta_I), (T,\widetilde{\otimes},\eta_I):\X \rightarrow \X\;.$$
\item $\eta$ serves as a monoidal transformation with respect to either structure:  i.e. both
$$\eta:(1_\V,1,1) \rightarrow (T,\otimes,\eta_I)\;\;\;\;\text{and}\;\;\;\;\eta:(1_\V,1,1) \rightarrow (T,\widetilde{\otimes},\eta_I)$$
are monoidal transformations.
\end{enumerate}
\end{PropSub}
\begin{proof}
1 follows from \cite{Kock:Comm} 2.1 and the remark at the beginning of section 3 there; see also \cite{Kock:SmComm} (2.1)-(2.2).  To deduce that $\eta$ is monoidal with respect to $\widetilde{\otimes}, \eta_I$, we may employ the first part of the proof of Theorem 3.2 of \cite{Kock:Comm}, which applies without the assumption that $\TT$ is commutative (as remarked also in \cite{Kock:BilCartClMnds} 1.5); note that $\widetilde{\otimes} = \psi$ in the notation of \cite{Kock:Comm}, \cite{Kock:BilCartClMnds}.  The statement that $\eta$ is monoidal with respect to $\otimes, \eta_I$ follows, as we now demonstrate.  For all objects $X,Y \in \X$ we have a triangular prism
$$
\xymatrix {
                                                                  & {X \boxtimes Y} \ar[dl]_{\eta_X \boxtimes \eta_Y} \ar[dr]^{\eta_{X \boxtimes Y}} \ar@{.>}[dd]_(.7){\sigma} & \\
{TX \boxtimes TY} \ar[rr]^(.6){\otimes_{X Y}} \ar[dd]_{\sigma}           &                                                                            & {T(X \boxtimes Y)} \ar[dd]_{T\sigma}\\
                                                                  & {Y \boxtimes X} \ar@{.>}[dl]_{\eta_Y \boxtimes \eta_X} \ar@{.>}[dr]^{\eta_{Y \boxtimes X}} & \\
{TY \boxtimes TX} \ar[rr]^(.6){\widetilde{\otimes}_{Y X}}               &                                                                            & {T(Y \boxtimes X)}\\
}
$$
in which $\sigma$ is the symmetry carried by $\X$.  The bottom face commutes.  Also, the left face commutes by the naturality of $\sigma$, the right face commutes by the naturality of $\eta$, and the front face commutes by the first characterization of $\otimes_{XY}$ given in \cite{Kock:BilCartClMnds},\S 1 (in which $\otimes$ is written instead as $\widetilde{\psi}$).  Hence since the vertical morphisms are isos, the top face also commutes.
\end{proof}

\begin{PropSub}\label{prop:mnd_mor_monoidal_wrt_tensors}
Let $\lambda:\TT \rightarrow \SSS$ be a morphism of $\X$-monads $\TT = (T,\eta^\TT,\mu^\TT)$, $\SSS = (S,\eta^\SSS,\mu^\SSS)$ on $\uX$.  Then $\lambda$ serves as a monoidal transformation with respect to each of the following choices of monoidal structure on $T$, $S$:
\begin{description}
\item $\lambda:(T,\otimes^\TT,\eta^\TT_I) \rightarrow (S,\otimes^\SSS,\eta^\SSS_I)$
\item $\lambda:(T,\widetilde{\otimes}^\TT,\eta^\TT_I) \rightarrow (S,\widetilde{\otimes}^\SSS,\eta^\SSS_I)$
\end{description}
\end{PropSub}
\begin{proof}
This follows from \bref{prop:xnat_transf_commutes_w_tens_str} and the diagrammatic axioms for the monad morphism $\lambda$.
\end{proof}

\begin{RemSub} \label{rem:comm_inv_under_iso}
As a corollary to \bref{prop:mnd_mor_monoidal_wrt_tensors}, the property of commutativity is invariant under isomorphism of $\X$-monads.
\end{RemSub}

\begin{ThmSub}\label{thm:comm_sm_mnd}\textnormal{(Kock \cite{Kock:SmComm})}
Let $\TT = (T,\eta,\mu)$ be an ordinary monad on $\X$.  Then there is a bijection between the following kinds of structure on $\TT$:
\begin{enumerate}
\item $\X$-enrichments of $T$ making $\TT$ a commutative $\X$-monad on $\aeX$;
\item monoidal structures on $T$ making $\TT$ a symmetric monoidal monad.
\end{enumerate}
In particular, if $\TT$ is equipped with the structure of a symmetric monoidal monad, then the associated $\X$-enrichment of $T$ consists of the structure morphisms
$$\aeX(X,Y) \rightarrow \aeX(TX,TY)$$
gotten as the transposes of the following composites:
\begin{equation}\label{eqn:transp_tc}TX \btimes \aeX(X,Y) \xrightarrow{1 \btimes \eta} TX \btimes T\aeX(X,Y) \rightarrow T(X \btimes \aeX(X,Y)) \xrightarrow{T(\Ev)} TY\;.\end{equation}
\end{ThmSub}

\begin{DefSub}
Given a symmetric monoidal monad $\TT = (T,\eta,\mu)$ on $\X$, let $\TT^c = (\bT^c,\eta,\mu)$ denote the associated commutative $\X$-monad on $\aeX$ \pbref{thm:comm_sm_mnd}.
\end{DefSub}

\begin{PropSub} \label{thm:tc_via_tgrave}
Given a symmetric monoidal monad $\TT = (T,\eta,\mu)$ on $\X$, the associated $\X$-functor $\bT^c:\aeX \rightarrow \aeX$ is the composite
$$\aeX \xrightarrow{\eta_*\uX} T_*\aeX \xrightarrow{\grave{T}} \aeX\;,$$
where $\eta_*:1_{\XCAT} = (1_{\X})_* \rightarrow T_*$ is gotten by applying $(-)_*$ \pbref{par:ch_base} to the symmetric monoidal transformation $\eta$.
\end{PropSub}
\begin{proof}
$\eta_*\uX$ is identity-on-objects, and $\grave{T}$ acts as $T$ on objects, so both $\bT^c$ and the indicated composite are given as $T$ on objects.  For all $X,Y \in \aeX$, the associated structure morphism of the composite $\X$-functor $\grave{T} \circ \eta_*\uX$ is the composite 
$$\aeX(X,Y) \xrightarrow{\eta} T\aeX(X,Y) \xrightarrow{\grave{T}_{X Y}} \aeX(TX,TY)\;,$$
whose transpose $TX \btimes \aeX(X,Y) \rightarrow TY$ is (by the definition of $\grave{T}$, \bref{par:ch_base}) exactly the composite \eqref{eqn:transp_tc} employed in defining $\bT^c_{X Y}$.
\end{proof}

\begin{PropSub} \label{thm:enradj_assoc_smadj_ind_commmnd}
Let $\xymatrix {\X \ar@/_0.5pc/[rr]_F^(0.4){\eta}^(0.6){\varepsilon}^{\top} & & \W \ar@/_0.5pc/[ll]_G}$ be a symmetric monoidal adjunction with $\X$,$\W$ symmetric monoidal closed, and let $\TT$ be the induced symmetric monoidal monad on $\X$.  Then the associated commutative $\X$-monad $\TT^c$ on $\aeX$ coincides with the $\X$-monad $\TT'$ induced by the $\X$-adjunction
$$
\xymatrix {
\aeX \ar@/_0.5pc/[rr]_{\acute{F}}^(0.4){\eta}^(0.6){\varepsilon}^{\top} & & G_*\aeW \ar@/_0.5pc/[ll]_{\grave{G}}
}
$$
gotten via \bref{thm:assoc_enr_adj}.  In particular, $\TT'$ is commutative.
\end{PropSub}
\begin{proof}
Letting $\TT = (T,\eta,\mu)$, we have that $\TT^c = (\bT^c,\eta,\mu)$.  By \bref{thm:assoc_enr_adj}, the underlying ordinary adjunction of $\acute{F} \nsststile{\varepsilon}{\eta} \grave{G}$ is $F \nsststile{\varepsilon}{\eta} G$, so the underlying ordinary monad of $\TT'$ is equal to that of $\TT$; hence $\TT' = (\bT',\eta,\mu)$, where $\bT' = \grave{G}\acute{F}$.  Further, $\bT'= \bT^c$, since by \bref{thm:tc_via_tgrave}, $T^c$ is the composite
$$\aeX \xrightarrow{\eta_*\uX} (GF)_*\aeX \xrightarrow{\widegrave{GF}} \aeX\;,$$
which is equal to $\grave{G}\acute{F} = T'$, as noted at \eqref{eqn:expr_for_ggrave_of_facute}.
\end{proof}

\section{The symmetric monoidal closed Eilenberg-Moore adjunction}

Let $\X = (\X,\boxtimes,I)$ be a closed symmetric monoidal category, and let $\TT = (T,\eta,\mu)$ be a symmetric monoidal monad on $\X$, which we identify with the corresponding commutative $\X$-monad $\TT^c$ on $\uX$ via the bijection \bref{thm:comm_sm_mnd}.  Let $F \nsststile{\varepsilon}{\eta} G:\X^\TT \rightarrow \X$ be the (ordinary) Eilenberg-Moore adjunction.  Following \cite{Jac:SemWkngContr}, \cite{Seal:TensMndActn}, we make the following definition

\begin{DefSub}\label{def:tens_prod_algs}
We say that $\TT$ has \textit{tensor products of algebras} if for each pair of $\TT$-algebras $A = (X,a)$, $B = (Y,b)$ there is an associated coequalizer in $\X^\TT$ of the (reflexive) pair
$$(T(TX \boxtimes TY),\mu) \xrightarrow{Tm_{X Y}} (TT(X \boxtimes Y), \mu) \xrightarrow{\mu} (T(X \boxtimes Y),\mu)\;,$$
$$(T(TX \boxtimes TY),\mu) \xrightarrow{T(a \boxtimes b)} (T(X \boxtimes Y),\mu)\;,$$
where $m_{X Y}$ is the structure morphism for the monoidal functor $T$.  We denote the associated coequalizer, as object of $\X^\TT$, by $A \otimes B$.
\end{DefSub}

\begin{RemSub} \label{rem:bihomom}
By \cite{Jac:SemWkngContr} 4.1 or \cite{Seal:TensMndActn} 2.3.4, the tensor product $A \otimes B$ of $\TT$-algebras $A$ and $B$ \pbref{def:tens_prod_algs} {represents bimorphisms} in the sense that $\X^\TT(A \otimes B,C) \cong \X^\TT(A,B;C)$, naturally in $C \in \X^\TT$, where $\X^\TT(A,B;C)$ is the class of all \textit{bimorphisms} (or \textit{$\TT$-bihomomorphisms}) from $A,B$ to $C$ (\cite{Jac:SemWkngContr} 5,  \cite{Seal:TensMndActn} 2.1).
\end{RemSub}

\begin{PropSub}
If $\X$ has reflexive coequalizers and $T$ preserves reflexive coequalizers, then $\TT$ has tensor products of algebras \pbref{def:tens_prod_algs}.
\end{PropSub}
\begin{proof}
Apply \cite{Lin:CoeqCatsAlgs}, Corollary 3.
\end{proof}

\begin{ThmSub}\label{thm:smclosed_em_adj}
Let $\TT$ be a symmetric monoidal monad on a closed symmetric monoidal category $\X$ with equalizers, and suppose that $\TT$ has tensor products of algebras \pbref{def:tens_prod_algs}.
\begin{enumerate}
\item The Eilenberg-Moore adjunction $F \nsststile{\varepsilon}{\eta} G:\X^\TT \rightarrow \X$ acquires the structure of a symmetric monoidal adjunction, with $\X^\TT = (\X^\TT,\otimes,FI)$ a closed symmetric monoidal category.
\item Given $\TT$-algebras $A = (X,a), B = (Y,b)$, the underlying $\X$-object of the internal hom $\underline{\X^\TT}(A,B)$ in $\X^\TT$ is the equalizer in $\X$ of the morphisms
\begin{equation}\label{eqn:pair_for_which_hom_is_equalizer}
\uX(X,Y) \xrightarrow{T_{X Y}} \uX(TX,TY) \xrightarrow{\uX(TX,b)} \uX(TX,Y) \;,
\end{equation}
$$\uX(X,Y) \xrightarrow{\uX(a,Y)} \uX(TX,Y)\;.$$
\item The structure morphisms
$$G\underline{\X^\TT}(A,B) \rightarrow \uX(GA,GB)\;,\;\;A,B \in \X^\TT$$
of the closed functor associated to the symmetric monoidal functor $G:\X^\TT \rightarrow \X$ via \bref{par:cl_func} are simply the equalizers of 2.
\end{enumerate}
\end{ThmSub}
\begin{proof}
This is almost entirely clear from \cite{Jac:SemWkngContr} 5.3, \cite{Seal:TensMndActn} (in particular, 2.5.5 and 2.7.3), and \cite{Kock:ClsdCatsGenCommMnds}.  In particular, $\X^\TT$ is a closed symmetric monoidal category (\cite{Jac:SemWkngContr} 5.3), via the same symmetric monoidal structure given in \cite{Seal:TensMndActn}.  The internal hom employed in \cite{Jac:SemWkngContr} 5.3 is the same used in defining a (not-necessarily monoidal) closed structure on $\X^\TT$ in \cite{Kock:ClsdCatsGenCommMnds}, namely the equalizer appearing in 2.  By \cite{Seal:TensMndActn} 2.7.3, the given adjunction is monoidal; indeed the monoidality of the adjunction follows from the strong monoidality of the left adjoint, by \cite{Ke:Doctr} 1.5.  The \textit{symmetry} of this monoidal adjunction does not seem to have been verified or asserted in any of our sources, so we now perform this verification.  In order to show that $F \nsststile{\varepsilon}{\eta} G$ is symmetric monoidal it suffices, again by by \cite{Ke:Doctr} 1.5, to show that $F$ is strong symmetric monoidal, and all that remains for this is to verify the symmetry of the monoidal functor $F$.  Given $X,Y \in \X$, we must show that the following square commutes
$$
\xymatrix{
FX \otimes FY \ar[r] \ar[d] & F(X \boxtimes Y) \ar[d]^{F\sigma_{X Y}} \\
FY \otimes FX \ar[r]        & F(Y \boxtimes X) 
}
$$
where the top and bottom sides are the structure morphisms for the monoidal functor $F$, the left side is the symmetry in $\X^\TT$, and $\sigma$ is the symmetry in $\X$.  But the two composite $\TT$-homomorphisms in this diagram are induced the $\TT$-bihomomorphisms \pbref{rem:bihomom} gotten as the respective composites in the following square
$$
\xymatrix{
{TX \boxtimes TY} \ar[r] \ar[d]^{\sigma_{TX TY}} & T(X \boxtimes Y) \ar[d]^{T\sigma_{X Y}} \\
{TY \boxtimes TX} \ar[r]                         & T(Y \boxtimes X)
}
$$
in which the top and bottom sides are the structure morphisms for the monoidal functor $T$.  But $T$ is a symmetric monoidal functor, so the latter square commutes and hence the former does as well.
\end{proof}

\begin{RemSub} \label{rem:smccl_emadj_wknd_hyp}
We can of course weaken the hypotheses of \bref{thm:smclosed_em_adj} slightly, by assuming only that the needed equalizers \eqref{eqn:pair_for_which_hom_is_equalizer} exist rather than requiring that $\X$ has all equalizers.  Note that this weaker assumption exactly coincides with the condition of the existence of the Eilenberg-Moore $\X$-category $\uX^\TT$ \pbref{par:em_adj} for the corresponding $\X$-monad \pbref{thm:comm_sm_mnd}.  Indeed, the underlying $\X$-object of the internal hom $\underline{\X^\TT}(A,B)$ in $\X^\TT$ as given defined in \bref{thm:smclosed_em_adj} is the same as the hom-object $\uX^\TT(A,B)$ for the Eilenberg-Moore $\X$-category.
\end{RemSub}

\begin{PropSub}\label{thm:commmnd_sm_em_adj_dets_em_vadj}
Let $\X$ be a closed symmetric monoidal category, and let $\TT$ be a commutative $\X$-monad on $\uX$ whose Eilenberg-Moore $\X$-category exists \pbref{par:em_adj}.  Suppose that (the corresponding symmetric monoidal monad) $\TT$ has tensor products of algebras \pbref{def:tens_prod_algs}.  Let $F \nsststile{\varepsilon}{\eta} G:\sL \rightarrow \X$ be the Eilenberg-Moore adjunction for the ordinary monad $\TT$, equipped with its associated structure as a symmetric monoidal adjunction \pbref{thm:smclosed_em_adj} (in which $\sL$ is closed).  Then the associated $\X$-adjunction $\acute{F} \nsststile{\varepsilon}{\eta} \grave{G}:G_*\uL \rightarrow \uX$ \pbref{thm:assoc_enr_adj} coincides with the Eilenberg-Moore $\X$-adjunction $F^\TT \nsststile{\varepsilon}{\eta} G^\TT:\uX^\TT \rightarrow \uX$ determined by the given $\uX$-monad $\TT$.
\end{PropSub}
\begin{proof}
Firstly, these two $\X$-adjunctions have the same underlying ordinary adjunction, namely the Eilenberg-Moore adjunction.  Next, for each pair of objects $A = (X,a)$ and $B = (Y,b)$ in $\sL = \X^\TT$, the hom-object $(G_*\uL)(A,B) = G\uL(A,B)$ is the underlying $\X$-object of the internal hom $\uL(A,B)$ in $\sL$; by definition, this $\X$-object is the equalizer
\begin{equation}\label{eqn:hom_as_equ}
G\uL(A,B) \rightarrow \X(X,Y) = \X(GA,GB)
\end{equation}
of the morphisms \eqref{eqn:pair_for_which_hom_is_equalizer}.  Moreover, by \bref{thm:smclosed_em_adj} 3, this equalizer morphism is the structure morphism of the closed functor determined by $G$, which by \bref{par:cl_func} is equally the structure morphism $\grave{G}_{A B}$ of the $\X$-functor $\grave{G}:G_*\uL \rightarrow \uX$.  On the other hand, the structure morphism
$$G^\TT_{A B}:\uX^\TT(A,B) \rightarrow \X(GA,GB) = \X(X,Y)$$
for the Eilenberg-Moore forgetful $\X$-functor is, by definition \pbref{par:em_adj}, the equalizer of the same pair of morphisms \eqref{eqn:pair_for_which_hom_is_equalizer}.

Hence, the hom-objects of the $\X$-categories $G_*\uL$ and $\uX^\TT$ coincide, as do the structure morphisms of the $\X$-faithful $\X$-functors $\grave{G}$ and $G^\TT$, so by \bref{thm:crit_for_equal_faithful_vfunctors}, the categories $G_*\uL$ and $\uX^\TT$ are identical and $\grave{G} = G^\TT$.

Lastly, the diagram of $\X$-functors
$$
\xymatrix{
\uX \ar[d]_{\acute{F}} \ar[r]^{F^\TT} \ar[dr]^T & \uX^\TT \ar[d]^{G^\TT = \grave{G}} \\
G_*\uL \ar[r]_{\grave{G} = G^\TT}               & \uX
}
$$
commutes by \bref{thm:enradj_assoc_smadj_ind_commmnd}, so since $\grave{G} = G^\TT$ is $\X$-faithful and $\ob F^\TT = \ob \acute{F}$, it follows that $F^\TT = \acute{F}$.
\end{proof}

\begin{ExaSub} \label{exa:rmod_smclosed_monadic}
Returning to the example in \bref{exa:rmod_monad_comm}, we show in \bref{thm:rmod_commutative_monadic} that $\RMod(\X) \cong \X^\TT$ for a commutative $\X$-monad $\TT$ with tensor products of algebras, and moreover that the given isomorphism commutes strictly with both the \textit{free} and the \textit{forgetful} functors.  Identifying the Eilenberg-Moore adjunction with the free-forgetful adjunction $F \dashv G:\RMod(\X) \rightarrow \X$, the latter acquires the structure of a symmetric monoidal adjunction, via \bref{thm:smclosed_em_adj}, with $\RMod(\X)$ symmetric monoidal closed.  This adjunction also underlies an $\X$-enriched adjunction, which by \bref{thm:commmnd_sm_em_adj_dets_em_vadj} can be described both as the Eilenberg-Moore $\X$-adjunction for $\TT$ and as the $\X$-adjunction associated to the symmetric monoidal adjunction $F \dashv G$ via \bref{thm:assoc_enr_adj}.
\end{ExaSub}

\begin{RemSub}\label{rem:standard_cotensors_in_talg}
In the situation of \bref{thm:commmnd_sm_em_adj_dets_em_vadj}, we know by \bref{thm:assoc_enr_adj} that the Eilenberg-Moore $\X$-category $\uX^\TT = G_*\uL$ is cotensored, with cotensors $[X,E]$ $(X \in \X, E \in \uX^\TT)$ obtained as the internal homs $\uL(FX,E)$ in $\sL = \X^\TT$.  But there is also another canonical choice of cotensors:  As noted in \cite{Kock:Dist} \S 2, one obtains a cotensor $[X,E]$ by equipping the hom $\uX(X,GE)$ with what we may call the \textit{pointwise} $\TT$-algebra structure.
\end{RemSub}

\begin{ExaSub}\label{exa:std_cotensors_in_rmod}
Returning to the example of \bref{exa:rmod_smclosed_monadic}, where $\X^\TT \cong \RMod(\X)$, let us assume for the sake of illustration that $\X$ is a locally small well-pointed cartesian closed category \pbref{par:well_pointed_ccc}.  Then the underlying set functor $\X \rightarrow \Set$ sends $R$ to a commutative ring, and for each $X \in \X$ and $E \in \RMod(\X)$, the cotensor $[X,E]$ formed in the manner of \bref{rem:standard_cotensors_in_talg} consists of the function space $\uX(X,E)$ equipped with the familiar pointwise $R$-module structure.
\end{ExaSub}

\section{Commutative idempotent monads and symmetric monoidal closed reflections}\label{sec:comm_idm_mnds}

\begin{LemSub}\label{sec:univ_biunit}
Let $\SSS = (S,\rho,\lambda)$ be an idempotent $\V$-monad on $\uV$.  Then for all objects $U,V$ of $\uV$ and each object $W$ of the associated $\V$-reflective-subcategory $\uV^{(\SSS)}$ of $\uV$,
$$\uV(\rho_U \otimes \rho_V,W):\uV(SU \otimes SV,W) \rightarrow \uV(U \otimes V,W)$$
is an isomorphism in $\V$ (and hence also induces a bijection between the associated hom-classes of $\V$).
\end{LemSub}
\begin{proof}
By \bref{thm:refl_subcats_are_orth_subcats}, $\uV^{(\SSS)} = \uV_{\Sigma_S}$ and $(\Sigma_S,(\Sigma_S)^{\downarrow_\V})$ is a $\V$-prefactorization-system on $\uV$.  Hence since $\rho_U,\rho_V \in \Sigma_S$, we deduce by \bref{prop:upperclass_closed_under_tensors} that each of the two composable morphisms
$$U \otimes V \xrightarrow{U \otimes \rho_V} U \otimes SV \xrightarrow{\rho_U \otimes SV} SU \otimes SV$$
lies in $\Sigma_S$, so their composite $\rho_U \otimes \rho_V$ lies in $\Sigma_S$.  Hence $\rho_U \otimes \rho_V \bot_\V W$, yielding the needed result.
\end{proof}

\begin{PropSub} \label{thm:idm_mnd_comm}
Every idempotent $\V$-monad on $\uV$ is commutative.
\end{PropSub}
\begin{proof}
Let $\SSS = (S,\rho,\lambda)$ be an idempotent $\V$-monad on $\uV$.  By \bref{thm:mon_structs_assoc_enr_mnd}, $\rho$ is a monoidal transformation with respect to each of the monoidal structures $\otimes, \widetilde{\otimes}$.  Hence, for all $U,V \in \V$, the diagram
$$
\xymatrix{
                        & U \otimes V \ar[dl]_{\rho_U \otimes \rho_V} \ar[dr]^{\rho_{U \otimes V}}       & \\
SU \otimes SV \ar[rr]^m &                                                                               & S(U \otimes V)
}
$$
commutes for each $m = \otimes_{U V}, \widetilde{\otimes}_{U V}$, so by \bref{sec:univ_biunit}, $\otimes_{U V} = \widetilde{\otimes}_{U V}$.
\end{proof}

\begin{LemSub}\label{thm:idm_mnd_tens_prod_algs}
Let $\SSS = (S,\rho,\lambda)$ be an idempotent $\V$-monad on $\uV$.  Then (the associated symmetric monoidal monad) $\SSS$ has tensor products of algebras \pbref{def:tens_prod_algs}.  In particular, identifying $\V^\SSS$ with the reflective subcategory $\V^{(\SSS)}$ determined by $\SSS$, the tensor product $U \widehat{\otimes} V$ of $\SSS$-algebras $U,V \in \V^\SSS$ is simply 
$$U \widehat{\otimes} V = S(U \otimes V)\;.$$
\end{LemSub}
\begin{proof}
As $\SSS$-algebras, $U$ and $V$ carry structure morphisms $\rho_U^{-1}:SU \rightarrow U$ and $\rho_V^{-1}:SV \rightarrow V$, respectively.  The tensor product $U \widehat{\otimes} V$ (if it exists) is defined \pbref{def:tens_prod_algs} as the coequalizer in $\V^\SSS$ of the morphisms
$$S(SU \otimes SV)  \xrightarrow{S(m_{U \otimes V})} SS(U \otimes V) \xrightarrow{\lambda_{U \otimes V}} S(U \otimes V)$$
$$S(SU \otimes SV) \xrightarrow{S(\rho_U^{-1} \otimes \rho_U^{-1})} S(U \otimes V)\;,$$
where $m$ is the monoidal structure carried by $S$.  But these morphisms are equal, since
$$
\xymatrix {
S(SU \otimes SV) \ar[rr]^{S(m_{U \otimes V})} & & SS(U \otimes V) \ar[d]^{\lambda_{U \otimes V}} \\
S(U \otimes V) \ar[u]^{S(\rho_U \otimes \rho_V)}_\wr \ar@{=}[rr] \ar[urr]_{S\rho_{U \otimes V} } & & S(U \otimes V)
}
$$
commutes (as $\rho$ is a monoidal transformation).  Hence $S(U \otimes V)$ serves as their coequalizer in $\V$, but $S(U \otimes V)$ lies in the full subcategory $\V^{(\SSS)} = \V^\SSS$ of $\V$ and hence serves also as coequalizer in this subcategory.
\end{proof}

\begin{ThmSub}\label{thm:enr_refl_smadj}
Let $K \nsststile{}{\rho} J : \W \hookrightarrow \uV$ be a $\V$-reflection.
\begin{enumerate}
\item Then the underlying ordinary adjunction $K \nsststile{}{\rho} J : \W \hookrightarrow \V$ carries the structure of a symmetric monoidal adjunction, with $\W = (\W,\widehat{\otimes},KI)$ a symmetric monoidal closed category.
\item The associated $\V$-adjunction $\acute{K} \nsststile{}{\rho} \grave{J} : J_*\uW \rightarrow \uV$ \pbref{thm:assoc_enr_adj} is identical to the given $\V$-reflection.
\item For $U,V \in \W$, the internal hom $\uW(U,V)$ is simply $\uV(U,V)$.
\end{enumerate}
\end{ThmSub}
\begin{proof}
Letting $\SSS = (S,\rho,\lambda)$ be the induced idempotent $\V$-monad on $\uV$, we have by \bref{thm:em_adj_idm_mnd} that the Eilenberg-Moore $\V$-category for $\SSS$ exists and that we may identify the Eilenberg-Moore $\V$-adjunction with the the given $\V$-reflection, so that $\W = \V^\SSS$.   Further, $\SSS$ is commutative, by \bref{thm:idm_mnd_comm}, and in view of \bref{thm:idm_mnd_tens_prod_algs} and \bref{rem:smccl_emadj_wknd_hyp} we may apply \bref{thm:smclosed_em_adj} and \bref{thm:commmnd_sm_em_adj_dets_em_vadj} in order to obtain 1 and 2.  By 2, $J_*\uW$ is equally the given $\V$-reflective-subcategory $\W$ of $\uV$, so 3 follows as
$$\uW(U,V) = J\uW(U,V) = (J_*\uW)(U,V) = \W(U,V) = \uV(U,V)\;.$$
\end{proof}

The preceding theorem yields an alternative proof of the following result of Day \cite{Day:Refl}:

\begin{CorSub}\label{thm:day_refl}
Let $\V$ be a closed symmetric monoidal category, let $\W$ be a (full, replete) reflective subcategory of $\V$, with associated adjunction $K \nsststile{}{\rho} J : \W \hookrightarrow \V$, and suppose that
$$\forall V \in \V, W \in \W \;:\; \uV(V,W) \in \W\;.$$
Then the given adjunction acquires the structure of a symmetric monoidal adjunction, with $\W$ a closed symmetric monoidal category.
\end{CorSub}
\begin{proof}
It suffices to show that for each $V \in \V$, $W \in \W$, the morphism $\uV(\rho_V,W):\uV(KV,W) \rightarrow \uV(V,W)$ is an iso in $\V$, for then by \bref{prop:enr_of_ord_adj} the given adjunction is the underlying ordinary adjunction of a $\V$-reflection, and hence an invocation of \bref{thm:enr_refl_smadj} yields the needed structure.  To this end, observe that for each $U \in \V$, we have a commutative square of classes
$$
\xymatrix {
\V(U,\uV(KV,W)) \ar[rr]^{\V(U,\uV(\rho_V,W))} \ar[d]_\wr & & \V(U,\uV(V,W)) \ar[d]_\wr \\
\V(KV,\uV(U,W)) \ar[rr]^{\V(\rho_V,\uV(U,W))}             & & \V(V,\uV(U,W))
}
$$
in which the left and right sides are isomorphisms.  Since $\uV(U,W) \in \W$ by our hypothesis, the bottom side is an isomorphism, so the top side is also an isomorphism.  Hence by the Yoneda lemma, $\uV(\rho_V,W)$ is an isomorphism as needed.
\end{proof}

%% file: fin_comm_alg_th.tex
\chapter{Enriched finitary commutative algebraic theories}
\setcounter{subsection}{0}

Whereas \textit{universal algebra} or \textit{general algebra} studies in generality the varieties of algebraic objects defined by operations and equations, it was established by Lawvere \cite{Law:PhdTh} that each variety of finitary algebras can be described as a category $\T$, called an \textit{algebraic theory}, that embodies all possible derived operations and equations generic to the algebras of that variety.  Hence algebraic theories or \textit{Lawvere theories} describe varieties of algebras in a \textit{presentation-independent} way and are related to the \textit{clones} of universal algebra.  Algebraic theories and monads permit the consideration of algebraic objects in categories $\X$ other than $\Set$, such as the category of topological spaces.  But the usual $\Set$-based algebraic theories are insufficient for describing topological $R$-modules, since such theories allow only operations of the form $X^n \rightarrow X$, whereas a topological $R$-module $X$ must carry an action $R \times X \rightarrow X$ continuous with respect to the product of topological spaces.  This limitation is overcome when one employs instead \textit{$\X$-enriched} algebraic theories $\T$, which could also be called \textit{enriched single-sorted finite product theories}.  The theory of finitary algebraic theories enriched in a cartesian closed category $\X$ was developed in unpublished work of Gray \cite{Gray:UnAlgCCC} and treated briefly for the more general \textit{$\pi$-categories} $\X$ by Borceux and Day \cite{BoDay}.

Building upon work of Linton in the non-enriched case \cite{Lin:AutEqCats}, Borceux and Day give a definition of \textit{commutative} $\X$-enriched finitary theory.  Commutative theories are those in which the operations commute with each other in a strong sense:  For example, binary operations $*,\circ:X^2 \rightarrow X$ on a set $X$ commute in this sense if and only if
$$(w * x) \circ (y * z) = (w \circ y) * (x \circ z)$$
for all $w,x,y,z \in X$.  In the $\Set$-based case, $R$-modules for a commutative ring $R$ are the algebras of a commutative theory.  The aim of this chapter is to confirm the following, for suitable cartesian closed categories $\X$:
\begin{enumerate}
\item Every commutative $\X$-enriched finitary theory $\T$ determines an associated commutative $\X$-monad $\TT$ with \textit{tensor products of algebras} (\bref{def:comm_mnd},\bref{def:tens_prod_algs}).
\item The $\X$-category of algebras $\Alg{\T}$ of an $\X$-theory $\T$ is equivalent to that of the associated monad $\TT$, and its full subcategory of \textit{normal} $\T$-algebras \pbref{def:talgs} is \textit{isomorphic} to $\uX^\TT$.
\item The category of $R$-module objects for a commutative ring object $R$ in $\X$ is isomorphic to the category of normal $\T$-algebas for a commutative $\X$-theory $\T$.
\end{enumerate}
For such \textit{finitary} theories, the existence of tensor products of algebras in 1 follows from the existence of all reflexive coequalizers in the category of algebras.  While the equivalence $\Alg{\T} \simeq \uX^\TT$ in 2 is known and 1 and 3 are unsurprising and ought to be known, it seems that 1 and 3 have never been proved in the literature.  By proving these statements we make available an essential class of examples for the theory developed in Chapters \bref{ch:smc_adj_comm_mnd}, \bref{ch:nat_acc_dist}, and \bref{ch:dcompl_vint}, including our paradigmatic example of $R$-module objects in $\X$.

The reader may skip the present chapter if content to accept Theorem \bref{thm:rmod_commutative_monadic}.  We begin in \bref{sec:fin_alg_ccc} with a largely self-contained treatment of finitary enriched algebraic theories sufficient to enable proofs of 1 and 3 in sections \bref{sec:comm_alg_th_ccc} and \bref{sec:rmods_alg}, respectively.

\section{Finitary algebra in a cartesian closed category}\label{sec:fin_alg_ccc}

\begin{DefSub}\label{def:xtheory}
Let $\X$ be a category with finite products.  A \textit{(single-sorted finitary algebraic) $\X$-theory} consists of an $\X$-category $\T$ whose objects are the finite ordinal numbers $0,1,2,...$, together with an assignment to each object $n$ a cone $(\pi_i:n \rightarrow 1)_{i = 1}^n$ that exhibits $n$ as an $n$-th $\X$-power of $1$ in $\T$.
\end{DefSub}

\begin{ParSub}\label{par:th_in_presence_of_finite_copowers}
For the remainder of this section, let us assume (unless otherwise indicated) that $\X$ is a cartesian closed category (with designated finite products) equipped with finite copowers $n \cdot 1$ of the terminal object~$1$.
\begin{enumerate}
\item Let $\N$ denote the category whose objects are finite ordinals and whose morphisms are arbitrary functions between these ordinals, so that $\N(m,n)$ may be identified with $n^m$.
\item We may form the free $\X$-category $\N_\X$ on $\N$ (just as per \cite{Ke:Ba} 2.5), with objects those of $\N$ and with $\N_\X(m,n) = n^m \cdot 1$.  The underlying ordinary category of $\N_\X$ may be identified with $\N$, and then for each ordinary functor $Q:\N \rightarrow \B$ into an $\X$-category $\B$, there is a unique $\X$-functor $\N_\X \rightarrow \B$ whose underlying ordinary functor is $Q$.  We shall simply write $\N$ to denote $\N_\X$.
\item There is a canonical functor $J = (-)\cdot 1:\N \rightarrow \X$.  The induced $\X$-functor $J:\N \rightarrow \uX$ (2) is $\X$-fully-faithful, since its structure morphisms $J_{m n}:\N(m,n) \rightarrow \uX(m \cdot 1, n \cdot 1)$ are the composite isomorphisms $n^m \cdot 1 \cong (n \cdot 1)^m \cong \uX(m \cdot 1,n \cdot 1)$.
\item $\N$ has finite $\X$-copowers.  Given finite ordinals $m,n$, there is an associated $\X$-copower cocone $(\iota_i:m \rightarrow n \times m)_{i = 1}^n$ in $\N$.   Hence $\N^\op$ has finite $\X$-powers, with designated $\X$-power cones $(\pi_i:n \times m \rightarrow m)_{i = 1}^n$ in $\N^\op$ gotten as simply the cocones $(\iota_i)$ of $\N$.
\item For such a category $\X$, an $\X$-theory may be defined equivalently as an $\X$-category $\T$ equipped with an identity-on-objects $\X$-functor $\tau:\N^\op \rightarrow \T$ which preserves finite $\X$-powers.  $\T$ then has finite $\X$-powers, with a canonical $\X$-power cone $(\pi_i:\tau(n \times m) \rightarrow \tau(m))_{i = 1}^n$ for each pair of ordinals $m,n$ gotten by applying $\tau$ to the canonical $\X$-copower cocone $(\iota_i)$ in $\N$.
\end{enumerate}
\end{ParSub}

\begin{DefSub}\label{def:talgs}
Let $\T$ be an $\X$-theory.
\begin{enumerate}
\item A \textit{$\T$-algebra} is an $\X$-functor $A:\T \rightarrow \uX$ which preserves finite $\X$-powers.
\item A \textit{normal $\T$-algebra} is an $\X$-functor $A:\T \rightarrow \uX$ which sends each of the designated $\X$-power cones $(\pi_i:n \rightarrow 1)_{i = 1}^n$ in $\T$ to the designated $\X$-power cone in $\uX$.  (It then follows that $A$ preserves all finite $\X$-powers).
\item Given $\T$-algebras $A,B$, we refer to $\X$-natural transformations $h:A \rightarrow B$ as \textit{$\T$-homomorphisms}.
\item We let $\Alg{\T}$ denote the full subcategory of $\XCAT(\T,\uX)$ consisting of $\T$-algebras, and we let $\Alg{\T}_!$ denote the full subcategory of $\Alg{\T}$ consisting of normal $\T$-algebras.
\end{enumerate}
\end{DefSub}

\begin{ExaSub}
The representable $\X$-functors $\T(t,-)$ $(t \in \T)$ on an $\X$-theory $\T$ are $\T$-algebras (but need not be normal).
\end{ExaSub}

\begin{RemSub}\label{rem:first_alt_desc_talgs}
For an $\X$-theory $\T$, one can easily show that an $\X$-functor $A:\T \rightarrow \uX$ is a $\T$-algebra if and only if the composite 
$$\N^\op \xrightarrow{\tau} \T \xrightarrow{A} \uX$$
preserves finite $\X$-powers.  Analogously, $A$ is a normal $\T$-algebra if and only if this composite strictly preserves the designated $\X$-power cones $(\pi_i:n \rightarrow 1)_{i = 1}^n$ of $\N^\op$. 
\end{RemSub}

\begin{RemSub}
Given any $\X$-functor $C:\N^\op \rightarrow \uX$ and an object $n \in \N$, the canonical cocone presenting $n$ as a copower $n \cdot 1$ in $\N$ determines a comparison morphism $\kappa_n:Cn \rightarrow (C1)^n$ in $\uX$, and these morphisms constitute an $\X$-natural transformation $\kappa:C \rightarrow (C1)^{(-)}$.  $C$ preserves finite $\X$-powers iff $\kappa$ is iso, and $C$ strictly preserves the designated $\X$-power cones $(\pi_i:n \rightarrow 1)_{i = 1}^n$ of $\N^\op$ iff $\kappa$ is an identity.

Hence, given an $\X$-theory $\T$ and an $\X$-functor $A:\T \rightarrow \uX$, $A$ is a $\T$-algebra iff the comparison transformation
\begin{equation}\label{eqn:compn_trans}
\xymatrix{
                                    & \T \ar[dr]^{A} \ar@{}[d]_\Downarrow^\kappa & \\
\N^\op \ar[ur]^\tau \ar[rr]_{(A1)^{(-)}} & & \uX 
}
\end{equation}
is iso, and $A$ is a normal $\T$-algebra iff $\kappa$ is an identity, i.e. iff the triangle commutes.
\end{RemSub}

\begin{RemSub}\label{rem:alt_desc_normal_talg}
A normal $\T$-algebra $A$ is uniquely determined by the choice of object $X := A1$ together with the family of morphisms
$$A_{m n}:\T(m,n) \rightarrow \uX(X^m,X^n)\;,\;\;\;\;(m,n \in \N)\;,$$
which in turn are uniquely determined by just the morphisms
$$A_{m 1}:\T(m,1) \rightarrow \uX(X^m,X^1)\;,\;\;\;\;(m \in \N)\;.$$
It follows that a normal $\T$-algebra can be defined alternatively as an object $X$ of $\X$ together with morphisms
\begin{equation}\label{eqn:ops} A_m:\T(m,1) \times X^m \rightarrow X\;,\;\;\;\;(m \in \N)\end{equation}
such that each of the following diagrams commute
\begin{equation}
\xymatrix{
1 \times X^m \ar[r]^(.4){u \times 1} \ar[dr]_{\sim} & \T(m,m) \times X^m \ar[d]^{A_{m m}} & & \T(m,n) \times \T(l,m) \times X^l \ar[r]^(.65){c \times 1} \ar[d]_{1 \times A_{l m}} & \T(l,n) \times X^l \ar[d]^{A_{l n}}\\                                                 
& X^m                                 & & \T(m,n) \times X^m \ar[r]^{A_{m n}} & {X^n\;,}
}                                  
\end{equation}
where $u$ and $c$ are the unit and composition morphisms for $\T$, respectively, and each of the morphisms
$$A_{m n}:\T(m,n) \times X^m \rightarrow X^n\;,\;\;\;\;(m,n \in \N)$$
is induced by the family
$$\T(m,n) \times X^m \xrightarrow{\T(m,\pi_i) \times 1} \T(m,1) \times X^m \xrightarrow{A_m} X\;,\;\;\;\;(i = 1,...,n)\;.$$
Further, it is equivalent to require the commutativity of the rightmost diagram in just those case where $n = 1$.
\end{RemSub}

\begin{RemSub}\label{rem:alt_desc_thomom}
Given a $\T$-homomorphism $h:A \rightarrow B$, the naturality of $h$ and the preservation of finite powers by $A$ and $B$ imply that the components $h_m$ of $h$ can be expressed in terms of just the component $h_1:A1 \rightarrow B1$ via the commutativity of the following diagram.
$$
\xymatrix {
Am \ar[d]_{\wr} \ar[r]^{h_m} & Bm \ar[d]^\wr \\
(A1)^m \ar[r]^{h_1^m}               & (B1)^m
}
$$
In view of \bref{rem:alt_desc_normal_talg}, it follows that a $\T$-homomorphism of normal $\T$-algebras $h:(X,(A_m)) \rightarrow (Y,(B_m))$ may be defined equivalently as a morphism $h:X \rightarrow Y$ such that each diagram
\begin{equation}\label{eqn:alt_desc_thomom}
\xymatrix{
\T(m,1) \times X^m \ar[d]_{A_m} \ar[rr]^{1 \times h^m} & & \T(m,1) \times Y^m \ar[d]^{B_m} \\
X \ar[rr]^h                                                    & & Y
}
\end{equation}
commutes.
\end{RemSub}

\begin{DefSub}\label{def:alt_def_norm_talg}
The alternative definition of normal $\T$-algebra given in \bref{rem:alt_desc_normal_talg} applies equally in the case that $\X$ is not necessarily cartesian closed but is merely assumed to have (designated) finite products.  Similarly, the notion of $\T$-homomorphism (of normal $\T$-algebras) is also available, in view of \bref{rem:alt_desc_thomom}.
\end{DefSub}

\begin{ParSub}\label{par:assns_for_monadicity_of_algt}
For the remainder of the section, we assume that $\X$ is a countably-complete and~\hbox{-cocomplete} cartesian closed category, and we let $\T$ be an $\X$-theory.  Since the set of objects of $\T$ is countable, we can form the $\X$-category of $\X$-functors $[\T,\uX]$ (as per \cite{Ke:Ba} 2.2), and we obtain full sub-$\X$-categories $\Alg{\T}$ and $\Alg{\T}_!$ of $[\T,\uX]$ whose underlying ordinary categories are the categories of $\T$-algebras and normal $\T$-algebras \pbref{def:talgs}.
\end{ParSub}

\begin{PropSub}\label{thm:alg_iso_nalg}
There is an equivalence of $\X$-categories $\Alg{\T} \simeq \Alg{\T}_!$.
\end{PropSub}
\begin{proof}
Given a $\T$-algebra $A$, the comparison transformation $\kappa:A\tau \Rightarrow (A1)^{(-)}:\N^\op \rightarrow \uX$ \eqref{eqn:compn_trans} is iso.  Concretely, the components of $\kappa$ are isomorphisms $\kappa_n:An \rightarrow X^n$ $(n \in \N)$ in $\uX$, where we let $X := A1$, so we can define an $\X$-functor $NA:\T \rightarrow \uX$ on objects by $(NA)(n) := X^n$ and on homs by taking each $(NA)_{m n}$ to be the following composite:
$$
\xymatrix{
\T(m,n) \ar[dr]_{A_{m n}} \ar[rr]^{(NA)_{m n}} &                                                       & \uX((NA)m,(NA)n) \\
                                            & \uX(Am,An) \ar[ur]^\sim_{\uX(\kappa^{-1}_m,\kappa_n)} &
}
$$
It is then immediate that the morphisms $\kappa_n$ constitute an $\X$-natural isomorphism $A \Rightarrow NA:\T \rightarrow \uX$, which we refer to as $\rho$ in order to distinguish it from $\kappa:A\tau \Rightarrow (A1)^{(-)}:\N^\op \rightarrow\uX$.

In order to show that $NA$ is a normal $\T$-algebra, it suffices by \bref{rem:first_alt_desc_talgs} to show that the composite $\N^\op \xrightarrow{\tau} \T \xrightarrow{NA} \uX$ is simply $X^{(-)}$.  This composite $(NA)\tau$ acts on objects as $X^{(-)}$, and its structure morphism $((NA)\tau)_{m n}$ for each pair $m,n \in \N^\op$ appears within a commutative diagram
$$
\xymatrix{
\N(n,m) \ar[rr]^{((NA)\tau)_{m n}} \ar[dd]_{\tau_{m n}} & & \uX(X^m,X^n) \ar[dr]^{\uX(\kappa_m,X^n)}_\sim & \\
& & & \uX(Am,X^n) \\
\T(m,n) \ar[rr]^{A_{m n}} & & \uX(Am,An) \ar[ur]^\sim_{\uX(Am,\kappa_n)}
}
$$
But if we substitute $X^{(-)}_{m n}$ for $((NA)\tau)_{m n}$ in this diagram, the resulting diagram also commutes, by the $\X$-naturality of $\kappa:A\tau \Rightarrow X^{(-)}$.  Hence since the morphisms on the right side of the diagram are iso, $((NA)\tau)_{m n} = X^{(-)}_{m n}$ as needed.

Thus for each $A \in \Alg{\T}$ we have an associated isomorphism $\rho_A:A \rightarrow NA$ in $\Alg{\T}$ with $NA \in \Alg{\T}_!$.  Since for each $B \in \Alg{\T}_!$, the morphism
$$\Alg{\T}(\rho_A,B):\Alg{\T}(NA,B) \rightarrow \Alg{\T}(A,B)$$
is iso, we obtain by \bref{thm:enradj_detd_by_radj_and_univ_arr} a left $\X$-adjoint $N:\Alg{\T} \rightarrow \Alg{\T}_!$ to the inclusion $\Alg{\T}_! \hookrightarrow \Alg{\T}$, and $N$ is an equivalence of $\X$-categories since the unit $\rho$ is iso.
\end{proof}

\begin{DefSub}
Given $\X$-theories $\sS,\T$, a \textit{(normal) morphism of $\X$-theories} $\phi:\sS \rightarrow \T$ is an identity-on-objects $\X$-functor which strictly preserves finite $\X$-powers.  Equivalently, $\phi:\sS \rightarrow \T$ is an $\X$-functor such that
$$
\xymatrix{
\sS \ar[rr]^{\phi} &                                      & \T \\
                                    & \N^\op \ar[ul]^\sigma \ar[ur]_{\tau} &
}
$$
commutes, where $\sigma$ and $\tau$ are the associated $\X$-functors \pbref{par:th_in_presence_of_finite_copowers}.

Given a morphism of $\X$-theories $\phi:\sS \rightarrow \T$, the $\X$-functor $[\phi,\uX]:[\T,\uX] \rightarrow [\sS,\uX]$ clearly restricts to an $\X$-functor $\phi^\natural:\Alg{\T} \rightarrow \Alg{\sS}$ and further to an $\X$-functor $\phi^!:\Alg{\T}_! \rightarrow \Alg{\sS}_!$.
\end{DefSub}

\begin{ParSub}\emptybox
\begin{enumerate}
\item $\N^\op$ itself is an $\X$-theory whose associated $\X$-functor simply the identity $\N^\op \rightarrow \N^\op$.
\item There is an equivalence of $\X$-categories $\Alg{\N^\op} \simeq \uX$, which sends an $\N^\op$-algebra $A$ to $A1$.  This restricts to an isomorphism of $\X$-categories $\Alg{\N^\op}_! \cong \uX$.
\item Given an $\X$-theory $\T$, the associated $\X$-functor $\tau:\N^\op \rightarrow \T$ is the unique (normal) morphism from $\N^\op$ to $\T$, and we obtain a commutative diagram of $\X$-functors
\begin{equation}\label{eqn:alg_forg_funcs_diag}
\xymatrix{
\Alg{\T}_! \ar@{^{(}->}[r]^\simeq \ar[d]^{\tau^!}       & \Alg{\T} \ar[d]^{\tau^\natural} \\
\Alg{\N^\op}_! \ar@{^{(}->}[r]^{\simeq} \ar[d]_{\cong} & \Alg{\N^\op} \ar[d]^{\simeq} \\
\uX \ar@{=}[r]                                         & \uX
}
\end{equation}
in which $\simeq$ indicates an equivalence and $\cong$ an isomorphism.  We denote the leftmost and rightmost vertical composites by $G_!$ and $G_\natural$, respectively, and we call these the \textit{forgetful $\X$-functors}.  Each is simply given by evaluation at $1$.
\end{enumerate}
\end{ParSub}

\begin{ParSub}\label{par:ladjs_to_alg_funcs}
Let $\phi:\sS \rightarrow \T$ be a morphism of $\X$-theories.  Since $\X$ is countably\hbox{-cocomplete} and $\sS,\T$ have countably many objects, the $\X$-functor $[\phi,\uX]:[\T,\uX] \rightarrow [\sS,\uX]$ has a left $\X$-adjoint $\Lan_\phi$, given by pointwise left $\X$-Kan-extension along $\phi$; see \cite{Ke:Ba} or \cite{Dub}.  By an argument as in \cite{BoDay} 2.2.1, if an $\X$-functor $A:\sS \rightarrow \uX$ preserves finite $\X$-powers, then its pointwise left $\X$-Kan-extension $\Lan_\phi A:\T \rightarrow \uX$ also preserves finite $\X$-powers.  Hence $\Lan_\phi$ restricts to a left $\X$-adjoint $\phi_\natural$  for $\phi^\natural:\Alg{\T} \rightarrow \Alg{\sS}$.  Hence we also obtain a left $\X$-adjoint $\phi_!$ to $\phi^!:\Alg{\T}_! \rightarrow \Alg{\sS}_!$ by composing with the equivalences of \bref{thm:alg_iso_nalg}.  Letting $\y s := \sS(s,-) : \sS \rightarrow \uX$, $s \in \sS$, note that
$$\phi_\natural \y s = \Lan_\phi \y s \cong \T(\phi(s),-)$$
$\X$-naturally in $s \in \sS$.
\end{ParSub}

\begin{DefSub}
Given an $\X$-theory $\T$, we define
\begin{description}
\item $F_\natural := (\uX \xrightarrow{\simeq} \Alg{\N^\op} \xrightarrow{\tau_\natural} \Alg{\T})$
\item $F_! := (\uX \xrightarrow{\cong} \Alg{\N^\op}_! \xrightarrow{\tau_!} \Alg{\T}_!)$\;,
\end{description}
so that $F_\natural$,$F_!$ are left $\X$-adjoint to the forgetful $\X$-functors $G_\natural$,$G_!$, respectively.
\end{DefSub}

\begin{RemSub}\label{rem:charn_f_nat}
One readily computes that $F_\natural$ is given by 
$$F_\natural X = \int^{n \in \N} \T(\tau(n),-) \times X^n$$
$\X$-naturally in $X \in \uX$.  Observe also that the free $\T$-algebra on a finite copower $n \cdot 1$ of the terminal object $1$ is 
$$F_\natural(n \cdot 1) = \tau_\natural((n \cdot 1)^{(-)}) \cong \tau_\natural(\N(-,n)) \cong \T(\tau(n),-)\;,$$
$\X$-naturally in $n \in \N$.
\end{RemSub}

\begin{DefSub}\label{def:mnds_assoc_th}
Let $\TT_\natural$, $\TT_!$ be the $\X$-monads induced by the $\X$-adjunctions
\begin{enumerate}
\item[] $F_\natural \dashv G_\natural:\Alg{\T} \rightarrow \uX$
\item[] $F_! \dashv G_!:\Alg{\T}_! \rightarrow \uX$
\end{enumerate}
respectively.
\end{DefSub}

\begin{RemSub}\label{rem:charn_t_nat}
By \bref{rem:charn_f_nat}, the endofunctor $T_\natural = G_\natural F_\natural$ is given by
$$T_\natural X = \int^{n \in \N} \T(\tau(n),1) \times X^n$$
$\X$-naturally in $X \in \uX$.
\end{RemSub}

\begin{PropSub}\label{thm:assoc_mnds_iso}
The $\X$-monads $\TT_\natural$, $\TT_!$ are isomorphic.
\end{PropSub}
\begin{proof}
Letting $\SSS$ be the $\X$-monad induced by the equivalence $N \nsststile{}{\rho} E:\Alg{\T}_! \hookrightarrow \Alg{\T}$, the unit $\rho$ is iso, so the morphism of $\X$-monads $\rho:\ONEONE_{\Alg{\T}} \rightarrow \SSS$ \pbref{par:adj_and_mnd} is an isomorphism.  Hence, applying the monoidal functor $[F_\natural,G_\natural]$ \pbref{thm:mon_func_detd_by_adj} to this isomorphism, we find that 
$$\TT_\natural = [F_\natural,G_\natural](\ONEONE_{\Alg{\T}}) \cong [F_\natural,G_\natural](\SSS) = \TT_!\;.$$
\end{proof}

\begin{PropSub}\label{thm:nalg_forg_func_cr_refl_coeqs}\emptybox
\begin{enumerate}
\item $G_!:\Alg{\T}_! \rightarrow \uX$ creates $\X$-coequalizers of reflexive pairs.
\item $G_\natural:\Alg{\T} \rightarrow \uX$ detects, preserves, and reflects $\X$-coequalizers of reflexive pairs.
\end{enumerate}
\end{PropSub}
\begin{proof}
It suffices to prove the first statement, since the second follows, as we have an equivalence of $\X$-categories $\Alg{\T}_! \hookrightarrow \Alg{\T}$ which commutes with $G_!,G_\natural$ \eqref{eqn:alg_forg_funcs_diag}.  Given a reflexive pair of morphisms $g,h:A \rightarrow B$ in $\Alg{\T}_!$ and an $\X$-coequalizer $p:G_!B\rightarrow Z$ in $\uX$ of the reflexive pair $G_! g = g_1, G_! h = h_1$, we must show firstly that there is a unique morphism $q:B \rightarrow C$ in $\uX$ with $G_! q = p$, and secondly that $q$ is an $\X$-coequalizer of $g,h$.  Let $X := G_! A$, $Y := G_!B $.  Since $p$ is a coequalizer of the reflexive pair $g_1,h_1$ in $\X$, it follows by \cite{Joh:Ele} A1.2.12 that for each $n \in \N$, $p^n:Y^n \rightarrow Z^n$ is a coequalizer of $g_1^n,h_1^n:X^n \rightarrow Y^n$, and by \bref{thm:lim_mono_epi_in_base_are_enr} these are $\X$-coequalizers in $\X$.  Hence we have an $\X$-coequalizer $q_n := p^n:Bn \rightarrow Cn := Z^n$ of each pair of components $g_n = g_1^n$, $h_n = h_1^n$, so we obtain a pointwise $\X$-coequalizer $q:B \rightarrow C$ of $g,h$ in $[\T,\uX]$.  

To see that $C$ is a normal $\T$-algebra, observe that for each designated $\X$-power projection $\pi_i:n \rightarrow 1$ in $\T$ we have a naturality square
$$
\xymatrix{
Y^n \ar[r]^{q_n = p^n} \ar[d]_{B\pi_i} & Z^n \ar[d]^{C\pi_i} \\
Y \ar[r]^{q_1 = p}                       & {Z}
}
$$
for $q$, but since $B\pi_i = \pi_i$ is the designated projection (since $B$ is normal), this square also commutes when we replace $C\pi_i$ by the designated projection $\pi_i$, so since $p^n$ is epi, $C\pi_i = \pi_i$.

Hence $q$ serves as the needed $\X$-coequalizer in $\Alg{\T}_!$.  In view of the characterization of $\T$-homomorphisms of normal $\T$-algebras at \eqref{eqn:alt_desc_thomom}, it remains only to show that the structure morphisms $C_{n}:\T(n,1) \times Z^n \rightarrow Z$ of the normal $\T$-algebra $(Z,(C_n))$ are uniquely determined by the property that $p:Y \rightarrow Z$ is a $\T$-homomorphism $(Y,(B_n)) \rightarrow (Z,(C_n))$, i.e. by the commutativity of the squares
$$
\xymatrix{
\T(n,1) \times Y^n \ar[d]_{B_n} \ar[r]^{1 \times p^n} & \T(n,1) \times Z^n \ar[d]_{C_n} \\
Y \ar[r]_p & Z
}
$$
for each $n \in \N$.  But this is immediate, since $1 \times p^n$ is epi (since $p^n$ is epi and $\T(n,1) \times (-):\X \rightarrow \X$ is a left adjoint).
\end{proof}

\begin{ThmSub}\label{thm:nalg_monadic}
Let $\T$ be an $\X$ theory, where $\X$ is a countably-complete and~\hbox{-cocomplete} cartesian closed category.
\begin{enumerate}
\item The $\X$-adjunction $F_\natural \dashv G_\natural:\Alg{\T} \rightarrow \uX$ is $\X$-monadic \pbref{def:monadic}.
\item The $\X$-adjunction $F_! \dashv G_!:\Alg{\T}_! \rightarrow \uX$ is strictly $\X$-monadic \pbref{def:monadic}.
\end{enumerate}
\end{ThmSub}
\begin{proof}
Using \bref{thm:nalg_forg_func_cr_refl_coeqs}, we can invoke the Crude Monadicity Theorem \pbref{thm:crude_mndcity}.
\end{proof}

\section{Commutative algebraic theories in a cartesian closed category}\label{sec:comm_alg_th_ccc}

\begin{ParSub}\label{par:power_morphs}
Given an $\X$-theory $\T$ \pbref{sec:fin_alg_ccc}, let us denote the designated $m$-th $\X$-power of each object $n$ \pbref{par:th_in_presence_of_finite_copowers} of $\T$ by $[m,n]$, so that $[m,n]$ is simply the product of ordinals $m \times n$.  We therefore have canonical morphisms $[m,-]_{n k}:\T(n,k) \rightarrow \T(m \times n,m \times k)$ for all finite ordinals $m,n,k$.  Observe that each $[m,n]$ serves as a cotensor $[m \cdot 1,n]$ in $\T$ of $n$ by the $m$-th copower $m \cdot 1$ of the terminal object $1$ of $\uX$, since
$$\T(l,[m,n]) \cong \T(l,n)^m \cong \uX(1,\T(l,n))^m \cong \uX(m \cdot 1,\T(l,n))$$
$\X$-naturally in $l \in \T$.
\end{ParSub}

The following notion was defined in \cite{BoDay}, building upon the $\Set$-based case in \cite{Lin:AutEqCats}.

\begin{DefSub}\label{def:comm_th}
An $\X$-theory $\T$ is \textit{commutative} if the following diagram commutes for all finite ordinals $m,n$
\begin{equation}\label{eqn:commutative_th}
\xymatrix{
\T(m,1)\times\T(n,1) \ar[d]_{[n,-]_{m 1} \times 1} \ar[r]^(.45){1 \times [m,-]_{n 1}} & \T(m,1)\times\T(m \times n,m) \ar[d]^c \\
\T(m \times n,n)\times\T(n,1) \ar[r]^c & \T(m \times n, 1)
}
\end{equation}
where $c$ denotes the composition morphism for $\T$.
\end{DefSub}

\begin{LemSub}\label{thm:emb_theory_into_kleisli_op}
Let $\T$ be an $\X$-theory, with $\X$ a countably-complete and~\hbox{-cocomplete} cartesian closed category.  Then there is an $\X$-fully-faithful $\X$-functor $\T \rightarrowtail \uX_{\TT_\natural}^\op$ sending each object $n$ of $\T$ to the copower $n \cdot 1$ in $\X$ of the terminal object $1$ of $\X$ (where $\uX_{\TT_\natural}$ denotes the Kleisli $\X$-category of the associated $\X$-monad $\TT_\natural$, \bref{def:mnds_assoc_th}).
\end{LemSub}
\begin{proof}
Let $\TT := \TT_\natural$.  The comparison $\X$-functor (\cite{Dub} II.1) $K:\uX_\TT \rightarrow \uX^\TT$ induced by the Kleisli $\X$-adjunction is $\X$-fully-faithful.  Moreover, we have a (strictly) commutative diagram
$$
\xymatrix{
\uX_\TT \ar@{ >->}[r]^K & \uX^\TT & \Alg{\T} \ar[l]_L^\simeq \\
& \uX \ar[ul]^{F_\TT} \ar[u]|{F^\TT} \ar[ur]_{F_\natural} & 
}
$$
in which the upper row consists of the comparison $\X$-functors, the rightmost of which is an equivalence by \bref{thm:nalg_monadic}.  Also, the Yoneda embedding $\T^\op \rightarrowtail [\T,\uX]$ restricts to an $\X$-fully-faithful $\X$-functor $\y:\T^\op \rightarrow \Alg{\T}$, and as noted in \bref{rem:charn_f_nat}, $\y n = \T(n,-) \cong F_\natural(n \cdot 1)$ for each finite ordinal $n$.  Hence the $\X$-fully-faithful composite
$$\T^\op \xrightarrow{\y} \Alg{\T} \xrightarrow{L} \uX^\TT$$
sends each object $n \in \T$ to $L\y n \cong LF_\natural(n \cdot 1) = F^\TT(n \cdot 1) = KF_\TT(n \cdot 1) = K(n \cdot 1)$, and it follows that the given composite factors up to isomorphism through $K$ via an $\X$-fully-faithful $\X$-functor given on objects by $n \mapsto n \cdot 1$.
\end{proof}

\begin{ParSub}
Given a symmetric monoidal closed category $\X = (\X,\boxtimes,I)$ and an $\X$-monad $\TT = (T,\eta,\mu)$ on $\uX$, there is for each pair of objects $X \in \X$, $Y \in \uX_\TT^\op$ a cotensor $[X,Y]$ in $\uX_\TT^\op$ given as simply the monoidal product $X \boxtimes Y$ in $\X$.  The associated unit morphism $\xi:X \rightarrow \X_\TT^\op([X,Y],Y) = \uX(Y,T(X \boxtimes Y))$ is simply the transpose of $\eta_{X \boxtimes Y}:X \boxtimes Y \rightarrow T(X \boxtimes Y)$.  In analogy with \bref{par:power_morphs}, \bref{def:comm_th}, we have canonical morphisms $[X,-]_{Y I}:\uX_\TT^\op(Y,I) \rightarrow \uX_\TT^\op(X \boxtimes Y, X)$.
\end{ParSub}

\begin{LemSub}\label{thm:tensor_morphs_via_kleisli_comp}
Let $\X = (\X,\boxtimes,I)$ be a symmetric monoidal closed category, and let $\TT$ be an $\X$-monad on $\uX$.  Then the canonical morphisms $\otimes_{X Y},\widetilde{\otimes}_{X Y}:TX \boxtimes TY \rightarrow T(X \boxtimes Y)$ \pbref{def:comm_mnd}, $X,Y \in \X$, may be obtained from the lower-left and upper-right composites (resp.) in the following square
\begin{equation}\label{eqn:tensor_morphs_via_kleisli_comp}
\xymatrix{
\uX_\TT^\op(X,I)\boxtimes\uX_\TT^\op(Y,I) \ar[d]^{[Y,-]_{X I} \boxtimes 1} \ar[r]^(.45){1 \boxtimes [X,-]_{Y I}} & \uX_\TT^\op(X,I)\boxtimes\uX_\TT^\op(X \boxtimes Y,X) \ar[d]^c \\
\uX_\TT^\op(X \boxtimes Y,Y)\boxtimes\uX_\TT^\op(Y,I) \ar[r]^c & \uX_\TT^\op(X \boxtimes Y,I)
}
\end{equation}
by composing with the canonical isomorphisms 
$$\uX_\TT^\op(X,I)\boxtimes\uX_\TT^\op(Y,I) = \uX(I,TX)\boxtimes\uX(I,TY) \cong TX \boxtimes TY$$
and 
$$\uX_\TT^\op(X \boxtimes Y, I) = \uX(I,T(X \boxtimes Y)) \cong T(X \boxtimes Y)\;.$$  Hence $\TT$ is commutative iff each such square commutes.
\end{LemSub}
\begin{proof}
We will prove given statement relating $\widetilde{\otimes}_{X Y}$ to the upper-right composite; the proof of analogous statement regarding $\otimes_{X Y}$ is analogous.  Observe that the composite in question appears along the left side of the following diagram
$$
\renewcommand{\objectstyle}{\scriptstyle}
\renewcommand{\labelstyle}{\scriptstyle}
\xymatrix@!0@C=12ex @R=6ex{
\uX_\TT^\op(X,I)\boxtimes\uX_\TT^\op(Y,I) \ar[d]_{1 \boxtimes [X,-]_{Y I}} \ar@{=}[rr] & & \uX(I,TX)\boxtimes\uX(I,TY) \ar[rr]^\sim & & TX \boxtimes TY \ar[d]^{1 \boxtimes \tilde{t}_{XY}''} \ar@{=}[rr] & & TX \boxtimes TY \ar[dd]|{t'_{X TY}} \ar@/^7ex/[dddd]^{\widetilde{\otimes}_{XY}} \\
\uX_\TT^\op(X,I)\boxtimes\uX_\TT^\op(X\boxtimes Y,X) \ar[ddd]_c \ar@{=}[rr] & & \uX(I,TX)\boxtimes\uX(X,T(X\boxtimes Y)) \ar[d]_{1\boxtimes T} \ar[rr]^\sim & & TX \boxtimes \uX(X,T(X\boxtimes Y)) \ar[d]^{1 \boxtimes T} & & \\
 & & \uX(I,TX)\boxtimes \uX(TX,TT(X\boxtimes Y)) \ar[d]_c \ar[rr]^\sim & & TX \boxtimes \uX(TX,TT(X\boxtimes Y)) \ar[d]^\Ev & & T(X \boxtimes TY) \ar[d]^{Tt''_{XY}}\\
 & & \uX(I,TT(X\boxtimes Y)) \ar[d]_{\uX(1,\mu)} \ar[rr]^\sim & & TT(X \boxtimes Y) \ar[d]^\mu \ar@{=}[rr] & & TT(X\boxtimes Y) \ar[d]^\mu \\
 \uX_\TT^\op(X \boxtimes Y,I) \ar@{=}[rr] && \uX(I,T(X \boxtimes Y)) \ar[rr]^\sim & & T(X\boxtimes Y) \ar@{=}[rr] & & T(X \boxtimes Y) 
}
$$
in which $\tilde{t}_{XY}''$ is the transpose of $t''_{XY}:X \boxtimes TY \rightarrow T(X \boxtimes Y)$ \pbref{def:comm_mnd}.  It suffices to show that the diagram commutes.  Observe first that the rectangle at the bottom-left commutes by the definition of the composition morphism for $\uX_\TT$.  The cell at the extreme right side of the diagram commutes by the definition of $\widetilde{\otimes}$.  Also the three squares in middle column clearly commute, as does the square at the bottom-right.  Regarding the large rectangle at the top-right, we find that the transposes of the composites therein are the two composites on the periphery of the following diagram, which clearly commutes.
$$
\xymatrix {
TY \ar[r] \ar[d]_{\tilde{t}''_{XY}} & \uX(X,X \boxtimes TY) \ar[dl]^{\uX(X,t''_{XY})} \ar[r]^T & \uX(TX,T(X \boxtimes TY)) \ar[d]^{\uX(TX,Tt''_{XY})} \\
\uX(X,T(X\boxtimes Y)) \ar[rr]_T & & \uX(TX,TT(X\boxtimes Y))
}
$$
Regarding the remaining rectangle, at the top-left, we must show that the morphisms
\begin{enumerate}
\item[] $\uX(I,TY) = \uX_\TT^\op(Y,I) \xrightarrow{[X,-]_{Y I}} \uX_\TT^\op(X \boxtimes Y,X) = \uX(X,T(X \boxtimes Y))$
\item[] $\uX(I,TY) \xrightarrow{\sim} TY \xrightarrow{\tilde{t}''_{XY}} \uX(X,T(X\boxtimes Y))$
\end{enumerate}
are equal, but the first is the transpose of the path around the top, right, and bottom sides of the following commutative diagram, and the second is the transpose of the left side.
$$
\renewcommand{\objectstyle}{\scriptstyle}
\renewcommand{\labelstyle}{\scriptscriptstyle}
\xymatrix@!0@C=12ex @R=6ex{
X \boxtimes \uX(I,TY) \ar@{=}[rr] \ar[d]_{\wr} & & X \boxtimes \uX_\TT^\op(Y,I) \ar[rrr]^{\xi \boxtimes 1} & & & \uX_\TT^\op(X \boxtimes Y,Y) \boxtimes \uX_\TT^\op(Y,I) \ar@{=}[ddl] \ar[dddd]^c \\
X \boxtimes TY \ar@/_13ex/[dddd]^{t''_{X Y}} \ar[d]_{\upsilon \boxtimes 1} \ar[drr]^{\xi \boxtimes 1} & & & & & \\
\uX(Y,X \boxtimes Y) \boxtimes TY \ar[d]_{T \boxtimes 1} \ar[rr]_(.45){\uX(1,\eta) \boxtimes 1} & & \uX(Y,T(X \boxtimes Y)) \boxtimes TY \ar[d]_{T \boxtimes 1} \ar@{<-}[rr]^\sim & & \uX(Y,T(X\boxtimes Y)) \boxtimes \uX(I,TY) \ar[d]^{T \boxtimes 1} & \\
\uX(TY,T(X\boxtimes Y)) \boxtimes TY \ar@{=}[d] \ar[rr]^{\uX(1,T\eta) \boxtimes 1} & & \uX(TY,TT(X\boxtimes Y)) \boxtimes TY \ar[dll]|{\uX(1,\mu) \boxtimes 1} \ar[d]^\Ev \ar@{<-}[rr]^\sim & & \uX(TY,TT(X \boxtimes Y)) \boxtimes \uX(I,TY) \ar[d]^c & \\
\uX(TY,T(X \boxtimes Y)) \boxtimes TY \ar[d]^\Ev & & TT(X \boxtimes Y) \ar[dll]_\mu \ar@{<-}[rr]^\sim & & \uX(I,TT(X \boxtimes Y)) \ar@{=}[r] \ar[dr]|{\uX(1,\mu)} & \uX_\TT^\op(X \boxtimes Y,I) \ar@{=}[d] \\
T(X \boxtimes Y) \ar@{<-}[rrrrr]^\sim & & & & & \uX(I,T(X \boxtimes Y))
}
$$

\end{proof}

\begin{ThmSub}\label{thm:mnds_assoc_comm_xth_are_comm}
Let $\T$ be a commutative $\X$-theory, where $\X$ is a countably-complete and~\hbox{-cocomplete} cartesian closed category.  Then the associated (isomorphic) $\X$-monads $\TT_\natural$, $\TT_!$ are commutative and have tensor products of algebras \pbref{def:tens_prod_algs}, and
$$\uX^{\TT_\natural} \simeq \Alg{\T} \simeq \Alg{\T}_! \cong \uX^{\TT_!}\;.$$
\end{ThmSub}
\begin{proof}
The given isomorphisms and equivalences were established in \bref{thm:nalg_monadic}, so it suffices to show that $\TT_\natural$ is commutative, since it then follows by \bref{rem:comm_inv_under_iso} that $\TT_!$ is also commutative and further, by \bref{thm:nalg_forg_func_cr_refl_coeqs}, that both have tensor products of algebras.

Let $\TT := \TT_\natural$, $\TT = (T,\eta,\mu)$.  By \bref{rem:enr_nat_of_tensor_maps}, we have $\X$-natural transformations
$$
\xymatrix{
\uX \times \uX \ar[r]^{T \times T} \ar[d]_\times & \uX \times \uX \ar@{=>}[dl]|{\otimes,\widetilde{\otimes}} \ar[d]^\times \\
\uX \ar[r]_T & \uX
}
$$
and we must show that $\otimes = \widetilde{\otimes}$.  But we claim that the upper-right composite $\X$-functor is a left $\X$-Kan-extension of its restriction along $J \times J:\N \times \N \rightarrow \uX \times \uX$ \pbref{par:th_in_presence_of_finite_copowers}, so that by \cite{Ke:Ba} 4.43 it then suffices to show that the restrictions of $\otimes,\widetilde{\otimes}$ along $J \times J$ are equal.  Indeed, using \bref{rem:charn_t_nat}, \bref{rem:charn_f_nat}
$$
\begin{array}{lll}
TX \times TY & = & (\int^{m \in \N} \T(\tau(m),1) \times X^m) \times (\int^{n \in \N} \T(\tau(n),1) \times Y^n) \\
& \cong & \int^{m,n \in \N} \T(\tau(m),1) \times \T(\tau(n),1) \times X^m \times Y^n \\
 & \cong & \int^{m,n \in \N} T(m \cdot 1) \times T(n \cdot 1) \times \uX(m \cdot 1,X) \times \uX(n \cdot 1,Y)\\
 & = & \int^{m,n \in \N} TJm \times TJn \times \uX(Jm,X) \times \uX(Jn,Y) \\
 & = & (\Lan_{J \times J}(TJ-\:\times\:TJ-))(X,Y)
\end{array}
$$
$\X$-naturally in $X,Y \in \uX$.

Hence it now suffices to show that $\otimes_{m \cdot 1, n \cdot 1} = \widetilde{\otimes}_{m \cdot 1, n \cdot 1}$ for all finite copowers $m \cdot 1$, $n \cdot 1$ of the terminal object $1$ of $\X$.  But the equality of these morphisms is equivalent to the commutativity of the square \eqref{eqn:tensor_morphs_via_kleisli_comp} with $X := m \cdot 1$, $Y := n \cdot 1$, and via the $\X$-fully-faithful $\X$-functor $\T \rightarrowtail \uX_\TT^\op$ \pbref{thm:emb_theory_into_kleisli_op} this reduces to the commutativity of the analogous square \eqref{eqn:commutative_th} characterizing the commutativity of $\T$.
\end{proof}

\section{$R$-modules in a category}\label{sec:rmods_alg}

\begin{ParSub}\label{par:th_of_rmods}
Let $\X$ be a category with finite products, and let $R$ be a commutative (unital) ring object in $\X$.  We shall define an $\X$-theory $\T_R$ with
$$\T_R(m,n) := R^{n \times m}\;,\;\;\;\;(m,n\;\text{finite ordinals})\;.$$
As the composition morphisms, we take the \textit{matrix multiplication} morphisms
$$R^{n \times m} \times R^{m \times l} = \T_R(m,n)\times\T_R(l,m) \xrightarrow{c} \T_R(l,n) = R^{n \times l}\;,$$
induced by the composites
$$R^{n \times m} \times R^{m \times l} \xrightarrow{\pi_i \times \pi_k} R^m \times R^m \xrightarrow{((\pi_j,\pi_j))_{j=1}^m} (R \times R)^m \xrightarrow{(\cdot)^m} R^m \xrightarrow{(+)_m} R\;,$$
($i = 1,...,n$, $k = 1,...,l$), where $(\cdot)$ is the multiplication carried by $R$ and $(+)_m$ is the $m$-fold addition.  We take the identity morphism
$$\delta:1 \rightarrow R^{n \times n} = \T_R(n,n)\;,\;\;\;\;(\text{$n$ a finite ordinal})$$
to be the \textit{identity matrix} induced by the morphisms
$$\delta_{i j} : 1 \rightarrow R\;,\;\;\;\;(i,j = 1,...,n)$$
defined as the multiplicative identity of $R$ if $i = j$ and the zero otherwise.
\end{ParSub}

\begin{PropSub}\label{thm:tr_enr_theory}
$\T_R$ is an $\X$-theory.
\end{PropSub}
\begin{proof}
Let $\SET$ be a category of classes \pbref{par:cat_classes} in which lie the hom-classes of $\X$.  For each object $X$ of $\X$, the functor $\X(X,-):\X \rightarrow \SET$ preserves products, so the commutative ring object $R$ in $\X$ is sent to a commutative ring object $\X(X,R)$ in $\SET$.  In order to verify that $\T_R$ satisfies the diagrammatic laws for an $\X$-category, it suffices (by Yoneda) to verify for each $X \in \X$ the commutativity of the associated diagrams in $\SET$ obtained by applying the product-preserving functor $\X(X,-)$.  In this way, we reduce this verification to the case of $\X = \SET$, where the associativity and unit laws are simply the familiar associativity and unit laws for matrix multiplication.

Having thus deduced that $\T_R$ is an $\X$-category, we claim that for each object $n$ of $\T_R$, the morphisms
$$p_i : 1 \rightarrow R^n = \T_R(n,1)\;,\;\;\;\;(i = 1,...,n)$$
in $\X$ induced by $\delta_{i j}:1 \rightarrow R$ $(j = 1,...,n)$ serve as the components
$$p_i:n \rightarrow 1\;,\;\;\;\;(i = 1,...,n)$$
in $\T_R$ of an $\X$-power cone.  Indeed, for each object $m \in \T_R$, the (ordinary) functor $\T_R(m,-):\T_R \rightarrow \X$ sends each $p_i$ to the composite
$$R^{n \times m} = \T_R(m,n) \xrightarrow{(p_i,1)} \T_R(n,1)\times\T_R(m,n) = R^{1 \times n} \times R^{n \times m} \xrightarrow{c} \T_R(m,1) = R^m\;,$$
which we claim is merely the projection $\pi_i$ in $\X$; again we may reduce the verification of this claim to the case of $\X = \SET$, where it is immediate.
\end{proof}

\begin{PropSub}
The $\X$-theory $\T_R$ is commutative \pbref{def:comm_th}.
\end{PropSub}
\begin{proof}
One computes that the canonical morphisms
$$[m,-]_{n 1}:R^n = \T(n,1) \rightarrow \T(m \times n,m \times 1) = R^{m \times (m \times n)}$$
employed in the definition of commutativity (\bref{def:comm_th}, \bref{par:power_morphs}) are induced by the morphisms
$$\theta_{j j'}:R^n \rightarrow R^n\;,\;\;\;\;(j,j' = 1,...,m)$$
where $\theta_{j j'}$ is the identity morphism if $j = j'$ and the zero morphism otherwise.  We must show that the diagram
$$
\xymatrix{
R^m \times R^n \ar[d]_{[n,-]_{m 1} \times 1} \ar[r]^(.45){1 \times [m,-]_{n 1}} & R^m\times R^{m \times (m \times n)} \ar[d]^c \\
R^{n \times (m \times n)}\times R^n \ar[r]^c & R^{m \times n}
}
$$
commutes for all $m,n$, where $c$ denotes the matrix multiplication morphism.  Again using the Yoneda lemma as in \bref{thm:tr_enr_theory}, we reduce to the case of $\X = \SET$, where it is straightforward to verify that both composites send a pair of row vectors $(U = (U_j)_{j = 1}^m,V = (V_i)_{i = 1}^n)$ to the matrix $(U_j \cdot V_i)_{j=1,...,m,i=1,...,n}$.
\end{proof}

\begin{ThmSub}\label{thm:rmod_cat_of_algs_comm_th}
Given a commutative ring object $R$ in a category $\X$ with finite products, the category $\RMod(\X)$ of $R$-module objects in $\X$ is isomorphic to the category $\Alg{\T_R}_!$ of normal $\T_R$-algebras for the commutative $\X$-theory $\T_R$.
\end{ThmSub}
\begin{proof}
Given an $R$-module object $M$ in $\X$, we can define for each finite ordinal $n$ an associated `linear combination' morphism
$$A_n:\T(n,1) \times M^n = R^{1 \times n} \times M^n \rightarrow M$$
as the composite
$$R^n \times M^n \xrightarrow{(\pi_i,\pi_i)_{i=1}^n} (R \times M)^n \xrightarrow{(\cdot)^n} M^n \xrightarrow{(+)_n} M$$
where $(\cdot)$ is the scalar multiplication carried by $M$ and $(+)_n$ is the $n$-fold addition.  In order to show that $(M,(A_n))$ is a normal $\T_R$-algebra \pbref{def:alt_def_norm_talg}, we must show that the diagrams
$$
\xymatrix{
1 \times M^m \ar[r]^(.4){\delta \times 1} \ar[dr]_{\sim} & R^{m \times m} \times M^m \ar[d]|{(A_m \cdot (\pi_j \times 1))_{j=1}^m} & & R^{1 \times m} \times R^{m \times l} \times M^l \ar[r]^(.65){c \times 1} \ar[d]|{1 \times (A_l \cdot (\pi_j \times 1))_{j=1}^m} & R^{1 \times l} \times M^l \ar[d]^{A_l}\\                                                 & M^m                                 & & R^{1 \times m} \times M^m \ar[r]^{A_m} & {M^1\;,}
}
$$
commute.  Again this reduces to the case of $\X = \SET$ as in \bref{thm:tr_enr_theory}, where the verification is straightforward.

Conversely, given a normal $\T_R$-algebra $(X,(A_n))$, we can equip $X$ with the structure of an $R$-module object by taking as addition morphism the composite
$$X^2 \cong 1 \times X^2 \xrightarrow{(e,e) \times 1} R^2 \times X^2 \xrightarrow{A_2} X$$
(where $e$ is the multiplicative identity for $R$), as scalar multiplication simply $A_1:R^1 \times X^1 \rightarrow X$, and as additive identity $1 \xrightarrow{\sim} R^{1 \times 0} \times X^0 \xrightarrow{A_0} X$.  In order to verify the diagrammatic axioms for an $R$-module, one again reduces to the straightforward case of $\X = \SET$.

One may verify through another reduction to $\X = \SET$ that these processes are inverse to one another and further that every $\T_R$-homomorphism is a homomorphism of the corresponding $R$-module objects, and vice-versa.
\end{proof}

\begin{ThmSub} \label{thm:rmod_commutative_monadic}
Let $R$ be a commutative ring object in a countably-complete and~\hbox{-cocomplete} cartesian closed category $\X$.  Then the category $\RMod(\X)$ of $R$-module objects in $\X$ is isomorphic to the (ordinary) category of Eilenberg-Moore algebras of a commutative $\X$-monad $\TT$ with tensor products of algebras \pbref{def:tens_prod_algs}.
\end{ThmSub}
\begin{proof}
This follows from \bref{thm:rmod_cat_of_algs_comm_th} and \bref{thm:mnds_assoc_comm_xth_are_comm}.
\end{proof}

%% file: nat_acc_dist.tex
\chapter{The natural and accessible distribution monads}\label{ch:nat_acc_dist}
\setcounter{subsection}{0}

Given a commutative ring object $R$ in a suitable cartesian closed category $\X$, whose objects we call \textit{spaces}, we call the objects of the category $\sL$ of $R$-modules in $\X$ \textit{linear spaces}.  We call the morphisms in these categories simply \textit{maps} and \textit{linear maps}, respectively.  As an example, we may take $\X$ to be the category $\Conv$ of \textit{convergence spaces} \pbref{exa:conv_sm_sp}, which includes all topological spaces as a full subcategory; with $R = \RR$ or $\CC$, $\sL$ is then the category of \textit{convergence vector spaces}.  Another example is the category $\X = \Smooth$ of Fr\"olicher's \textit{smooth spaces} \pbref{exa:conv_sm_sp}, which includes paracompact smooth manifolds as a full subcategory, and here $\sL$ is the category of \textit{smooth vector spaces}.  In each case we can associate to every space $X$ a canonical linear space $[X,R]$ of maps $X \rightarrow R$, \textit{continuous} maps in our first example, and \textit{smooth} in our second.  Indeed, $[X,R]$ may be characterized as the \textit{cotensor} of $R$ in the $\X$-enriched category $\sL$.  Considering any locally compact Hausdorff space $X$ as an object of $\X = \Conv$, the (continuous) linear maps 
$$\mu:[X,R] \rightarrow R$$
are the \textit{$R$-valued Radon measures of compact support} on $X$ \pbref{exa:nat_dist_conv}.  On the other hand, given any separable paracompact smooth manifold $X$, as object of $\X = \Smooth$, the (smooth) linear maps $\mu:[X,R] \rightarrow R$ are the \textit{Schwartz distributions of compact support} on $X$ \pbref{exa:nat_dist_smooth}.

In general, we call the morphisms $[X,R] \rightarrow R$ in $\sL$ \textit{natural distributions} on $X$.  For suitable $\X$, the category of linear spaces $\sL$ is \textit{symmetric monoidal closed}, so we can form a canonical linear space $DX = \sL([X,R],R) = [X,R]^*$ whose elements are natural distributions.  Further, the resulting endofunctor $D:\X \rightarrow \X$ is part of an $\X$-enriched monad $\DD$ on $\uX$.  With categories of smooth spaces in mind, Kock has called $\DD$ the \textit{Schwartz double-dualization monad} \cite{Kock:ProResSynthFuncAn,Kock:Dist}.

Generalizing the above, we work in the present chapter with a given symmetric monoidal  adjunction $F \dashv G:\sL \rightarrow \X$ between given symmetric monoidal closed categories $\X$ and $\sL$.  We give a succinct definition of the natural distribution monad $\DD$ in this setting \pbref{def:nat_dist}, and we observe in \pbref{sec:nat_distn_dbl_dualn} that
$$DX = [X,R]^* \cong (FX)^{**}\;,$$
so that $DX$ is the double-dual of the \textit{free} linear space $FX$ on $X$.
Further, the $\X$-monad $\DD$ is induced, up to isomorphism, by a composite $\X$-adjunction
$$
\xymatrix {
\uX \ar@/_0.5pc/[rr]_F^(0.4){}^(0.6){}^{\top} & & \sL \ar@/_0.5pc/[ll]_G \ar@/_0.5pc/[rr]_{(-)^*}^(0.4){}^(0.6){}^{\top} & & {\sL^\op\;,} \ar@/_0.5pc/[ll]_{(-)^*}
}
$$
in which $(-)^*$ is the $\X$-functor sending a linear space $E$ to its dual $E^*$, so that the $\X$-monad $\HH$ induced by the rightmost $\X$-adjunction may be called the \textit{double-dualization monad}.  Hence $\DD$ may be obtained from $\HH$ via the leftmost $\X$-adjunction; in symbols, $\DD \cong [F,G](\HH)$ \pbref{thm:mon_func_detd_by_adj}.

Under suitable hypotheses, we next form the \textit{idempotent} $\sL$-enriched monad $\tHH$ induced by $\HH$ \pbref{sec:assoc_idm_mnd}, whose associated reflective subcategory $\tL \hookrightarrow \sL$ consists of the \textit{$\HH$-complete} objects \pbref{def:tcomplete}, which we call \textit{functionally complete} linear spaces \pbref{def:func_compl}.  Whereas the space of natural distributions $DX \cong (FX)^{**}$ is the double-dual of the free span $FX$ of $X$, we define the space of \textit{accessible distributions} $\tD X$ as the \textit{functional completion} $\tH FX$ of the free span.  Hence $\tD X$ embeds as a \textit{subspace} of $DX$, and so in our examples above consists again of certain functionals $\mu:[X,R] \rightarrow R$.  We define the \textit{accessible distribution monad} as $\tDD = [F,G](\tHH)$.

We prove that if the cotensors $[X,R]$ in $\sL$ are \textit{reflexive}, then $\tDD \cong \DD$, so that accessible distributions are the same as natural distributions.  Since Butzmann proved that the convergence vector spaces $[X,R]$ are indeed reflexive \cite{Bu}, this result is applicable in the example of $\X = \Conv$.

Since $\tDD$ is induced by the composite $\X$-adjunction
\begin{equation}\label{eq:fg_compl_composite_adj}
\xymatrix {
\X \ar@/_0.5pc/[rr]_F^(0.4){}^(0.6){}^{\top} & & \sL \ar@/_0.5pc/[ll]_G \ar@/_0.5pc/[rr]^(0.4){}^(0.6){}^{\top} & & {\tL\;,} \ar@{_{(}->}@/_0.5pc/[ll]
}
\end{equation}
the linear space $\tD X = \tH FX$ of accessible distributions on $X$ is equally the \textit{free functionally complete linear space} on $X$.  Hence, in the case that the cotensors $[X,R]$ are reflexive, and for $\X = \Conv$ in particular, the linear space $DX$ of natural distributions is the free functionally complete linear space on $X$.  Explicitly, given any continuous linear map $f:X \rightarrow E$ into a functionally complete convergence vector space $E$, there is a unique continuous linear map $f^\sharp:DX \rightarrow E$ such that the diagram
$$
\xymatrix{
X \ar[r]^{\delta_X} \ar[dr]_f & DX \ar@{-->}[d]^{f^\sharp} \\
                              & E
}
$$
commutes, where the map $\delta_X$ sends each point of $X$ to its associated \textit{Dirac functional}.  In particular, this characterizes the space of $R$-valued compactly-supported Radon measures on a locally compact space $X$ via a universal property.  These results shed light on a long-standing question of Lawvere as to what universal property might be possessed by \hbox{$\delta_X:X \rightarrow DX$} in the context of a ringed topos $\X$ \cite{Kock:ProResSynthFuncAn}.  Lawvere's question was prompted by a theorem of Waelbroeck \cite{Wae} to the effect that the linear space of Schwartz distributions of compact support has a universal property with respect to \textit{complete bornological vector spaces} \cite{HN}.

We show that the $\X$-monad $\tDD$ is always \textit{commutative}, and in view of a formal connection noted by Kock \cite{Kock:ProResSynthFuncAn,Kock:Dist} between the notion of commutative monad and the Fubini equation, we obtain as a corollary a Fubini theorem for accessible distributions \pbref{thm:fub_for_acc_distns}.  In the example of convergence spaces, where $\DD \cong \tDD$, we deduce that $\DD$ is commutative; the resulting Fubini theorem for natural distributions, which appears also in the author's recent paper \cite{Lu:Fu}, generalizes to arbitrary convergence spaces the classical Fubini theorem for compactly supported Radon measures on locally compact spaces \cite{Bou}, in particular yielding an entirely new proof of the latter.  These results again shed light on a further long-standing problem of Lawvere and Kock \cite{Kock:ProResSynthFuncAn,Kock:Dist} regarding the apparent non-commutativity of the Schwartz double-dualization monad.

In order to prove that $\tDD$ is commutative, we note that since $\tL$ is an $\sL$-enriched reflective subcategory of $\uL$, it follows by \bref{thm:enr_refl_smadj} that $\tL$ is symmetric monoidal closed and that the rightmost adjunction in \eqref{eq:fg_compl_composite_adj} is symmetric monoidal.  Hence, reasoning informally, the underlying ordinary adjunction of the composite \eqref{eq:fg_compl_composite_adj} is symmetric monoidal, so its induced monad $\tDD$ is symmetric monoidal and hence commutative \pbref{thm:comm_sm_mnd}.  But in the interest of rigour, we must also (in effect) verify that the resulting commutative $\X$-enriched \textit{structure} on (the ordinary monad underlying) $\tDD$ coincides with the given $\X$-enriched structure carried by $\tDD$, and this follows from theory developed in Chapter \bref{ch:smc_adj_comm_mnd}.

\section{The natural distribution monad}\label{sec:nat_distn_mnd}

\begin{ParSub} \label{par:smc_adj}
Throughout this chapter, we will consider a given symmetric monoidal adjunction $F \nsststile{\varepsilon}{\eta} G:\sL \rightarrow \X$, as in \eqref{eqn:smcadj}, where $\X = (\X,\btimes,I)$ and $\sL = (\sL,\otimes,R)$ are \textit{closed} symmetric monoidal categories.
By \bref{thm:assoc_enr_adj}, there is an associated $\X$-enriched adjunction $\acute{F} \nsststile{\varepsilon}{\eta} \grave{G}:G_*\aeL \rightarrow \aeX$ \eqref{eqn:assoc_enr_adj} whose underlying ordinary adjunction coincides with that of $F \nsststile{\varepsilon}{\eta} G$.
\end{ParSub}

\begin{ExaSub}\label{exa:linear_sp}
Given a commutative ring object $R$ in a countably-complete and~\hbox{-cocomplete} cartesian closed category $\X$, we obtain an example of an adjunction as in \bref{par:smc_adj} by taking $\sL := \RMod(\X)$ to be the category of $R$-modules in $\X$; see \bref{exa:rmod_smclosed_monadic}.  For example, if we take $\X$ to be the category $\Conv$ of \textit{convergence spaces} \pbref{exa:conv_sm_sp} and let $R := \RR$ or $\CC$, then $\RMod(\X)$ is the category of \textit{convergence vector spaces} \pbref{exa:rmodx_smc_and_xenr}.  With $\X = \Smooth$ the category of Fr\"olicher's \textit{smooth spaces} \pbref{exa:conv_sm_sp} and $R := \RR$, we obtain as $\RMod(\X)$ the category of \textit{smooth vector spaces} \pbref{exa:rmodx_smc_and_xenr}.
\end{ExaSub}

With examples of the sort given in \bref{exa:linear_sp} in mind, we shall at times employ the following synthetic terminology in order to render the exposition more intuitive.

\begin{DefSub}\label{def:first_synth_terminology}\emptybox
\begin{enumerate}
\item We call the objects of $\X$ \textit{spaces} and the morphisms of $\X$ \textit{maps}.
\item We call the objects of $\sL$ \textit{linear spaces} and the morphisms of $\sL$ \textit{linear maps}.
\item Given a linear space $E$, we call $GE$ \textit{the underlying space} of $E$, and given a linear map $h$, we call $Gh$ \textit{the underlying map} of $h$.
\item Given a space $X$, we call $FX$ the \textit{free linear space} on $X$, or the \textit{free span} of $X$.
\end{enumerate}
\end{DefSub}

\begin{DefSub} \label{def:nat_dist}
Given data as in \bref{par:smc_adj}, the $\X$-category $\eL := G_*\aeL$ is cotensored (by \bref{thm:assoc_enr_adj}), so we have in particular an $\X$-adjunction
\begin{equation}\label{eq:hom_cotensor}
\xymatrix {
\aeX \ar@/_0.5pc/[rr]_{[-,R]}^(0.4){}^(0.6){}^{\top} & & {\eL^\op\;.} \ar@/_0.5pc/[ll]_{\eL(-,R)}
}
\end{equation}
We call the $\X$-monad $\DD = (\bD,\delta^\DD,\kappa^\DD)$ on $\aeX$ induced by this $\X$-adjunction the \textit{natural distribution monad}.  Hence 
$$\bD X = \eL([X,R],R)\;\;\;\;\text{$\X$-naturally in $X \in \uX$.}$$
Employing the synthetic terminology of \bref{def:first_synth_terminology}, $DX = G\uL([X,R],R)$ is therefore the underlying space of the linear space $\uL([X,R],R)$, which we call \textit{the linear space of natural distributions} on $X$.
\end{DefSub}

\begin{ExaSub}\label{exa:nat_dist_conv}
In our example where $\sL$ is the category of convergence vector spaces \pbref{exa:linear_sp}, we know by \bref{exa:std_cotensors_in_rmod} that one can form cotensors $[X,R]$ in $\eL$ ($X \in \X$) by equipping the space $\aeX(X,R)$ of continuous $R$-valued functions with the pointwise vector space structure.  The associated cotensor unit morphism $\delta^\DD_X:X \rightarrow \eL([X,R],R)$ sends each $x \in X$ to the \textit{Dirac functional}, given by $f \mapsto f(x)$.  For each convergence space $X \in \X$, $DX = \eL([X,R],R)$ is the convergence vector space of continuous linear functionals $[X,R] \rightarrow R$.  In the case that $X$ is a locally compact Hausdorff topological space, it is well-known that the convergence structure on $G[X,R] = \aeX(X,R)$ (namely, \textit{continuous convergence}) coincides with \textit{uniform convergence on compact sets}; see, e.g., \cite{McCNta} III.1.  Hence the elements of $DX$ may be identified with the $R$-valued Radon measures of compact support on $X$.
\end{ExaSub}

\begin{ExaSub}\label{exa:nat_dist_smooth}
In our example where $\sL$ is the category of smooth vector spaces \pbref{exa:linear_sp}, we can form cotensors $[X,E]$ in $\sL$ by equipping the space $\uX(X,GE)$ of smooth $E$-valued functions on a smooth space $X$ with the pointwise vector space structure.  For each smooth space $X \in \X$, $DX = \sL([X,\RR],\RR)$ is the space of smooth linear functionals $[X,\RR] \rightarrow \RR$.  When $X$ is a separable paracompact smooth manifold (considered as a smooth space) the elements of $DX$ are the Schwartz distributions of compact support on $X$; indeed, in view of \cite{Fro:CccAnSmthMaps} 6, this is established in \cite{FroKr} 5.1.
\end{ExaSub}

\begin{ExaSub}\label{exa:monadic_exa_integ_notn}
Generalizing \bref{exa:nat_dist_conv} and \bref{exa:nat_dist_smooth}, we have a large class of examples in which $F \dashv G$ is the Eilenberg-Moore adjunction of a suitable symmetric monoidal monad $\TT$ \pbref{exa:em_cat_as_exa_smcadj}.  In this general setting, we can again form each cotensor $[X,R]$ in $\sL$ in such a way that its underlying space is $\uX(X,GR)$ \pbref{rem:standard_cotensors_in_talg}.  For the sake of illustration, we shall sometimes assume that $\X$ is a locally small \textit{well-pointed} cartesian closed category \pbref{par:well_pointed_ccc} and work with elements of the underlying sets of the objects of $\X$.  In this case, we employ the following integral notation
\begin{equation}\label{eq:scalar_integ_nat_distn}\int_x f(x)\;d\mu := \int f\;d\mu := \mu(f)\end{equation}
for each map $f:X \rightarrow R$ and each element $\mu:[X,R] \rightarrow R$ of $DX$.  In the example of a locally compact Hausdorff topological space $X$, considered in the setting of convergence spaces \pbref{exa:nat_dist_conv}, this accords with the identification of each Radon measure of compact support with its associated integration functional $\mu \in DX$.
\end{ExaSub}

\section{Embeddings of linear spaces}

We work with given data as in \bref{par:smc_adj}.

\begin{PropSub}\label{thm:epi_lepi_xepi}
For a morphism $e:E_1 \rightarrow E_2$ in $\sL$, the following are equivalent
\begin{enumerate}
\item $f$ is epi (resp. mono) in $\sL$.
\item $f$ is $\X$-epi (resp. $\X$-mono) in $\sL = G_*\uL$.
\item $f$ is $\sL$-epi (resp. $\sL$-mono) in $\uL$.
\end{enumerate}
\end{PropSub}
\begin{proof}
We will prove the equivalence of the given statements regarding epis; the proof regarding monos is analogous.  By \bref{thm:lim_mono_epi_in_base_are_enr}, 1 is equivalent to 3.  Also, we known that 2 implies 1.  Further, 3 implies 2, since if $\uL(f,E):\uL(E_2,E) \rightarrow \uL(E_1,E)$ is mono in $\sL$ (where $E \in \sL$) then $G\uL(f,E) = (G_*\uL)(f,E)$ is mono in $\X$ since $G$ is a right adjoint.
\end{proof}

\begin{PropSub}\label{thm:strmono_lstrongmono_xstrmono}
For a morphism $m:M \rightarrow E$ in $\sL$, the following are equivalent:
\begin{enumerate}
\item $m$ is a strong mono in $\sL$;
\item $m$ is an $\X$-strong-mono in $\sL = G_*\uL$;
\item $m$ is an $\sL$-strong-mono in $\uL$.
\end{enumerate}
\end{PropSub}
\begin{proof}
The equivalence of 1 and 3 is given by \bref{thm:ord_str_mono_in_base_is_enriched}.  Next, 2 implies 1, since if $m$ is an $\X$-strong-mono, then since $\X$-enriched orthogonality implies ordinary, we deduce by \bref{thm:epi_lepi_xepi} that $m$ is a strong mono.  To see that 3 implies 2, suppose 3 holds.  Since $G$ is right-adjoint and hence preserves pullbacks, $\sL$-orthogonality of morphisms in $\uL$ implies $\X$-orthogonality in $G_*\uL$, so it follows by \bref{thm:epi_lepi_xepi} that $m$ is an $\X$-strong mono.
\end{proof}

\begin{DefSub}
As per \bref{def:enr_str_mon}, we call any morphism satisfying the equivalent conditions in \bref{thm:strmono_lstrongmono_xstrmono} an \textit{embedding} in $\sL$.
\end{DefSub}

\section{Natural distributions via double-dualization}\label{sec:nat_distn_dbl_dualn}

Again we work with data given as in \bref{par:smc_adj}.

\begin{DefSub} \label{def:dualization}\emptybox
\begin{enumerate}
\item For each object $E$ of $\sL$, we call the object $E^* := \uL(E,R)$ of $\sL$ the \textit{dual} of $E$ and $E^{**} = (E^*)^*$ the \textit{double-dual} of $E$.
\item We call the $\sL$-adjunction 
$$
\xymatrix {
\aeL \ar@/_0.5pc/[rr]_{\aeL(-,R)}^(0.4){}^(0.6){}^{\top} & & {\aeL^\op} \ar@/_0.5pc/[ll]_{\aeL(-,R)}
}
$$
the \textit{dualization $\sL$-adjunction} and $(-)^* := \aeL(-,R)$ \textit{the dualization $\sL$-functor}.
\item The $\sL$-monad $\HH = (\bH ,\fd,\fk)$ on $\aeL$ induced by the dualization $\sL$-adjunction is called the \textit{double-dualization $\sL$-monad}.
\item An object $E \in \sL$ is said to be \textit{reflexive} if the morphism $\fd_E:E \rightarrow HE = E^{**}$ is an isomorphism in $\sL$.
\item An object $E \in \sL$ is said to be \textit{(functionally) separated} if $\fd_E$ is an embedding.
\end{enumerate}
\end{DefSub}

\begin{ExaSub}\label{exa:unit_reflexive}
The unit object $R$ of $\sL$ is always reflexive.  Indeed, we have a canonical isomorphism $q:R \xrightarrow{\sim} R^*$, and the diagram
$$
\xymatrix{
R \ar[rr]^{\fd_R} \ar[dr]_q & & R^{**} \ar[dl]^{q^*} \\
& R^*
}
$$
commutes, as the exponential transpose of each composite is the canonical morphism \hbox{$R \otimes R \rightarrow R$}.  The inverse of the isomorphism $\fd_R$ is the composite
$$R^{**} \xrightarrow{\sim} R \otimes R^{**} \xrightarrow{q \otimes 1} R^* \otimes R^{**} \xrightarrow{\Ev} R\;.$$
Indeed, this composite is a retraction of the isomorphism $\fd_R$ since the diagram
$$
\xymatrix{
R \ar[d]_{\fd_R} \ar[r]^\sim & R \otimes R \ar[d]|{1 \otimes \fd_R} \ar[r]_{q \otimes 1} & R^* \otimes R \ar[d]|{1 \otimes \fd_R} \ar[r]_\Ev & R \\
R^{**} \ar[r]_\sim & R \otimes R^{**} \ar[r]_{q \otimes 1} & R^* \otimes R^{**} \ar[ur]_\Ev & 
}
$$
commutes and the composite of the top row is the identity.  Since $q:R \rightarrow \uL(R,R)$ furnishes the identity morphism on $R$ in the $\sL$-category $\uL$, the inverse $\fd_R^{-1}:R^{**} \rightarrow R$ may be interpreted as `the morphism given by evaluation at the identity morphism on $R$'.  Indeed, for the class of examples \bref{exa:monadic_exa_integ_notn}, with $\X$ a locally small well-pointed cartesian closed category, $\fd_R^{-1}$ is thus characterized.
\end{ExaSub}

\begin{ExaSub} \label{exa:refl_conv_vect}
When $\sL$ is the category of convergence vector spaces \pbref{exa:linear_sp}, the following results of Butzmann \cite{Bu} show that reflexive objects are abundant:
\begin{enumerate}
\item For a Hausdorff locally convex topological vector space $E$ (considered as an object of $\sL$), the double-dual $E^{**}$ (taken in $\sL$) is the Cauchy-completion of $E$ (with $\fd_E:E \rightarrow E^{**}$ as the associated embedding) and $E$ is reflexive in $\sL$ if and only if $E$ is Cauchy-complete.
\item Every space $\aeX(X,R)$ of continuous $R$-valued functions on a convergence space $X \in \X$ is reflexive in $\sL$ when endowed with the pointwise $R$-vector-space structure.  As noted in \bref{exa:nat_dist_conv}, each such convergence vector space $\aeX(X,R)$ is a cotensor $[X,R]$ in the $\X$-category $\eL = G_*\aeL$.
\end{enumerate}
\end{ExaSub}

\begin{PropSub} \label{prop:dist_via_dd}
Choosing the cotensors in $G_*\aeL$ as in \bref{thm:assoc_enr_adj}, the following hold:
\begin{enumerate}
\item $\DD$ is the $\X$-monad on $\aeX$ induced by the composite $\X$-adjunction
\begin{equation}\label{eq:composite_adj_inducing_nat_dist_mnd}
\xymatrix {
\aeX \ar@/_0.5pc/[rr]_{\acute{F}}^(0.4){\eta}^(0.6){\varepsilon}^{\top} & & G_*\aeL \ar@/_0.5pc/[ll]_{\grave{G}} \ar@/_0.5pc/[rr]_{G_*(\aeL(-,R))}^(0.4){}^(0.6){}^{\top} & & {G_*(\aeL^\op)\;,} \ar@/_0.5pc/[ll]_{G_*(\aeL(-,R))}
}
\end{equation}
in which the rightmost $\X$-adjunction is obtained by applying $G_*:\LCAT \rightarrow \XCAT$ to the dualization $\sL$-adjunction $\aeL(-,R) \dashv \aeL(-,R)$ \pbref{def:dualization}.
\item $\DD$ may be obtained from $\HH$ by first applying the 2-functor $G_*:\LCAT \rightarrow \XCAT$ and then applying the monoidal functor $[\acute{F},\grave{G}]:\XCAT(G_*\aeL,G_*\aeL) \rightarrow \XCAT(\aeX,\aeX)$ \pbref{thm:mon_func_detd_by_adj}, so that
\begin{equation}\label{eq:dist_via_dd}\DD = [\acute{F},\grave{G}](G_*(\HH))\;.\end{equation}
\end{enumerate}
\end{PropSub}
\begin{proof}
By definition, $\DD$ is induced by the `hom-cotensor' $\X$-adjunction \eqref{eq:hom_cotensor}, which by \bref{thm:assoc_enr_adj} is equal to the composite \eqref{eq:composite_adj_inducing_nat_dist_mnd}.  The $\sL$-monad on $\aeL$ induced by dualization $\sL$-adjunction is $\HH$, so the $\X$-monad on $G_*\aeL$ induced by the rightmost $\X$-adjunction in \eqref{eq:composite_adj_inducing_nat_dist_mnd} is $G_*(\HH)$.  Hence $\DD = [\acute{F},\grave{G}](G_*(\HH))$.
\end{proof}

\begin{RemSub}
Employing the synthetic terminology of \bref{def:first_synth_terminology} and choosing cotensors as in \bref{prop:dist_via_dd}, the linear space of natural distributions $\uL([X,R],R)$ \pbref{def:nat_dist} on a space $X$ is equally the double-dual $HFX = (FX)^{**}$ of the free span $FX$ of $X$.
\end{RemSub}

\begin{PropSub}\label{thm:transp_of_delta}
For each object $X$ of $\X$, the transpose of $\delta^\DD_X:X \rightarrow DX = GHFX$ under $F \nsststile{\varepsilon}{\eta} G$ is $\fd_{FX}:FX \rightarrow HFX$.
\end{PropSub}
\begin{proof}
By \eqref{eq:dist_via_dd}, $\delta^\DD_X$ is the composite $X \xrightarrow{\eta_X} GFX \xrightarrow{G\fd FX} GHFX$.
\end{proof}

\section{Functional completeness, completion, closure, and density} \label{sec:compl}

Working with data as given in \bref{par:smc_adj}, we now suppose that $\sL$ is finitely well-complete \pbref{def:enr_fwc}.

\begin{RemSub} \label{rem:base_fwc}
In order for $\sL$ to be finitely well-complete, it suffices to assume that $\X$ is finitely well-complete and that the right adjoint $G:\sL \rightarrow \X$ \textit{detects limits} \cite{AHS} 13.22.  Indeed, under these alternate assumptions, $\sL$ is clearly finitely complete.  Further, it is known (e.g. \cite{Wy:QuTo}, 10.5) that right adjoints (such as $G$) preserve strong monomorphisms.  It follows that since $\X$ has arbitrary fiber products of strong monomorphisms and $G$ detects fiber products, $\sL$ has arbitrary fiber products of strong monomorphisms.
\end{RemSub}

\begin{ExaSub}\label{exa:smc_adj_with_l_fwc}\emptybox
\begin{enumerate}
\item Let $\X$ be a finitely well-complete symmetric monoidal closed category, and let $\TT$ be a symmetric monoidal monad on $\X$ with tensor products of algebras \pbref{def:tens_prod_algs}.  Then the category $\sL := \X^\TT$ of $\TT$-algebras is symmetric monoidal closed, and the Eilenberg-Moore adjunction is symmetric monoidal \pbref{thm:smclosed_em_adj}; further, since the forgetful functor detects limits, $\sL$ is finitely well-complete by \bref{rem:base_fwc}.
\item Specializing 1, in our example \bref{exa:linear_sp}, the category $\sL = \RMod(\X)$ of $R$-module objects in a countably-complete and~\hbox{-cocomplete} cartesian closed category is finitely well-complete as soon as $\X$ is so.
\item Specializing 2, the categories of convergence vector spaces and smooth vector spaces \pbref{exa:linear_sp} are finitely well-complete, since both $\Conv$ and $\Smooth$ are topological over $\Set$ and hence finitely well-complete by \bref{exa:fwc}.
\end{enumerate}
\end{ExaSub}

\begin{ParSub}
By \bref{rem:ind_idempt_mnd_on_v}, we may form the idempotent $\sL$-monad $\tHH$ induced by the double-dualization $\sL$-monad $\HH$ \pbref{thm:ind_idmmnd_and_refl}.
\end{ParSub}

\begin{DefSub}\label{def:func_compl}\emptybox
\begin{enumerate}
\item We call the idempotent $\sL$-monad $\tHH = (\tH,\tpartial,\tgamma)$ induced by the double-dualization $\sL$-monad $\HH$ \textit{the functional completion $\sL$-monad}.
\item We let $\tL$ denote the $\sL$-reflective-subcategory of $\uL$ determined by $\tHH$.  We denote the induced $\sL$-reflection by $K \nsststile{}{\tpartial} J:\tL \hookrightarrow \uL$.
\item We say that an object $E$ of $\sL$ is \textit{functionally complete} if it lies in $\tL$.
\item For each object $E$ of $\sL$, we call $\tH E = KE$ the \textit{functional completion of $E$}.
\item We denote the unique morphism of $\sL$-monads $\tHH \rightarrow \HH$ \pbref{thm:addnl_facts_on_ind_idmmnd_refl} by $i$.
\end{enumerate}
\end{DefSub}

\begin{RemSub}\label{rem:fcompl_subsp_ddual}
For each linear space $E \in \sL$, the associated component of the unit $\tfd$ of $\tHH$ is a linear map $\tfd_E:E \rightarrow \tH E$ from $E$ into its functional completion.  The associated component of the morphism $i:\tHH \rightarrow \HH$ is an embedding $i_E:\tH E \rightarrowtail HE = E^{**}$ of the functional completion of $E$ into the double-dual of $E$ \pbref{thm:addnl_facts_on_ind_idmmnd_refl}.  

Consider the class of examples \bref{exa:smc_adj_with_l_fwc} 1, with $\X$ a locally small well-pointed cartesian closed category.  For each $E \in \sL$, the associated embedding $i_E$ is obtained (via \bref{def:completion} and in turn \bref{prop:enr_factn_sys_det_lowercls_monos}) as a certain intersection of embeddings, and by choosing this intersection appropriately we may assume that the embedding $i_E$ is an inclusion $\tH E \hookrightarrow E^{**}$.
\end{RemSub}

\begin{ParSub}\label{par:basic_charns_func_completeness}
Per \bref{def:induced_idm_mnd_and_refl}, $\tHH$ is also called the \textit{$\HH$-completion} monad, and in the terminology of \bref{def:tcomplete}, its associated $\sL$-reflective subcategory $\tL = \uL_{(\HH)}$ consists of the \textit{$\HH$-complete} objects of $\sL$.  Hence the following conditions on an object $E$ of $\sL$ are equivalent:
\begin{enumerate}
\item $E$ is functionally complete.
\item $E$ is $\HH$-complete.
\item Given any morphism $h:E_1 \rightarrow E_2$ for which $h^{**}:E_1^{**} \rightarrow E_2^{**}$ is iso, $h$ is $\sL$-orthogonal to $E$ --- i.e. the morphism $\uL(h,E):\uL(E_2,E) \rightarrow \uL(E_1,E)$ is an iso in $\sL$.
\item Given any morphism $h:E_1 \rightarrow E_2$ for which $\uL(h,R):\uL(E_2,R) \rightarrow \uL(E_1,R)$ is iso in $\sL$, the morphism $\uL(h,E):\uL(E_2,E) \rightarrow \uL(E_1,E)$ is iso in $\sL$.
\end{enumerate}
\end{ParSub}
\begin{proof}
The equivalence of 1 and 2 holds by the definition of $\tHH$ \pbref{thm:ind_idmmnd_and_refl}, and the equivalence of 2 and 3 holds by the definition of $\HH$-completeness \pbref{def:tcomplete}.  3 and 4 are equivalent, since the dualization $\sL$-adjunction induces $\HH$ and hence, by \bref{thm:all_ladjs_ind_given_monad_inv_same_morphs}, its left adjoint $(-)^{*}$ inverts exactly the same morphisms as $H = (-)^{**}$.
\end{proof}

\begin{ExaSub}\label{exa:refl_impl_func_compl}
Every reflexive linear space is functionally complete, since if $E \in \sL$ is reflexive then $E \cong E^{**} = HE \in \uL_{(\HH)} = \tL$ \pbref{thm:ind_idmmnd_and_refl} and $\tL$ is a replete subcategory of $\sL$.  In the example of convergence vector spaces \pbref{exa:smc_adj_with_l_fwc}, every Cauchy-complete locally convex Hausdorff topological vector space is a reflexive object of $\sL$ \pbref{exa:refl_conv_vect} and hence is functionally complete.
\end{ExaSub}

\begin{ExaSub}
By \bref{exa:unit_reflexive} and \bref{exa:refl_impl_func_compl}, the unit object $R$ of $\sL$ is reflexive and hence functionally complete.  Note also that the functional completeness of $R$ is immediate from characterization 4 in \bref{par:basic_charns_func_completeness}.
\end{ExaSub}

\begin{RemSub}\label{rem:func_compl_obs_sep}
By \bref{def:tcomplete}, every functionally complete object of $\sL$ is $\HH$-separated, i.e. functionally separated \pbref{def:dualization}.
\end{RemSub}

\begin{PropSub}\label{thm:homs_func_complete}\emptybox
\begin{enumerate}
\item For objects $E', E \in \sL$, if $E$ is functionally complete, then the internal hom $\uL(E',E)$ is functionally complete.
\item Given objects $X \in \X$, $E \in \sL$, if $E$ is functionally complete then the cotensor $[X,E]$ in $\sL$ is functionally complete.
\end{enumerate}
\end{PropSub}
\begin{proof}
1.  Since $\tL$ is an $\sL$-reflective-subcategory of $\uL$, $\tL$ is closed under $\sL$-enriched weighted limits in $\uL$, so in particular, $\tL$ is closed under cotensors in $\uL$.  But $\uL(E',E)$ is a cotensor of $E$ in $\uL$, so since $E \in \tL$, this cotensor lies in $\tL$.  2.  By \bref{thm:assoc_enr_adj}, $\uL(FX,E)$ serves as a cotensor $[X,E]$ in $\sL = G_*\uL$, so 1 applies.
\end{proof}

\begin{DefSub}\label{def:func_cl_dens}\emptybox
\begin{enumerate}
\item We say that an embedding in $\uL$ is \textit{functionally closed} if it is $\Sigma_H$-closed \pbref{def:compl_sep_closed_dense}.
\item Given an embedding $m$ in $\uL$, we call the $\Sigma_H$-closure $\overline{m}$ of $m$ \pbref{def:sigma_closure} \textit{the functional closure} of $m$.
\item We say that a morphism in $\uL$ is \textit{functionally dense} if it is $\Sigma_H$-dense \pbref{def:compl_sep_closed_dense}.
\end{enumerate}
\end{DefSub}

\begin{PropSub} \label{thm:mnd_morph_i}
The classes of functionally dense morphisms and functionally closed embeddings constitute an $\sL$-factorization-system $(\Dense{\Sigma_H},\ClEmb{\Sigma_H})$ on $\uL$.  For each $E \in \uL$ we have a commutative diagram
$$
\xymatrix@R=0.5ex {
E \ar[dr]_{\tpartial_E} \ar[rr]^{\fd_E} &                        & {HE = E^{**}}  \\
                                             & \tH E \ar@{ >->}[ur]_{i_E} &    
}
$$
in which the unit component $\tpartial_E$ is functionally dense and the component $i_E$ of $i:\tHH \rightarrow \HH$ is a functionally closed embedding.
\end{PropSub}
\begin{proof}
This follows from \bref{thm:dense_closed_factn_sys}, \bref{def:completion}, and \bref{thm:addnl_facts_on_ind_idmmnd_refl}.
\end{proof}

\begin{RemSub}\label{rem:cmpl_as_closure}\emptybox
\begin{enumerate}
\item By \bref{thm:str_img_factns_for_fwc_base}, $\uL$ has $\sL$-strong image factorizations \pbref{def:image}, so we can take the \textit{image} of each morphism of $\uL$.
\item By \bref{thm:closure_of_image}, the embedding \hbox{$i_E:\tH E \rightarrowtail E^{**}$} presents the functional completion $\tH E$ of $E$ as the functional closure of the image of $\fd_E:E \rightarrow E^{**}$.
\end{enumerate}
\end{RemSub}

\begin{PropSub} \label{thm:compl_sm_adj}
The ordinary adjunction underlying $K \nsststile{}{\tfd} J : \tL \hookrightarrow \uL$ carries the structure of a symmetric monoidal adjunction, with $\tL$ a closed symmetric monoidal category.  The associated $\sL$-adjunction \hbox{$\acute{K} \nsststile{}{\tfd} \grave{J} : J_*\left(\aetL\right) \rightarrow \aeL$} \pbref{thm:assoc_enr_adj} is simply the same $\sL$-adjunction $K \nsststile{}{\tfd} J$.
\end{PropSub}
\begin{proof}
This follows from \bref{thm:enr_refl_smadj}.
\end{proof}

\section{Accessible distributions} \label{sec:acc_distn}

We continue to work with data as given in \bref{par:smc_adj}, and as in \bref{sec:compl} we again suppose that $\sL$ is finitely well-complete.

\begin{ParSub} \label{par:mnd_morph_tdd_to_dd}
Applying the 2-functor $G_*:\LCAT \rightarrow \XCAT$ to the monad morphism $i:\tHH \rightarrow \HH$ \pbref{thm:mnd_morph_i}, we obtain a monad morphism $G_*(i):G_*(\tHH) \rightarrow G_*(\HH)$ in $\XCAT$.  Next, applying the monoidal functor $[\acute{F},\grave{G}]:\XCAT(G_*\aeL,G_*\aeL) \rightarrow \XCAT(\aeX,\aeX)$ \pbref{thm:mon_func_detd_by_adj}, we obtain a morphism of $\X$-monads
\begin{equation}\label{eq:morph_td_to_d}[\acute{F},\grave{G}](G_*(\tHH)) \xrightarrow{[\acute{F},\grave{G}](i) = GiF} [\acute{F},\grave{G}](G_*(\HH)),\end{equation}
but by \eqref{eq:dist_via_dd}, $[\acute{F},\grave{G}](G_*(\HH)) = \DD$ is the natural distribution monad.  This leads to the following definition.
\end{ParSub}

\begin{DefSub}\label{def:acc_distns}\emptybox
\begin{enumerate}
\item We call the $\X$-monad $\tDD := [\acute{F},\grave{G}](G_*(\tHH))$ \eqref{eq:morph_td_to_d} on $\uX$ \textit{the accessible distribution monad}.
\item We denote the morphism of $\X$-monads $[\acute{F},\grave{G}](i)$ \eqref{eq:morph_td_to_d} by $j^{\tDD}:\tDD \rightarrow \DD$.
\item For each object $X \in \X$, we call $\tH FX$ \textit{the linear space of accessible distributions} on $X$.
\end{enumerate}
\end{DefSub}

\begin{RemSub}\label{rem:acc_distns}Employing our lexicon of synthetic terminology \pbref{def:first_synth_terminology}, we note the following.
\begin{enumerate}
\item For a space $X \in \X$, $\tD X$ is the underlying space of the linear space of accessible distributions on $X$.
\item The linear space of accessible distributions $\tH FX$ on $X$ is equally the functional completion of the free span $FX$ of $X$.
\item Each component $j^{\tDD}_X = GiFX$ of $j^{\tDD}:\tDD \rightarrow \DD$ is the underlying map of the functionally closed embedding $i_{FX}:\tH FX \rightarrowtail HFX = (FX)^{**}$ into the linear space of natural distributions $(FX)^{**}$.
\item By \bref{rem:cmpl_as_closure}, the embedding $i_{FX}$ presents $\tH FX$ as the functional closure of the image of $\fd_{FX}:FX \rightarrow (FX)^{**}$.
\item Since $\fd_{FX}$ is the transpose of \hbox{$\delta^\DD_X:X \rightarrow DX = GHFX$} under $F \nsststile{\varepsilon}{\eta} G$ (by \bref{thm:transp_of_delta}), there is thus a sense in which $i_{FX}$ presents $\tH FX$ as the functional closure of the span of the Dirac distributions in the linear space of natural distributions.
\item In contrast with \bref{prop:dist_via_dd} 1, $\tDD$ is induced by the composite $\X$-adjunction
\begin{equation}\label{eq:comp_adj_ind_ddt}
\xymatrix {
\aeX \ar@/_0.5pc/[rr]_{\acute{F}}^(0.4){\eta}^(0.6){\varepsilon}^{\top} & & G_*\aeL \ar@/_0.5pc/[ll]_{\grave{G}} \ar@/_0.5pc/[rr]_{G_*(\bK)}^(0.4){}^(0.6){}^{\top} & & {G_*\teL\;.} \ar@/_0.5pc/[ll]_{G_*(\bJ)}
}
\end{equation}
\end{enumerate}
\end{RemSub}

\begin{RemSub}\label{rem:tdx_free_fc_linear_sp}
By \bref{rem:acc_distns} 6, the linear space $\tH FX$ of accessible distributions on $X$ is the \textit{free} functionally complete linear space on $X$.
\end{RemSub}

\begin{RemSub}\label{rem:acc_dist_are_nat_dist}
Consider the class of examples \bref{exa:smc_adj_with_l_fwc} 1, with $\X$ a locally small well-pointed cartesian closed category.  For any choice of cotensors $[X,R]$ in $\sL$ $(X \in \X)$, we have isomorphisms $HFX \cong \uL([X,R],R)$ $(X \in \X)$.  Recall that $i_{FX}:\tH FX \rightarrowtail HFX = (FX)^{**}$ is a certain intersection of strong monomorphisms, and by choosing this intersection appropriately we may assume that the composite embedding $\tH FX \rightarrowtail HFX \cong \uL([X,R],R)$ induces an inclusion of the underlying sets.  Hence, regardless of the choice of cotensors $[X,R]$ we may assume that each accessible distribution $\mu \in \tD X$ is a natural distribution $\mu:[X,R] \rightarrow R$.
\end{RemSub}

\begin{LemSub} \label{thm:partial_fx_dense}
Let $X \in \X$, and suppose that the cotensor $[X,R]$ in $G_*\uL$ is reflexive (in $\sL$).  Then $\fd_{FX}:FX \rightarrow HFX = (FX)^{**}$ is functionally dense.
\end{LemSub}
\begin{proof}
Since $\Sigma_H \subs \Dense{\Sigma_H}$, it suffices to show that $\fd_{FX}$ is inverted by $H$.

Having chosen the cotensor $[X,R]$ in $G_*\aeL$ as in \bref{thm:assoc_enr_adj} 2, we have that
$$[X,R] = \aeL(FX,R) = (FX)^*\;,$$
so since $[X,R]$ is reflexive, the unit component
$$\fd_{(FX)^*}:(FX)^* \rightarrow H((FX)^*) = (FX)^{***}$$
is an isomorphism.

But $\fd_{(FX)^*}$ has a retraction $\fd_{FX}^*$,
as follows.  Indeed, as
$$\fd_{(FX)^*}:(FX)^* \rightarrow \aeL(\aeL((FX)^*,R),R)$$
is the transpose of 
$$\Ev:(FX)^* \otimes \aeL((FX)^*,R) \rightarrow R\;,$$
we find that the composite
$$(FX)^* \xrightarrow{\fd_{(FX)^*}} \aeL(\aeL((FX)^*,R),R) \xrightarrow{\fd_{FX}^* \;=\; \aeL(\fd_{FX},R)} \aeL(FX,R) = (FX)^*$$
has transpose
$$(FX)^* \otimes FX \xrightarrow{1 \otimes \fd_{FX}} (FX)^* \otimes \aeL((FX)^*,R) \xrightarrow{\Ev} R\;,$$
which is equal to $\Ev:\aeL(FX,R) \otimes FX \rightarrow R$, so that $\fd_{FX}^* \cdot \fd_{(FX)^*} = 1_{(FX)^*}$ as needed.

Hence, since $\fd_{(FX)^*}$ is an isomorphism, its retraction $\fd_{FX}^*:(FX)^{***} \rightarrow (FX)^*$ is an isomorphism.  Applying $(-)^*:\sL^\op \rightarrow \sL$, we obtain an isomorphism
$$H\fd_{FX} = \fd_{FX}^{**} : HFX = (FX)^{**} \rightarrow (FX)^{****} = HHFX\;.$$
\end{proof}

\begin{ThmSub} \label{thm:dist_via_compl}
Suppose that for each $X \in \X$, the cotensor $[X,R]$ in $G_*\uL$ is reflexive (in $\sL$).  
\begin{enumerate}
\item Then the morphism of $\X$-monads $j^{\tDD}:\tDD \rightarrow \DD$ \eqref{eq:morph_td_to_d} is an isomorphism.
\item Moreover, each embedding $i_{FX} : \tH FX \rightarrowtail HFX$ $(X \in \X)$ is an isomorphism in $\sL$.
\end{enumerate}
\end{ThmSub}
\begin{proof}
Since $j^{\tDD} = GiF$, it suffices to prove 2.  Since $i_{FX}$ is a functionally closed embedding, it suffices by \bref{thm:props_sigma_cl_dens} to show that $i_{FX}$ is functionally dense.  We know that $\fd_{FX}$ factors as
$$FX \xrightarrow{\tpartial_{FX}} \tH FX \overset{i_{FX}}{\rightarrowtail} HFX\;,$$
but by \bref{thm:partial_fx_dense}, $\fd_{FX}$ is functionally dense, so by \bref{thm:props_sigma_cl_dens}, $i_{FX}$ is functionally dense as well.
\end{proof}

\begin{CorSub}\label{thm:nat_dist_free_fc_linsp_when_cot_refl}
Letting $X \in \X$, if the cotensor $[X,R]$ in $\sL$ is reflexive, then the linear space $\uL([X,R],R)$ of natural distributions on $X$ is the free functionally complete linear space on $X$.
\end{CorSub}

\begin{ExaSub}\label{exa:acc_nat_dist_coincide_for_conv}
In the example in which $\X = \Conv$ is the category of convergence spaces \pbref{exa:smc_adj_with_l_fwc}, the cotensors $[X,R]$ in $G_*\aeL$ are reflexive in $\sL$ by \bref{exa:refl_conv_vect}, so by \bref{thm:dist_via_compl} and \bref{thm:nat_dist_free_fc_linsp_when_cot_refl}, we obtain the following:
\begin{enumerate}
\item $\tDD \cong \DD$.
\item For each convergence space $X$, the convergence vector space $\uL([X,R],R)$ of natural distributions on $X$ is the free functionally complete linear space on $X$.
\end{enumerate}
Further, by \bref{rem:acc_dist_are_nat_dist} we may assume that each component $j^{\tDD}_X:\tD X \rightarrow DX$ induces an inclusion of the underlying sets, whence since $j^{\tDD}_X$ is an isomorphism, $\tD X = DX$ and $j^{\tDD}_X$ is the identity morphism.  It then follows that the $\X$-monads $\tDD$, $\DD$ are identical and that $j^{\tDD}:\tDD \rightarrow \DD$ is the identity.  Hence, for convergence spaces, accessible distributions are the same as natural distributions. 
\end{ExaSub}

\section{An abstract Fubini theorem}

\begin{PropSub}\label{thm:td_comm}
Given data as in \bref{par:smc_adj} with $\sL$ finitely well-complete \pbref{def:enr_fwc}, the $\X$-monad $\tDD$ is commutative.
\end{PropSub}
\begin{proof}
Recall that the underlying ordinary adjunction of \hbox{$\bK \nsststile{}{\tfd} \bJ : \teL \hookrightarrow \aeL$} carries the structure of a symmetric monoidal adjunction, with $\tL$ closed \pbref{thm:compl_sm_adj}.  Further, this symmetric monoidal adjunction determines an $\sL$-adjunction \hbox{$\acute{K} \nsststile{}{\tfd} \grave{J} : J_*\left(\aetL\right) \rightarrow \aeL$} which is in fact just the same $\sL$-adjunction $\bK \nsststile{}{\tfd} \bJ$ \pbref{thm:compl_sm_adj}. 

Hence, by \bref{rem:acc_distns} 6, $\tDD$ is the $\X$-monad induced by the composite $\X$-adjunction
\begin{equation}
\xymatrix {
\aeX \ar@/_0.5pc/[rr]_{\acute{F}}^(0.4){}^(0.6){}^{\top} & & G_*\aeL \ar@/_0.5pc/[ll]_{\grave{G}} \ar@/_0.5pc/[rr]_{G_*(\acute{K})}^(0.4){}^(0.6){}^{\top} & & {G_*J_*(\aetL)\;,} \ar@/_0.5pc/[ll]_{G_*(\grave{J})}
}
\end{equation}
which by \bref{thm:compn_assoc_enr_adj} is the same as the $\X$-adjunction
\begin{equation}\label{eq:enradj_kf_gj}
\xymatrix {
\aeX \ar@/_0.5pc/[rr]_{\wideacute{KF}}^(0.4){}^(0.6){}^{\top} & & (GJ)_*\aetL \ar@/_0.5pc/[ll]_{\widegrave{GJ}}
}
\end{equation}
associated to the composite symmetric monoidal adjunction
$$
\xymatrix {
\X \ar@/_0.5pc/[rr]_F^(0.4){}^(0.6){}^{\top} & & \sL \ar@/_0.5pc/[ll]_G \ar@/_0.5pc/[rr]_K^(0.4){}^(0.6){}^{\top} & & {\tL\;.} \ar@{_{(}->}@/_0.5pc/[ll]_J
}
$$
Hence by \bref{thm:enradj_assoc_smadj_ind_commmnd}, $\tDD$ is commutative.
\end{proof}

In view of \bref{thm:dist_via_compl}, \bref{thm:td_comm} yields the following.

\begin{ThmSub} \label{thm:general_fubini}
Let $\xymatrix {\X \ar@/_0.5pc/[rr]_F^(0.4){}^(0.6){}^{\top} & & \sL \ar@/_0.5pc/[ll]_G}$ be a symmetric monoidal adjunction, with $\X$ and $\sL$ symmetric monoidal closed and $\sL$ finitely well-complete \pbref{def:enr_fwc}.  Suppose that each cotensor $[X,R]$ in $G_*\uL$ is reflexive, where $X \in \X$ and $R$ is the unit object of $\sL$.  Then the natural distribution monad $\DD$ is commutative.
\end{ThmSub}

\begin{CorSub}\label{thm:nat_dist_on_conv_comm}
The natural distribution monad $\DD$ on the category $\X = \Conv$ of convergence spaces \pbref{exa:nat_dist_conv} is commutative.
\end{CorSub}
\begin{proof}
In the example of convergence spaces \pbref{exa:smc_adj_with_l_fwc}, the cotensors $[X,R]$ in $G_*\aeL$ are reflexive in $\sL$ by \bref{exa:refl_conv_vect}.
\end{proof}

A formal connection between the notion of commutative monad and the Fubini equation was noted by Kock with regard to the Schwartz double-dualization monad $\DD$ on a ringed topos in \cite{Kock:ProResSynthFuncAn}, yet $\DD$ was not proved commutative and Kock later stated that (in a general sense) $\DD$ ``is not commutative'' (\cite{Kock:Dist}, pg. 97).  But we have proved above that $\DD$ is in fact commutative under certain hypotheses \pbref{thm:general_fubini}, and that $\tDD$ is always commutative \pbref{thm:td_comm}, and as corollaries we obtain Fubini-type theorems for accessible and natural distributions, as follows.

\begin{CorSub} \label{thm:fub_for_acc_distns}
Suppose data as in \bref{exa:smc_adj_with_l_fwc} \textnormal{1} are given, and assume that $\X$ is a locally small well-pointed cartesian closed category.
\begin{enumerate}
\item Given accessible distributions $\mu:[X,R] \rightarrow R$ and $\nu:[Y,R] \rightarrow R$, there is a unique accessible distribution $\mu \otimes \nu:[X \times Y, R] \rightarrow R$ such that
\begin{equation}\label{eqn:fubini}\int_y \int_x f(x,y) \;d\mu d\nu = \int f \;d(\mu \otimes \nu) = \int_x \int_y f(x,y) \;d\nu d\mu\end{equation}
for all maps $f:X \times Y \rightarrow R$.
\item Suppose that the cotensors $[X,R]$ in $G_*\uL$ are reflexive $(X \in \X)$.  Given morphisms $\mu:[X,R] \rightarrow R$ and $\nu:[Y,R] \rightarrow R$ in $\sL$, there is a unique morphism $\mu \otimes \nu:[X \times Y, R] \rightarrow R$ in $\sL$ such that \eqref{eqn:fubini} holds for all maps $f:X \times Y \rightarrow R$.
\end{enumerate}
\end{CorSub}
\begin{proof}
Note that the Fubini equation \eqref{eqn:fubini} can be rewritten in the notation of \textit{lambda calculus} as
\begin{equation}\label{eqn:lambda_fubini}\nu(\lambda y.\mu(\lambda x.f(x,y))) = (\mu \otimes \nu)(f) = \mu(\lambda x.\nu(\lambda y.f(x,y)))\;.\end{equation}

The morphisms $t'_{XY}:DX \times Y \rightarrow D (X \times Y)$ and $t''_{XY}:X \times D Y \rightarrow D (X \times Y)$ \pbref{def:tens_str} are given by
$$t'_{XY}(\mu,y) = \lambda f.\mu(\lambda x.f(x,y))\;,\;\;\;\;t''_{XY}(x,\nu) = \lambda f.\nu(\lambda y.f(x,y))\;.$$
Next note that the components of the counit $\sigma$ of the $\X$-adjunction $[-,R] \dashv \eL(-,R)$ inducing $\DD$ \pbref{def:nat_dist} are the morphisms $\sigma_E:E \rightarrow [\eL(E,R),R]$ of $\sL$ given by $\sigma_E(e) = \lambda \phi.\phi(e)$.  Note also that for each $X \in \aeX$, the morphisms $\kappa^\DD_X:DDX \rightarrow DX$ are given as
$$\kappa^\DD_X = \eL(\sigma_{[X,R]},R) : \eL([D X,R],R) \rightarrow \eL([X,R],R)\;.$$
Using these facts, it is straightforward to compute that the maps $\otimes_{XY},\widetilde{\otimes}_{X Y}:D X \times D Y \rightarrow D (X \times Y) = \eL([X \times Y,R],R)$ of \bref{def:comm_mnd} send a pair $(\mu,\nu) \in DX \times DY$ to the functionals $\otimes_{XY}(\mu,\nu),\widetilde{\otimes}_{X Y}(\mu,\nu)$ whose values at each $f \in [X \times Y,R]$ are the left- and right-hand-sides (respectively) of the Fubini equation \eqref{eqn:lambda_fubini} (equivalently, \eqref{eqn:fubini}).  Hence 2 follows from \bref{thm:general_fubini}.

The components of $j := j^{\tDD}:\tD \rightarrow D$ induce inclusions of the underlying sets \pbref{rem:acc_dist_are_nat_dist}, and by \bref{prop:mnd_mor_monoidal_wrt_tensors}, the diagram
$$
\xymatrix {
\tD X \times \tD Y \ar@{^{(}->}[d]_{j_X \times j_Y} \ar[r]^{m^{\tDD}} & \tD (X \times Y) \ar@{^{(}->}[d]^{j_{X \times Y}} \\
DX \times DY \ar[r]^{m^\DD} & D(X \times Y)
}
$$
commutes, both with $m^{\tDD} := \otimes^{\tDD}$, $m^\DD := \otimes^\DD$ and also with $m^{\tDD} := \widetilde{\otimes}^{\tDD}$, $m^\DD := \widetilde{\otimes}^\DD$.  Since $\tDD$ is commutative \pbref{thm:td_comm}, 1 now follows.
\end{proof}

\begin{CorSub}\label{thm:sc_fub_for_conv_sp}
Let $X,Y \in \Conv$ be convergence spaces, and let $\mu:[X,R] \rightarrow R$ and $\nu:[Y,R] \rightarrow R$ be continuous linear functionals, where $R = \RR$ or $\CC$.  Then there is a unique continuous linear functional $\mu \otimes \nu:[X \times Y, R] \rightarrow R$ such that \eqref{eqn:fubini} holds for all continuous maps $f:X \times Y \rightarrow R$.
\end{CorSub}

%% file: dcompl_vint.tex
\chapter{Distributional completeness and vector-valued integration} \label{ch:dcompl_vint}
\setcounter{subsection}{0}

In the present chapter, we develop an abstract theory of vector-valued integration with respect to natural and accessible distributions \pbref{ch:nat_acc_dist}, and so in particular with respect to compactly-supported Radon measures and Schwartz distributions.  Further, we examine the relation between this vector integration and notions of completeness in linear spaces.

In particular, we study in a general setting three distinct approaches to vector integration, and we show that they all coincide in appropriate contexts.  Firstly, we reinterpret Pettis' 1938 notion of vector integral \cite{Pett}, employed also by Bourbaki in a different setting \cite{Bou}, showing that it may described in terms of monad morphisms.  Secondly, we study an approach to vector integration on the basis of a universal property of the space of distributions, inspired by work of Schwartz \cite{Schw} and Waelbroeck \cite{Wae} and formulated for the monad $\DD$ on a ringed topos by Lawvere and Kock \cite{Kock:ProResSynthFuncAn}.  Thirdly, we pursue an approach via monads introduced recently by the author \cite{Lu:AlgThVectInt} and closely related to recent work of Kock \cite{Kock:Dist}, in which integration valued in a space $E$ is seen as a family of \textit{algebraic operations} governed by equations, so that algebras of a suitable \textit{distribution monad} $\MM$ are construed as linear spaces in which the integral may take its value.

In this chapter we work with a given commutative monad $\LL$ on a suitable symmetric monoidal closed category $\X$, and we consider the $\X$-enriched Eilenberg-Moore adjunction \hbox{$F \dashv G:\sL \rightarrow \uX$}.  Having in mind the case of the category $\sL = \RMod(\X)$ of $R$-modules in a suitable cartesian closed category $\X$, we apply the terminology of Chapter \bref{ch:nat_acc_dist}, calling the objects of $\X$ and $\sL$ \textit{spaces} and \textit{linear spaces}, respectively, and the morphisms \textit{maps} and \textit{linear maps}.  In particular, we have again the examples of convergence spaces and smooth spaces noted in \bref{ch:nat_acc_dist}.  We define a notion of \textit{abstract distribution monad} $(\MM,\Delta,\xi)$, which includes as examples the natural and accessible distribution monads $\DD$ and $\tDD$, the former serving as terminal object among abstract distribution monads.

For each linear space $E$, there is a canonical map $ME \rightarrow E^{**}$ from the space of \textit{(abstract) distributions} $ME$ on $E$ to the double-dual $E^{**}$ of $E$, and these maps constitute a morphism of $\X$-monads $\xi:\MM \rightarrow \HH$ into the double-dualization monad $\HH$ on $\sL$.  Given a map $f:X \rightarrow E$ and a distribution $\mu \in MX$, we call the value of the composite map
$$MX \xrightarrow{Mf} ME \xrightarrow{\xi_E} E^{**}$$
at $\mu$ the \textit{Dunford-type integral}, written $\dint{} f\;d\mu$.  For \textit{separated} linear spaces $E$ \pbref{sec:nat_distn_dbl_dualn}, we may ask whether there exists a (necessarily unique) element $e \in E$ sent by the canonical embedding $E \rightarrowtail E^{**}$ to the Dunford integral, in which case we call $e$ the \textit{Pettis-type integral}, written $\pint{} f\;d\mu$.  In the example of convergence spaces, with $\MM = \DD$, these are the Dunford- and Pettis-type integrals employed by Bourbaki \cite{Bou} for Hausdorff locally convex topological vector spaces $E$ and locally compact spaces $X$ \pbref{sec:vint_pett_bou}.  Given a separated linear space $E$, the existence of Pettis-type integrals for \textit{all} incoming maps and natural distributions is tantamount to the existence of a (necessarily unique) map $a$ making the diagram
$$
\xymatrix{
ME \ar[dr]_a \ar[rr]^{\xi_E} &                  & E^{**} \\
                                     & E \ar@{ >->}[ur] & 
}
$$
commute.  If $E$ satisfies this condition, we call $E$ an \textit{$\MM$-Pettis} linear space.

In the Schwartz-Waelbroeck approach, given a linear space $E$ we ask instead whether for each map $f:X \rightarrow E$ there exists a unique linear map $f^\sharp$ such that the diagram
$$
\xymatrix{
X \ar[r]^{\delta_X} \ar[dr]_f & MX \ar@{-->}[d]^{f^\sharp} \\
                              & E
}
$$
commutes.  If the space $E$ satisfies this condition, then we define the integral of $f$ with respect to a distribution $\mu \in MX$ as $\int f \;d\mu := f^\sharp(\mu)$.  Working in the context of enriched category theory, we demand moreover that $\delta_X$ determines an \textit{isomorphism} of spaces
\begin{equation}\label{eqn:mdist_compl_intro}\uX(X,E) \cong \sL(MX,E)\end{equation}
for each $X \in \X$, in which case we call $E$ an \textit{$\MM$-distributionally complete} linear space.  For an $R$-module $E$ in a ringed topos, the natural isomorphism \eqref{eqn:mdist_compl_intro} in the case $\MM = \DD$ was called the \textit{Waelbroeck property} by Kock in \cite{Kock:ProResSynthFuncAn}.  We show that every functionally complete linear space \pbref{def:func_compl} is $\tDD$-distributionally complete.  Further, we prove the following:
\begin{quotation}
\noindent\textit{A linear space $E$ is $\tDD$-Pettis if and only if $E$ is $\tDD$-distributionally complete and separated.}
\end{quotation}
We show that the latter class of linear spaces constitutes a symmetric monoidal closed reflective subcategory of $\sL$, thus shedding light on a long-standing problem of Kock \cite{Kock:ProResSynthFuncAn} as to the potential reflectivity of the subcategory of $\RMod(\X)$ consisting of $R$-modules with the Waelbroeck property.

We show that both $\MM$-Pettis linear spaces and $\MM$-distributionally complete linear spaces are \textit{$\MM$-algebras}.  Indeed, the map $a:ME \rightarrow E$ witnessing that a space $E$ is $\MM$-Pettis serves as an $\MM$-algebra structure on the underlying space of $E$.  For an $\MM$-distributionally complete space $E$, the unique linear extension $a:ME \rightarrow E$ of the identity map $1_E$ serves as \hbox{$\MM$-algebra} structure.  In each case, the integral (of Pettis- or Schwartz-Waelbroeck-type) of a map $f:X \rightarrow E$ with respect to a distribution $\mu \in MX$ may be expressed in terms of the $\MM$-algebra structure as the result of applying the composite
$$MX \xrightarrow{Mf} ME \xrightarrow{a} E$$
to $\mu$.  For an \textit{arbitrary} $\MM$-algebra $(Z,a)$, this prompts us to define
$$\int f\;d\mu := (a \cdot Mf)(\mu)$$
for all $f:X \rightarrow Z$ and $\mu \in MX$, thus generalizing both the Pettis- and Schwartz-Waelbroeck-type integrals.  While the carrier $Z$ of an $\MM$-algebra is not a priori a linear space, but rather just an object of $\X$, a monad morphism $\Delta:\LL \rightarrow \MM$ associated to $\MM$ serves to equip each $\MM$-algebra with the structure of a linear space.  In the case of the accessible distribution monad $\tDD$, we show that the resulting notion of vector-valued integration in $\tDD$-algebras only \textit{slightly} generalizes the Pettis- and Schwartz-Waelbroeck-type integrals (which coincide in this case).  Indeed, whereas we show that the $\tDD$-Pettis linear spaces embed as a full reflective subcategory of the category of $\tDD$-algebras, we prove also that an arbitrary $\tDD$-algebra is the image of a $\tDD$-Pettis space as soon as its underlying linear space is \textit{separated}.

We show that the \textit{linearity} of the integral carried by each $\MM$-algebra is tantamount to the equality of two canonical morphisms 
$$\otimes^\Delta_{X Y},\widetilde{\otimes}^\Delta_{X Y}:MX \boxtimes LY \rightarrow M(X \boxtimes Y)$$
for each pair of objects $X,Y$ in the symmetric monoidal closed category $\X = (\X,\boxtimes,I)$.  In this case, we say that the morphism of $\X$-monads $\Delta:\LL \rightarrow \MM$ is \textit{relatively commutative}, and we then call $\MM$ a \textit{linear} abstract distribution monad.  We prove that the natural and accessible distribution monads $\DD$ and $\tDD$ are indeed always linear.  Our notion of relatively commutative morphism of monads generalizes the notion of commutative monad, in that an $\X$-monad $\TT$ on $\X$ is commutative if and only if the identity morphism on $\TT$ is relatively commutative.

Further, we show that if $\MM$ is commutative, then the vector-valued integral carried by $\MM$-algebras satisfies a Fubini-type theorem.  This applies in particular to the accessible distribution monad $\tDD$, which is always commutative \pbref{thm:td_comm}.

In the example of convergence spaces, where $\tDD \cong \DD$, we obtain several corollaries for vector integration with respect to \textit{natural distributions} and compactly-supported Radon measures in particular, many of which we summarize in detail in \pbref{sec:vint_pett_bou}.

\section{On the vector integrals of Pettis and Bourbaki}\label{sec:vint_pett_bou}

Given a Banach space $E$, let us recall Pettis' 1938 definition of the integral of a function $f:X \rightarrow E$ with respect to a measure $\mu$ on a measurable space $X$ \cite{Pett}.  We say that $f$ is \textit{scalarly $\mu$-integrable} if for each continuous linear functional $\varphi:E \rightarrow R$ ($R = \RR$ or $\CC$), the composite
$$X \xrightarrow{f} E \xrightarrow{\varphi} R$$
is $\mu$-integrable.  Given a vector $e \in E$, we say that $e$ is a \textit{Pettis integral} of $f$ with respect to $\mu$ if $f$ is scalarly $\mu$-integrable and for each continuous linear functional $\varphi$ on $E$, $\varphi(e) = \int \varphi \cdot f \;d\mu$.  Since the continuous linear functionals $\varphi \in E^*$ on $E$ separate points, the Pettis integral is unique if it exists, and we denote it by $\pint{x} f(x) \;d\mu = \pint{} f\;d\mu$.  Hence the defining property of the Pettis integral is the equation
$$\varphi\left(\:\pint{x} f(x)\;d\mu\right) = \int_x \varphi(f(x)) \;d\mu\;.$$
Regardless of whether the Pettis integral exists, if $f$ is scalarly $\mu$-integrable then we can consider the functional $E^* \rightarrow R$ given by $\varphi \mapsto \int \varphi \cdot f\;d\mu$.  The latter functional is an element of the double-dual $E^{**}$ of the Banach space $E$ and is called the \textit{Dunford integral} \cite{Dun}.  Denoting the Dunford integral by $\dint{} f\;d\mu$, the map $f$ is Pettis integrable iff $\dint{} f\;d\mu$ lies in the image of the canonical embedding of Banach spaces $E \rightarrow E^{**}$.

These definitions of Dunford and Pettis were adopted by Bourbaki \cite{Bou} in relation to Hausdorff locally convex topological vector spaces $E$ and Radon measures on locally compact Hausdorff spaces $X$.  In this case we can interpret the resulting definitions within the wider setting of the category $\X := \Conv$ of convergence spaces \pbref{exa:linear_sp} and in relation to the natural distribution monad $\DD$ on $\X$ \pbref{exa:nat_dist_conv}, as a motivating example for the more general development that shall follow in subsequent sections.  For the sake of illustration, let $f:X \rightarrow E$ be a continuous map on a compact Hausdorff topological space $X$, and let $\mu:[X,R] \rightarrow R$ be an $R$-valued Radon measure on $X$ \pbref{exa:nat_dist_conv}.  Hence $\mu \in DX = \sL([X,R],R)$, and, applying the map $Df:DX \rightarrow DE$, we obtain a continuous linear functional $(Df)(\mu):[E,R] \rightarrow R$ on the convergence vector space $[E,R]$ of $R$-valued functions on $E$, given by $g \mapsto \mu(g \cdot f) = \int g \cdot f \;d\mu$.  Forming the dual $E^{*}$ and double-dual $E^{**}$ in the category $\sL = \RMod(\X)$ of convergence vector spaces, $E^{*}$ is a subspace of $[E,R]$, so we obtain a continuous linear map
$$\xi^\DD_E:DE \rightarrow E^{**} = HE$$
sending each continuous linear functional $[E,R] \rightarrow R$ to its restriction $E^* \hookrightarrow [E,R] \rightarrow R$; in fact, $\xi^\DD_E$ is the component at $E$ of a monad morphism $\DD \rightarrow \HH$ (\bref{def:can_mnd_mor_distns},\bref{thm:inv_charn_of_morphs_for_d},\bref{exa:desc_of_xi}).  Applying $\xi^\DD_E$ to the element $(Df)(\mu)$ of $DE$, we obtain an element $\Phi$ of $E^{**}$, so that $\Phi$ is equally the result of applying the composite
$$DX \xrightarrow{Df} DE \xrightarrow{\xi^\DD_E} E^{**}$$
to $\mu \in DX$.  Explicitly, $\Phi = (\xi^\DD_E \cdot Df)(\mu)$ is a functional $E^* \rightarrow R$ given by $\Phi(\varphi) = ((Df)(\mu))(\varphi) = \mu(\varphi \cdot f) = \int \varphi \cdot f \;d\mu$ --- i.e. $\Phi$ is the Dunford integral $\dint{} f \;d\mu$, which Bourbaki calls simply the \textit{integral} of $f$ with respect to $\mu$ (\cite{Bou} \S 3).  Bourbaki studies sufficient conditions under which the integral \textit{lies in $E$}, so that the associated Pettis-type integral exists.

More generally, let $E$ be a convergence vector space, and suppose that the canonical map $\fd_E:E \rightarrow E^{**}$ is an embedding, so that in our terminology, $E$ is a \textit{functionally separated} object of $\sL$ \pbref{def:dualization}.  For example, any Hausdorff locally convex topological vector space is functionally separated \pbref{exa:refl_conv_vect}.  Given an arbitrary natural distribution $\mu \in DX$ on a convergence space $X \in \X$ and a continuous map $f:X \rightarrow E$, we can define the Dunford-type integral $\dint{} f \;d\mu$ just as above, and we may again ask also whether the Pettis-type integral $\pint{} f \;d\mu$ exists.  The latter may be characterized as the unique element of $E$ such that $\fd_E(\pint{} f \;d\mu) = \dint{} f \;d\mu$.

\begin{PropSub}\label{thm:pett_cvs}
Let $E$ a functionally separated convergence vector space.  Then the following are equivalent:
\begin{enumerate}
\item For each $R$-valued Radon measure $\mu$ on any compact Hausdorff topological space $X$ and each continuous map $f:X \rightarrow E$, the Pettis-type integral $\pint{} f \;d\mu$ exists.
\item For each natural distribution $\mu \in DX$ on any convergence space $X$ and each continuous map $f:X \rightarrow E$, the Pettis-type integral $\pint{} f \;d\mu$ exists.
\item For each natural distribution $\mu \in DE$ on $E$ itself, the Pettis-type integral $\int 1_E \;d\mu$ of the identity map $1_E:E \rightarrow E$ exists.
\item There exists a (necessarily unique) continuous map $a_E$ such that the diagram
$$
\xymatrix{
DE \ar[dr]_{a_E} \ar[rr]^{\xi^\DD_E} &                    & E^{**} \\
                                 & E \ar@{ >->}[ur]_{\fd_E} & 
}
$$
commutes.
\end{enumerate}
\end{PropSub}
\begin{proof}
Recall that for any continuous map $f:X \rightarrow E$ and any $\mu \in DX$, the Dunford-type integral $\dint{} f\;d\mu$ is obtained by applying the composite
\begin{equation}\label{eqn:dun_int_comp}DX \xrightarrow{Df} DE \xrightarrow{\xi^\DD_E} E^{**}\end{equation}
to $\mu$.  Hence 2 requires exactly that for each $f$ the image of this composite lies in the image of $\fd_E:E \rightarrow E^{**}$, and since $\fd_E$ is an embedding of convergence spaces, this is equivalent to the statement that the given composite factors through $\fd_E$.  The equivalence of 2, 3, and 4 follows immediately.  2 obviously implies 1, so it suffices to show that 1 implies 3.  Letting $\mu \in DE$, it suffices to show that there exists a continuous map $f:K \rightarrow E$ on a compact Hausdorff topological space $K$ and some $\mu' \in DK$ such that $(Df)(\mu') = \mu$.  To this end, first observe that the diagram
$$
\xymatrix{
                   & E \ar[dl]_{\delta^\DD_E} \ar[dr]^{\fd_E} & \\
DE \ar[rr]^{\xi^\DD_E} &   & E^{**}
}
$$
commutes, where $\delta^\DD$ is the unit of $\DD$ \pbref{exa:nat_dist_conv}.  Hence since $\fd_E$ is an embedding, $\delta^\DD_E$ is an embedding, so the composite
$$E \xrightarrow{\delta^\DD_E} \sL([E,R],R) \hookrightarrow \uX(\uX(E,R),R)$$
is an embedding.  Hence $E$ is a \textit{$c$-embedded} convergence space (\cite{BeBu} 1.5.20).  Therefore, it is immediate from \cite{BeBu} 3.5.21 that there is a compact subspace $K \subs E$ and some $\mu' \in DK$ such that $(D\iota_K)(\mu') = \mu$, where $\iota_K:K \hookrightarrow E$ is the inclusion.   Since $E$ is $c$-embedded, its subspace $K$ is $c$-embedded (by \cite{BeBu} 1.5.27), so since $K$ is compact, we deduce by \cite{BeBu} 1.5.28 that $K$ is a topological convergence space.  Further, by \cite{BeBu} 1.5.24, $K$ is \textit{functionally Hausdorff} and hence Hausdorff.  Therefore, $K$ is a compact Hausdorff topological space (qua convergence space), and the needed result is established.
\end{proof}

According to Definition \bref{thm:pett_cvs} below, which we formulate in a more general setting, if a functionally separated convergence vector space $E$ satisfies the equivalent conditions of \bref{thm:pett_cvs} then we say that $E$ is \textit{$\DD$-Pettis}.

For a Hausdorff locally convex topological vector space $E$, the double-dual $E^{**}$ in $\sL$ coincides with the usual \textit{Cauchy-completion} of $E$ \pbref{exa:refl_conv_vect}, with the canonical morphism $\fd_E:E \rightarrow E^{**}$ serving as the associated embedding.  Hence each Dunford-type integral $\dint{} f\;d\mu$ lies in the Cauchy-completion of $E$.  If $E$ is Cauchy-complete then $\fd_E$ is an isomorphism (i.e. $E$ is a \textit{reflexive} convergence vector space, \bref{def:dualization}) and hence $E$ is $\DD$-Pettis.

More generally then, every reflexive convergence vector space $E$ is $\DD$-Pettis.  Still more generally, we show in \bref{exa:dunf_int_wrt_nat_dist_on_convsp_lies_in_func_compl} that each Dunford-type integral $\dint{} f\;d\mu$ lies in the \textit{functional completion} $\tH E \hookrightarrow E^{**}$ of $E$ (\bref{def:func_compl}, \bref{rem:fcompl_subsp_ddual}).  In particular then, every convergence vector space $E$ that is \textit{functionally complete} \pbref{def:func_compl} is $\DD$-Pettis.

Moreover, in \bref{sec:distnl_compl_ls} and \bref{sec:distnl_compl_pett_for_acc_distns} we employ the techniques of enriched orthogonality developed in \bref{sec:cmpl_cl_dens} in order to define and study a weaker abstract notion of completeness, namely \hbox{\textit{$\DD$-distributional completeness}}, which exactly characterizes $\DD$-Pettis convergence vector spaces.  Indeed, our general result \bref{thm:dist_compl_sep_vs_pett} entails the following, via \bref{rem:d_distnl_compl_for_conv}:
\begin{quotation}
\noindent\textit{A convergence vector space $E$ is $\DD$-Pettis if and only if $E$ is $\DD$-distributionally complete and functionally separated.}
\end{quotation}
We show that the $\DD$-distributionally complete separated convergence vector spaces form a symmetric monoidal closed reflective sub-$\sL$-category of the category $\sL$ of convergence vector spaces (\bref{thm:dist_compl_sep_vs_pett},\bref{rem:d_distnl_compl_for_conv}).  Hence for each convergence vector space $E$, we can form the \textit{$\DD$-distributional completion} $\dH E$ of $E$, which embeds into the functional completion $\tH E$, and again each Dunford-type integral $\dint{} f\;d\mu$ lies in the $\DD$-distributional completion.

Given a $\DD$-Pettis convergence vector space $E$, \bref{thm:pett_spaces_are_malgs} below entails that the associated map $a_E:DE \rightarrow E$ is in fact a $\DD$-algebra structure on the underlying convergence space $E$.  Morever, given a continuous map $f:X \rightarrow E$ and a natural distribution $\mu \in DX$, the associated Pettis integral $\pint{} f\;d\mu$ is the value of the composite map $DX \xrightarrow{Df} DE \xrightarrow{a_E} E$ at $\mu$.  This motivates the following definition, which we formulate in a more general context in \bref{def:vint_in_malgs}.  Given an \textit{arbitrary} $\DD$-algebra $(Z,a)$, a continuous map $f:X \rightarrow Z$, and a natural distribution $\mu \in DX$, we define the \textit{integral} of $f$ with respect to $\mu$ as
$$\int f \;d\mu := (a \cdot Df)(\mu)\;,$$
so that the integral is an element of $Z$.  In fact, this too is a \textit{vector-valued} integral, since each $\DD$-algebra carries the structure of a convergence vector space.  Indeed, recalling that $\sL$ may be identified with the category of algebras of a monad $\LL$, we define a monad morphism $\Delta:\LL \rightarrow \DD$ (\bref{def:can_mnd_mor_distns},\bref{thm:inv_charn_of_morphs_for_d}) and thus obtain an associated functor $\uX^\Delta:\uX^\DD \rightarrow \uX^\LL = \sL$ which sends each $\DD$-algebra to its \textit{underlying convergence vector space} \pbref{par:underlying_linear_space}.  For each $\mu \in DX$, we show in \bref{thm:lin_of_vint_wrt_nat_acc_distn_wpt} that the associated mapping
$$\uX(X,Z) \rightarrow Z\;,\;\;\;\;f \mapsto \int f\;d\mu$$
is a continuous linear map.  We show also that our scalar Fubini theorem \bref{thm:sc_fub_for_conv_sp} (equivalently, the commutativity of $\DD$, \bref{thm:nat_dist_on_conv_comm}) yields also a Fubini-type theorem for this vector-valued integration \pbref{thm:vv_fub_conv_sp}.

Thus the latter notion of vector-valued integration in $\DD$-algebras subsumes the Pettis-type integral with respect to natural distributions, including compactly-supported Radon measures.  But in fact we show that integration in $\DD$-algebras only \textit{slightly} generalizes the Pettis-type integral:  Whereas the $\DD$-Pettis convergence vector spaces embed as a full reflective subcategory of the category $\X^\DD$ of $\DD$-algebras (\bref{thm:pett_td_algs_refl_in_dalg},\bref{rem:d_distnl_compl_for_conv}), we show that an arbitrary $\DD$-algebra $(Z,a)$ arises from a $\DD$-Pettis space if and only if the underlying convergence vector space of $(Z,a)$ is separated \pbref{rem:dalg_pett_if_sep_for_conv}.

We shall see also that many of the above results hold also for categories $\X$ other than $\Conv$, provided we employ the \textit{accessible} distribution monad $\tDD$ rather than $\DD$.  In the case of $\Conv$, we have shown that $\tDD$ coincides with $\DD$ \pbref{exa:acc_nat_dist_coincide_for_conv}.

\section{Canonical morphisms determined by a monad on algebras}\label{sec:can_mor_mnd_on_algs}

\begin{PropSub}\label{thm:can_mnd_mor}
In a 2-category $\K$, let $f \nsststile{\varepsilon}{\eta} g:B \rightarrow A$ be an adjunction with induced monad $\TT = (t,\eta,\mu)$, and let $\SSS = (s,\phi,\nu)$ be a monad on $B$.
\begin{enumerate}
\item The following diagram of 2-cells commutes.
$$
\xymatrix{
fg \ar[dd]_{\phi fg} \ar[dr]_\varepsilon \ar[rr]^{fg\phi} &                & fgs \ar[dd]^{\varepsilon s}\\
                                                       & 1_B \ar[dr]_\phi &     \\
sfg \ar[rr]_{s\varepsilon}                                &                & s
}
$$
\item Letting $\varsigma:fg \rightarrow s$ be the common composite 2-cell in 1,
$$(g\varsigma,g):\TT \rightarrow \SSS$$
is a morphism of monads.
\end{enumerate}
\end{PropSub}
\begin{proof}
1 follows immediately from the axioms for a 2-category.  Regarding 2, first observe that the unit law for $(g\varsigma,g):\TT \rightarrow \SSS$ holds since
$$
\xymatrix {
*!/r2ex/+{tg = gfg} \ar[rr]^{g\varsigma} \ar[dr]^{g\varepsilon} &                                                 & gs \\
                                                             & g \ar[ur]^{g\phi}                               &    \\
                                                             & g \ar[uul]^{\eta g} \ar[u]_{1} \ar[uur]_{g\phi} &    
}
$$
commutes.  Further, the diagram
$$
\xymatrix {
fgfg \ar[r]^{fg\varsigma} \ar[dd]_{\varepsilon fg} & fgs \ar[rr]^{\varsigma s} \ar[dr]|{\varepsilon s} &                                & ss \ar[dd]^{\nu} \\
                                                &                                                & s \ar@{=}[dr] \ar[ur]|{\phi s} & \\
fg \ar[rrr]_{\varsigma}                         &                                                &                                & s
}
$$
commutes; whiskering on the left by $g$, we obtain the associativity law for $(g\varsigma,g)$.
\end{proof}

\begin{ParSub} \label{par:data_mnd_on_em_cat}
For the remainder of this section, we suppose that $\V$ has equalizers, and we work with given data as follows:
\begin{enumerate}
\item Let $\TT = (T,\eta,\mu)$ be a $\V$-monad on a $\V$-category $\A$.
\item Let $\SSS = (S,\phi,\nu)$ be a $\V$-monad on $\B := \A^\TT$.
\end{enumerate}
We denote by $F \nsststile{\varepsilon}{\eta} G:\B := \A^\TT \rightarrow \A$ the Eilenberg-Moore $\V$-adjunction for $\TT$.
\end{ParSub}

\begin{DefSub} \label{def:can_mnd_mor}
Suppose data as in \bref{par:data_mnd_on_em_cat} are given.  Applying \bref{thm:can_mnd_mor} with respect to the Eilenberg-Moore $\V$-adjunction for $\TT$, we call the resulting morphism of $\V$-monads $(\zeta^{\TT\SSS},G):\TT \rightarrow \SSS$ \textit{the canonical morphism}, where $\zeta^{\TT\SSS} := G\varsigma$.
\end{DefSub}

\begin{PropSub}\label{thm:vfunc_ind_by_can_mor_is_forgetful}
Suppose data as in \bref{par:data_mnd_on_em_cat} are given.  Then the $\V$-functor \linebreak[4] \hbox{$G^{\zeta}:\B^\SSS \rightarrow \A^\TT = \B$} induced by the canonical morphism $(\zeta,G):\TT \rightarrow \SSS$ \pbref{def:can_mnd_mor} is simply the Eilenberg-Moore forgetful $\V$-functor \hbox{$G^\SSS:\B^\SSS \rightarrow \B$}.
\end{PropSub}
\begin{proof}
Since the diagram
$$
\xymatrix {
\B^\SSS \ar[r]^{G^{\zeta}} \ar[d]_{G^\SSS} & *!/l2ex/+{\A^\TT = \B} \ar[d]^{G^\TT\:=\:G} \\
\B \ar[r]^G                                & \A                            
}
$$
commutes and $G$ is $\V$-faithful, it suffices to show that $G^{\zeta} = G^\SSS$ on objects.  Given an $\SSS$-algebra $(B,b:SB \rightarrow B)$,
$$G^{\zeta}(B,b) = (GB,GFGB = TGB \xrightarrow{\zeta_B} GSB \xrightarrow{Gb} GB)\;,$$
whereas $G^\SSS(B,b) = B = (GB,G\varepsilon_B:GFGB \rightarrow GB)$, but the diagram
$$
\xymatrix {
GFGB \ar[rr]^{\zeta_B} \ar[dr]_{G\varepsilon_B}           &                                  & GSB \ar[dd]^{Gb} \\
                                                       & GB \ar@{=}[dr] \ar[ur]_{G\phi_B} &     \\
                                                       &                                  & GB 
}
$$
commutes.
\end{proof}

\begin{DefSub}\label{def:Delta_ts}
Given data as in \bref{par:data_mnd_on_em_cat}, we let $\Delta^{\TT\SSS} := [F,G](\phi) : \TT = [F,G](\ONEONE_\B) \rightarrow [F,G](\SSS)$ denote the morphism of $\V$-monads obtained by applying the monoidal functor $[F,G]:\VCAT(\B,\B) \rightarrow \VCAT(\A,\A)$ \pbref{thm:mon_func_detd_by_adj} to the unique monad morphism $\phi:\ONEONE_\B \rightarrow \SSS$ \pbref{par:adj_and_mnd} (whose underlying 2-cell is simply the unit $\phi$ of $\SSS$).
\end{DefSub}

\begin{PropSub}\label{thm:mnd_mor_gsf_to_s}
Given data as in \bref{par:data_mnd_on_em_cat}, there is a morphism of $\V$-monads
$$(GS\varepsilon,G):[F,G](\SSS) \rightarrow \SSS\;.$$
\end{PropSub}
\begin{proof}
Let $[F,G](\SSS) = (GSF,\delta,\kappa)$, and let $\xi := GS\varepsilon : GSFG \rightarrow GS$.  The unit law for the candidate morphism of $\V$-monads $(\xi,G)$ holds since
$$
\xymatrix{
GSFG \ar[rrr]^{\xi\:=\:GS\varepsilon}  &                                                               &                                                & GS \\
                                  & GFG \ar[ul]|{G\phi FG} \ar[r]^{G\varepsilon}                     & G \ar[ur]^{G\phi}                              &    \\
                                  & G \ar[u]^{\eta G} \ar[ur]_{1_G} \ar@/^1ex/[uul]^{\delta G}  &                                                &
}
$$
commutes, and the associativity law holds since 
$$
\xymatrix {
*!/r3ex/+{GSFGSFG} \ar[rr]^{GSF\xi\:=\:GSFGS\varepsilon} \ar[dd]_{\kappa G} \ar[drr]|{GS\varepsilon SFG} & & GSFGS \ar[rr]^{\xi S\:=\:GS\varepsilon S}             & & GSS \ar[dd]^{G\nu} \\
                                                                                      & & GSSFG \ar[dll]|{G\nu FG} \ar[urr]_{GSS\varepsilon}  & &                    \\
GSFG \ar[rrrr]_{\xi\:=\:GS\varepsilon}                                                     & &                                                  & & GS
}
$$
commutes.
\end{proof}

\begin{DefSub}\label{def:xi_ts}
Given data as in \bref{par:data_mnd_on_em_cat}, we define $\xi^{\TT\SSS} := GS\varepsilon$, so that the morphism of $\V$-monads given in \bref{thm:mnd_mor_gsf_to_s} is denoted by $(\xi^{\TT\SSS},G):[F,G](\SSS) \rightarrow \SSS$.
\end{DefSub}

The following proposition assembles within a commutative diagram the morphisms of $\V$-monads defined above.

\begin{PropSub}\label{thm:comm_diag_can_mnd_morphs}
Suppose that $\V$ has equalizers.  Let $\TT$ be a $\V$-monad on a $\V$-category $\A$ with Eilenberg-Moore $\V$-adjunction $F \dashv G:\B := \A^\TT \rightarrow \A$, and let $\SSS$ be a $\V$-monad on $\B$.  Then we obtain a commutative diagram of $\V$-monads
\begin{equation}\label{eqn:comm_diag_can_mnd_morphs}
\xymatrix {
\TT \ar[dr]_{(\Delta^{\TT\SSS},1_\A)} \ar[rr]^{(\zeta^{\TT\SSS},G)} &                                         & \SSS \\
                                                                    & [F,G](\SSS) \ar[ur]_{(\xi^{\TT\SSS},G)} & 
}
\end{equation}
in which the given morphisms of $\V$-monads are as defined in \bref{def:can_mnd_mor}, \bref{def:Delta_ts}, \bref{def:xi_ts}.
\end{PropSub}
\begin{proof}
Indeed, the following diagram commutes by the definition of $\zeta^{\TT\SSS}$ \pbref{def:can_mnd_mor} and \bref{thm:can_mnd_mor}.
$$
\xymatrix{
*!/r2ex/+{TG\:=\:GFG} \ar[rr]^{\zeta^{\TT\SSS}} \ar[dr]|{\Delta^{\TT\SSS} G\:=\:G\phi FG} &                                           & GS \\
                                                                                     & GSFG \ar[ur]_{\xi^{\TT\SSS}\:=\:GS\varepsilon} &
}
$$
\end{proof}

\begin{CorSub}\label{thm:comm_diag_funcs_ind_by_can_mnd_morphs}
The commutative diagram \eqref{eqn:comm_diag_can_mnd_morphs} induces via \bref{prop:enr_func_induced_by_mnd_mor} a commutative triangle of $\V$-functors
\begin{equation}\label{eqn:comm_diag_funcs_ind_by_can_mnd_morphs}
\xymatrix{
*!/r3ex/+{\B = \A^\TT} &                                      & {\B^\SSS} \ar[ll]_{G^{\zeta^{\TT\SSS}}} \ar[dl]^{G^{\xi^{\TT\SSS}}}\\
       & \A^{[F,G](\SSS)} \ar[ul]^{\A^{\Delta^{\TT\SSS}}} &
}
\end{equation}
in which the top side is \textnormal{(}by \bref{thm:vfunc_ind_by_can_mor_is_forgetful}\textnormal{)} the forgetful $\V$-functor.
\end{CorSub}

\begin{PropSub}\label{thm:free_gsf_alg_via_xi}
For each $A \in \A$, the $\X$-functor $G^{\xi^{\TT\SSS}}$ \pbref{thm:comm_diag_funcs_ind_by_can_mnd_morphs} sends the free $\SSS$-algebra $(SFA,\nu_{FA})$ on $FA$ to the free $[F,G](\SSS)$-algebra on $A$.
\end{PropSub}
\begin{proof}
$$G^{\xi^{\TT\SSS}}(SFA,\nu_{FA}) = (GSFA,GSFGSFA \xrightarrow{\xi^{\TT\SSS} SFA} GSSFA \xrightarrow{G\nu FA} GSFA)$$
and since by definition $\xi^{\TT\SSS} = GS\varepsilon$, the latter $[F,G](\SSS)$-algebra structure is the composite
$$GSFGSFA \xrightarrow{GS\varepsilon SFA} GSSFA \xrightarrow{G\nu FA} GSFA\;,$$
which is the free $[F,G](\SSS)$-algebra structure on $A$.
\end{proof}

\begin{CorSub}\label{thm:t_alg_obtained_from_free_gsf_alg_on_a_via_delta_is_sfa}
For each object $A \in \A$, the object of $\B = \A^\TT$ obtained by applying $\A^{\Delta^{\TT\SSS}}$ to the free $[F,G](\SSS)$-algebra on $A$ is $SFA$.
\end{CorSub}
\begin{proof}
This follows from \bref{thm:free_gsf_alg_via_xi} and \bref{thm:comm_diag_funcs_ind_by_can_mnd_morphs}.
\end{proof}

\section{Canonical morphisms for the distribution monads}

As we continue our study of the natural and accessible distribution monads $\DD$ and $\tDD$ of Chapter \bref{ch:nat_acc_dist}, we shall again work with a symmetric monoidal adjunction as in \bref{par:smc_adj}, but now we assume that this adjunction is monadic, as follows.

\begin{ParSub}\label{par:data_for_dcompl_and_pett}\emptybox
\begin{enumerate}
\item Let $\X = (\X,\btimes,I)$ be a closed symmetric monoidal category, and suppose that $\X$ is finitely well-complete \pbref{def:enr_fwc}.
\item Let $\LL = (L,\eta,\mu)$ be a commutative $\X$-monad on $\uX$, and suppose that $\LL$ has tensor products of algebras \pbref{def:tens_prod_algs}.
\end{enumerate}
Then by \bref{thm:smclosed_em_adj}, the Eilenberg-Moore adjunction $F \nsststile{\varepsilon}{\eta} G:\sL \rightarrow \X$ for $\LL$ acquires the structure of a symmetric monoidal adjunction, in which the category $\sL := \X^\LL$ of $\LL$-algebras is a symmetric monoidal closed category $(\sL,\otimes,R)$.  Also, as noted in \bref{exa:smc_adj_with_l_fwc}, $\sL$ is finitely well-complete.  Further, by \bref{thm:commmnd_sm_em_adj_dets_em_vadj}, the $\X$-adjunction $\acute{F} \nsststile{\varepsilon}{\eta} \grave{G}:G_*\uL \rightarrow \uX$ determined by this symmetric monoidal adjunction (via \bref{thm:assoc_enr_adj}) is simply the Eilenberg-Moore $\X$-adjunction $F^\LL \nsststile{}{} G^\LL:\uX^\LL \rightarrow \uX$.  We employ the following notation and terminology:
\begin{enumerate}
\item Let $F \nsststile{\varepsilon}{\eta} G:\sL \rightarrow \uX$ denote the Eilenberg-Moore $\X$-adjunction for $\LL$.
\item We call the $\X$-monad $\HH = (H,\fd,\fk)$ obtained by applying the 2-functor $G_*:\LCAT \rightarrow \XCAT$ to the double-dualization $\sL$-monad \pbref{def:dualization} \textit{the double-dualization $\X$-monad}.
\end{enumerate}
\end{ParSub}

\begin{ExaSub}
Examples of the above data are noted in \bref{exa:smc_adj_with_l_fwc}.
\end{ExaSub}

\begin{DefSub}\label{def:can_mnd_mor_distns}
We now define a diagram of $\X$-monad morphisms
\begin{equation}\label{eqn:can_mnd_mor_distns}
\xymatrix{
                                                                               & \DD \ar[dr]^{(\xi^{\DD},G)}   &     \\
\LL \ar[ur]^{(\Delta^{\DD},1_{\uX})} \ar[dr]_{(\Delta^{\tDD},1_{\uX})} \ar[rr]|{(\zeta,G)} &                         & \HH \\
                                                                               & \tDD \ar[ur]_{(\xi^{\tDD},G)} & 
}
\end{equation}
as follows.  Recalling that $\DD = [F,G](\HH)$ \pbref{prop:dist_via_dd}, we define $\Delta^{\DD} := \Delta^{\LL\HH}$ \pbref{def:Delta_ts} and $\xi^{\DD} := \xi^{\LL\HH}$ \pbref{def:xi_ts}.  Recalling also that $\tDD = [F,G](\tHH)$ \pbref{sec:acc_distn}, we define $\Delta^{\tDD} := \Delta^{\LL\tHH}$, but we take $(\xi^{\tDD},G)$ to be the composite
$$\tDD = [F,G](\tHH) \xrightarrow{(\xi^{\LL\tHH},G)} \tHH \xrightarrow{(i,1_\sL)} \HH\;,$$
where $i$ is as defined in \bref{def:func_compl}.  We define $\zeta := \zeta^{\LL\HH}$, so that $(\zeta,G)$ is the canonical morphism \pbref{def:can_mnd_mor}.
\end{DefSub}

\begin{PropSub}\label{thm:can_mnd_mor_distns}
The diagram \eqref{eqn:can_mnd_mor_distns} commutes, as does the following diagram
\begin{equation}\label{eqn:jtd_mor_abs_distn_mnd}
\xymatrix{
                                                             & \DD \ar[dr]^{(\xi^{\DD},G)}                         &     \\
\LL \ar[ur]^{(\Delta^{\DD},1_{\uX})} \ar[dr]_{(\Delta^{\tDD},1_{\uX})}   &                                               & \HH \\
                                                             & \tDD \ar[uu]|{(j^{\tDD},1_{\uX})} \ar[ur]|{(\xi^{\tDD},G)} &
}
\end{equation}
where $j^{\tDD}$ is as given in \bref{def:acc_distns}.
\end{PropSub}
\begin{proof}
It suffices to show that the latter diagram commutes, since the upper triangle in \eqref{eqn:can_mnd_mor_distns} commutes by \bref{thm:comm_diag_can_mnd_morphs}.  The leftmost triangle in this latter diagram commutes, since it may be obtained by applying the monoidal functor $[F,G]:\XCAT(\sL,\sL) \rightarrow \XCAT(\uX,\uX)$ \pbref{thm:mon_func_detd_by_adj} to the triangle
$$
\xymatrix{
                                               & \HH                 \\
\ONEONE_\sL \ar[ur]^\fd \ar[dr]_\tpartial &                     \\
                                               & \tHH \ar[uu]_i
}
$$
which commutes since the unit $\fd$ of $\HH$ is the unique morphism of $\X$-monads $\ONEONE_\sL \rightarrow \HH$ \pbref{par:adj_and_mnd}.  Regarding the rightmost triangle, observe that the upper triangle in the following diagram commutes
$$
\xymatrix{
*!/r2ex/+{\tD G = G\tH FG} \ar[d]_{j^{\tDD} G\:=\:GiFG} \ar[rr]^{\xi^{\LL\tHH}\:=\:G\tH \varepsilon} \ar[drr]|{\xi^{\tDD}} & & G\tH \ar[d]^{Gi} \\
*!/r2ex/+{DG = GHFG} \ar[rr]_{\xi^{\DD}\:=\:\xi^{\LL\HH} = GH\varepsilon} & & GH \\
}
$$
by the definition of $\xi^{\tDD}$, so since the periphery commutes by the naturality of $i$, the lower triangle commutes, as needed.
\end{proof}

\begin{DefSub}\label{def:iota}
For each object $E \in \sL$, there is a canonical morphism $\iota_E:E^* \rightarrow [GE,R]$ in $\sL$ gotten as the transpose of $G\fd_E:GE \rightarrow GHE = \sL(E^*,R)$ under the `hom-cotensor' adjunction for $R$ \eqref{eq:hom_cotensor}.
\end{DefSub}

In defining the morphisms of $\X$-monads $\Delta^\DD$, $\xi^\DD$ in \bref{def:can_mnd_mor_distns}, we have assumed that the cotensors $[X,R]$ in $\sL$ employed in defining $DX = \sL([X,R],R)$ and the $\X$-monad $\DD$ are chosen as $[X,R] = (FX)^*$ (per \bref{thm:assoc_enr_adj}).  Any other choice of cotensors yields (by \bref{thm:uniq_adj}) an isomorphic `hom-cotensor' $\X$-adjunction $[-,R] \dashv \sL(-,R):\sL^\op \rightarrow \uX$ and hence an isomorphic induced $\X$-monad $\DD$. Hence \bref{def:can_mnd_mor_distns} implicitly furnishes definitions applicable for any choice of cotensors $[X,R]$.  The following provides a characterization of $\Delta^\DD$, $\xi^\DD$ that is independent of the choice of cotensors.

\begin{PropSub}\label{thm:inv_charn_of_morphs_for_d}
Let $X \in \X$ and $E \in \sL$.
\begin{enumerate}
\item $\Delta^\DD_X:LX \rightarrow DX$ is the underlying map of the transpose $FX \rightarrow \uL([X,R],R)$ of $\delta^\DD_X:X \rightarrow G\uL([X,R],R) = DX$ under $F \dashv G$.
\item $\zeta_E:LGE \rightarrow GHE$ is the underlying map of the transpose $FGE \rightarrow HE$ of \linebreak[4]\hbox{$G\fd_E:GE \rightarrow GHE$} under $F \dashv G$.
\item $\xi^\DD_E:DGE \rightarrow GHE$ is exactly the map
$$\sL(\iota_E,R):\sL([GE,R],R) \rightarrow \sL(E^*,R)\;,$$
where $\iota_E:E^* \rightarrow [GE,R]$ is the canonical map defined in \bref{def:iota}.
\end{enumerate}
\end{PropSub}
\begin{proof}
\emptybox
1. In view of the definition of $\Delta^\DD = G\fd F :GF \rightarrow GHF$, the claim follows from \bref{thm:transp_of_delta}.

2. Since the transpose in question is the composite
$$FGE \xrightarrow{FG\fd_E} FGHE \xrightarrow{\varepsilon HE} HE\;,$$
this follows from \bref{thm:can_mnd_mor} and the definition of $\zeta = \zeta^{\LL \HH}$ \pbref{def:can_mnd_mor}.

3. Employing the cotensors $[X,R] = (FX)^{*}$, the transpose $\iota_E:E^* \rightarrow [GE,R]$ of $G\fd_E:GE \rightarrow GHE = \sL(E^*,R)$ under the `hom-cotensor' adjunction $[-,R] \nsststile{}{\delta^\DD} \sL(-,R):\sL^\op \rightarrow \uX$ is $\iota_E = \varepsilon_E^*:E^* \rightarrow [GE,R] = (FGE)^*$, since
$$
\xymatrix{
GE \ar@/_4ex/@{=}[ddr] \ar[dr]|{\eta_{GE}} \ar[rr]^{\delta^\DD_{GE}} & & *!/l6ex/+{GHFGE = \sL((FGE)^*,R)} \ar[d]^{GH\varepsilon_E \:=\: \sL(\varepsilon_E^*,R)} \\
& GFGE \ar[ur]|{G\fd FGE} \ar[d]|{G\varepsilon_E} & *!/l4ex/+{GHE = \sL(E^*,R)}\\
& GE \ar[ur]_{G\fd_E} & 
}
$$
commutes (using the fact that that $\DD = [G,F](\HH)$ under this choice of cotensors).  But by definition 
$$\xi^\DD_E = GH\varepsilon_E = \sL(\varepsilon^*_E,R)\;\;:\;\;GHFGE = \sL((FGE)^*,R) \rightarrow GHE = \sL(E^*,R)\;.$$ 
\end{proof}

\begin{ExaSub}\label{exa:desc_of_xi}
Let us assume for the sake of illustration that our base category $\X$ is cartesian closed, locally small, and well-pointed \pbref{par:well_pointed_ccc}, as in \bref{exa:monadic_exa_integ_notn}, and let us choose the cotensors $[X,R]$ in $\sL$ as in \bref{rem:standard_cotensors_in_talg}, \bref{exa:monadic_exa_integ_notn}, so that the elements of $[X,R]$ are maps $f:X \rightarrow GR$ in $\X$.  For each $E \in \sL$, one finds that the underlying map of $\iota_E:E^* \rightarrow [GE,R]$ \pbref{def:iota} is simply the inclusion $G_{E R}:\sL(E,R) \hookrightarrow \uX(GE,GR)$.  Hence, by \bref{thm:inv_charn_of_morphs_for_d} 3, we obtain the following:
\begin{enumerate}
\item $\xi^\DD_E:DGE = \sL([GE,R],R) \rightarrow \sL(E^*,R) = GHE$ sends each $\mu:[GE,R] \rightarrow R$ in $\sL$ to the restriction $E^{*} \rightarrow R$ of $\mu$, given by $\varphi \mapsto \mu(\varphi) = \int \varphi\;d\mu$.
\item Given a map $f:X \rightarrow GE$ in $\X$ and a natural distribution $\mu \in DX$, the composite
$$DX \xrightarrow{Df} DGE \xrightarrow{\xi^\DD_E} G(E^{**})$$
sends $\mu$ to the functional $E^* \rightarrow R$ given by $\varphi \mapsto \int \varphi \cdot f \;d\mu$.  We call this functional \textit{the Dunford-type integral} and denote it by $\dint{} f\;d\mu$.
\end{enumerate}
\end{ExaSub}

\section{Abstract distribution monads}

We continue to work with data as given in \bref{par:data_for_dcompl_and_pett}.

\begin{DefSub}\label{def:abs_distn_mnd}\emptybox
\begin{enumerate}
\item An \textit{abstract distribution monad} $(\MM,\Delta,\xi)$ consists of an $\X$-monad $\MM = (M,\delta,\kappa)$ on $\uX$ equipped with $\X$-natural transformations $\Delta, \xi$ yielding a commutative diagram of $\X$-monad morphisms
\begin{equation}\label{eqn:mnd_morphs_for_abs_distn_mnd}
\xymatrix {
\LL \ar[dr]_{(\Delta,1_{\uX})} \ar[rr]^{(\zeta,G)} &                               & \HH \\
                                                & \MM \ar[ur]_{(\xi,G)} & 
}
\end{equation}
in which $\HH$ is the double-dualization $\X$-monad on $\sL$ \pbref{par:data_for_dcompl_and_pett} and $(\zeta,G)$ is the canonical morphism \pbref{def:can_mnd_mor}.  
\item A \textit{morphism} of abstract distribution monads $\lambda:(\MM,\Delta,\xi) \rightarrow (\MM',\Delta',\xi')$ consists of a morphism of $\X$-monads $\lambda:\MM \rightarrow \MM'$ such that the diagram
$$
\xymatrix{
                                                             & \MM' \ar[dr]^{(\xi,G)}                         &     \\
\LL \ar[ur]^{(\Delta,1_{\uX})} \ar[dr]_{(\Delta,1_{\uX})}   &                                               & \HH \\
                                                             & \MM \ar[uu]|{(\lambda,1_{\uX})} \ar[ur]_{(\xi,G)} &
}
$$
commutes.
\end{enumerate}
We shall often write simply the name of the $\X$-monad $\MM$ to denote an abstract distribution monad; we then denote the associated morphisms by $\delta^\MM$, $\kappa^\MM$, $\Delta^\MM$, and $\xi^\MM$.  Abstract distribution monads clearly form a category with the given morphisms.
\end{DefSub}

\begin{ExaSub}\emptybox
\begin{enumerate}
\item The natural distribution monad $\DD$ and the accessible distribution monad $\tDD$ carry the structure of abstract distribution monads $(\DD,\Delta^\DD, \xi^\DD)$, $(\tDD,\Delta^{\tDD}, \xi^{\tDD})$ \pbref{def:can_mnd_mor_distns}.
\item By \bref{thm:can_mnd_mor_distns}, we have a morphism of abstract distribution monads $j^{\tDD}:\tDD \rightarrow \DD$.
\item The monad $\LL$ itself carries the structure of an abstract distribution monad, with $\Delta^\LL := 1_L$ and $\xi^\LL := \zeta$.  In fact, $\LL$ is an initial object in the category of abstract distribution monads \pbref{def:abs_distn_mnd}, since for any abstract distribution monad $\MM$, $\Delta^\MM$ is the unique morphism of abstract distribution monads $\LL \rightarrow \MM$.
\end{enumerate}
\end{ExaSub}

\begin{ExaSub}\label{exa:d_td_for_conv_ident_as_abs_distn_mnds}
In the example of convergence spaces \pbref{exa:smc_adj_with_l_fwc}, $j^{\tDD}:\tDD \rightarrow \DD$ is an isomorphism of abstract distribution monads, by \bref{exa:acc_nat_dist_coincide_for_conv}.  Further, by \bref{exa:acc_nat_dist_coincide_for_conv}, the $\X$-monads $\tDD$, $\DD$ may be taken to be identical, with $j^{\tDD}$ the identity, and in this case the abstract distribution monads $(\tDD,\Delta^{\tDD},\xi^{\tDD})$ and $(\DD,\Delta^{\DD},\xi^{\DD})$ are then also identical.
\end{ExaSub}

For the remainder of the section, we fix an abstract distribution monad $(\MM,\Delta,\xi)$, $\MM = (M,\delta,\kappa)$.

\begin{RemSub}\label{rem:dun_pett_ints_wrt_abs_distn}
For the sake of illustration, assume for the moment that $\X$ is cartesian closed, locally small, and well-pointed \pbref{par:well_pointed_ccc}.  Given a linear space $E \in \sL$, a map $f:X \rightarrow E$ in $\X$, and an element $\mu$ of $MX$, we call the element of $E^{**}$ obtained by applying
\begin{equation}\label{eqn:comp_for_dun_integ}MX \xrightarrow{Mf} MGE \xrightarrow{\xi^\MM_E} GHE = G(E^{**})\end{equation}
to $\mu$ \textit{the Dunford-type integral} of $f$ with respect to $\mu$, and we denote this element by $\dint{} f\;d\mu$.  Further, if $E$ is separated, then given an element $e \in E$, we say that $e$ is a \textit{Pettis-type integral} of $f$ with respect to $\mu$ if $\fd_E(e) = \dint{} f\;d\mu$.  Since $\fd_E:E \rightarrowtail E^{**}$ is injective, the Pettis-type integral is unique if it exists, and we denote it by $\pint{} f\;d\mu$.
\end{RemSub}

\begin{ExaSub}\label{exa:dunf_int_wrt_acc_dist_lies_in_func_compl}
Assume that $\X$ is cartesian closed, locally small, and well-pointed, and consider the case where $\MM = \tDD$.  By its definition, the morphism $\xi^{\tDD}_E:\tD GE \rightarrow GHE = G(E^{**})$ factors through the embedding of linear spaces $i_E:\tH E \rightarrowtail E^{**}$.  Taking the latter embedding to be an inclusion, as in \bref{rem:fcompl_subsp_ddual}, we find that for each map $f:X \rightarrow E$ in $\X$ the associated composite \eqref{eqn:comp_for_dun_integ} factors through the inclusion $\tH E \hookrightarrow E^{**}$, so that each Dunford-type integral $\dint{} f\;d\mu$ $(\mu \in \tD X)$ lies in the functional completion $\tH E$ of $E$.
\end{ExaSub}

\begin{ExaSub}\label{exa:dunf_int_wrt_nat_dist_on_convsp_lies_in_func_compl}
Consider the example of convergence spaces \pbref{exa:smc_adj_with_l_fwc}, in which $j^{\tDD}:\tDD \rightarrow \DD$ is an isomorphism of abstract distribution monads \pbref{exa:d_td_for_conv_ident_as_abs_distn_mnds}.  For any convergence vector space $E \in \sL$, any continuous map $f:X \rightarrow E$, and any natural distribution $\mu \in DX$, we deduce by \bref{exa:dunf_int_wrt_acc_dist_lies_in_func_compl} that the Dunford-type integral $\dint{} f\;d\mu$ lies in the functional completion $\tH E$ of $E$.
\end{ExaSub}

\begin{ParSub}\label{par:underlying_linear_space}
Each $\MM$-algebra $(X,a)$ carries the structure of a linear space (i.e. an object of $\sL = \uX^\LL$, \bref{def:first_synth_terminology}), via the $\X$-functor
$$\uX^\Delta:\uX^\MM \rightarrow \uX^\LL = \sL$$
associated to $\Delta:\LL \rightarrow \MM$ via \bref{prop:enr_func_induced_by_mnd_mor}.  We call $\uX^\Delta(X,a)$ \textit{the underlying linear space} of $(X,a)$.  Since $\uX^\Delta$ commutes with the forgetful $\X$-functors, the underlying space $G\uX^\Delta(X,a)$ of $\uX^\Delta(X,a)$ is simply $X$.
\end{ParSub}

\begin{ExaSub}\label{exa:und_linear_spaces_of_free_d_and_td_algs}\emptybox
\begin{enumerate}
\item The underlying linear space of each free $\DD$-algebra $(DX,\kappa^\DD_X)$ $(X \in \X)$ is the linear space of natural distributions $\uL([X,R],R) = (FX)^{**} = HFX$ on $X$.
\item The underlying linear space of each free $\tDD$-algebra $(\tD X,\kappa^{\tDD}_X)$ $(X \in \X)$ is the linear space of accessible distributions $\tH FX$ on $X$.
\begin{proof}
Recalling that $\DD = [F,G](\HH)$ and $\tDD = [F,G](\tHH)$, each of 1 and 2 is obtained by an invocation of \bref{thm:t_alg_obtained_from_free_gsf_alg_on_a_via_delta_is_sfa}.
\end{proof}
\end{enumerate}
\end{ExaSub}

\begin{ParSub}\label{par:underlying_malg_of_halg}
The diagram of $\X$-monad morphisms \eqref{eqn:mnd_morphs_for_abs_distn_mnd} determines a commutative diagram of $\X$-functors
\begin{equation}\label{eqn:xfunctors_for_abs_distn_mnd}
\xymatrix{
*!/r3ex/+{\sL = \uX^\LL} &                                             & {\sL^\HH} \ar[ll]_{G^\zeta} \ar[dl]^{G^\xi}\\
                               & \uX^\MM \ar[ul]^{\uX^\Delta} &
}
\end{equation}
via \bref{prop:enr_func_induced_by_mnd_mor}, and by \bref{thm:vfunc_ind_by_can_mor_is_forgetful}, $G^\zeta$ is simply the forgetful $\X$-functor.  In particular, then, every $\HH$-algebra $(E,b)$ determines an $\MM$-algebra, via $G^\xi$, whose underlying linear space is simply $E$.
\end{ParSub}

\begin{ExaSub}\label{exa:r_malg}
The unit object $R$ of $\sL$ is reflexive \pbref{exa:unit_reflexive}, and it follows readily that $(R,\fd_R^{-1}:HR \xrightarrow{\sim} R)$ is an $\HH$-algebra, whence $R$ is the underlying linear space of an $\MM$-algebra with structure morphism
$$MGR \xrightarrow{\xi_R} GHR \xrightarrow{G\fd_R^{-1}} GR\;.$$
\end{ExaSub}

\begin{ExaSub}\label{exa:r_dalg}
Suppose $\X$ is cartesian closed, locally small, and well-pointed, and consider the example of $\MM = \DD$.  One deduces via \bref{exa:desc_of_xi} and \bref{exa:unit_reflexive} that the $\DD$-algebra structure $a:DGR = \sL([GR,R],R) \rightarrow GR$ associated to $R$ via \bref{exa:r_malg} sends each natural distribution $\mu:[GR,R] \rightarrow R$ on the underlying space $GR$ of $R$ to the result of evaluating $\mu$ at the identity map $1_{GR}:R \rightarrow R$; i.e.
$$a(\mu) = \mu(1_{GR}) = \int 1_{GR} \;d\mu\;.$$
\end{ExaSub}

\begin{PropSub}\label{thm:malg_struct_on_dx}
Given a space $X \in \X$, the space of natural distributions $DX$ carries the structure of an $\MM$-algebra, whose underlying linear space is the linear space of natural distributions $(FX)^{**} = HFX$.
\end{PropSub}
\begin{proof}
By \bref{par:underlying_malg_of_halg}, the free $\HH$-algebra $(HFX,\fk_{FX})$ on $FX$ determines an $\MM$-algebra \linebreak[4] $G^\xi(HFX,\fk_{FX})$ whose underlying linear space is again $HFX$, and whose underlying $\X$-object is $GHFX = DX$.
\end{proof}

\begin{ExaSub}\label{exa:malg_str_on_dx}\emptybox
\begin{enumerate}
\item Since $\DD$ itself is an abstract distribution monad, \bref{thm:malg_struct_on_dx} endows $DX$ with the structure of a $\DD$-algebra, and, as one might hope, this $\DD$-algebra is simply the free $\DD$-algebra on $X$.  Indeed, $G^{\xi^\DD}(HFX,\fk_{FX}) = (DX,\kappa^\DD_X)$.
\begin{proof}
Since $\DD = [F,G](\HH)$, this follows from \bref{thm:free_gsf_alg_via_xi}.
\end{proof}
\item Since $\tDD$ is an abstract distribution monad, \bref{thm:malg_struct_on_dx} endows $DX$ with the structure of a $\tDD$-algebra.
\end{enumerate}
\end{ExaSub}

\begin{ThmSub}\label{thm:mor_from_absdistnmnd_to_d}
For each abstract distribution monad $(\MM,\Delta^\MM,\xi^\MM)$ there is a unique morphism of abstract distribution monads \pbref{def:abs_distn_mnd}
$$j = j^\MM:\MM \rightarrow \DD$$
to the natural distribution monad $\DD$.
\end{ThmSub}
\begin{proof}
Define $j:M \rightarrow D$ to be the composite $\X$-natural transformation
$$j := (M \xrightarrow{M\eta} MGF \xrightarrow{{\xi^\MM} F} GHF = D )\;.$$
The follow diagram of morphisms of $\X$-endofunctors \pbref{par:general_mnd_mor} commutes
$$
\xymatrix{
D \ar[dr]^{(\xi^{\DD},G)}                       &   \\
                                                & H \\
M \ar[uu]|{(j,1_{\uX})} \ar[ur]|{(\xi^{\MM},G)} &
}
$$
since the following diagram commutes
$$
\xymatrix{
MG \ar@/_6ex/@{=}[dd] \ar[d]|{M\eta G} \ar[dr]^{jG} & \\
MGFG \ar[d]|{MG\varepsilon} \ar[r]_{\xi^\MM FG} & *!/l5ex/+{DG = GHFG} \ar[d]|{\xi^\DD\:=\:GH\varepsilon} \\
MG \ar[r]^{\xi^\MM} & GH 
}
$$
by a triangular inequality for $F \nsststile{\varepsilon}{\eta} G$ and the naturality of $\xi^\MM$.  Hence the following diagram of functors (obtained via \bref{par:alg_endofunc}) commutes,
$$
\xymatrix{
\X^D \ar[dd]_{\X^j} &\\
     & \sL^H \ar[ul]_{G^{\xi^\DD}} \ar[dl]^{G^{{\xi^\MM}}} \\
\X^M  & 
}
$$
so in particular we note that
\begin{equation}\label{eqn:coin_of_algs}G^{\xi^\MM}(HFX,\fk_{FX}) = \X^j G^{\xi^\DD}(HFX,\fk_{FX}) = \X^j(DX,\kappa^\DD_X)\end{equation}
since $G^{\xi^\DD}(HFX,\fk_{FX}) = (DX,\kappa^\DD_X)$ by \bref{exa:malg_str_on_dx}.

Next observe that for each $X \in \X$, the diagram
$$
\xymatrix{
MX \ar[rr]^{j_X} \ar[dr]|{M\eta_X} & & *!/l3ex/+{GHFX = DX} \\
& MGFX \ar[ur]|{{\xi^\MM} FX} & \\
X \ar@/_10ex/[uurr]|{\delta^\DD_X} \ar[uu]|{\delta^\MM_X} \ar[r]^{\eta_X} & GFX \ar[u]|{\delta^\MM GFX} \ar@/_2ex/[uur]|{G\fd FX} & 
}
$$
commutes, using the definition of $\delta^\DD$ and the assumption that $({\xi^\MM},G)$ is a morphism of $\X$-monads; hence $j_X$ extends $\delta^\DD_X$ along $\delta^\MM_X$.  In order to show that $j$ is a morphism of $\X$-monads, it therefore suffices by \bref{thm:mnd_mor_as_homom} to show that $j_X$ serves as an $M$-homomorphism
$$j_X:(MX,\kappa^\MM_X) \rightarrow \X^j(DX,\kappa^\DD_X)\;,$$
but the $M$-algebra $\X^j(DX,\kappa^\DD_X)$ is identical to the $\MM$-algebra $G^{\xi^\MM}(HFX,\fk_{FX})$ by \eqref{eqn:coin_of_algs}.  Moreover, the factors of $j_X$ serve as $\MM$-homomorphisms
$$(MX,\kappa^\MM_X) \xrightarrow{M\eta_X} (MGFX,\kappa^\MM_{GFX}) \xrightarrow{{\xi^\MM}_{FX}} G^{\xi^\MM}(HFX,\fk_{FX})\;,$$
(the latter by \bref{thm:mnd_mor_as_homom}, since $({\xi^\MM},G)$ is a morphism of $\X$-monads), so $j_X$ is an $\MM$-homomorphism.

Now we have a diagram
\begin{equation}\label{eqn:abs_distn_mnd_diag_for_j}
\xymatrix{
                                                             & \DD \ar[dr]^{(\xi^{\DD},G)}                         &     \\
\LL \ar[ur]^{(\Delta^{\DD},1_{\uX})} \ar[dr]_{(\Delta^{\MM},1_{\uX})}   &                                               & \HH \\
                                                             & \MM \ar[uu]|{(j,1_{\uX})} \ar[ur]_{(\xi^{\MM},G)} &
}
\end{equation}
consisting of $\X$-monad morphisms, and in order to show that $j$ is a morphism of abstract distribution monads it remains only to prove that the leftmost triangle commutes.  To this end we shall employ \bref{thm:equality_of_mnd_morphs_via_free_algs}.  For each free $\DD$-algebra $(DX,\kappa^\DD_X)$, we compute as follows, using \eqref{eqn:coin_of_algs}, \bref{exa:malg_str_on_dx} 1, and the commutativity of \eqref{eqn:xfunctors_for_abs_distn_mnd} for each of the abstract distribution monads $\MM,\DD$:
$$
\begin{array}{lllll}
\uX^{\Delta^\DD}(DX,\kappa^\DD_X) & = & \uX^{\Delta^\DD}G^{\xi^\DD}(HFX,\fk_{FX}) & = & G^\zeta(HFX,\fk_{FX}) \\
& = & \uX^{{\Delta^\MM}}G^{{\xi^\MM}}(HFX,\fk_{FX}) & = & \uX^{{\Delta^\MM}}\uX^j(DX,\kappa^\DD_X)\;.
\end{array}
$$

For uniqueness, observe that any morphism of abstract distribution monads $j:\MM \rightarrow \DD$ determines a commutative diagram \eqref{eqn:abs_distn_mnd_diag_for_j}, so $\uX^j:\uX^\DD \rightarrow \uX^\MM$ sends each free $\DD$-algebra $(DX,\kappa^\DD_X) = G^{\xi^\DD}(HFX,\fk_{FX})$ \pbref{exa:malg_str_on_dx} to
$\uX^j(DX,\kappa^\DD_X) = \uX^j G^{\xi^\DD}(HFX,\fk_{FX}) = \uX^{\xi^\MM}(HFX,\fk_{FX})\;,$
which is independent of $j$, and by \bref{thm:equality_of_mnd_morphs_via_free_algs}, $j$ is thus uniquely determined.
\end{proof}

\begin{CorSub}
The natural distribution monad $\DD$ is a terminal object in the category of abstract distribution monads \pbref{def:abs_distn_mnd}.
\end{CorSub}

\begin{RemSub}
For the sake of illustration, suppose for the moment that $\X$ is cartesian closed, locally small, and well-pointed.  Let $\MM$ be an abstract distribution monad, $E \in \sL$ a linear space, $f:X \rightarrow E$ a map in $\X$, and $\mu$ an element of $MX$.  Then, by the commutativity of \eqref{eqn:abs_distn_mnd_diag_for_j}, the Dunford-type integral $\dint{} f\;d\mu \in E^{**}$ \pbref{rem:dun_pett_ints_wrt_abs_distn} is equally the Dunford-type integral with respect to the associated natural distribution $j^\MM_X(\mu) \in DX$; i.e. $\dint{} f\;d\mu = \dint{} f\;dj^\MM(\mu)$.  Consequently, for a separated linear space $E$, the existence of the Pettis-type integral $\pint{} f\;d\mu$ is equivalent to that of $\pint{} f\;dj^\MM(\mu)$, and these Pettis-type integrals are identical if they exist.
\end{RemSub}

\section{An algebraic formalism for vector integration with respect to abstract distributions}

Again working with data as in \bref{par:data_for_dcompl_and_pett}, in the present section we shall assume for the sake of illustration that $\X$ is cartesian closed, locally small, and well-pointed.  Let $\MM$ be an abstract distribution monad.  We shall forgo the notational distinction between a linear space $E \in \sL$ and its underlying space $GE \in \X$.

\begin{DefSub}\label{def:scalar_integ_wrt_abs_distn}
Given a space $X \in \X$, the morphism $j^\MM_X:MX \rightarrow DX = \sL([X,R],R)$ associates to each $\mu \in MX$ a natural distribution $j^\MM_X(\mu):[X,R] \rightarrow R$, so that for each map $f:X \rightarrow R$ in $\X$ we can define the \textit{integral of $f$ with respect to $\mu$} as
$$\int f \;d\mu := (j^\MM_X(\mu))(f)\;,$$
which is equally the integral $\int f\;dj^\MM_X(\mu)$ with respect to the natural distribution $j^\MM_X(\mu)$ as defined in \bref{exa:monadic_exa_integ_notn}.  Hence, in the case that $\DD = \MM$ (so that $j^\MM = 1_\DD$), this definition specializes to \eqref{eq:scalar_integ_nat_distn}.
\end{DefSub}

\begin{ParSub}\label{par:rec_r_malg}
Recall that $R$ is the underlying linear space of an $\MM$-algebra, obtained by applying the $\X$-functor $G^{\xi^\MM}:\sL^\HH \rightarrow \uX^\MM$ to the $\HH$-algebra $(R,\fd_R^{-1})$ \pbref{exa:r_malg}.
\end{ParSub}

\begin{PropSub}\label{thm:scalar_integ_via_str_map}
Given $f:X \rightarrow R$ in $\X$ and $\mu \in MX$, the integral $\int f\;d\mu$ may be obtained by applying the composite $MX \xrightarrow{Mf} MR \xrightarrow{a} R$
to $\mu$, where $a$ is the $\MM$-algebra structure carried by $R$ \pbref{par:rec_r_malg}; i.e.
$$\int f\;d\mu = (a \cdot Mf)(\mu)\;.$$
\end{PropSub}
\begin{proof}
Since $j^\MM:\MM \rightarrow \DD$ is a morphism of abstract distribution monads, the diagram of $\X$-functors
$$
\xymatrix{
\uX^\DD \ar[dd]_{\uX^{j^\MM}} & \\
                & \sL^\HH \ar[ul]_{G^{\xi^\DD}} \ar[dl]^{G^{\xi^\MM}}\\
\uX^\MM &                
}
$$
commutes.  Hence the $\MM$-algebra $(R,a) := G^{\xi^\MM}(R,\fd_R^{-1})$ may be obtained from the $\DD$-algebra $(R,b) := G^{\xi^\DD}(R,\fd_R^{-1})$ by applying $\uX^{j^\MM}$, so that the structure morphism $a$ is the composite
$$MR \xrightarrow{j^\MM_R} DR \xrightarrow{b} R\;.$$
Further, by the naturality of $j^\MM$, the diagram
$$
\xymatrix{
MX \ar[r]^{Mf} \ar[d]_{j^\MM_X} & MR \ar[dr]^a \ar[d]_{j^\MM_R} & \\
DX \ar[r]_{Df} & DR \ar[r]_b & R
}
$$
commutes.  Hence, using the characterization of $b$ given in \bref{exa:r_dalg} and the fact that $Df = \sL([f,R],R):\sL([X,R],R) \rightarrow \sL([R,R],R)$, we compute that
$$
\begin{array}{lllllll}
(a \cdot Mf)(\mu) & = & (b \cdot Df \cdot j^\MM_X)(\mu) & = & ((Df \cdot j^\MM_X)(\mu))(1_R) & & \\
                  & = & (j^\MM_X(\mu))(1_R \cdot f)   & = & (j^\MM_X(\mu))(f) & = &\int f \;d\mu\;,
\end{array}
$$
where $1_R:R \rightarrow R$ is the identity morphism in $\X$.
\end{proof}

Proposition \bref{thm:scalar_integ_via_str_map} motivates the following definition, which generalizes \bref{def:scalar_integ_wrt_abs_distn} and hence also \eqref{eq:scalar_integ_nat_distn}.

\begin{DefSub}\label{def:vint_in_malgs}
Given an $\MM$-algebra $(Z,a)$, a map $f:X \rightarrow Z$, and an element $\mu$ of $MX$, we define \textit{the integral of $f$ with respect to $\mu$} as the element of $Z$ obtained by applying the composite $MX \xrightarrow{Mf} MZ \xrightarrow{a} Z$ to $\mu$, so that
$$\int_x f(x)\;d\mu := \int f\;d\mu := (a \cdot Mf)(\mu)\;.$$
\end{DefSub}

\begin{RemSub}
We shall see in \bref{thm:pett_spaces_are_malgs} that any linear space $E$ that is \textit{$\MM$-Pettis} (i.e. has enough Pettis-type integrals, \bref{def:pett_obs}, \bref{def:charn_pett_malgs_via_pett_integs}) carries the structure of an $\MM$-algebra, in such a way that the integral defined in \bref{def:vint_in_malgs} generalizes the Pettis-type integral \pbref{rem:integ_in_pett_malg_is_pett_integ}.
\end{RemSub}

\begin{RemSub}\label{rem:int_notn_appl_arb_mnd}
Although we have assumed that $\MM$ is an abstract distribution monad, notice that the integral notation in Definition \bref{def:vint_in_malgs} may be applied with respect to any monad $\MM$ on an arbitrary concrete category $\X$, and indeed this suffices also for the remainder of this section.
\end{RemSub}

\begin{ExaSub}\label{exa:real_integral_conv_sp}
Consider our example with $\X = \Conv$ the category of convergence spaces \pbref{exa:nat_dist_conv} and $R = \RR$.  $\RR$ is a $\DD$-algebra \pbref{par:rec_r_malg}, and for a continuous map $f:X \rightarrow \RR$ on a locally compact Hausdorff topological space $X$, the resulting integral \pbref{def:vint_in_malgs} with respect to $\mu \in DX$ coincides with the usual integral with respect to the compactly supported Radon measure $\mu$.
\end{ExaSub}

\begin{RemSub}
Observe that, in the notation of \bref{def:vint_in_malgs}, the structure map $a:MZ \rightarrow Z$ of an $\MM$-algebra $(Z,a)$ is given by $a(\mu) = \int 1_Z \;d\mu = \int_z z \;d\mu$.
\end{RemSub}

\begin{PropSub}\label{thm:compn_of_integrand}
Given an $\MM$-algebra $(Z,a)$, maps $f:X \rightarrow Y$ and $g:Y \rightarrow Z$, and an element $\mu$ of $MX$, 
$$\int g \cdot f \;d\mu = \int g \;d(Mf)(\mu)\;.$$
\end{PropSub}
\begin{proof}
$$\int g \cdot f \;d\mu = (a \cdot M(g \cdot f))(\mu) = (a \cdot Mg \cdot Mf)(\mu) = \int g \;d(Mf)(\mu)\;.$$
\end{proof}

\begin{PropSub}\label{thm:homom_via_integ}
Given $\MM$-algebras $(Z,a),(W,b)$ and a map $h:Z \rightarrow W$, $h$ is an $\MM$-homomorphism if and only if
$$h(\int_x f(x)\;d\mu) = \int_x h(f(x)) \;d\mu$$
for every map $f:X \rightarrow Z$ and every element $\mu$ of $MX$.
\end{PropSub}
\begin{proof}
By the definition of the integral, the given condition states that
$$h \cdot a \cdot Mf = b \cdot M(h \cdot f)$$
for every map $f:X \rightarrow Z$.  But $h$ is an $\MM$-homomorphism if and only if this condition holds for $f = 1_Z$, and the result follows.
\end{proof}

\begin{ExaSub}
Consider again our example with $\X = \Conv$ the category of convergence spaces \pbref{exa:nat_dist_conv} and $R = \RR$.  Since $\RR$ is a $\DD$-algebra \pbref{exa:real_integral_conv_sp}, we can form the $n$-fold product $\RR^n$ in $\X^\DD$ for any set $n$ (as the forgetful functor $\X^\DD \rightarrow \X$ creates limits).  For any continuous map $f:X \rightarrow \RR^n$ and any $\mu \in DX$ we can form the integral $\int f \;d\mu \in \RR^n$ via Definition \bref{def:vint_in_malgs}, and since the projections $\pi_i:\RR^n \rightarrow \RR$ $(i \in n)$ are $\DD$-homomorphisms, the $i$-th coordinate of the integral is
$$\pi_i(\int_x f(x)\;d\mu) = \int_x \pi_i(f(x))\;d\mu\;,$$
so that in the case that $X$ is a locally compact Hausdorff topological space, the integral $\int f \;d\mu$ is simply the familiar coordinatewise integral of $f$ with respect to the compactly supported Radon measure $\mu$.
\end{ExaSub}

\section{On the linearity of vector-valued integration}

Let $\X = (\X,\boxtimes,I)$ be a symmetric monoidal closed category, let $\LL = (L,\eta,\mu)$, $\MM = (M,\delta,\kappa)$ be $\X$-monads on $\uX$.

\begin{DefSub}
Let $\Delta:\LL \rightarrow \MM$ be a morphism of $\X$-monads.  
\begin{enumerate}
\item For objects $X,Y \in \X$, let $\otimes^\Delta_{X Y}$ denote the composite
$$MX \boxtimes LY \xrightarrow{t''^L_{MX Y}} L(MX \boxtimes Y) \xrightarrow{Lt'^M_{X Y}} LM(X \boxtimes Y) \xrightarrow{\Delta M} MM(X \boxtimes Y) \xrightarrow{\kappa} M(X \boxtimes Y)\;,$$
and let $\widetilde{\otimes}^\Delta_{X Y}$ denote
$$MX \boxtimes LY \xrightarrow{t'^M_{X LY}} M(X \boxtimes LY) \xrightarrow{Mt''^L_{X Y}} ML(X \boxtimes Y) \xrightarrow{M\Delta} MM(X \boxtimes Y) \xrightarrow{\kappa} M(X \boxtimes Y)\;.$$
\item We say that the morphism $\Delta$ is \textit{relatively commutative} if $\otimes^\Delta_{X Y} = \widetilde{\otimes}^\Delta_{X Y}$ for all $X,Y \in \X$.
\end{enumerate}
\end{DefSub}

\begin{ExaSub}
An $\X$-monad $\TT$ on $\uX$ is commutative if and only if the identity morphism on $\TT$ is relatively commutative.
\end{ExaSub}

\begin{RemSub}\label{rem:nat_rel_tensors}
Observe that the morphisms $\otimes^\Delta_{X Y}, \widetilde{\otimes}^\Delta_{X Y}$ are $\X$-natural in $X,Y \in \uX$, since $t''^L$, $t'^M$, $\Delta$, and $\kappa$ are $\X$-natural transformations \pbref{rem:enr_nat_of_tens_str}.
\end{RemSub}

\begin{ParSub}\label{par:data_for_lin_integ}
For the remainder of this section, we assume that $\LL$ is commutative and has tensor products of algebras \pbref{def:tens_prod_algs}, and further that the Eilenberg-Moore $\X$-category for $\LL$ exists \pbref{par:em_adj}.
\end{ParSub}

\begin{ExaSub}
Take $(\MM,\Delta,\xi)$ to be an abstract distribution monad \pbref{def:abs_distn_mnd}.
\end{ExaSub}

\begin{ParSub}\label{par:cotensors_in_lalg}
Recall that for an $\LL$-algebra $E := (Z,b)$ and an object $X \in \X$, there is an associated cotensor $[X,E]$ in the Eilenberg-Moore $\X$-category $\uX^\LL$, whose underlying $\X$-object is simply the internal hom $\uX(X,Z)$ in $\X$ \pbref{rem:standard_cotensors_in_talg}.  In particular, the structure morphism for $[X,E]$ is the composite
$$L\uX(X,Z) \xrightarrow{\varsigma} \uX(X,LZ) \xrightarrow{\uX(X,b)} \uX(X,Z)$$
where $\varsigma$ is the \textit{cotensorial strength} \cite{Kock:Dist} for $\LL$, which is the transpose of the composite
$$X \boxtimes L\uX(X,Z) \xrightarrow{t''^L} L(X \boxtimes \uX(X,Z)) \xrightarrow{L\Ev} LZ\;.$$
For the remainder of the section, we shall implicitly endow $\uX(X,Z)$ with this $\LL$-algebra structure when $Z$ is the carrier of a given $\LL$-algebra.
\end{ParSub}

\begin{ParSub}\label{par:lalg_from_malg}
Via the $\X$-functor $\uX^\Delta:\uX^\MM \rightarrow \uX^\LL$ \pbref{prop:enr_func_induced_by_mnd_mor}, each $\MM$-algebra $(Z,a)$ determines an $\LL$-algebra with structure morphism $a \cdot \Delta_Z$, and in the present section, we will implicitly endow the carrier $Z$ with this $\LL$-algebra structure.
\end{ParSub}

\begin{DefSub}\label{def:extenders}
Given an $\MM$-algebra $(Z,a)$ and an object $X \in \X$, we define the \textit{extender} $\Omega^a_X$ as the composite
$$\Omega^a_X := \left(\uX(X,Z) \xrightarrow{M_{X Z}} \uX(MX,MZ) \xrightarrow{\uX(1,a)} \uX(MX,Z)\right)\;.$$
In particular, then, we have morphisms
$$\Omega^{\kappa_Y}_X:\uX(X,MY) \rightarrow \uX(MX,MY)$$
for all $X,Y \in \X$.
\end{DefSub}

\begin{ParSub}
In the situation of \bref{def:extenders}, the objects $\uX(X,Z),\uX(MX,Z)$ carry the structure of $\LL$-algebras, via \bref{par:lalg_from_malg} and \bref{par:cotensors_in_lalg}, so we can ask whether the extender $\Omega^a_X$ is an $\LL$-homomorphism.  We show below that $\Delta$ is relatively commutative if and only if each extender $\Omega^a_X$ is an $\LL$-homomorphism \pbref{thm:rel_comm_iff_extenders_lhomom}.
\end{ParSub}

\begin{LemSub}\label{thm:split_epi_final}
Let $\TT$ be a monad on a category $\A$, and let $e:(A,a) \rightarrow (B,b)$ be a $\TT$-homomorphism whose underlying morphism $e:A \rightarrow B$ is split epi in $\A$.  Then $e$ is \textit{final} with respect to the forgetful functor $\A^\TT \rightarrow \A$; i.e. if $f:B \rightarrow C$ in $\A$ is such that $f \cdot e:(A,a) \rightarrow (C,c)$ is a $\TT$-homomorphism, then $f:(B,b) \rightarrow (C,c)$ is a $\TT$-homomorphism.
\end{LemSub}

\begin{LemSub}\label{thm:tens_str_and_transp}
Let $h:\uX(X,Y) \rightarrow \uX(X',Y')$ be a morphism in $\X$ with transpose \hbox{$\overline{h}:X' \boxtimes \uX(X,Y) \rightarrow Y'$}.  Then the diagram
$$
\xymatrix{
X' \boxtimes L\uX(X,Y) \ar[r]^{1 \boxtimes Lh} \ar[d]_{t''^L} & \X' \boxtimes L\uX(X',Y') \ar[r]^{1 \boxtimes \varsigma} & X' \boxtimes \uX(X',LY') \ar[d]^\Ev \\
L(X' \boxtimes \uX(X,Y)) \ar[rr]_{L\overline{h}} & & LY'
}
$$
commutes.
\end{LemSub}
\begin{proof}
In the following diagram
$$
\xymatrix{
X' \boxtimes L\uX(X,Y) \ar[d]_{t''^L} \ar[r]^{1 \boxtimes Lh} & X' \boxtimes L\uX(X',Y') \ar[d]^{t''^L} \ar[r]^{1 \boxtimes \varsigma} & X' \boxtimes \uX(X',LY') \ar[d]^\Ev \\
L(X' \boxtimes \uX(X,Y)) \ar[r]_{L(1 \boxtimes h)} & L(X' \boxtimes \uX(X',Y')) \ar[r]_{L\Ev} & LY
}
$$
the leftmost square commutes by the naturality of $t''$ \pbref{rem:enr_nat_of_tensor_maps}, and the rightmost commutes by the definition of $\varsigma$.  But the composite of the bottom row is $L\overline{h}$.
\end{proof}

\begin{LemSub}\label{thm:lemma_on_tens_str}
For all $X,Z \in \X$, the diagram
$$
\xymatrix{
MX \boxtimes \uX(X,Z) \ar[r]^{t'^M} \ar[d]_{1 \boxtimes M} & M(X \boxtimes \uX(X,Z)) \ar[d]^{M\Ev} \\
MX \boxtimes \uX(MX,MZ) \ar[r]_\Ev & MZ
}
$$
commutes.
\end{LemSub}
\begin{proof}
By the definition of $t'^M$, the transpose of the upper-right composite is the composite along the upper-right periphery of the following commutative diagram.
$$
\xymatrix{
\uX(X,Z) \ar[r] \ar@{=}[dr] & \uX(X,X \boxtimes \uX(X,Z)) \ar[d]^{\uX(1,\Ev)} \ar[rr]^M & & \uX(MX,M(X \boxtimes \uX(X,Z))) \ar[d]^{\uX(1,M\Ev)} \\
& \uX(X,Z) \ar[rr]_M & & \uX(MX,MZ)
}
$$
\end{proof}

\begin{LemSub}\label{thm:transp_of_extender}
Given $X,Y \in \X$, the extender $\Omega^{\kappa_Y}_X:\uX(X,MY) \rightarrow \uX(MX,MY)$ has transpose
$$MX \boxtimes \uX(X,MY) \xrightarrow{t'^M} M(X \boxtimes \uX(X,MY)) \xrightarrow{M\Ev} MMY \xrightarrow{\kappa_Y} MY\;.$$
\end{LemSub}
\begin{proof}
By \bref{thm:lemma_on_tens_str}, we have a commutative diagram
$$
\xymatrix{
MX \boxtimes \uX(X,MY) \ar[r]^{t'^M} \ar[d]_{1 \boxtimes M} & M(X \boxtimes \uX(X,MY)) \ar[d]^{M\Ev} & \\
MX \boxtimes \uX(MX,MMY) \ar[r]_\Ev & MMY \ar[r]_{\kappa_Y} & MY
}
$$
in which the lower-left composite is the transpose of $\Omega^{\kappa_Y}_X$.
\end{proof}

\begin{ThmSub}\label{thm:rel_comm_iff_extenders_lhomom}
The following are equivalent:
\begin{enumerate}
\item $\Delta:\LL \rightarrow \MM$ is relatively commutative;
\item Each extender $\Omega^{\kappa_Y}_X:\uX(X,MY) \rightarrow \uX(MX,MY)$ is an $\LL$-homomorphism, where $X,Y \in \X$;
\item Each extender $\Omega^a_X:\uX(X,Z) \rightarrow \uX(MX,Z)$ is an $\LL$-homomorphism, where $(Z,a)$ is an $\MM$-algebra and $X \in \X$.
\end{enumerate}
\end{ThmSub}
\begin{proof}
Let us first prove the equivalence of 1 and 2.  By definition, the extender $\Omega = \Omega^{\kappa_Y}_X$ is an $\LL$-homomorphism if and only if the following diagram commutes.
$$
\xymatrix{
L\uX(X,MY) \ar[d]_\varsigma \ar[r]^{L\Omega} & L\uX(MX,MY) \ar[d]^\varsigma \\
\uX(X,LMY) \ar[d]_{\uX(1,\Delta MY)} & \uX(MX,LMY) \ar[d]^{\uX(1,\Delta MY)} \\
\uX(X,MMY) \ar[d]_{\uX(1,\kappa_Y)}  & \uX(MX,MMY) \ar[d]_{\uX(1,\kappa_Y)} \\
\uX(X,MY) \ar[r]_{\Omega} & \uX(MX,MY)
}
$$
Recalling that $\Omega$ is the composite
$$\uX(X,MY) \xrightarrow{M} \uX(MX,MMY) \xrightarrow{\uX(1,\kappa_Y)} \uX(MX,MY)\;,$$
we find that the transposes of the two composites in the above diagram form the periphery of the following diagram.
$$
\xymatrix@C=2.8ex{
MX \boxtimes L\uX(X,MY) \ar[d]_{1 \boxtimes \varsigma} \ar[rrr]^{1 \boxtimes L\Omega} & & & MX \boxtimes L\uX(MX,MY) \ar[d]^{1 \boxtimes \varsigma} \\
MX \boxtimes \uX(X,LMY) \ar[d]_{1 \boxtimes \uX(1,\Delta MY)} \ar[r]^(.45){1 \boxtimes M} & MX \boxtimes \uX(MX,MLMY) \ar[r]^(.7)\Ev & MLMY \ar[d]_{M\Delta MY} & MX \boxtimes \uX(MX,LMY) \ar[d]^\Ev \\
MX \boxtimes \uX(X,MMY) \ar[dd]_{1 \boxtimes \uX(1,\kappa_Y)} & & MMMY \ar[dd]_{M\kappa_Y} & LMY \ar[d]^{\Delta MY} \\
& & & MMY \ar[d]^{\kappa_Y} \\
MX \boxtimes \uX(X,MY) \ar[r]_(.45){1 \boxtimes M} & MX \boxtimes \uX(MX,MMY) \ar[r]_(.7)\Ev & MMY \ar[r]_{\kappa_Y} & MY
}
$$
But the rectangle at the bottom-left commutes, so $\Omega$ is an $\LL$-homomorphism if and only if the upper-right cell commutes.  We now show that upper composite $\theta$ on the periphery of the latter cell is
\begin{equation}\label{eqn:upper_comp_via_tensor}MX \boxtimes L\uX(X,MY) \xrightarrow{\otimes^\Delta} M(X \boxtimes \uX(X,MY)) \xrightarrow{M\Ev} MMY \xrightarrow{\kappa_Y} MY\end{equation}
and further that the lower composite $\widetilde{\theta}$ is
\begin{equation}\label{eqn:lower_comp_via_tensor}MX \boxtimes L\uX(X,MY) \xrightarrow{\widetilde{\otimes}^\Delta} M(X \boxtimes \uX(X,MY)) \xrightarrow{M\Ev} MMY \xrightarrow{\kappa_Y} MY\end{equation}

To this end, first observe that $\theta$ appears as the composite of the upper-right path on the periphery of the following diagram.
$$
\xymatrix{
MX \boxtimes L\uX(X,MY) \ar[d]_{t''^L} \ar[rr]^{1 \boxtimes L\Omega} & & MX \boxtimes L\uX(MX,MY) \ar[d]^{1 \boxtimes \varsigma} \\
L(MX \boxtimes \uX(X,MY)) \ar[d]_{Lt'^M} & & MX \boxtimes \uX(MX,LMY) \ar[d]^\Ev \\
LM(X \boxtimes \uX(X,MY)) \ar[d]_{\Delta M} \ar[r]^(.65){LM\Ev} & LMMY \ar[r]^{L\kappa_Y} & LMY \ar[d]^{\Delta MY} \\
MM(X \boxtimes \uX(X,MY)) \ar[d]_{\kappa} \ar[r]^(.65){MM\Ev} & MMMY \ar[d]^{\kappa MY} \ar[r]^{M\kappa_Y} & MMY \ar[d]^{\kappa_Y} \\
M(X \boxtimes \uX(X,MY)) \ar[r]_(.6){M\Ev} & MMY \ar[r]_{\kappa_Y} & MY
}
$$
whereas the composite \eqref{eqn:upper_comp_via_tensor} appears as the lower-right path around the periphery.  The lower half of the diagram clearly commutes.  Moreover, we deduce by \bref{thm:tens_str_and_transp} and \bref{thm:transp_of_extender} that the upper cell also commutes.  Hence $\theta$ is the composite \eqref{eqn:upper_comp_via_tensor}, as needed.

Next observe that $\widetilde{\theta}$ appears as the upper-right path on the periphery of the following diagram
$$
\xymatrix{
MX \boxtimes L\uX(X,MY) \ar[r]^{1 \boxtimes \varsigma} \ar[d]_{t'^M} & MX \boxtimes \uX(X,LMY) \ar[d]_{t'^M} \ar[r]^{1 \boxtimes M} & MX \boxtimes \uX(MX,MLMY) \ar[d]^\Ev \\
M(X \boxtimes L\uX(X,MY)) \ar[d]_{Mt''^L} \ar[r]^{M(1 \boxtimes \varsigma)} & M(X \boxtimes \uX(X,LMY)) \ar[r]^{M\Ev} & MLMY \ar[d]^{M\Delta MY} \\
ML(X \boxtimes \uX(X,MY)) \ar[d]_{M \Delta} \ar[urr]|{ML\Ev} & & MMMY \ar[ddl]|{\kappa MY} \ar[d]^{M\kappa_Y} \\
MM(X \boxtimes \uX(X,MY)) \ar[d]_{\kappa} \ar[urr]|{MM\Ev} & & MMY \ar[d]^{\kappa_Y} \\
M(X \boxtimes \uX(X,MY)) \ar[r]_{M\Ev} & MMY \ar[r]_{\kappa_Y} & MY
}
$$
whose lower-left composite is the composite \eqref{eqn:lower_comp_via_tensor}.  The top-left square commutes by the naturality of $t'$ \pbref{rem:enr_nat_of_tensor_maps}, and the top-right square commutes by \bref{thm:lemma_on_tens_str}.  The cell immediately below the top row commutes by the definition of $\varsigma$.  The remaining cells in the diagram clearly commute.  Hence $\widetilde{\theta}$ is the composite \eqref{eqn:lower_comp_via_tensor}, as needed.

We have thus shown that for all $X,Y \in \X$, the extender $\Omega^{\kappa_Y}_X:\uX(X,MY) \rightarrow \linebreak[4] \uX(MX,MY)$ is an $\LL$-homomorphism if and only if the composites
\eqref{eqn:lower_comp_via_tensor}, \eqref{eqn:upper_comp_via_tensor} are equal.  Hence, if $\Delta$ is relatively commutative, then $\otimes^\Delta = \widetilde{\otimes}^\Delta$, so $\Omega^{\kappa_Y}_X$ is an $\LL$-homomorphism.  Conversely, suppose that $\Omega^{\kappa_Y}_X$ is an $\LL$-homomorphism for all $X,Y \in \X$.  Letting $X,Z \in \X$, we shall show that $\otimes^\Delta_{X Z} = \widetilde{\otimes}^\Delta_{X Z}$.  Setting $Y := X \boxtimes Z$, we can express $\otimes^\Delta_{X Z}$ in terms of the composite \eqref{eqn:upper_comp_via_tensor} via the diagram
$$
\xymatrix{
MX \boxtimes LZ \ar[d]_{\otimes^\Delta} \ar[r]^(.4){1 \boxtimes L\gamma} & MX \boxtimes L\uX(X,X \boxtimes Z) \ar[rr]^{1 \boxtimes L\uX(1,\delta)} & & MX \boxtimes L\uX(X,M(X \boxtimes Z)) \ar[d]^{\otimes^\Delta} \\
M(X \boxtimes Z) \ar@{=}[dr] \ar[r]^(.4){M(1 \boxtimes \gamma)} & M(X \boxtimes \uX(X,X \boxtimes Z)) \ar[d]^{M\Ev} \ar[rr]^{M(1 \boxtimes \uX(1,\delta))} & & *!/l5ex/+{M(X \boxtimes \uX(X,M(X \boxtimes Z)))} \ar[d]^{M\Ev} \\
& M(X \boxtimes Z) \ar@{=}[drr] \ar[rr]^{M\delta} & & MM(X \boxtimes Z) \ar[d]^{\kappa} \\
& & & M(X \boxtimes Z)
}
$$
in which $\gamma:Z \rightarrow \uX(X,X \boxtimes Z)$ is the canonical morphism.  Indeed, the rectangle at the top commutes by the naturality of $\otimes^\Delta$ \pbref{rem:nat_rel_tensors}, and the remaining cells clearly commute.  But we also obtain a similar commutative diagram with $\widetilde{\otimes}^\Delta$ in place of $\otimes^\Delta$.  Hence since the composites \eqref{eqn:lower_comp_via_tensor}, \eqref{eqn:upper_comp_via_tensor} are equal, we deduce that $\otimes^\Delta_{X Z} = \widetilde{\otimes}^\Delta_{X Z}$ as needed, showing that $\Delta$ is relatively commutative.

Having thus established the equivalence of 1 and 2, it now suffices to assume that 2 holds and show that 3 follows.  Letting $(Z,a)$ be an arbitrary $\MM$-algebra, we may apply the $\X$-functor $\uX^\Delta:\uX^\MM \rightarrow \uX^\LL$ and thus deduce that $a$ is an $\LL$-homomorphism.  Hence $\uX(X,a):\uX(X,MZ) \rightarrow \uX(X,Z)$ is an $\LL$-homomorphism.  Further, $\uX(X,a)$ is split epi in $\X$, with section $\uX(X,\delta_Z)$, so by \bref{thm:split_epi_final}, it suffices to show that the composite 
$$\uX(X,MZ) \xrightarrow{\uX(X,a)} \uX(X,Z) \xrightarrow{\Omega} \uX(MX,Z)$$
is an $\LL$-homomorphism.  This composite appears as the bottom-left composite on the periphery of the following commutative diagram.
$$
\xymatrix{
\uX(X,MZ) \ar[d]_{\uX(1,a)} \ar[r]^(.4)M & *!/r2ex/+{\uX(MX,MMZ)} \ar[d]|{\uX(1,Ma)} \ar[r]^{\uX(1,\kappa_Z)} & \uX(MX,MZ) \ar[d]^{\uX(1,a)} \\
\uX(X,Z) \ar[r]_(.4)M & *!/r2ex/+{\uX(MX,MZ)} \ar[r]_{\uX(1,a)} & {\uX(MX,Z)}
}
$$
The composite of the top row is the extender $\Omega^{\kappa_Z}_X$, which by assumption is an $\LL$-homomorphism, so since the rightmost vertical morphism is an $\LL$-homomorphism, the common composite around the periphery of the diagram is an $\LL$-homomorphism.
\end{proof}

\begin{DefSub}
We say that an abstract distribution monad $(\MM,\Delta,\xi)$ is \textit{linear} if the morphism $\Delta:\LL \rightarrow \MM$ is relatively commutative.
\end{DefSub}

\begin{RemSub}
By \bref{thm:rel_comm_iff_extenders_lhomom}, an abstract distribution monad $(\MM,\Delta,\xi)$ is linear if and only if each extender $\Omega^a_X:\uX(X,Z) \rightarrow \uX(MX,Z)$ is a linear map, where $(Z,a)$ is an $\MM$-algebra and $X \in \X$.
\end{RemSub}

\begin{CorSub}
Let $\MM$ be an abstract distribution monad, and assume that $\X$ is a locally small well-pointed cartesian closed category.  Then the following are equivalent:
\begin{enumerate}
\item $\MM$ is linear.
\item For each $\MM$-algebra $(Z,a)$, each $X \in \X$, and each $\mu \in MX$, the associated map
$$\uX(X,Z) \rightarrow Z\;,\;\;\;\;f \mapsto \int f \;d\mu$$
is linear (i.e. is a morphism in $\sL = \uX^\LL$).
\end{enumerate}
\end{CorSub}
\begin{proof}
Letting $E$ be the underlying linear space of $(Z,a)$, the extender $\Omega:\uX(X,Z) \rightarrow \uX(MX,Z)$ is an $\LL$-homomorphism $[X,E] \rightarrow [MX,E]$ if and only if its transpose $\overline{\Omega}:MX \rightarrow \uX(\uX(X,Z),Z)$ factors through $G:\uX^\LL([X,E],E) \hookrightarrow \uX(\uX(X,Z),Z)$.  This is equivalent to the statement that for each $\mu \in MX$, the map $\overline{\Omega}(\mu):\uX(X,Z) \rightarrow Z$ is an $\LL$-homomorphism.  But the latter map is given by $f \mapsto (\Omega(f))(\mu) = (a \cdot Mf)(\mu) = \int f \;d\mu$.
\end{proof}

\begin{LemSub}\label{thm:extender_for_adj}
If $\MM$ is induced by an $\X$-adjunction $P \nsststile{\sigma}{\delta} Q:\B \rightarrow \uX$, then for all $X,Y \in \X$, the extender $\Omega^{\kappa_Y}_X:\uX(X,MY) \rightarrow \uX(MX,MY)$ is the composite
\begin{equation}\label{eqn:extender_for_adj}\uX(X,QPY) \xrightarrow{\sim} \B(PX,PY) \xrightarrow{Q_{PX PY}} \uX(QPX,QPY)\;.\end{equation}
\end{LemSub}
\begin{proof}
We have a commutative diagram
$$
\xymatrix{
\uX(X,QPY) \ar[d]_P \ar[dr]^M & \\
\B(PX,PQPY) \ar[d]_{\B(1,\sigma PY)} \ar[r]_Q & \uX(QPX,QPQPY) \ar[d]|{\uX(1,Q\sigma PY)\:=\:\uX(1,\kappa_Y)} \\
\B(PX,PY) \ar[r]_Q & \uX(QPX,QPY)
}
$$
in which the composite of the lower-left path around the periphery is the composite \eqref{eqn:extender_for_adj}.  But the upper-right path around the periphery is the extender $\Omega^{\kappa_Y}_X$.
\end{proof}

\begin{ThmSub}\label{thm:nat_distn_mnd_linear}
Given data as in \bref{par:data_for_dcompl_and_pett}, the natural distribution monad $\DD$ is a linear abstract distribution monad.
\end{ThmSub}
\begin{proof}
Recall that $\DD$ is induced by the `hom-cotensor' $\X$-adjunction $[-,R] \dashv \sL(-,R):\sL^\op \rightarrow \uX$, so by \bref{thm:extender_for_adj}, the extender $\Omega:\uX(X,DY) \rightarrow \uX(DX,DY)$ $(X,Y \in \X)$ is the composite
$$\uX(X,\sL([Y,R],R)) \xrightarrow{\sim} \sL([Y,R],[X,R]) \xrightarrow{\sL(-,R)} \uX(\sL([X,R],R),\sL([Y,R],R))$$
in which the leftmost morphism is an instance of the given hom-cotensor adjointness.  But the latter morphism underlies an isomorphism $[X,\uL([Y,R],R)] \xrightarrow{\sim} \uL([Y,R],[X,R])$ in $\sL = \uX^\LL$.  Further the rightmost morphism underlies the composite
$$\uL([Y,R],[X,R]) \xrightarrow{\uL(-,R)} \uL(\uL([X,R],R),\uL([Y,R],R)) \rightarrowtail [\sL([X,R],R),\uL([Y,R],R)]$$
in $\sL$.  Whereas by definition $DX = \sL([X,R],R)$, we have shown in \bref{exa:und_linear_spaces_of_free_d_and_td_algs} that the underlying linear space of the free $\DD$-algebra $(DY,\kappa_Y)$ is the internal hom $\uL([Y,R],R)$ in $\sL$; hence $\Omega$ is an $\LL$-homomorphism, so the result now follows from \bref{thm:rel_comm_iff_extenders_lhomom}.
\end{proof}

\begin{PropSub}\label{thm:sub_adm_of_linear_adm_is_linear}
Let $\lambda:\MM \rightarrow \NN$ be a morphism of abstract distribution monads \pbref{def:abs_distn_mnd}.  Suppose that $\NN$ is linear and the components of $\lambda$ are monomorphisms.  Then $\MM$ is linear.
\end{PropSub}
\begin{proof}
We have a diagram
$$
\xymatrix{
MX \boxtimes LY \ar[d]_{\lambda \boxtimes 1} \ar[r]^{t''^L} & L(MX \boxtimes Y) \ar[d]|{L(\lambda \boxtimes 1)} \ar[r]^{Lt'^M} & LM(X \boxtimes Y) \ar[d]|{L\lambda} \ar[r]^{\Delta^\MM M} & MM(X \boxtimes Y) \ar[d]|{\lambda \circ \lambda} \ar[r]^{\kappa^\MM} & M(X \boxtimes Y) \ar[d]^\lambda \\
NX \boxtimes LY \ar[r]_{t''^L} & L(NX \boxtimes Y) \ar[r]_{Lt'^\NN} & LN(X \boxtimes Y) \ar[r]_{\Delta^\NN N} & NN(X \boxtimes Y) \ar[r]_{\kappa^\NN} & N(X \boxtimes Y)
}
$$
in which the composite of the upper row is $\otimes^{\Delta^\MM}_{X Y}$ and that of the lower row is $\otimes^{\Delta^\NN}_{X Y}$.  The leftmost square commutes by the naturality of $t''^L$ \pbref{rem:enr_nat_of_tens_str}, the next-to-leftmost square commutes by \bref{prop:xnat_transf_commutes_w_tens_str}, and the rightmost square commutes since $\lambda$ is a monad morphism.  The remaining square is the periphery of the following diagram
$$
\xymatrix {
LM(X \boxtimes Y) \ar[dd]_{L\lambda} \ar[r]^{\Delta^\MM M} & MM(X \boxtimes Y) \ar[d]^{M\lambda} \\
& MN(X \boxtimes Y) \ar[d]^{\lambda N} \\
LN(X \boxtimes Y) \ar[ur]|{\Delta^\MM N} \ar[r]_{\Delta^\NN N} & NN(X \boxtimes Y)
}
$$
in which the upper cell commutes by the naturality of $\Delta^\MM$ and the lower cell commutes since $\lambda$ is a morphism of abstract distribution monads.  Hence the previous diagram commutes.  We can form an analogous commutative diagram whose upper and lower rows are $\widetilde{\otimes}^{\Delta^\MM}_{X Y}$ and $\widetilde{\otimes}^{\Delta^\NN}_{X Y}$, respectively.  Hence, since $\otimes^{\Delta^\NN}_{X Y} = \widetilde{\otimes}^{\Delta^\NN}_{X Y}$ and $\lambda_{X \boxtimes Y}$ is mono, it follows that $\otimes^{\Delta^\MM}_{X Y} = \widetilde{\otimes}^{\Delta^\MM}_{X Y}$.
\end{proof}

\begin{CorSub}\label{thm:j_mono_impl_mm_linear}
Let $\MM$ be an abstract distribution monad, and suppose that the components of the unique morphism of abstract distribution monads $j^\MM:\MM \rightarrow \DD$ are mono.  Then $\MM$ is linear.
\end{CorSub}
\begin{proof}
This follows from \bref{thm:nat_distn_mnd_linear} and \bref{thm:sub_adm_of_linear_adm_is_linear}.
\end{proof}

\begin{CorSub}
Given data as in \bref{par:data_for_dcompl_and_pett}, the accessible distribution monad $\tDD$ is a linear abstract distribution monad.
\end{CorSub}
\begin{proof}
By \bref{rem:acc_distns}, \bref{thm:can_mnd_mor_distns}, the components of $j^{\tDD}:\tDD \rightarrow \DD$ are the underlying morphisms in $\X$ of monomorphisms in $\sL$ and hence are monomorphisms, so this follows from \pbref{thm:j_mono_impl_mm_linear}.
\end{proof}

\begin{CorSub}\label{thm:lin_of_vint_wrt_nat_acc_distn_wpt}
Let data as in \bref{par:data_for_dcompl_and_pett} be given, and assume that $\X$ is a locally small well-pointed cartesian closed category.  For each $\DD$-algebra (resp. $\tDD$-algebra) $(Z,a)$, each $X \in \X$, and each $\mu \in DX$ (resp. $\mu \in \tD X$), the associated map
$$\uX(X,Z) \rightarrow Z\;,\;\;\;\;f \mapsto \int f \;d\mu$$
is linear (i.e. is a morphism in $\sL = \uX^\LL$).
\end{CorSub}

\section{A Fubini theorem for vector-valued integration}

Let $\X = (\X,\boxtimes,I)$ be a symmetric monoidal closed category, and let $\MM = (M,\delta,\kappa)$ be an $\X$-monad on $\uX$.

\begin{DefSub}\label{def:f_a}
Given an $\MM$-algebra $(Z,a)$ and a morphism $f:X \boxtimes Y \rightarrow Z$ in $\X$, let us define
\begin{enumerate}
\item[] $f_a := \left(X \boxtimes MY \xrightarrow{t''_{X Y}} M(X \boxtimes Y) \xrightarrow{Mf} MZ \xrightarrow{a} Z\right)$,
\item[] $f^a := \left(MX \boxtimes Y \xrightarrow{t'_{X Y}} M(X \boxtimes Y) \xrightarrow{Mf} MZ \xrightarrow{a} Z\right)$,
\end{enumerate}
where $t'$, $t''$ are as defined in \bref{def:tens_str}.
\end{DefSub}

\begin{RemSub}
For the sake of illustration, suppose that $\X$ is a locally small well-pointed cartesian closed category.  Then, given data as in \bref{def:f_a}, the morphisms $f_a$, $f^a$ are given by
\begin{enumerate}
\item $f_a(x,\nu) = \int_y f(x,y) \;d\nu\;,\;\;\;(x,\nu) \in X \times MY$,
\item $f^a(\mu,y) = \int_x f(x,y) \;d\mu\;,\;\;\;(\mu,y) \in MX \times Y$
\end{enumerate}
in the notation of \bref{rem:int_notn_appl_arb_mnd}, \bref{def:vint_in_malgs}.
\end{RemSub}
\begin{proof}
The morphisms $t''_{X Y}:X \times MY \rightarrow M(X \times Y)$ and $t'_{X Y}:MX \times Y \rightarrow M(X \times Y)$ are given by
\begin{enumerate}
\item[] $t''_{X Y}(x,\nu) = (M(\lambda y.(x,y)))(\nu)$,
\item[] $t'_{X Y}(\mu,y) = (M(\lambda x.(x,y)))(\mu)$,
\end{enumerate}
where we have employed the notation of lambda calculus, so that for each $y \in Y$, the expression $\lambda x.(x,y)$ denotes the morphism $X \rightarrow X \times Y$ given by $x \mapsto (x,y)$.  Hence
\begin{enumerate}
\item[] $f_a(x,\nu) = (a \cdot Mf \cdot M(\lambda y.(x,y)))(\nu) = (a \cdot M(\lambda y.f(x,y)))(\nu) = \int_y f(x,y) \;d\nu\;$,
\item[] $f^a(\mu,y) = (a \cdot Mf \cdot M(\lambda x.(x,y)))(\mu) = (a \cdot M(\lambda x.f(x,y)))(\mu) = \int_x f(x,y) \;d\mu\;$.
\end{enumerate}
\end{proof}

\begin{PropSub}\label{thm:abs_vv_fubini}
The $\X$-monad $\MM$ is commutative if and only if for each $\MM$-algebra $(Z,a)$ and each morphism $f:X \boxtimes Y \rightarrow Z$ in $\X$, the composites
\begin{enumerate}
\item[] $MX \boxtimes MY \xrightarrow{t'_{X MY}} M(X \boxtimes MY) \xrightarrow{Mf_a} MZ \xrightarrow{a} Z$
\item[] $MX \boxtimes MY \xrightarrow{t''_{MX Y}} M(MX \boxtimes Y) \xrightarrow{Mf^a} MZ \xrightarrow{a} Z$
\end{enumerate}
are equal.  If so, then both are equal to the composite
\begin{enumerate}
\item[] $MX \boxtimes MY \xrightarrow{\otimes_{X Y}} M(X \boxtimes Y) \xrightarrow{Mf} MZ \xrightarrow{a} Z\;.$
\end{enumerate}
\end{PropSub}
\begin{proof}
The first composite is
$$
\begin{array}{lll}
a \cdot Mf_a \cdot t'_{X MY} & = & a \cdot M(a \cdot Mf \cdot t''_{X Y}) \cdot t'_{X MY} \\
 & = & a \cdot Ma \cdot MMf \cdot Mt''_{X Y} \cdot t'_{X MY} \\
 & = & a \cdot \kappa_Z \cdot MMf \cdot Mt''_{X Y} \cdot t'_{X MY} \\
 & = & a \cdot Mf \cdot \kappa_{X \boxtimes Y} \cdot Mt''_{X Y} \cdot t'_{X MY} \\
 & = & a \cdot Mf \cdot \widetilde{\otimes}_{X Y}
\end{array}
$$
and, similarly, the second composite is $a \cdot Mf \cdot \otimes_{X Y}$, where the morphisms $\widetilde{\otimes}_{X Y}, \otimes_{X Y}$ are defined in \bref{def:comm_mnd}.  But the commutativity of $\MM$ is the condition that $\widetilde{\otimes} = \otimes$, which therefore entails the equality of the two given composites.  Conversely, taking $(Z,a)$ to be the free $\MM$-algebra $(M(X \boxtimes Y),\kappa_{X \boxtimes Y})$ and letting $f := \delta_{X \boxtimes Y}$, the upper and lower composites $a \cdot Mf \cdot \widetilde{\otimes}_{X Y}$ and $a \cdot Mf \cdot \otimes_{X Y}$ become exactly $\widetilde{\otimes}_{X Y}$ and $\otimes_{X Y}$, respectively.
\end{proof}

Kock \cite{Kock:Dist} has claimed that for an arbitrary $\X$-monad $\TT$ on a cartesian closed category $\X$, the equality of morphisms $\otimes_{X Y} = \widetilde{\otimes}_{X Y}$ that is required in order that $\TT$ be commutative may be interpreted as a form of Fubini's Theorem.  The following proposition substantiates this claim in the well-pointed case.

\begin{CorSub}\label{thm:vv_fubini_arb_mnd}
Let $\X$ be a locally small well-pointed cartesian closed category, and let $\MM$ be an $\X$-monad on $\uX$.  Then $\MM$ is commutative if and only if
\begin{equation}\label{eq:vv_fub}\int_y \int_x f(x,y) \;d\mu d\nu = \int f \;d(\mu \otimes \nu) = \int_x \int_y f(x,y) \;d\nu d\mu\;.\end{equation}
for each $\MM$-algebra $(Z,a)$, each morphism $f:X \times Y \rightarrow Z$ in $\X$, and each pair $(\mu,\nu) \in MX \times MY$.  (Here $\mu \otimes \nu := \otimes_{X Y}(\mu,\nu)$.)
\end{CorSub}
\begin{proof}
The first composite in \bref{thm:abs_vv_fubini} sends $(\mu,\nu)$ to
$$
\begin{array}{lllll}
(a \cdot Mf_a \cdot M(\lambda x.(x,\nu)))(\mu) & = & (a \cdot M(\lambda x.f_a(x,\nu)))(\mu) & & \\
 & = & (a \cdot M(\lambda x.\int_y f(x,y) \;d\nu))(\mu) & = & \int_x \int_y f(x,y) \;d\nu\;d\mu\;. \\
\end{array}
$$
Similarly, the second sends $(\mu,\nu)$ to the leftmost value in \eqref{eq:vv_fub}; the third sends $(\mu,\nu)$ to the middle value.
\end{proof}

\begin{ThmSub}\label{thm:vv_fub_acc_distn}
Suppose data as in \bref{par:data_for_dcompl_and_pett} are given, with $\X$ a locally small well-pointed cartesian closed category.  Let $(Z,a)$ be a $\tDD$-algebra, let $f:X \times Y \rightarrow Z$ in $\X$, and let $\mu,\nu$ be accessible distributions on $X,Y$, respectively.  Then
\begin{equation}\label{eq:vv_fub_2}\int_y \int_x f(x,y) \;d\mu d\nu = \int f \;d(\mu \otimes \nu) = \int_x \int_y f(x,y) \;d\nu d\mu\;.\end{equation}
\end{ThmSub}
\begin{proof}
Since the accessible distribution monad $\tDD$ is commutative \pbref{thm:td_comm}, this follows from \bref{thm:vv_fubini_arb_mnd}.
\end{proof}

\begin{CorSub}\label{thm:vv_fub_conv_sp}
Let $(Z,a)$ be an algebra of the natural distribution monad on the category of convergence spaces \pbref{exa:nat_dist_conv}.  Let $f:X \times Y \rightarrow Z$ be a continuous map, and let $\mu:[X,R] \rightarrow R$, $\nu:[Y,R] \rightarrow R$ be continuous linear functionals.  Then the equation \eqref{eq:vv_fub_2} holds.
\end{CorSub}
\begin{proof}
This follows from \pbref{thm:nat_dist_on_conv_comm} via \pbref{thm:vv_fubini_arb_mnd}.
\end{proof}

\section{Pettis linear spaces}

In the present section, we work with data as given in \bref{par:data_for_dcompl_and_pett}, and we fix an abstract distribution monad $(\MM,\Delta,\xi)$.  Write $\MM = (M,\delta,\kappa)$.

\begin{DefSub}\label{def:pett_obs}\emptybox
\begin{enumerate}
\item We say that an object $E$ of $\sL$ is an \textit{$\MM$-Pettis object} if $E$ is functionally separated and there exists a (necessarily unique) morphism $a_E$ in $\X$ such that
$$
\xymatrix{
MGE \ar[dr]_{a_E} \ar[rr]^{\xi_E} &                          & GHE \\
                                  & GE \ar[ur]_{G\fd_E} & 
}
$$
commutes.  ($a_E$ is unique if it exists, since $\fd_E$ is mono and $G$ preserves monos.)
\item Given an object $E \in \sL$ and a morphism $a:MGE \rightarrow GE$, we say that \textit{$a$ is an $\MM$-Pettis structure for $E$} if $E$ is functionally separated and the diagram formed as above commutes --- equivalently, if $E$ is $\MM$-Pettis and $a = a_E$.
\item We denote by $\pL{\MM}$  the full sub-$\X$-category of $\sL$ consisting of the $\MM$-Pettis objects.
\end{enumerate}
\end{DefSub}

\begin{ExaSub}\label{def:charn_pett_malgs_via_pett_integs}
Suppose that $\X$ is cartesian closed, locally small, and well-pointed, and suppose also that the faithful functor $U = \X(1,-):\X \rightarrow \Set$ has a left adjoint (as is the case when $U$ is topological, and in particular when $\X$ is $\Conv$ or $\Smooth$).  Then, for a separated linear space $E \in \sL$, the following are equivalent:
\begin{enumerate}
\item $E$ is $\MM$-Pettis. 
\item For each map $f:X \rightarrow E$ in $\X$ and each $\mu \in MX$, the Pettis-type integral $\pint{} f\;d\mu$ exists \pbref{rem:dun_pett_ints_wrt_abs_distn}.
\end{enumerate}
\end{ExaSub}
\begin{proof}
Under the present hypotheses, it is easy to show that every strong mono in $\X$ is $U$-cartesian.  In particular, $G\fd_E$ is a strong mono in $\X$ and hence is $U$-cartesian (since $\fd_E$ is a strong mono in $\sL$ and any right adjoint preserves strong monos, \cite{Wy:QuTo}  10.5).  But one readily finds that 2 holds iff $U\xi_E$ factors through the monomorphism $UG\fd_E$ in $\Set$, and the result follows.
\end{proof}

\begin{ThmSub}\label{thm:pett_spaces_are_malgs}
For any $\MM$-Pettis linear space $E \in \sL$, the underlying space $GE$ is an $\MM$-algebra when equipped with the $\MM$-Pettis structure $a_E:MGE \rightarrow GE$ of $E$.
\end{ThmSub}
\begin{proof}
The periphery of the following diagram commutes since $(\xi,G)$ is a monad morphism, so since the lower two cells commute and $G\fd_E$ is mono, the upper cell commutes; i.e. the unit law for $(GE,a_E)$ holds.
$$
\xymatrix{
GE \ar[rr]^{\delta_{GE}} \ar@{=}[dr] \ar@/_3ex/[ddr]_{G\fd_E} & & MGE \ar[dl]|{a_E} \ar@/^3ex/[ddl]^{\xi_E} \\
 & GE \ar[d]|{G\fd_E} & \\
 & GHE &
}
$$
The diagrammatic associativity law for $(GE,a_E)$ appears as the square on the left side of the following diagram.
$$
\xymatrix{
 & & MGHE \ar[r]^{\xi_{HE}} & GHHE \ar[dd]^{G\fk_E} \\
MMGE \ar[d]_{\kappa_{GE}} \ar[r]^{Ma_E} \ar@/^2ex/[urr]^{M\xi_E} & MGE \ar[d]_{a_E}  \ar[r]|{\xi_E} \ar[ur]|{MG\fd_E} & GHE \ar[dr]_1 \ar[ur]|{GH\fd_E} & \\
MGE \ar@/_3ex/[rrr]|{\xi_E} \ar[r]^{a_E} & GE \ar[rr]|{G\fd_E} & & GHE
}
$$
All the other cells in the diagram commute, and the periphery commutes since $(\xi,G)$ is a morphism of $\X$-monads, so since $G\fd_E$ is mono, the square in question commutes.
\end{proof}

\begin{RemSub}\label{rem:integ_in_pett_malg_is_pett_integ}
Assume for the sake of illustration that $\X$ is cartesian closed, locally small, and well-pointed.  If $E$ is an $\MM$-Pettis linear space, then $(GE,a_E)$ is an $\MM$-algebra, so for any map $f:X \rightarrow GE$ and any $\mu \in MX$ we can take the associated integral $\int f\;d\mu = (a_E \cdot Mf)(\mu)$ \pbref{def:vint_in_malgs}, which is equally the Pettis-type integral $\pint{} f\;d\mu$ \pbref{rem:dun_pett_ints_wrt_abs_distn}.
\end{RemSub}

\begin{DefSub}
Given an $\MM$-Pettis object $E \in \sL$, we denote by $AE$ the $\MM$-algebra $(GE,a_E)$.
\end{DefSub}

\begin{ThmSub}\label{thm:pett_obs_malgs}\emptybox
\begin{enumerate}
\item Given $\MM$-Pettis linear spaces $E,E' \in \sL$ and a linear map $h:E \rightarrow E'$, the underlying map $Gh$ of $h$ is an $\MM$-homomorphism $Gh:(GE,a_E) \rightarrow (GE',a_{E'})$ (and hence preserves the integral, \bref{thm:homom_via_integ}).
\item For all $\MM$-Pettis objects $E,E' \in \sL$, there is a unique morphism $A_{E E'}$ in $\X$ such that the triangle
$$
\xymatrix{
\pL{\MM}(E,E') \ar[dr]_{A_{E E'}} \ar@{=}[r] & \sL(E,E') \ar@{ >->}[r]^{G_{E E'}} & \uX(GE,GE') \\
& \uX^\MM(AE,AE') \ar@{ >->}[ur]_{G^\MM_{AE AE'}} &
}
$$
commutes.
\end{enumerate}
\end{ThmSub}
\begin{proof}
Although 2 entails 1, it is instructive to first show 1 before `internalizing' the resulting argument in order to prove 2.  We have a triangular prism
$$
\xymatrix{
MGE \ar[dd]_{MGh} \ar[dr]_{a_E} \ar[rr]^{\xi_E} & & GHE \ar[dd]^{GHh} \\
 & GE \ar[ur]_{G\fd_E} \ar[dd]|(.3){Gh} & \\
MGE' \ar[dr]_{a_{E'}} \ar@{.>}[rr]^(.65){\xi_{E'}} & & GHE' \\
& GE' \ar[ur]_{G\fd_{E'}} & 
}
$$
in which the top, bottom, back, and right faces commute, so since $G\fd_{E'}$ is mono, the left face commutes --- i.e. $Gh$ is an $\MM$-homomorphism.

For 2, it suffices to show that
$$
\xymatrix{
\sL(E,E') \ar[r]^{G_{E E'}} & \uX(GE,GE') \ar[rr]^{\uX(a_E,GE')} \ar[dr]_{M_{GE GE'}} &                                         & \uX(MGE,GE') \\
                            &                                                         & \uX(MGE,MGE') \ar[ur]|{\uX(MGE,a_{E'})} & 
}
$$
is a fork, since $G^\MM_{AE AE'}$ is the equalizer of the given pair.  We have a diagram
$$
\xymatrix{
& \sL(E,E') \ar@/_4ex/[ddl]_{MG} \ar[d]^G \ar[r]^{GH} & \uX(GHE,GHE') \ar[d]_{(G\fd_E)^*} \ar@/^10ex/[dd]^{\xi_E^*} \\
& \uX(GE,GE') \ar[dl]|M \ar[d]^{a_E^*} \ar[r]^{(G\fd_{E'})_*} & \uX(GE,GHE') \ar[d]_{a_E^*} \\
\uX(MGE,MGE') \ar[r]^{(a_{E'})_*} \ar@/_4ex/[rr]|{(\xi_{E'})_*} & \uX(MGE,GE') \ar[r]^{(G\fd_{E'})_*} & \uX(MGE,GHE')
}
$$
in which the prospective fork occurs, where we have written $f_*$, $f^*$ to denote the morphisms $\uX(X,f)$, $\uX(f,X)$ associated to morphisms $f$ and objects $X$ in $\X$.  Observe that the upper square commutes by the $\X$-naturality of $G\fd$.  Moreover, aside from the triangle that occurs within the prospective fork, each cell in the diagram commutes.  Also, the periphery commutes by the $\X$-naturality of $\xi$, and the result follows since $(G\fd_{E'})_* = \uX(MGE,G\fd_{E'})$ is mono (as $G\fd_{E'}$ is mono and hence $\X$-mono, by \bref{thm:lim_mono_epi_in_base_are_enr}).
\end{proof}

\begin{CorSub}
The composite $\X$-functor $\pL{\MM} \hookrightarrow \sL \xrightarrow{G} \uX$ factors through \linebreak[4] \hbox{$G^\MM:\uX^\MM \rightarrow \uX$} via an $\X$-functor $A:\pL{\MM} \rightarrow \uX^\MM$ which sends each $\MM$-Pettis object $E$ to its associated $\MM$-algebra $AE = (GE,a_E)$.
\end{CorSub}
\begin{proof}
Since the forgetful $\X$-functor $G^\MM$ is $\X$-faithful, this follows from \bref{thm:pett_obs_malgs} 2 by \bref{thm:factn_through_faithful_vfunctor}.
\end{proof}

\begin{PropSub}\label{thm:diag_funcs_pett}
The diagram of $\X$-functors
\begin{equation}\label{eqn:diag_funcs_pett}
\xymatrix{
\pL{\MM} \ar[d]_A \ar@{^{(}->}[r] & *!/l3ex/+{\sL = \uX^\LL} \ar[d]^G \\
\uX^\MM \ar[ur]|{\uX^\Delta} \ar[r]_{G^\MM} & \uX
}
\end{equation}
commutes, where $\uX^\Delta$ is the $\X$-functor determined by the given morphism of $\X$-monads $\Delta:\LL \rightarrow \MM$ via \bref{prop:enr_func_induced_by_mnd_mor}.
\end{PropSub}
\begin{proof}
By \bref{prop:enr_func_induced_by_mnd_mor}, the lower triangle commutes.  Hence since the periphery commutes and $G$ is $\X$-faithful, it suffices to show that the upper triangle commutes on objects.  Given an object $E \in \pL{\MM}$, we obtain an $\LL$-algebra
$$\uX^\Delta AE = \uX^\Delta(GE,a_E) = (GE,LGE \xrightarrow{\Delta_{GE}} MGE \xrightarrow{a_E} GE)\;,$$
but $E$ is an object of $\sL = \uX^\LL$ and so $E$ itself is an $\LL$-algebra $E = (GE,G\varepsilon_E)$, and our aim is to show that these $\LL$-algebras are the same.  Hence it suffices to show that the upper triangle in the following diagram commutes.
$$
\xymatrix{
*!/r3ex/+{LGE = GFGE} \ar@/_5ex/[ddrr]_{\zeta_E} \ar[rr]^{\Delta_{GE}} \ar[dr]|{G\varepsilon_E} & & MGE \ar[dl]|{a_E} \ar[dd]^{\xi_E} \\
& GE \ar[dr]|{G\fd_E} & \\
& & GHE
}
$$
But the rightmost triangle commutes, and the leftmost cell commutes by the definition of the canonical morphism $\zeta$, so since the periphery commutes and $G\fd_E$ is mono, the result follows.
\end{proof}

\begin{ThmSub}\label{thm:pett_obs_iso_to_subcat_malgs}\emptybox
\begin{enumerate}
\item The $\X$-functor $A:\pL{\MM} \rightarrow \uX^\MM$ is $\X$-fully-faithful and injective on objects, so $A$ restricts to an isomorphism between the category of $\MM$-Pettis objects $\pL{\MM}$ and a full sub-$\X$-category $\uX^\MM_{\textnormal{P}}$ of $\uX^\MM$, whose objects we call \emph{Pettis $\MM$-algebras}.
\item In particular, the resulting isomorphism of $\X$-categories $\uX^\MM_{\textnormal{P}} \xrightarrow{\sim} \pL{\MM}$ is simply a restriction of $\uX^\Delta:\uX^\MM \rightarrow \uX^\LL = \sL$.
\end{enumerate}
\end{ThmSub}
\begin{proof}
Let $A':\pL{\MM} \rightarrow \uX^\MM_{\textnormal{P}}$ be the corestriction of $A$ to its full image.  Since the upper triangle in \eqref{eqn:diag_funcs_pett} commutes, $\uX^\Delta$ restricts to a retraction $P$ of the $\X$-functor $A'$.  Since $A'$ is by definition surjective on objects, $A'$ and $P$ are bijective (and mutually inverse) on objects.  Also, since the lower triangle in \eqref{eqn:diag_funcs_pett} commutes and $G^\MM$ is $\X$-faithful, $\uX^\Delta$ is $\X$-faithful, so $P$ is $\X$-faithful and hence $A'$, $P$ are mutually-inverse isomorphisms of $\X$-categories.
\end{proof}

\begin{PropSub}\label{thm:basic_charn_pett_malg}
An $\MM$-algebra $(X,a)$ is Pettis if and only if $a$ serves as an $\MM$-Pettis structure for the underlying linear space $E := \uX^\Delta(X,a)$ of $(X,a)$.
\end{PropSub}
\begin{proof}
If $a$ serves as an $\MM$-Pettis structure for $E$, then $E \in \pL{\MM}$ and $AE = (GE,a) = (X,a)$, so $(X,a)$ is Pettis.  Conversely, if $(X,a) = AE'$ for some $E' \in \pL{\MM}$ then $a$ is the $\MM$-Pettis structure for $E'$, but by \bref{thm:pett_obs_iso_to_subcat_malgs}, $E = \uX^\Delta(X,a) = \uX^\Delta AE' = E'$.
\end{proof}

\section{Distributionally complete linear spaces}\label{sec:distnl_compl_ls}

We continue to work with given data as in \bref{par:data_for_dcompl_and_pett}, and we let $(\MM,\Delta,\xi)$ be an abstract distribution monad, with $\MM = (M,\delta,\kappa)$.  

\begin{ParSub}\label{par:preamble_for_dist_compl}
For each space $X \in \X$, $MX$ carries the structure of a linear space (i.e. an object of $\sL = \uX^\LL$, \bref{def:first_synth_terminology}), namely the underlying linear space of the free $\MM$-algebra on $X$, written as $\uX^\Delta(MX,\kappa_X)$ \pbref{par:underlying_linear_space}.  In the present section, we endow $MX$ with this structure.  Moreover, we omit all notational distinctions between linear spaces $E \in \sL$ and their underlying spaces $GE \in \X$, and similarly for linear maps $h$ in $\sL$ and their underlying maps $Gh$ in $\X$.

By \bref{thm:mnd_mor_as_homom}, the components of the morphism of $\X$-monads $\Delta:\LL \rightarrow \MM$ are linear maps 
$$\Delta_X:FX \rightarrow MX\;,$$
i.e. morphisms in $\sL = \uX^\LL$, where $FX = (LX,\mu_X)$ is the free linear space on $X$.  Moreover, $\Delta_X$ is the transpose under $F \nsststile{\varepsilon}{\eta} G:\sL \rightarrow \uX$ of the unit component
$$\delta_X:X \rightarrow MX\;.$$
\end{ParSub}

\begin{DefSub}\label{def:dist_compl}
Given an abstract distribution monad $(\MM,\Delta,\xi)$, we say that a linear space $E \in \sL$ is \textit{$\MM$-distributionally complete} if for each space $X \in \X$, the morphism \hbox{$\Delta_X:FX \rightarrow MX$} is $\sL$-orthogonal to $E$ in $\uL$.
\end{DefSub}

\begin{RemSub}\label{rem:dist_compl_expl}
Given $X \in \X$, $E \in \sL$, recall from \bref{def:orth_subcat_sigma} that $\Delta_X$ is $\sL$-orthogonal to $E$ (written $ \Delta_X \bot_\sL E$) if and only if
$$\uL(\Delta_X,E):\uL(MX,E) \rightarrow \uL(FX,E)$$
is an isomorphism in $\sL$.  Since $G:\sL \rightarrow \X$ reflects isomorphisms and the $\X$-category $\sL$ may be described as $G_*\uL$ \pbref{par:data_for_dcompl_and_pett}, this is equivalent to the requirement that the morphism
\begin{equation}\label{eqn:char_dist_compl_via_xorth}\sL(\Delta_X,E):\sL(MX,E) \rightarrow \sL(FX,E)\end{equation}
in $\X$ be iso --- i.e. that $\Delta_X$ be $\X$-orthogonal to $E$ in $\sL$.

Abusing notation, we will denote the class of linear maps $\{\Delta_X\;|\; X \in \X\}$ by $\Delta$, so that in the notation of \bref{def:orth_subcat_sigma}, the full sub-$\X$-category of $\sL$ consisting of $\MM$-distributionally complete objects is denoted by $\sL_\Delta$.
\end{RemSub}

\begin{ExaSub}\label{exa:dist_compl_for_nat_acc_distn_mnds}\emptybox
\begin{enumerate}
\item By \bref{exa:und_linear_spaces_of_free_d_and_td_algs}, the underlying linear space of each free $\DD$-algebra $(DX,\kappa^\DD_X)$ $(X \in \X)$ is the linear space of natural distributions $HFX = (FX)^{**}$ on $X$, and by the definition of $\Delta^\DD$ \pbref{def:can_mnd_mor_distns}, the associated linear map
$$\Delta^\DD_X = \fd_{FX}:FX \rightarrow HFX = (FX)^{**}$$
is the component at $FX$ of the unit $\fd$ of the double-dualization monad $\HH$.  Hence a linear space $E \in \sL$ is $\DD$-distributionally complete iff each of these linear maps $\fd_{FX}$ is $\sL$-orthogonal (equivalently, $\X$-orthogonal) to $E$.
\item By \bref{exa:und_linear_spaces_of_free_d_and_td_algs}, the underlying linear space of each free $\tDD$-algebra $(\tD X,\kappa^{\tDD}_X)$ $(X \in \X)$ is the linear space of accessible distributions $\tH FX$ on $X$.  The associated linear map
$$\Delta^{\tDD}_X = \tfd_{FX}:FX \rightarrow \tH FX$$
is the component at $FX$ of the unit $\tfd$ of the idempotent $\sL$-monad $\tHH$, which is induced by the functional completion $\sL$-reflection $K \nsststile{}{\tfd} J:\tL \hookrightarrow \uL$ \pbref{def:func_compl}.  Hence a linear space $E \in \sL$ is $\tDD$-distributionally complete iff each of these reflection morphisms $\tfd_{FX}$ is $\sL$-orthogonal (equivalently, $\X$-orthogonal) to $E$.  \begin{samepage}In particular, we deduce the following by \bref{thm:refl_subcats_are_orth_subcats}:
\begin{equation}\label{thm:func_compl_impl_td_dist_compl}\textit{Every functionally complete linear space is $\tDD$-distributionally complete.}\end{equation}
\end{samepage}
\end{enumerate}
\end{ExaSub}

\begin{PropSub}\label{thm:charn_dist_compl_via_dirac}
A linear space $E \in \sL$ is $\MM$-distributionally complete if and only if for each space $X \in \X$, the composite
\begin{equation}\label{eq:charn_dist_compl_via_dirac} \sL(MX,E) \overset{G}{\rightarrowtail} \uX(MX,E) \xrightarrow{\uX(\delta_X,E)} \uX(X,E)\end{equation}
is an isomorphism in $\X$.
\end{PropSub}
\begin{proof}
Since we have an $\X$-adjunction $F \nsststile{\varepsilon}{\eta} G:\sL \rightarrow \uX$, the composite
$$\sL(FX,E) \xrightarrow{G} \uX(FX,E) \xrightarrow{\uX(\eta_X,E)} \uX(X,E)$$
is an isomorphism in $\X$.  Composing the morphism \eqref{eqn:char_dist_compl_via_xorth} with this isomorphism, one obtains by the $\X$-functoriality of $G$ the composite
$$\sL(MX,E) \xrightarrow{G} \uX(MX,E) \xrightarrow{\uX(\Delta_X,E)} \uX(FX,E) \xrightarrow{\uX(\eta_X,E)} \uX(X,E)\;,$$
which is the same as \eqref{eq:charn_dist_compl_via_dirac} since $\Delta_X \cdot \eta_X = \delta_X$ in $\X$.
\end{proof}

\begin{RemSub}\label{rem:eltwise_distn_compl}
If a linear space $E$ is $\MM$-distributionally complete, then in particular, the isomorphism \eqref{eq:charn_dist_compl_via_dirac} induces a bijection between the associated hom-classes of the underlying ordinary categories, so that for each map $f:X \rightarrow E$ in $\X$ there is a unique linear map $f^\sharp:MX \rightarrow E$ such that
$$
\xymatrix{
X \ar[r]^{\delta_X} \ar[dr]_f & MX \ar[d]^{f^\sharp} \\
& E
}
$$
commutes in $\X$.  $\MM$-distributional completeness can be construed as the $\X$-enriched analogue of this criterion.
\end{RemSub}

\begin{ThmSub}\label{thm:m_distnly_cpl_sp_are_m_algs}
Let $E \in \sL$ be an $\MM$-distributionally complete linear space.  Then the underlying space of $E$ carries the structure of an $\MM$-algebra.  In particular, there is a unique linear map $b_E:ME \rightarrow E$ whose underlying map serves as $\MM$-algebra structure.  Moreover,
\begin{enumerate}
\item $b_E$ is the counit $b_E:ME \rightarrow E$ of the representation $\sL(M-,E) \cong \uX(-,E)$ \pbref{thm:charn_dist_compl_via_dirac}.  
\item $b_E$ is the unique linear map whose underlying map is a retraction of \hbox{$\delta_E:E \rightarrow ME$}.
\end{enumerate}
\end{ThmSub}
\begin{proof}
In view of \bref{rem:eltwise_distn_compl}, the counit $b_E$ satisfies 2.  Hence it suffices to show that the rightmost square in the following diagram commutes.
$$
\xymatrix{
ME \ar[r]_{\delta_{ME}} \ar[d]_{b_E} & MME \ar[d]|{Mb_E} \ar[r]_{\kappa_E} & ME \ar[d]^{b_E} \\
E \ar[r]^{\delta_E} & ME \ar[r]^{b_E} & E
}
$$
But the leftmost square commutes, and the periphery commutes since the composites on the top and bottom sides are identity morphisms.  Further, the maps in the rightmost square are all linear; indeed, $Mb_E$ and $\kappa_E$ are $\MM$-homomorphisms with respect to the free $\MM$-algebra structures and so, via $\uX^\Delta:\uX^\MM \rightarrow \uX^\LL = \sL$ are linear maps of the underlying linear spaces.  Hence since $E$ is $\MM$-distributionally complete, it follows by \bref{rem:eltwise_distn_compl} that the rightmost square commutes.
\end{proof}

\begin{PropSub}\label{thm:dist_compl_to_malg_to_same_lin_sp}
Given an $\MM$-distributionally complete linear space $E \in \sL$, the $\X$-functor $\uX^\Delta:\uX^\MM \rightarrow \uX^\LL = \sL$ sends the associated $\MM$-algebra $(E,b_E)$ to $E$ itself.
\end{PropSub}
\begin{proof}
We must show that the $\LL$-algebra structure
$$LE \xrightarrow{\Delta_E} ME \xrightarrow{b_E} E$$
coincides with the structure $e:LE \rightarrow E$ carried by the $\LL$-algebra $E$ itself.  The latter structure $e$ may be characterized as the unique linear map $e:FE \rightarrow E$ such that $e \cdot \eta_E = 1_E$ in $\X$.  Hence since
$$
\xymatrix{
LE \ar[r]^{\Delta_E} & ME \ar[r]^{b_E} & E \\
E \ar[u]_{\eta_E} \ar[ur]|{\delta_E} \ar[urr]_{1_E} & & 
}
$$
commutes and both $\Delta_E$ and $b_E$ are linear maps, we must have $b_E \cdot \Delta_E = e$.
\end{proof}

\begin{ThmSub}\label{thm:lin_map_betw_mdist_compl_sp_is_m_homom}\emptybox
\begin{enumerate}
\item Every linear map $h:E' \rightarrow E$ between $\MM$-distributionally complete linear spaces is an $\MM$-homomorphism $h:(E',b_{E'}) \rightarrow (E,b_E)$ between the associated $\MM$-algebras (and hence preserves the integral, \bref{thm:homom_via_integ}).
\item There is an $\X$-fully-faithful and injective-on-objects $\X$-functor $B:\sL_\Delta \rightarrowtail \uX^\MM$, sending each $\MM$-distributionally complete object $E \in \sL$ to its associated $\MM$-algebra $BE = (E,b_E)$, such that the following diagram commutes.
\begin{equation}\label{eqn:diag_funcs_distn_compl}
\xymatrix{
\sL_\Delta \ar@{ >->}[d]_B \ar@{^{(}->}[r] & *!/l3ex/+{\sL = \uX^\LL} \ar[d]^G \\
\uX^\MM \ar[ur]|{\uX^\Delta} \ar[r]_{G^\MM} & \uX
}
\end{equation}
\item There are isomorphisms
$$\uX^\Delta_{Z,BE}:\uX^\MM(Z,BE) \xrightarrow{\sim} \sL(\uX^\Delta Z,E)$$
$\X$-natural in $Z \in \uX^\MM$, $E \in \sL_\Delta$.
\end{enumerate}
\end{ThmSub}
\begin{proof}
It is instructive to first prove 1, before `internalizing' the argument in order to prove 2 and 3.  In the following diagram
$$
\xymatrix{
E' \ar[d]_{\delta_{E'}} \ar[r]^h & E \ar[d]^{\delta_E} \\
ME' \ar[d]_{b_{E'}} \ar[r]^{Mh} & ME \ar[d]^{b_E} \\
E' \ar[r]_h & E
}
$$
the upper square commutes, and the periphery commutes since the composites on the left and right sides are identity morphisms.  But each of the morphisms in the lower square is linear; indeed, $Mh$ is an $\MM$-homomorphism of free $\MM$-algebras and hence, via $\uX^\Delta:\uX^\MM \rightarrow \uX^\LL = \sL$, is a linear map of the underlying linear spaces.  It follows by \bref{rem:eltwise_distn_compl} that the lower square commutes.

Notice that the argument still goes through if we simply let $E'$ be the underlying linear space $E' = \uX^\Delta(Z,b')$ of an arbitrary given $\MM$-algebra $(Z,b')$, since the $\MM$-algebra structure $b'$ is an $\MM$-homomorphism and hence a linear map $b':ME' \rightarrow E'$.  By formulating an $\X$-enriched analogue of the above argument, we shall now show that there is a unique morphism $B_{E' E}$ such that
$$
\xymatrix{
\sL(E',E) \ar@{ >->}[rr]^{G_{E' E}} \ar[dr]_{B_{E' E}} & & \uX(E',E) \\
 & \uX^\MM((E',b'),(E,b_E)) \ar@{ >->}[ur]|{G^\MM_{(E',b') (E,b_E)}} &
}
$$
commutes.  For this it suffices to show that
$$
\xymatrix{
\sL(E',E) \ar[r]^{G_{E' E}} & \uX(E',E) \ar[rr]^{\uX(b',E)} \ar[dr]_{M_{E' E}} &                                         & \uX(ME',E) \\
                            &                                                         & \uX(ME',ME) \ar[ur]|{\uX(ME',b_E)} & 
}
$$
is a fork, since $G^\MM_{(E',b') (E,b_E)}$ is the equalizer of the given pair.  To this end, first observe that the upper composite in the prospective fork appears within the following commutative diagram
$$
\xymatrix{
\sL(E',E) \ar[d]_G \ar[r]^{\sL(b',E)} & \sL(ME',E) \ar[d]^G \\
\uX(E',E) \ar@{=}[dr] \ar[r]^{\uX(b',E)} & \uX(ME',E) \ar[d]^{\uX(\delta_{E'},E)} \\
& {\uX(E',E)\;,}
}
$$
and by \bref{thm:charn_dist_compl_via_dirac}, the composite on the right side is an isomorphism $\phi:\sL(ME',E) \xrightarrow{\sim} \uX(E',E)$ since $E$ is $\MM$-distributionally complete.  The lower composite in the prospective fork appears within the following commutative diagram
$$
\xymatrix{
& & \sL(ME',ME) \ar[d]^G \ar[rr]^{\sL(ME',b_E)} & & \sL(ME',E) \ar[d]^G \\
& \uX^\MM(F^\MM E',F^\MM E) \ar[r]_{G^\MM} \ar[ur]^{\uX^\Delta} & \uX(ME',ME) \ar[d]|{\uX(\delta_{E'},ME)} \ar[rr]^{\uX(ME',b_E)} & & \uX(ME',E) \ar[d]^{\uX(\delta_{E'},E)} \\
\sL(E',E) \ar[r]_{G} & \uX(E',E) \ar[r]^{\uX(E',\delta_E)} \ar@/_3ex/@{=}[rrr] \ar[ur]|M \ar[u]^{F^\MM} & \uX(E',ME) \ar[rr]^{\uX(E',b_E)} & & {\uX(E',E)\;,} \\
& & & &
}
$$
whose right side is the same isomorphism $\phi$.  In particular, both the upper and lower composites in the prospective fork factor through $G:\sL(ME',E) \rightarrow \uX(ME',E)$, via evident morphisms $u,l:\sL(E',E) \rightarrow \sL(ME',E)$, respectively.  But by the commutativity of the preceding diagrams, the composites $\phi \cdot u$, $\phi \cdot l$ are both equal to $G_{E' E}:\sL(E',E) \rightarrow \uX(E',E)$.  Hence, since $\phi$ is iso, $u = l$ and hence the upper and lower composites in the prospective fork are equal.

Hence we have commutative triangles
$$
\xymatrix @C=12ex @R=6ex {
\sL(E',E) \ar[r]^(.4){B_{E' E}} \ar@{ >->}[dr]_{G_{E' E}} & \uX^\MM((E',b'),(E,b_E)) \ar[r]^(.6){\uX^\Delta_{(E',b') (E,b_E)}} \ar@{ >->}[d]|{G^\MM_{(E',b') (E,b_E)}} & \sL(E',E) \ar@{ >->}[dl]^{G_{E' E}} \\
& \uX(E',E) & 
}
$$
in which $G_{E' E},G^\MM_{(E',b'),(E,b_E)}$ are mono.  It follows that $B_{E' E},\uX^\Delta_{(E',b') (E,b_E)}$ are mutually-inverse isomorphisms, so 3 is established.

An $\X$-fully-faithful $\X$-functor $B$ making \eqref{eqn:diag_funcs_distn_compl} commute is now obtained by \bref{thm:factn_through_faithful_vfunctor}.  The commutativity of the upper triangle in \eqref{eqn:diag_funcs_distn_compl} also entails that $B$ is injective on objects.
\end{proof}

\section{Distributional completeness and Pettis spaces for accessible distributions}\label{sec:distnl_compl_pett_for_acc_distns}

Again we work with given data as in \bref{par:data_for_dcompl_and_pett}.  Beware that in the present section we shall elide applications of the change-of-base 2-functor $G_*:\LCAT \rightarrow \XCAT$.

\begin{DefSub}
Given an abstract distribution monad $\MM$, we denote by $\sL_{(\MM)}$ the full subcategory of $\sL$ consisting of the $\MM$-distributionally complete separated objects.  Depending on context, we shall also construe $\sL_{(\MM)}$ as a full sub-$\X$-category (resp. sub-$\sL$-category) of $\sL$ (resp. $\uL$).
\end{DefSub}

In the present section, we shall prove the following, and along the way, we will define and employ notions of \textit{$\tDD$-distributional completion, closedness, and density} \pbref{def:distnl_compl_cl_dens}.

\begin{ThmSub}\label{thm:dist_compl_sep_vs_pett}\emptybox
\begin{enumerate}
\item The $\tDD$-distributionally complete separated linear spaces are exactly the $\tDD$-Pettis linear spaces.
\item The resulting full subcategory $\sL_{(\tDD)} = \pL{\tDD}$ of $\sL$ is $\sL$-reflective in $\uL$ (and hence also $\X$-reflective in $\sL = G_*\uL$).
\item The category $\sL_{(\tDD)} = \pL{\tDD}$ carries the structure of a symmetric monoidal closed category, and the associated reflection $\dK \dashv \dJ:\sL_{(\tDD)} \hookrightarrow \sL$ underlies a symmetric monoidal adjunction.
\end{enumerate}
\end{ThmSub}

\begin{RemSub}\label{rem:d_distnl_compl_for_conv}
For the example of $\X = \Conv$, where $\tDD$ coincides with $\DD$ \pbref{exa:d_td_for_conv_ident_as_abs_distn_mnds}, we may replace $\tDD$ by $\DD$ in the statement of Theorem \bref{thm:dist_compl_sep_vs_pett}.  Further, in this case all the theory developed in the present section applies also with $\DD$ in place of $\tDD$.
\end{RemSub}

\begin{RemSub}\label{rem:func_compl_linsp_tddistnly_compl_sep}
Every functionally complete linear space is $\tDD$-distributionally complete by \eqref{thm:func_compl_impl_td_dist_compl} and separated by \bref{rem:func_compl_obs_sep} and so, by \bref{thm:dist_compl_sep_vs_pett}, is $\tDD$-Pettis.
\end{RemSub}

\begin{RemSub}\label{rem:initial_rmks_for_pf_of_dist_compl_pett_thm}
In order to prove \bref{thm:dist_compl_sep_vs_pett}, it suffices to prove statements 1 and 2, since 3 follows from 2 by \bref{thm:enr_refl_smadj}.
\end{RemSub}

\begin{ParSub}
In the present section, we will denote the $\X$-monad $\tDD$ by $(\tD,\tdelta,\tkappa)$ and the associated abstract distribution monad by $(\tDD,\tDelta,\txi)$.  Recall from \bref{exa:dist_compl_for_nat_acc_distn_mnds} 2, that a linear space $E \in \sL$ is $\tDD$-distributionally complete iff each linear map
\begin{equation}\label{eqn:Delta_fx_for_td}\tDelta_X = \tfd_{FX}:FX \rightarrow \tH FX\;\;\;\;(X \in \X)\end{equation}
is $\sL$-orthogonal to $E$ in $\uL$.  Again abusing notation as in \bref{rem:dist_compl_expl}, we will denote the class of all the linear maps \eqref{eqn:Delta_fx_for_td} by $\tDelta$.
\end{ParSub}

The following shows that the theory of completion, closure, and density developed in \bref{sec:cmpl_cl_dens}, \bref{sec:sep_cpl_mnds_class_morphs} is applicable with regard to $\tDD$-distributional completeness.

\begin{PropSub}\emptybox
\begin{enumerate}
\item The $\sL$-category $\uL$, the $\sL$-monad $\HH$, and the class $\tDelta \subs \Mor\uL$ satisfy the assumptions of \bref{par:data_enr_orth_subcat_subord_adj}.  
\item The $\tDD$-distributionally complete separated linear spaces may be equivalently characterized as the $\tDelta$-complete $\HH$-separated objects of $\uL$ \pbref{def:compl_sep_closed_dense}, so that
$$\sL_{(\tDD)} = \uL_{(\HH,\tDelta)}$$
as $\sL$-categories.
\end{enumerate}
\end{PropSub}
\begin{proof}
Firstly, since $\sL$ is finitely well-complete, the cotensored $\sL$-category $\uL$ is $\sL$-finitely-well-complete by \bref{thm:base_fwc}.  Next, $\tH$ inverts each morphism in the class $\tDelta$, and by \bref{thm:addnl_facts_on_ind_idmmnd_refl}, $\tH$ inverts exactly the same morphisms as $H$.  Hence $\tDelta \subs \Sigma_H$.  Thus 1 is established, and 2 follows immediately from the definitions.
\end{proof}

\begin{CorSub}\label{thm:td_dist_compl_sep_obs_refl}
$\sL_{(\tDD)} = \uL_{(\HH,\tDelta)}$ is the $\sL$-reflective-subcategory of $\uL$ determined by the idempotent $\sL$-monad $\HH_{\tDelta}$ \pbref{def:tsep_sigma_compl_mnd}.
\end{CorSub}

\begin{DefSub}\label{def:distnl_compl_cl_dens}\emptybox
\begin{enumerate}
\item We call the idempotent $\sL$-monad $\dHH := \HH_{\tDelta}$ \pbref{thm:td_dist_compl_sep_obs_refl} the \textit{separated $\tDD$-distributional completion monad}, and we write $\dHH = (\dH,\dfd,\dfk)$.
\item We let $\dK \nsststile{}{\dfd} \dJ:\sL_{(\tDD)} \hookrightarrow \uL$ denote the associated $\sL$-reflection.
\item An embedding $m:M \rightarrow E$ in $\uL$ is \textit{$\tDD$-distributionally closed} if $m$ is $\tDelta$-closed \pbref{def:compl_sep_closed_dense}, i.e. if $m \in \tDelta^{\downarrow_\sL}$ \pbref{def:enr_orth_fact}.
\item A morphism $h:E_1 \rightarrow E_2$ in $\uL$ is \textit{$\tDD$-distributionally dense} if $h$ is $\tDelta$-dense \pbref{def:compl_sep_closed_dense}, i.e. if $f \in \ClEmb{\tDelta}^{\uparrow_\sL}$.
\end{enumerate}
\end{DefSub}

\begin{PropSub}\label{thm:func_cl_dens_via_dist_cl_dens}\emptybox
\begin{enumerate}
\item Every functionally closed embedding is $\tDD$-distributionally closed.
\item Every $\tDD$-distributionally dense morphism is functionally dense.
\end{enumerate}
\end{PropSub}
\begin{proof}
Recall that an embedding $m$ in $\uL$ is functionally closed iff $m \in \Sigma_H^{\downarrow_\sL}$, and since $\tDelta \subs \Sigma_H$ it follows that $\Sigma_H^{\downarrow_\sL} \subs \tDelta^{\downarrow_\sL}$.  Therefore $\ClEmb{\Sigma_H} \subs \ClEmb{\tDelta}$ and 1 is established.  We therefore have also that $\Dense{\tDelta} = \ClEmb{\tDelta}^{\uparrow_\sL} \subs \ClEmb{\Sigma_H}^{\uparrow_\sL} = \Dense{\Sigma_H}$, so 2 is obtained.
\end{proof}

\begin{PropSub}\label{thm:mnd_morphs_th_dh_h}
There are unique morphisms of $\sL$-monads $k$, $\wideparen{i}$ as in the diagram
\begin{equation}\label{eqn:mnd_morphs_th_dh_h}
\xymatrix {
                                         & \tHH \ar[dr]^i &          \\
\dHH \ar[ur]^{k} \ar[rr]_{\wideparen{i}} &                      & {\HH\;,}
}
\end{equation}
where $i$ is as in \bref{def:func_compl}.  Moreover, the diagram commutes, and the components of $\wideparen{i}$ are $\tDD$-distributionally closed embeddings. 
\end{PropSub}
\begin{proof}
Since $\tHH = \HH_{\Sigma_H}$, $\dHH = \HH_{\tDelta}$, and $\tDelta \subs \Sigma_H$, this follows from \bref{thm:mor_tsep_sigma_compl_mnds} and \bref{thm:uniq_morph_mnds_tsep_sigmacompl_mnd_to_t}.
\end{proof}

\begin{DefSub}
Let $\dDD := [F,G](\dHH)$ denote the $\X$-monad on $\uX$ obtained by applying the monoidal functor $[F,G]:\XCAT(\sL,\sL) \rightarrow \XCAT(\uX,\uX)$ \pbref{thm:mon_func_detd_by_adj} to the $\X$-monad $\dHH$.
\end{DefSub}

\begin{PropSub}\label{thm:dd_iso_td}\emptybox
\begin{enumerate}
\item The morphism of $\X$-monads
$$\dDD = [F,G](\dHH) \xrightarrow{[F,G](k) \:=\: GkF} [F,G](\tHH) = \tDD$$
is an isomorphism.
\item For each $X \in \X$, the morphism $k_{FX}:\dH FX \rightarrow \tH FX$ is an isomorphism in $\sL$.   Hence the linear space of accessible distributions $\tH FX$ on $X$ is isomorphic to the separated $\tDD$-distributional completion of the free span $FX$ of $X$.
\end{enumerate}
\end{PropSub}
\begin{proof}
In order to prove 1, it suffices to show that for each $X \in \X$, the morphism
$$\dD X = G\dH FX \xrightarrow{GkFX} G\tH FX = \tD X$$
is iso in $\X$, and for this it suffices to prove 2.  Since \eqref{eqn:mnd_morphs_th_dh_h} commutes and $k$ is a monad morphism, the following diagram commutes.
$$
\xymatrix{
FX \ar[d]_{\dfd_{FX}} \ar[r]^{\tfd_{FX}} & \tH FX \ar@{ >->}[d]^{i_{FX}} \\
\dH FX \ar[ur]|{k_{FX}} \ar@{ >->}[r]_{\wideparen{i}_{FX}} & HFX
}
$$
Recall that $i_{FX}$ is a functionally closed embedding \pbref{thm:mnd_morph_i} and hence by \bref{thm:func_cl_dens_via_dist_cl_dens} is a $\tDD$-distributionally closed embedding.  But $\wideparen{i}_{FX}$ is a $\tDD$-distributionally closed embedding \pbref{thm:mnd_morphs_th_dh_h}, so we deduce by \bref{thm:props_sigma_cl_dens} that $k_{FX}$ is a $\tDD$-distributionally closed embedding.  Further, since $\tfd_{FX} \in \tDelta$, $\tfd_{FX}$ is $\tDD$-distributionally dense by \bref{thm:props_sigma_cl_dens}.  Hence by \bref{thm:props_sigma_cl_dens}, $k_{FX}$ is $\tDD$-distributionally dense, so since $k_{FX}$ is also a $\tDD$-distributionally closed embedding, $k_{FX}$ is iso by \bref{thm:props_sigma_cl_dens}.
\end{proof}

\begin{CorSub}\label{thm:mnd_mor_td_to_dh}
There is a unique morphism of $\X$-monads $(\overline{\xi},G):\tDD \rightarrow \dHH$ such that the diagram
$$
\xymatrix{
                                                           & \dHH \ar[d]^{(k,1_\sL)} \\
\tDD \ar[ur]^{(\overline{\xi},G)} \ar[r]_{(\xi^{\LL\tHH},G)} & \tHH
}
$$
commutes
\end{CorSub}
\begin{proof}
The diagram
$$
\xymatrix{
*!/r4ex/+{\dDD = [F,G](\dHH)} \ar[d]_{([F,G](k),1_{\uX})} \ar[rr]^{(\xi^{\LL\dHH},G)} & & \dHH \ar[d]^{(k,1_\sL)} \\
*!/r4ex/+{\tDD = [F,G](\tHH)} \ar[rr]_{(\xi^{\LL\tHH},G)}                             & & \tHH
}
$$
commutes (where $\xi^{\LL\dHH}, \xi^{\LL\tHH}$ are as defined in \bref{def:xi_ts}) since the underlying $\X$-natural transformations of the morphisms therein yield a commutative diagram
$$
\xymatrix{
*!/r4ex/+{\dD G = G\dH FG} \ar[d]_{G k FG} \ar[r]^{\xi^{\LL\dHH} \:=\: G\dH \varepsilon} & {G\dH\;\:} \ar[d]^{Gk} \\
*!/r4ex/+{\tD G = G\tH FG} \ar[r]_{\xi^{\LL\tHH} \:=\: G\tH \varepsilon}                 & {G\tH\;.}
}
$$
Hence since $[F,G](k)$ is an isomorphism (by \bref{thm:dd_iso_td}), the result follows.
\end{proof}

\begin{ParSub}
Since $\tL$ and $\sL_{(\tDD)}$ are the $\X$-reflective-subcategories of $\sL$ determined by the idempotent $\X$-monads $\tHH$ and $\dHH$, respectively, we may identify $\tL$ and $\sL_{(\tDD)}$ with the Eilenberg-Moore $\X$-categories $\sL^{\tHH}$ and $\sL^{\dHH}$, respectively.  Under this identification, the Eilenberg-Moore forgetful $\X$-functors are identified with the inclusions into $\sL$, and hence the $\X$-functor $\sL^k:\sL^{\tHH} \rightarrow \sL^{\dHH}$ (which commutes with the forgetful $\X$-functors, \bref{prop:enr_func_induced_by_mnd_mor}) is identified with the inclusion $\tL \hookrightarrow \sL_{(\tDD)}$ \eqref{rem:func_compl_linsp_tddistnly_compl_sep}.
\end{ParSub}

\begin{CorSub}\label{thm:xfuncs_for_dist_compl_sep}
We obtain a commutative diagram of $\X$-functors
$$
\xymatrix{
*!/r4ex/+{\sL_{(\tDD)} = \sL^{\dHH}} \ar[dr]^{G^{\overline{\xi}}} & \\
*!/r4ex/+{\tL = \sL^{\tHH}} \ar@{^{(}->}[u]^{\sL^k} \ar[r]_{G^{\xi^{\LL\tHH}}} & {\uX^{\tDD}\;.}
}
$$
\end{CorSub}

\begin{LemSub}\label{thm:td_alg_str_on_distcpl_sep_ls_is_pett_str}
Let $E$ be a $\tDD$-distributionally complete separated linear space.  Then the $\tDD$-algebra structure $\tD GE \rightarrow GE$ obtained by applying $G^{\overline{\xi}}$ to $E$ is a $\tDD$-Pettis structure for $E$.
\end{LemSub}
\begin{proof}
The $\dHH$-algebra structure carried by $E$ is the inverse $\dfd_E^{-1}$ of the unit component \hbox{$\dfd_E:E \rightarrow \dH E$}.  The $\tDD$-algebra structure obtained by applying $G^{\overline{\xi}}$ to this $\dHH$-algebra is the composite 
$$\tD GE \xrightarrow{\overline{\xi}_E} G\dH E \xrightarrow{G\dfd_E^{-1}} GE\;,$$
which is a $\tDD$-Pettis structure for $E$ since the diagram
$$
\xymatrix@!0@R=8ex @C=10ex{
\tD GE \ar[rrrr]^{\txi_E} \ar[drr]^{\xi^{\LL\tHH}_E} \ar[dr]_{\overline{\xi}_E} & & & & GHE \\
& G\dH E \ar[dr]_{G\dfd_E^{-1}} \ar[r]_{Gk_E} & G \tH E \ar[urr]^{Gi_E} & & \\
& & GE \ar[u]|{G\tfd_E} \ar[uurr]_{G\fd_E} & & \\
}
$$
commutes, using \bref{thm:mnd_mor_td_to_dh}, the definition of $\txi$, and the fact that both $i$ and $k$ are monad morphisms.
\end{proof}

\begin{CorSub}\label{thm:func_ind_by_xibar_is_restn_of_a_and_b}
There is an inclusion $\sL_{(\tDD)} \hookrightarrow \pL{\tDD}$, and the following diagram of $\X$-functors commutes, where $A$ and $B$ are as in \bref{thm:diag_funcs_pett} and \bref{thm:lin_map_betw_mdist_compl_sp_is_m_homom}, respectively.
$$
\xymatrix{
\sL_{(\tDD)} \ar@{^{(}->}[d] \ar@{^{(}->}[r] \ar[dr]|{G^{\overline{\xi}}} & \pL{\tDD} \ar@{ >->}[d]^A \\
\sL_{\tDelta} \ar@{ >->}[r]_B & \uX^{\tDD}
}
$$
\end{CorSub}
\begin{proof}
By \bref{thm:td_alg_str_on_distcpl_sep_ls_is_pett_str}, the inclusion exists and the upper triangle commutes on objects.  All the $\X$-functors in the diagram commute with the $\X$-faithful forgetful $\X$-functors to $\uX$, so the upper triangle commutes and it suffices to show that the periphery commutes on objects.  Given an object $E \in \sL_{(\tDD)}$, since $E$ is $\tDD$-Pettis we have an associated $\tDD$-algebra $AE = (GE,a_E)$, and we know by \bref{thm:diag_funcs_pett} that the underlying linear space $\uX^{\tDelta}(GE,a_E)$ is simply $E$ itself.  But the $\tDD$-algebra structure $a_E:\tD GE \rightarrow GE$ is a $\tDD$-homomorphism \hbox{$a_E:(\tD GE,\tkappa_{GE}) \rightarrow (GE,a_E)$} and therefore, via $\uX^{\tDelta}$, is a linear map between the underlying linear spaces.  Hence since $a_E \cdot \tdelta_{GE} = 1_{GE}$ in $\X$, it is immediate from \bref{thm:m_distnly_cpl_sp_are_m_algs} that $a_E$ must coincide with the $\tDD$-algebra structure $b_E$ associated to $E$ via $B$.
\end{proof}

\begin{CorSub}\label{thm:th_fx_pettis_with_free_td_alg_str}
The linear space of accessible distributions $\tH FX$ on a space $X \in \X$ is $\tDD$-Pettis, with $\tDD$-Pettis structure $\tkappa_X:\tD\tD X \rightarrow \tD X$.
\end{CorSub}
\begin{proof}
By \bref{thm:td_alg_str_on_distcpl_sep_ls_is_pett_str} and \bref{thm:xfuncs_for_dist_compl_sep}, a $\tDD$-Pettis structure for $\tH FX$ may be obtained as the $\tDD$-algebra structure obtained by applying $G^{\xi^{\LL\tHH}}:\tL = \sL^{\tHH} \rightarrow \uX^{\tDD}$ to $\tH FX$, and by \bref{thm:free_gsf_alg_via_xi}, this is simply the free $\tDD$-algebra structure carried by $\tD X$.
\end{proof}

\begin{LemSub}
Let $E$ be a $\tDD$-Pettis linear space.  Then $E$ is $\tDD$-distributionally complete and separated.
\end{LemSub}
\begin{proof}
We already know that $E$ is separated.  Letting $X \in \X$, it suffices to show that
$$\sL(\tfd_{FX},E):\sL(\tH FX,E) \rightarrow \sL(FX,E)$$
is an isomorphism in $\X$.  By \bref{thm:th_fx_pettis_with_free_td_alg_str}, $\tH FX$ is $\tDD$-Pettis and is sent by $A:\pL{\tDD} \rightarrowtail \uX^{\tDD}$ to the free $\tDD$-algebra $(\tD X,\tkappa_X)$.  The key idea now is that we obtain isomorphisms
$$\pL{\tDD}(\tH FX,E) \cong \uX^{\tDD}(A\tH FX,AE) = \uX^{\tDD}((\tD X,\tkappa_X),(GE,a_E)) \cong \uX(X,GE) \cong \sL(FX,E)$$
in $\X$, via the $\X$-adjunctions $F^{\tDD} \dashv G^{\tDD}$ and $F \dashv G$ and the fact that $A$ is $\X$-fully-faithful.  Moreover, we have a diagram
$$
\xymatrix{
*!/r5ex/+{\pL{\tDD}(\tH FX,E) = \sL(\tH FX,E)} \ar[d]_A^\wr \ar[dr]|G \ar[rrr]^{\sL(\tfd_{FX},E)} & & & \sL(FX,E) \ar[d]^G \\
*!/r3ex/+{\X^{\tDD}((\tD X,\tkappa_X),(GE,c_E))} \ar[d]^\wr \ar[r]|{G^{\tDD}} & \uX(\tD X,GE) \ar[drr]|{\uX(\tdelta_X,GE)} \ar[rr]^{\uX(G\tfd FX,GE)} & & \uX(GFX,GE) \ar[d]^{\uX(\eta_X,GE)} \\
\uX(X,GE) \ar@{=}[rrr] & & & \uX(X,GE)
}
$$
which commutes, using the $\X$-functoriality of $G$, the definition of $\tdelta$, and the fact that $A$ commutes with the forgetful $\X$-functors to $\uX$ \pbref{thm:diag_funcs_pett}.  Hence since the composites on the left and right sides are isomorphisms (the latter by the $\X$-adjointness $F \dashv G$), the top side $\sL(\tfd_{FX},E)$ is also an isomorphism.
\end{proof}

\begin{ParSub}
Hence, in view of \bref{thm:td_alg_str_on_distcpl_sep_ls_is_pett_str}, the $\tDD$-distributionally complete separated linear spaces are exactly the $\tDD$-Pettis linear spaces, so by \bref{rem:initial_rmks_for_pf_of_dist_compl_pett_thm} and \bref{thm:td_dist_compl_sep_obs_refl}, Theorem \bref{thm:dist_compl_sep_vs_pett} is proved.
\end{ParSub}

\begin{PropSub}\label{thm:pett_td_algs_refl_in_dalg}
The $\X$-fully-faithful, injective-on-objects $\X$-functor
\begin{equation}\label{eqn:emb_a_for_td}A:\sL_{(\tDD)} = \pL{\tDD} \rightarrowtail \uX^{\tDD}\end{equation}
has a left $\X$-adjoint.  Hence its full image, the full sub-$\X$-category $\uX^{\tDD}_{\textnormal{P}}$ consisting of the Pettis $\tDD$-algebras \pbref{thm:pett_obs_iso_to_subcat_malgs}, is $\X$-reflective in $\uX^{\tDD}$.
\end{PropSub}
\begin{proof}
The composite $\uX^{\tDD} \xrightarrow{\uX^{\tDelta}} \uX^{\LL} = \sL \xrightarrow{\dK} \sL_{(\tDD)}$ serves as left $\X$-adjoint, as follows.  By \bref{thm:func_ind_by_xibar_is_restn_of_a_and_b} and \bref{thm:dist_compl_sep_vs_pett}, $A$ is a restriction of $B:\sL_{\tDelta} \rightarrowtail \uX^{\tDD}$, so by \bref{thm:lin_map_betw_mdist_compl_sp_is_m_homom} we have
$$\sL_{(\tDD)}(\dK\uX^{\tDelta}Z,E) \cong \sL(\uX^{\tDelta}Z,E) \cong \uX^{\tDD}(Z,BE) = \uX^{\tDD}(Z,AE)$$
$\X$-naturally in $Z \in \uX^{\tDD}$, $E \in \sL_{(\tDD)}$.
\end{proof}

Whereas $\tDD$-Pettis linear spaces `are' $\tDD$-algebras, via the embedding $A$ \eqref{eqn:emb_a_for_td}, the following provides a partial converse.

\begin{PropSub}\label{thm:td_pett_iff_sep}
A $\tDD$-algebra $(X,a)$ is Pettis if and only if its underlying linear space $E$ is separated.  In particular, the underlying linear space of a $\tDD$-algebra is $\tDD$-Pettis as soon as it is separated.
\end{PropSub}
\begin{proof}
By \bref{thm:basic_charn_pett_malg}, $(X,a)$ is Pettis if and only if $a$ is a $\tDD$-Pettis structure for $E$.  Hence one implication is immediate.  For the other, suppose that $E$ is separated.  We must show that the periphery of the following diagram commutes
$$
\xymatrix{
\tD GE \ar[dr]|{\xi^{\LL\tHH}_E} \ar@/_2ex/[ddr]_a \ar[rr]^{\txi_E} & & GHE \\
 & G\tH E \ar[ur]|{Gi_E} & \\
 & GE \ar[u]|{G\tfd_E} \ar@/_2ex/[uur]_{G\fd_E} &
}
$$
but the upper triangle commutes by the definition of $\txi$ and the rightmost triangle commutes since $i$ is a monad morphism, so it suffices to show that the leftmost triangle commutes.  Now $a$ is a $\tDD$-homomorphism $a:(\tD X,\kappa_X) \rightarrow (X,a)$ and so is a linear map between the underlying linear spaces, which by \bref{exa:und_linear_spaces_of_free_d_and_td_algs} are $\tH FX = \tH FGE$ and $E$, respectively.  Hence $a = Gc$ for a unique $c:\tH FGE \rightarrow E$ in $\sL$.  Since $(GE,a)$ is a $\tDD$-algebra, we have that $a \cdot \tdelta_{GE} = 1_{GE}$, and by the definition of $\tdelta$ this means that the composite
$$GE \xrightarrow{\eta_{GE}} GFGE \xrightarrow{G\tfd FGE} G\tH FGE \xrightarrow{a\:=\:Gc} GE$$
is the identity morphism.  It follows that the diagram
\begin{equation}\label{eqn:unit_eqn_for_linear_unit}
\xymatrix{
FGE \ar[rr]^{\varepsilon_E} \ar[dr]_{\tfd FGE} &                   & E \\
                                               & \tH FGE \ar[ur]_c &
}
\end{equation}
in $\sL$ commutes.  We have a diagram
$$
\xymatrix {
*!/r4ex/+{\tD GE = G\tH FGE} \ar[d]_{G\tfd \tH FGE \:=\: G\tH\tfd FGE} \ar[r]^{a \:=\: Gc} & GE \ar[d]^{G\tfd_E} \\
G\tH\tH FGE \ar[r]_{G\tH c}                       & G\tH E
}
$$
which commutes by the naturality of $G\tfd$, noting that $\tfd \tH = \tH\tfd$ since the $\X$-monad $\tHH$ is idempotent \pbref{par:idm_mnd}.  But by the commutativity of \eqref{eqn:unit_eqn_for_linear_unit}, the lower-left composite is $G\tH \varepsilon_E = \xi^{\LL\tHH}_E$.
\end{proof}

\begin{RemSub}\label{rem:dalg_pett_if_sep_for_conv}
For the example of convergence spaces \pbref{exa:smc_adj_with_l_fwc}, where $\tDD$ coincides with $\DD$ \pbref{exa:d_td_for_conv_ident_as_abs_distn_mnds}, we deduce by \bref{thm:td_pett_iff_sep} that a $\DD$-algebra is Pettis if and only if its underlying linear space is separated.
\end{RemSub}